\documentclass[11pt,oneside,english]{amsart}
\usepackage[latin9]{inputenc}
\usepackage[a4paper]{geometry}
\usepackage[active]{srcltx}
\usepackage{color}
\usepackage{babel}
\usepackage{amstext}
\usepackage{amsthm}
\usepackage{amssymb}
\usepackage{stackrel}
\usepackage{setspace}
\usepackage[unicode=true,pdfusetitle,
 bookmarks=true,bookmarksnumbered=false,bookmarksopen=false,
 breaklinks=false,pdfborder={0 0 1},backref=false,colorlinks=false]
 {hyperref}

\makeatletter
\numberwithin{equation}{section}
\numberwithin{figure}{section}
\theoremstyle{plain}
\newtheorem{thm}{\protect\theoremname}[section]
\theoremstyle{plain}
\newtheorem{question}[thm]{\protect\questionname}
\theoremstyle{definition}
\newtheorem{defn}[thm]{\protect\definitionname}
\theoremstyle{remark}
\newtheorem{rem}[thm]{\protect\remarkname}
\theoremstyle{plain}
\newtheorem{cor}[thm]{\protect\corollaryname}
\theoremstyle{definition}
\newtheorem{example}[thm]{\protect\examplename}
\theoremstyle{remark}
\newtheorem*{acknowledgement*}{\protect\acknowledgementname}
\theoremstyle{plain}
\newtheorem{lem}[thm]{\protect\lemmaname}
\theoremstyle{plain}
\newtheorem{fact}[thm]{\protect\factname}
\theoremstyle{plain}
\newtheorem{prop}[thm]{\protect\propositionname}
\theoremstyle{remark}
\newtheorem{notation}[thm]{\protect\notationname}
\theoremstyle{plain}
\newtheorem{conjecture}[thm]{\protect\conjecturename}

\newtheorem{theorem}{Theorem}

\AtEndDocument{\bigskip{\footnotesize%
  \textsc{Itay Glazer, Faculty of Mathematics and Computer Science, The  Weizmann Institute of
Science,  Rehovot 76100, ISRAEL.
} \\
  \textit{E-mail address}: \texttt{itayglazer@gmail.com} \\
    \textit{URL}: \texttt{https://sites.google.com/view/itay-glazer} \par
  \addvspace{\medskipamount}
}}
\AtEndDocument{\bigskip{\footnotesize
  \textsc{Yotam I. Hendel, Department of Mathematics, Northwestern University, 2033 Sheridan Road, Evanston, IL 60208, USA.
} \\
  \textit{E-mail address}: \texttt{yotam.hendel@gmail.com} \\
    \textit{URL}: \texttt{https://sites.google.com/view/yotam-hendel} 
}}

\makeatother

\providecommand{\acknowledgementname}{Acknowledgement}
\providecommand{\conjecturename}{Conjecture}
\providecommand{\corollaryname}{Corollary}
\providecommand{\definitionname}{Definition}
\providecommand{\examplename}{Example}
\providecommand{\factname}{Fact}
\providecommand{\lemmaname}{Lemma}
\providecommand{\notationname}{Notation}
\providecommand{\propositionname}{Proposition}
\providecommand{\questionname}{Question}
\providecommand{\remarkname}{Remark}
\providecommand{\theoremname}{Theorem}

\begin{document}
\title[On singularity properties of word maps and applications]{On singularity properties of word maps and applications to probabilistic
Waring type problems}
\author{Itay Glazer and Yotam I. Hendel}
\begin{abstract}
We study singularity properties of word maps on semisimple algebraic
groups and Lie algebras, generalizing the work of Aizenbud-Avni in
the case of the commutator map.

Given a word $w$ in a free Lie algebra $\mathcal{L}_{r}$, it induces
a word map $\varphi_{w}:\mathfrak{g}^{r}\rightarrow\mathfrak{g}$
for every semisimple Lie algebra $\mathfrak{g}$. Given two words
$w_{1}\in\mathcal{L}_{r_{1}}$ and $w_{2}\in\mathcal{L}_{r_{2}}$,
we define and study the convolution of the corresponding word maps
$\varphi_{w_{1}}*\varphi_{w_{2}}:=\varphi_{w_{1}}+\varphi_{w_{2}}:\mathfrak{g}^{r_{1}+r_{2}}\rightarrow\mathfrak{g}$.

We show that for any word $w\in\mathcal{L}_{r}$ of degree $d$, and
any simple Lie algebra $\mathfrak{g}$ with $\varphi_{w}(\mathfrak{g}^{r})\neq0$,
one obtains a flat morphism with reduced fibers of rational singularities
(abbreviated an (FRS) morphism) after taking $O(d^{4})$ self-convolutions
of $\varphi_{w}$. We deduce that a group word map of length $\ell$
becomes (FRS) at $(e,\ldots,e)\in\underline{G}^{r}$ after $O(\ell^{4})$
self-convolutions, for every semisimple algebraic group $\underline{G}$.

We furthermore bound the dimensions of the jet schemes of the fibers
of Lie algebra word maps, and the fibers of group word maps in the
case where $\underline{G}=\mathrm{SL}_{n}$. For the commutator word
$w_{0}=[X,Y]$, we show that $\varphi_{w_{0}}^{*4}$ is (FRS) for
any semisimple Lie algebra, obtaining applications in representation
growth of compact $p$-adic and arithmetic groups.

The singularity properties we consider, such as the (FRS) property,
are intimately connected to the point count of fibers over finite
rings of the form $\mathbb{Z}/p^{k}\mathbb{Z}$. This allows us to
relate them to properties of some natural families of random walks
on finite and compact $p$-adic groups. We explore these connections,
characterizing some of the singularity properties discussed in probabilistic
terms, and provide applications to $p$-adic probabilistic Waring
type problems.
\end{abstract}

\maketitle
\begin{scriptsize}

\tableofcontents{}

\end{scriptsize}

\global\long\def\nats{\mathbb{\mathbb{N}}}%
\global\long\def\reals{\mathbb{\mathbb{R}}}%
\global\long\def\ints{\mathbb{\mathbb{Z}}}%
\global\long\def\val{\mathbb{\mathrm{val}}}%
\global\long\def\Qp{\mathbb{Q}_{p}}%
\global\long\def\Zp{\mathbb{Z}_{p}}%
\global\long\def\ac{\mathrm{ac}}%
\global\long\def\complex{\mathbb{C}}%
\global\long\def\rats{\mathbb{Q}}%
\global\long\def\supp{\mathrm{supp}}%
\global\long\def\VF{\mathrm{VF}}%
\global\long\def\RF{\mathrm{RF}}%
\global\long\def\VG{\mathrm{VG}}%
\global\long\def\spec{\mathrm{Spec}}%
\global\long\def\Ldp{\mathcal{L}_{\mathrm{DP}}}%

\raggedbottom

\section{Introduction}

\subsection{Motivation}

\subsubsection{\label{subsec:Waring-type-problems}Waring type problems}

Our motivation goes back to the 18th century, with Lagrange's famous
four square theorem - every natural number can be represented as the
sum of at most four square integers (Lagrange, 1770). Waring then
considered the following generalization which was confirmed by Hilbert
in 1909 \cite{Hil909}: 
\begin{question}[Waring's problem, 1770]
Can every natural number be represented as the sum of $t(k)$ many
$k$-th powers, for some function $t(k)$ of $k$?
\end{question}

We now formulate Waring's problem in a language suitable for our purposes. 
\begin{defn}
\label{def:convolutions for poor}Let $\varphi:X\rightarrow G$ and
$\psi:Y\rightarrow G$ be maps from sets $X$ and $Y$ to a (semi-)group
$G$. Define the \textit{convolution} $\varphi*\psi:X\times Y\rightarrow G$
of $\varphi$ and $\psi$ by 
\[
\varphi*\psi(x,y)=\varphi(x)\cdot\psi(y).
\]
Furthermore, we denote by $\varphi^{*t}:X^{t}\rightarrow G$ the $t$-th
convolution power (or self-convolution) of $\varphi$. 
\end{defn}

In this language, the Waring problem asks whether the power map $\varphi_{k}:\nats\rightarrow\nats$
defined by $\varphi_{k}(x)=x^{k}$ becomes surjective after $t(k)$
self-convolutions. Here is an interesting variant.

Given a word $w=x_{i_{1}}^{n_{1}}\cdot\dots\cdot x_{i_{l}}^{n_{l}}\in F_{r}$,
with $x_{i_{j}}\in\{x_{1},\dots,x_{r}\}$ and $n_{j}\in\mathbb{Z}$,
and a group $G$, we associate the corresponding \textit{word map}
$\varphi_{w}:G^{r}\rightarrow G$ by $(g_{1},\dots,g_{r})\mapsto g_{i_{1}}^{n_{1}}\cdot\dots\cdot g_{i_{l}}^{n_{l}}$.
Note that $\varphi_{w_{1}}*\varphi_{w_{2}}=\varphi_{w_{1}*w_{2}}$,
where $w_{1}*w_{2}\in F_{r_{1}+r_{2}}$ is the concatenation of $w_{1}\in F_{r_{1}}$
and $w_{2}\in F_{r_{2}}$. We can then ask if for any $w\in F_{r}$,
there exists $t(w)$ such that $\varphi_{w}^{*t(w)}:G^{rt(w)}\rightarrow G$
is surjective for every $G$ in a nice family $\mathcal{G}$ of groups.
The minimal such $t(w)$ is called the \textit{$\mathcal{G}$-width}
of $w$. Similarly, one can take a word $w$ in a free Lie algebra
$\mathcal{L}_{r}$ on a finite set $\{X_{1},\dots,X_{r}\}$ and consider,
for any semisimple Lie algebra $\mathfrak{g}$, the corresponding
Lie algebra word maps $\varphi_{w}:\mathfrak{g}^{r}\rightarrow\mathfrak{g}$.

This kind of questions are referred to as Waring type problems, and
they have been extensively studied over the past few decades, in various
settings. We mention some of them: 
\begin{enumerate}
\item \textbf{$\mathcal{G}_{\text{Alg}}$, the family of simple algebraic
groups}. It is a classical theorem of Borel \cite{Bor83} (see also
\cite{Lar04}) that for any non-trivial word $w\in F_{r}$, and any
semisimple algebraic group $\underline{G}$, the map $\varphi_{w}:\underline{G}^{r}\rightarrow\underline{G}$
is dominant. In particular $\varphi_{w}^{*2}$ is surjective. 
\item \textbf{$\mathcal{G}_{\text{LieA}}$, the family of simple Lie algebras.}
In \cite{BGKP12}, an analogue of Borel's theorem was shown for semisimple
Lie algebras $\mathfrak{g}$ under the additional assumption that
$\varphi_{w}:\mathfrak{g}^{r}\rightarrow\mathfrak{g}$ is not identically
zero on $\mathfrak{sl}_{2}$. 
\item \textbf{$\mathcal{G}_{\text{fin}}$, the family of finite simple groups.}
The problem of finding the width of various words over (large enough)
finite simple groups has been studied extensively (e.g.~\cite{MZ96,SW97,LiS01,LaS08,Sha09,LaS09,LOST10}).
The state of the art is \cite{LST11}, where Larsen, Shalev and Tiep
showed that for any $w\in F_{r}$, the word map $\varphi_{w}^{*2}:G^{2r}\rightarrow G$
is surjective for every large enough $G\in\mathcal{G}_{\text{fin}}$. 
\item \textbf{$\mathcal{G}_{\text{cp}}$, the family of compact $p$-adic
groups.} In \cite{Jai08} it was shown that any word $w\in F_{r}$
has finite width with respect to any compact $p$-adic group. In \cite{AGKS13},
the authors showed that for any $w\in F_{r}$ and any simply connected
semisimple $\rats$-group $\underline{G}$, $\varphi_{w}^{*3}:\underline{G}(\Zp)^{3r}\rightarrow\underline{G}(\Zp)$
is surjective for every large enough prime $p$. 
\item \textbf{Other families. } Given a ring $R$, one can consider the
corresponding Waring problem for $\underline{G}(R)$ where $\underline{G}\in\mathcal{G}_{\text{Alg}}$.
For example, the cases of $R=\reals,\Qp$ were treated in \cite{HLS15},
and the case $R=\ints$ and $\underline{G}=\mathrm{SL}_{n}$ was studied
in \cite{AM19}, and it was shown that $\varphi_{w}^{*87}$ is surjective
for any $w\in F_{r}$, if $n$ is large enough. 
\end{enumerate}

\subsubsection{\label{subsec:Probabilistic-Waring-type}Probabilistic Waring type
problems}

Given that a word map $\varphi_{w}$ becomes surjective after enough
self-convolutions, the next step is to study its fibers $\varphi_{w}^{-1}(g)$,
and to ask how many self-convolutions does it take so that the fibers
of $\varphi_{w}^{*t}$ will be uniform in some sense. Such questions
are referred to as probabilistic Waring type problems, and have recently
received a significant amount of attention (e.g.~\cite{GS09,SS11,LaS12,AA16,Bors17,LST19}).
In particular, one can take each of the Waring type problems described
above and produce its probabilistic counterpart.

For $\mathcal{G}_{\text{fin}}$ and $\mathcal{G}_{\text{cp}}$, problems
as above can be phrased in terms of questions on certain random walks.
For simplicity of presentation we work here only with $\mathrm{SL}_{n}$.
Any word $w\in F_{r}$ induces a family $\{\varphi_{w,n,p,k}\}_{n,p,k}$
of word maps
\[
\varphi_{w,n,p,k}:\mathrm{SL}_{n}^{r}(\ints/p^{k}\ints)\rightarrow\mathrm{SL}_{n}(\ints/p^{k}\ints).
\]
One can then study the family of random walks induced by $\varphi_{w,n,p,k}$.
Write $\pi_{n,p,k}$ for the uniform probability measure on $\mathrm{SL}_{n}(\ints/p^{k}\ints)$,
and set $\tau_{w,n,p,k}:=(\varphi_{w,n,p,k})_{*}(\pi_{n,p,k})$. Then
the collection $\{\tau_{w,n,p,k}\}_{n,p,k}$ induces a family of random
walks on $\{\mathrm{SL}_{n}(\ints/p^{k}\ints)\}_{n,p,k}$, where the
probability to reach $g\in\mathrm{SL}_{n}(\ints/p^{k}\ints)$ at the
$t$-th step of the random walk is
\[
\tau_{w,n,p,k}^{*t}(g)=\frac{\left|(\varphi_{w,n,p,k}^{*t})^{-1}(g)\right|}{\left|\mathrm{SL}_{n}(\ints/p^{k}\ints)\right|^{rt}}.
\]
Since $\varphi_{w,n,p,k}$ is generating for $p$ large enough (e.g.~\cite[Theorem 2.3]{AGKS13}),
$\tau_{w,n,p,k}^{*t}$ converges to the uniform probability $\pi_{n,p,k}$
in any of the $L^{a}$-norms for $a\in[1,\infty]$, as $t$ goes to
$\infty$. Roughly speaking, the minimal number of steps $t_{a}$
it takes for $\tau_{w,n,p,k}^{*t_{a}}$ to be sufficiently close to
$\pi_{n,p,k}$ in the $L^{a}$-norm is called the $L^{a}$-mixing
time (see Definition \ref{def:mixing time}). A priori, $t_{a}$ might
depend on $p,k$ and $n$. It is natural to ask the following:
\begin{question}
\label{que:Mein probabilistic questions}Let $w\in F_{r}$ be a word.
\begin{enumerate}
\item Is there $\varepsilon(w)>0$, such that $\tau_{w,n,p,k}(g)<\left|\mathrm{SL}_{n}(\ints/p^{k}\ints)\right|^{-\varepsilon(w)}$
uniformly in $n,p$ and $k$, when $p$ is large enough?
\item Is there an upper bound on the $L^{a}$-mixing times of $\{\tau_{w,n,p,k}\}$
which is uniform in $n,p$ and $k$, when $p$ is large enough? 
\end{enumerate}
\end{question}

\subsubsection{\label{subsec:Algebraic-Waring-type}Algebraic Waring type problems}

Let $\varphi:X\rightarrow\underline{G}$ be a morphism from a smooth
variety $X$ to an algebraic group $\underline{G}$. Similarly to
the above problems, we can study the convolution powers $\varphi^{*t}:X^{t}\rightarrow\underline{G}$
of $\varphi$ using Definition \ref{def:convolutions for poor}. Similarly
to the usual convolution operation in analysis, the algebraic convolution
operation has a smoothing effect on morphisms. In \cite{GH19,GH},
it was shown that if $\varphi:X\rightarrow\underline{G}$ is (strongly)
dominant, then $\varphi^{*t}:X^{t}\rightarrow\underline{G}$ has increasingly
better singularity properties as $t$ tends to $\infty$ (see Section
\ref{sec:Singulrities-of-word} for further discussion).

Now, given $w\in F_{r}$ (resp.~$w\in\mathcal{L}_{r}$), we can study
self-convolutions $\varphi_{w}^{*t}$ of the corresponding word maps
$\varphi_{w}:\underline{G}^{r}\rightarrow\underline{G}$ (resp.~$\varphi_{w}:\mathfrak{g}^{r}\rightarrow\mathfrak{g}$)
for $\underline{G}\in\mathcal{G}_{\mathrm{Alg}}$ and $\mathfrak{g}\in\mathcal{G}_{\mathrm{LieA}}$.
The question of surjectivity of self-convolutions was fully answered
by Borel's Theorem for $\underline{G}\in\mathcal{G}_{\mathrm{Alg}}$
(see \cite{Bor83}), and partially answered in \cite{BGKP12} for
$\mathfrak{g}\in\mathcal{G}_{\mathrm{LieA}}$. Having the results
of \cite{GH19,GH} in mind, and since there is a strong connection
between the singularities of a variety and its point count over finite
rings (see Theorem \ref{thm: rational singularities and point count}),
we are led to the following questions, which are analogous to Question
\ref{que:Mein probabilistic questions} (brought here in non-precise
form):
\begin{question}
\label{que:questions for algebraic waring problem}Let $w\in F_{r}$
(resp.~$w\in\mathcal{L}_{r}$) be a word.
\begin{enumerate}
\item How large are the fibers of $\varphi_{w}$, and what can one say on
their singularities? 
\item Is there $t\in\nats$ such that $\varphi_{w}^{*t}$ has ''nice''
singularity properties, uniformly over $\underline{G}\in\mathcal{G}_{\mathrm{Alg}}$
or over $\mathfrak{g}\in\mathcal{G}_{\mathrm{LieA}}$? 
\end{enumerate}
\end{question}

Question \ref{que:questions for algebraic waring problem} is in some
sense a generalization of Question \ref{que:Mein probabilistic questions}.
If we take 
\[
\text{"nice"="flat, with geometrically irreducible fibers"},
\]
then the Lang-Weil bounds give a strong connection between Question
\ref{que:questions for algebraic waring problem} and the probabilistic
Waring problem for finite simple groups of Lie type. Here, to measure
how large are the fibers we use their dimension. This can already
be seen in \cite{LST19}, where this connection is being used frequently.
If we take 
\[
\text{"nice"="flat, with geometrically irreducible, reduced fibers of rational singularities"},
\]
and consider the log canonical threshold $\mathrm{lct}(\underline{G}^{r},\varphi_{w}^{-1}(g))$
(see e.g.~\cite{Mus12}), as a measurement of how bad are the singularities
of the fibers, then Question \ref{que:questions for algebraic waring problem}
becomes strongly connected to the probabilistic Waring problem for
compact $p$-adic groups.

In this paper we deal with Questions \ref{que:Mein probabilistic questions}
and \ref{que:questions for algebraic waring problem} in the cases
of $\mathcal{G}_{\mathrm{Alg}}$, $\mathcal{G}_{\mathrm{LieA}}$ and
$\mathcal{G}_{\mathrm{cp}}$. We answer Question \ref{que:questions for algebraic waring problem}
for semisimple Lie algebras, and provide partial answers for semisimple
algebraic groups (Sections \ref{sec:Singulrities-of-word}, \ref{sec:Lie-algebra-word},
\ref{sec:Proof-of-Theorems- Lie algebra} and \ref{sec:flatness-and-(FRS) of word maps}).
In Sections \ref{sec:The-(FRS)-property and counting points} and
\ref{sec:Applications-to-probabilistic} we study the connection between
Questions \ref{que:Mein probabilistic questions} and \ref{que:questions for algebraic waring problem},
and utilize it in order to deal with Question \ref{que:Mein probabilistic questions}
in Subsection \ref{subsec:Algebraic-random-walks word maps}. In addition,
in Section \ref{sec:The-commutator-map-revisited} we use the methods
developed in Sections \ref{sec:Lie-algebra-word} and \ref{sec:Proof-of-Theorems- Lie algebra}
to study the fibers of convolutions of the commutator map. As an application,
we provide improved bounds on the representation growth of compact
$p$-adic and arithmetic groups.

\subsection{\label{subsec:Main-results}Main results}

We divide the main results into several parts. We first discuss the
relative Waring problem in the setting of semisimple Lie algebras
and algebraic groups, through the study of singularity properties
of word maps, followed by a discussion on the special case of the
commutator map. We then provide algebraic interpretation to the various
probabilistic properties of word measures in terms of these singularity
properties. Combining these two parts, we provide applications to
the $p$-adic probabilistic Waring problem.

\subsubsection{Singularity properties of word maps}

Let $K$ be a field of characteristic $0$. 
\begin{defn}[{The (FRS) property \cite[Definition II]{AA16}}]
\label{def:(FRS)}~
\begin{enumerate}
\item A $K$-variety $X$ has \textit{rational singularities} if it is normal
and for any resolution of singularities $\pi:\widetilde{X}\to X$
we have $R^{i}\pi_{*}(\mathcal{O}_{\widetilde{X}})=0$ for $i\geq1$. 
\item Let $X$ and $Y$ be smooth $K$-varieties. We say that a morphism
$\varphi:X\rightarrow Y$ is \textit{(FRS)} if it is flat and if every
fiber of $\varphi$ is reduced and has rational singularities. 
\end{enumerate}
\end{defn}

\begin{defn}
\label{def:Lie algebra words-introduction}Let $\mathcal{A}_{r}$
(resp.~$\mathcal{L}_{r}$) be the free associative $K$-algebra (resp.~the
free $K$-Lie algebra) on a finite set $\{X_{1},\ldots,X_{r}\}$.
We call $w\in\mathcal{A}_{r}$ (resp.~$w\in\mathcal{L}_{r}$) an
\textit{algebra word} (resp.~a \textit{Lie algebra word}). Note that
both $\mathcal{A}_{r}$ and $\mathcal{L}_{r}$ have a natural gradation.
We define the \textit{degree} of $w$, denoted $\mathrm{deg}(w)$,
as the maximal grade $d\in\nats$ in which the image of $w$ is non-trivial. 
\end{defn}

In the following theorems we show that (Lie) algebra word maps obtain
good singularity properties after sufficiently many convolutions,
in a uniform way:

\begin{theorem}[Theorem \ref{thm: main thm Lie algebra word maps}]\label{thm A}Let
$\{w_{i}\in\mathcal{L}_{r_{i}}\}_{i\in\nats}$ be a collection of
Lie algebra words of degree at most $d\in\nats$. Then there exists
$0<C<10^{6}$, such that for any simple $K$-Lie algebra $\mathfrak{g}$
for which $\{\varphi_{w_{i}}:\mathfrak{g}^{r_{i}}\rightarrow\mathfrak{g}\}$
are non-trivial, we have the following: 
\begin{enumerate}
\item If $m\geq Cd^{3}$ then $\varphi_{w_{1}}*\ldots*\varphi_{w_{m}}$
is flat. 
\item If $m\geq Cd^{4}$ then $\varphi_{w_{1}}*\ldots*\varphi_{w_{m}}$
is (FRS). 
\end{enumerate}
\end{theorem}

\begin{theorem}[Theorem \ref{thm: main theorem for matrix word map}]\label{Thm B}Let
$\{w_{i}\in\mathcal{A}_{r_{i}}\}_{i\in\nats}$ be a collection of
algebra words of degree at most $d\in\nats$. Then for any $n\in\nats$
such that $\{\varphi_{w_{i}}:M_{n}^{r}\rightarrow M_{n}\}_{i\in\nats}$
are generating, with $M_{n}$ is the ring of $n\times n$-matrices
over $K$, the assertions of Theorem \ref{thm A} hold. \end{theorem}

Any word map $\varphi_{w}:\underline{G}^{r}\rightarrow\underline{G}$
can be degenerated to a map $\widetilde{\varphi}_{w}:\mathfrak{g}^{r}\rightarrow\mathfrak{g}$,
called the \textit{symbol} of $\varphi_{w}$, by taking the analogue
of the lowest order non-vanishing term in the Taylor expansion of
$\varphi_{w}$ near $(e,\ldots,e)\in\underline{G}^{r}$. The Baker-Campbell-Hausdorff
formula shows that $\widetilde{\varphi}_{w}$ is in fact a Lie algebra
word map (Proposition \ref{prop:symbol of a word map is a Lie algebra word map}).
Theorem \ref{thm A}, together with the fact that the (FRS) property
behaves well under deformations (Corollary \ref{cor:reduction to degeneration})
give the following:

\begin{theorem}[Theorem \ref{thm: (FRS) at (e,...,e)}]\label{thm C}
Let $d\in\nats$, let $\underline{G}$ be a connected semisimple $K$-group
and let $\{\varphi_{w_{i}}:\underline{G}^{r_{i}}\rightarrow\underline{G}\}_{i=1}^{m}$
be a collection of word maps with symbols $\widetilde{\varphi}_{w_{i}}$
of degree at most $d$. Then there exists $C>0$ such that for any
$m\geq C\cdot d^{4}$ the map $\varphi_{w_{1}}*\ldots*\varphi_{w_{m}}$
is (FRS) at $(e,\ldots,e)$. \end{theorem}

We now turn to Item (1) of Question \ref{que:questions for algebraic waring problem}.
To quantitatively measure the singularities of the fibers of word
maps, we introduce the following notions: 
\begin{defn}
Let $\varphi:X\rightarrow Y$ be a morphism between geometrically
irreducible smooth $K$-varieties. Let $J_{m}(\varphi):J_{m}(X)\rightarrow J_{m}(Y)$
be the corresponding $m$-th jet morphism (see Subsections \ref{subsec:The-commutator-word map}
and \ref{def:basic definition jet schemes}) and let $X_{\varphi(x),\varphi}$
be the fiber of $\varphi$ over $\varphi(x)$. 
\begin{enumerate}
\item $\varphi$ is called \textit{$\varepsilon$-flat} if for every $x\in X$
we have $\mathrm{dim}X_{\varphi(x),\varphi}\leq\mathrm{dim}X-\varepsilon\mathrm{dim}Y$. 
\item $\varphi$ is called \textit{$\varepsilon$-jet flat} (resp.~jet-flat)
if $J_{m}(\varphi)$ is $\varepsilon$-flat (resp.~flat) for every
$m\in\nats$. 
\end{enumerate}
Note that a $1$-flat morphism is just flat. It follows from \cite{Mus02}
that the notion of $\varepsilon$-jet flatness is closely related
to the log canonical threshold of the pair $(X,X_{\varphi(x),\varphi})$,
denoted $\mathrm{lct}(X,X_{\varphi(x),\varphi})$; the morphism $\varphi$
is $\varepsilon$-jet flat if and only if $\mathrm{lct}(X,X_{\varphi(x),\varphi})\geq\varepsilon\mathrm{dim}Y$
for all $x\in X$. 
\end{defn}

Let $w$ be a word in $\mathcal{A}_{r}$ or $\mathcal{L}_{r}$. An
easy calculation of dimensions shows that if $\varphi_{w}^{*t}$ is
jet-flat for some $t\in\nats$, then $\varphi_{w}$ is $\frac{1}{t}$-jet
flat. Since (FRS) morphisms are in particular jet-flat (see Corollary
\ref{cor: singularity properties through dim of jets}), one can deduce
the following from the proof of Theorems \ref{thm A} and \ref{Thm B}:

\begin{theorem}[Theorem \ref{thm:lower bounds on epsilon jet flatness}]\label{thm D}
Let $w_{1}\in\mathcal{A}_{r}$ and $w_{2}\in\mathcal{L}_{r}$ be words
of degree at most $d$. 
\begin{enumerate}
\item The map $\varphi_{w_{1}}:M_{n}^{r}\rightarrow M_{n}$ is $\frac{1}{25d^{3}}$-jet
flat, for every $n\in\nats$.
\item The map $\varphi_{w_{2}}:\mathfrak{g}^{r}\rightarrow\mathfrak{g}$
is $\frac{1}{2\cdot10^{5}d^{3}}$-jet flat, for every simple $K$-Lie
algebra $\mathfrak{g}$.
\end{enumerate}
\end{theorem}

We denote by $\ell(w)$ the \textsl{length} of a word $w\in F_{r}$
(e.g.~$xyx^{-1}y^{-1}$ has length $4$). Theorem \ref{thm D} can
be used to give lower bounds on the log canonical threshold of the
fibers of group word maps.

\begin{theorem}[Theorem \ref{thm:jet flatness of word maps}]\label{thm E}
For any $w\in F_{r}$, the map $\varphi_{w}:\mathrm{SL}_{n}^{r}\rightarrow\mathrm{SL}_{n}$
is $\frac{1}{10^{17}\ell(w)^{9}}$-jet flat. In particular, $\mathrm{lct}(\mathrm{SL}_{n}^{r},\varphi_{w}^{-1}(g))\geq\frac{\mathrm{dim}\mathrm{SL}_{n}}{10^{17}\ell(w)^{9}}$
for every $g\in\mathrm{SL}_{n}$. \end{theorem}

Theorem \ref{thm A} and Theorem \ref{thm D} provide an answer to
Question \ref{que:questions for algebraic waring problem} for algebra
and Lie algebra words. Theorem \ref{thm E} provides an answer to
Part (1) of Question \ref{que:questions for algebraic waring problem}
for $\underline{G}=\mathrm{SL}_{n}$. Probabilistic applications of
the above results are discussed in Subsection \ref{subsec:Algebraic-interpretation-of}.

\subsubsection{\label{subsec:The-commutator-word}The commutator word map}

Let $w_{0}=xyx^{-1}y^{-1}\in F_{r}$ and $w=[X,Y]\in\mathcal{L}_{r}$
be the commutator word and Lie algebra word. In \cite{AA16}, it was
shown that $\varphi_{w_{0}}^{*21}$ is (FRS) for all semisimple algebraic
groups $\underline{G}$, and that $\varphi_{w}^{*21}$ is (FRS) for
all classical Lie algebras (see \cite[Theorem VIII]{AA16}). These
bounds were improved to $2$ for $\mathrm{SL}_{n}$ and $\frak{sl}_{n}$
in \cite{Bud}, and to $11$ in the general case in \cite{Kap}. Using
a criterion of Musta\c{t}\u{a} for rational singularities \cite[Theorem 0.1]{Mus01},
and the methods developed in the course of the proof of Theorem \ref{thm A},
we improve these bounds and significantly simplify the proof of \cite{AA16}:

\begin{theorem}[Theorems \ref{thm:Commutator is (FRS) after 4 convolutions} and \ref{Thm:commutator is flat}]\label{thm F:commutator map is (FRS)}Let
$\underline{G}$ be a semisimple algebraic $K$-group, with Lie algebra
$\mathfrak{g}$. 
\begin{enumerate}
\item $\varphi_{w_{0}}^{*2}:\underline{G}^{4}\rightarrow\underline{G}$
and $\varphi_{w}^{*2}:\mathfrak{g}^{4}\rightarrow\mathfrak{g}$ are
flat. 
\item $\varphi_{w_{0}}^{*4}:\underline{G}^{8}\rightarrow\underline{G}$
and $\varphi_{w}^{*4}:\mathfrak{g}^{8}\rightarrow\mathfrak{g}$ are
(FRS).
\end{enumerate}
\end{theorem}
\begin{rem}
It is conjectured that $\varphi_{w_{0}}^{*2}$ and $\varphi_{w}^{*2}$
are (FRS) for any semisimple algebraic group $\underline{G}$. For
$\underline{G}$ not of type $A$, Theorem \ref{thm F:commutator map is (FRS)}
gives the best known bounds. 
\end{rem}

The fibers of the commutator map $\varphi_{w_{0}}$ are closely related,
via a classical theorem of Frobenius, to the representation growth
of compact $p$-adic groups and arithmetic groups (see \cite{AA16,AA18}
and \cite[Theorems 1.3-1.5]{Bud}). Thus, Theorem \ref{thm F:commutator map is (FRS)}
gives the following applications. 
\begin{cor}[{cf.~\cite[Theorem A]{AA16}, \cite[Theorem B]{AA18}}]
\label{cor: representation growth}Let $\underline{G}$ be a semisimple,
simply connected algebraic $\rats$-group, and let $G$ be either
\begin{enumerate}
\item a compact open subgroup of $\underline{G}(\mathbb{Q}_{p})$ for some
prime $p$; or 
\item $\underline{G}(\mathbb{Z})$, provided that $\underline{G}$ has $\rats$-rank$\geq2$. 
\end{enumerate}
Then there exists a constant $C$ such that, for any $N\in\nats$,
\[
\#\{N\text{-dimensional irreducible \ensuremath{\mathbb{C}}-representations of }G\}<C\cdot N^{7}.
\]
\end{cor}

We further have, 
\begin{cor}[{cf.~\cite[Theorem B]{AA16}}]
\label{cor for surface groups}Let $\underline{G}$ be a complex
semisimple algebraic group, and $\Sigma_{g}$ be a closed orientable
surface of genus $g\geq4$. Then the deformation variety 
\[
\mathrm{Hom}(\pi_{1}(\Sigma_{g}),\underline{G}(\complex))\simeq(\varphi_{w_{0}}^{*g})^{-1}(e)
\]
has rational singularities. 
\end{cor}

\subsubsection{\label{subsec:Algebraic-interpretation-of}Algebraic interpretation
of uniform mixing times and the $p$-adic probabilistic Waring type
problem }

Let $\mathcal{G}$ denote the collection of all algebraic group schemes
$\underline{G}$, with $\underline{G}_{\rats}$ semisimple and simply
connected. For a prime power $q=p^{r}$, we denote the unique unramified
extension of $\Qp$ of degree $r$ by $\mathbb{Q}_{q}$, its ring
of integers by $\mathbb{Z}_{q}$, and the maximal ideal of $\mathbb{Z}_{q}$
by $\mathfrak{m}_{q}$.

We study the family $\{\underline{G}(\mathbb{Z}_{q})\}_{\underline{G}\in\mathcal{\mathcal{G}},q}$
through the quotients $\{\underline{G}(\mathbb{Z}_{q}/\mathfrak{m}_{q}^{k})\}_{\underline{G}\in\mathcal{\mathcal{G}},q,k}$.
Let $w\in F_{r}$ and let $\tau_{w,\underline{G},q,k}:=(\varphi_{w})_{*}(\pi_{\underline{G}^{r}(\mathbb{Z}_{q}/\mathfrak{m}_{q}^{k})})$
be the pushforward of the uniform probability measure on $\underline{G}^{r}(\mathbb{Z}_{q}/\mathfrak{m}_{q}^{k})$
with respect to $\varphi_{w}:\underline{G}^{r}\rightarrow\underline{G}$. 
\begin{defn}
For a real function $f$ on a finite group $G$ and $a\in\reals_{\geq1}$,
set 
\[
\left\Vert f\right\Vert _{a}:=(\left|G\right|^{a-1}\sum\limits _{g\in G}\left|f(g)\right|^{a})^{1/a}\text{ ~~and~~ }\left\Vert f\right\Vert _{\infty}=\left|G\right|\cdot\max\limits _{g\in G}\left|f(g)\right|.
\]
\end{defn}

For a probability measure $\mu$ on a finite group $G$, set $S_{\mu}:=\{g\in G:\mu(g)>0\}$.
If $e\in S_{\mu}$ and $\langle S_{\mu}\rangle=G$, then the random
walk induced by $\mu$ converges to the uniform probability measure
on $G$, i.e.~$\underset{t\rightarrow\infty}{\mathrm{lim}}\left\Vert \mu^{*t}-\pi_{G}\right\Vert _{a}=0$
for any $a\in[1,\infty]$ (e.g.~see \cite[Theorem 4.9]{LeP17}).
In particular, by \cite{Jai08}, this applies to each measure $\tau_{w,\underline{G},q,k}$. 
\begin{defn}
\label{def:mixing time}The minimal $t_{a}\in\nats$ such that $\left\Vert \mu^{*t_{a}}-\pi_{G}\right\Vert _{a}<\frac{1}{2}$
is called the \textsl{$L^{a}$-mixing time} of the random walk. At
this point, rapid convergence occurs; for any $m\in\nats$ we have
$\left\Vert \mu^{*t_{a}m}-\pi_{G}\right\Vert _{a}<2^{-m}$ (e.g.~\cite[Lemma 4.18]{LeP17}). 
\end{defn}

We would now like to define a uniform version of mixing time for families
of word maps. Given a set of primes $S$, denote by $\mathcal{P}_{S}$
the set of all prime powers $p^{k}\in\nats$ where $p\notin S$. 
\begin{defn}
~\label{def:uniform mixing tme-introduction}We say that a word $w\in F_{r}$ 
\begin{enumerate}
\item is\textit{ almost $\{\mathbb{F}_{q}\}$-unifor}\textsl{m in $L^{a}$}
if $\forall\underline{G}\in\mathcal{\mathcal{G}}$ there exists a
finite set $S$ of primes such that 
\[
\underset{\mathcal{P}_{S}\ni q\rightarrow\infty}{\lim}\left\Vert \tau_{w,\underline{G},q,1}-\pi_{\underline{G}(\mathbb{F}_{q})}\right\Vert _{a}=0.
\]
\item is\textit{ almost $\{\ints_{q}\}$-unifor}\textsl{m in $L^{a}$} if
$\forall\underline{G}\in\mathcal{\mathcal{G}}$ there exists a finite
set $S$ of primes such that 
\[
\underset{\mathcal{P}_{S}\ni q\rightarrow\infty}{\lim}\left(\underset{k\in\nats}{\sup}\left\Vert \tau_{w,\underline{G},q,k}-\pi_{\underline{G}(\mathbb{Z}_{q}/\mathfrak{m}_{q}^{k})}\right\Vert _{a}\right)=0.
\]
\end{enumerate}
Let $\mathcal{F}$ be $\{\ints_{q}\}$ or $\{\mathbb{F}_{q}\}$. We
say that $w$ has an \textit{almost} $\mathcal{F}$\textit{-uniform
$L^{a}$-mixing time }$t_{a,\mathcal{F}}$ if it is minimal such that
$w^{*t_{a,\mathcal{F}}}$ is almost $\mathcal{F}$-uniform in $L^{a}$. 
\end{defn}

In Section \ref{sec:The-(FRS)-property and counting points} we utilize
the Lang-Weil bounds to obtain a number theoretic characterization
of flatness (see Theorem \ref{thm: flatness and counting points}).
Using \cite[Theorem A]{AA18} and \cite{CGH}, the (FRS) property
of a $\rats$-morphism $\varphi:X\rightarrow Y$ between smooth $\rats$-varieties
can be interpreted by certain uniform bounds on the fibers of $\varphi:X(\mathbb{Z}_{q}/\mathfrak{m}_{q}^{k})\rightarrow Y(\mathbb{Z}_{q}/\mathfrak{m}_{q}^{k})$
(see further discussion in Subsection \ref{Number theoretic characterization of the (FRS)}
in the simpler case where $q=p$ is prime). This allows us to give
an algebraic interpretation of the notion of almost uniformity in
$L^{\infty}$.

\begin{theorem}[Corollary \ref{cor:dictionary for word maps}]\label{thm:G- dictionary for Linfty-introduction}Let
$w\in F_{r}$. 
\begin{enumerate}
\item $w$ is almost $\{\mathbb{F}_{q}\}$-uniform in $L^{\infty}$ if and
only if $(\varphi_{w})_{\rats}:\underline{G}_{\rats}^{r}\rightarrow\underline{G}_{\rats}$
is flat and has geometrically irreducible fibers for every $\underline{G}\in\mathcal{G}$.
 
\item $w$ is almost $\{\mathbb{Z}_{q}\}$-uniform in $L^{\infty}$ if and
only if $(\varphi_{w})_{\rats}:\underline{G}_{\rats}^{r}\rightarrow\underline{G}_{\rats}$
is (FRS) and has geometrically irreducible fibers for every $\underline{G}\in\mathcal{G}$. 
\end{enumerate}
\end{theorem}

In Section \ref{sec:The-(FRS)-property and counting points} we furthermore
provide a number theoretic interpretation of $\varepsilon$-jet flatness:

\begin{theorem}[Theorem \ref{thm:epsilon jet flat and counting points}]\label{thm H-epsilon jet flat and counting points}Let
$\varphi:X\to Y$ be a dominant morphism between finite type $\ints$-schemes
$X$ and $Y$, with $X_{\rats},Y_{\rats}$ smooth and geometrically
irreducible, and let $0<\varepsilon\leq1$. Then the following are
equivalent: 
\begin{enumerate}
\item For any $0<\varepsilon'<\varepsilon$ there exists a finite set $S$
of primes, such that for every $q\in\mathcal{P}_{S}$, every $k\in\nats$
and every $y\in Y(\ints_{q}/\frak{m}_{q}^{k})$,
\[
\left|\varphi^{-1}(y)\right|<q^{k(\mathrm{dim}X_{\mathbb{Q}}-\varepsilon'\mathrm{dim}Y_{\mathbb{Q}})}.
\]
\item $\varphi_{\rats}:X_{\rats}\to Y_{\rats}$ is $\varepsilon$-jet flat. 
\end{enumerate}
\end{theorem}

Theorems \ref{thm E} and \ref{thm H-epsilon jet flat and counting points}
yield the following $p$-adic analogue of \cite{LaS12}.

\begin{theorem}[Corollary \ref{cor:bounds on fibers for SLn}]\label{Theorem I: bounds on fibers}Let
$w\in F_{r}$ be a word. Then for any $n\in\nats$ there exists a
finite set $S(n)$ of primes such that for every $q\in\mathcal{P}_{S(n)}$
and every $g\in\mathrm{SL}_{n}(\ints_{q}/\frak{m}_{q}^{k})$ we have
\[
\tau_{w,\mathrm{SL}_{n},q,k}(g)\leq\left|\mathrm{SL}_{n}(\ints_{q}/\frak{m}_{q}^{k})\right|^{-\frac{1}{2\cdot10^{17}\ell(w)^{9}}}.
\]

\end{theorem}

For any $\underline{G}\in\mathcal{G}$ and a prime $p$, write $\underline{G}^{1}(\Zp)$
for the first congruence subgroup of $\underline{G}(\Zp)$ (i.e. the
kernel of $\underline{G}(\Zp)\rightarrow\underline{G}(\mathbb{F}_{p})$),
and let $\mu_{\underline{G}^{1}(\Zp)}$ be the normalized Haar measure.
Using Theorem \ref{thm C} and an equivalent analytic criterion for
the (FRS) property (\cite[Theorem 3.4]{AA16}), we can provide $L^{\infty}$-bounds
for the random walk on $\underline{G}^{1}(\Zp)$ induced from $w$: 

\begin{theorem}[Corollary \ref{cor:result for G1(ZP)}]\label{Thm J:random walk on G1(Zp)}Let
$w\in F_{r}$ be a word. Then there exists $C>0$, such that for every
$\underline{G}\in\mathcal{G}$, every prime $p>p_{0}(\underline{G})$,
and every $t\geq C\ell(w)^{4}$, the pushforward measure $(\varphi_{w}^{*t})_{*}(\mu_{\underline{G}^{1}(\Zp)}^{rt})$
on $\underline{G}^{1}(\Zp)$ has continuous density.\end{theorem}

The following theorem provides an algebraic interpretation of the
notion of almost uniformity in $L^{1}$. This is a strengthening of
\cite[Theorem 2]{LST19}.

\begin{theorem}[Corollary \ref{cor:dictionary for word maps}]\label{thm:K- dictionary for L1}Let
$w\in F_{r}$. The following are equivalent: 
\begin{enumerate}
\item $w$ is almost $\{\mathbb{F}_{q}\}$-uniform in $L^{1}$. 
\item $w$ is almost $\{\ints_{q}\}$-uniform in $L^{1}$. 
\item $(\varphi_{w})_{\rats}:\underline{G}_{\rats}^{r}\rightarrow\underline{G}_{\rats}$
has a geometrically irreducible generic fiber, for every $\underline{G}\in\mathcal{G}$. 
\end{enumerate}
\end{theorem}

In \cite[Theorem 1]{LST19} it was shown that any convolution $w_{1}*w_{2}$
of two words is almost $\{\mathbb{F}_{q}\}$\textit{-}uniform in $L^{1}$.
Theorem \ref{thm:K- dictionary for L1} allows us to push this result
to compact $p$-adic groups:

\begin{theorem}[Corollary \ref{cor:-L1 mixing time of words}]\label{Theorem L-uniform in L1 after convolution}Let
$w_{1}\in F_{r_{1}}$ and $w_{2}\in F_{r_{2}}$ be words. Then $w=w_{1}*w_{2}$
is almost $\{\ints_{q}\}$-uniform in $L^{1}$.

\end{theorem}

It follows from \cite{GS09} that the commutator word is almost $\{\mathbb{F}_{q}\}$\textit{-}uniform
in $L^{1}$. Theorems \ref{thm F:commutator map is (FRS)}, \ref{thm:G- dictionary for Linfty-introduction}
and \ref{thm:K- dictionary for L1} give the following: 
\begin{cor}
\label{cor:uniform mixing time commutator}Let $w_{0}=g_{1}g_{2}g_{1}^{-1}g_{2}^{-1}\in F_{2}$
be the commutator word. 
\begin{enumerate}
\item $w_{0}*w_{0}*w_{0}*w_{0}$ is almost $\{\ints_{q}\}$-uniform in $L^{\infty}$. 
\item $w_{0}$ is almost $\{\ints_{q}\}$-uniform in $L^{1}$. 
\end{enumerate}
\end{cor}

\subsection{Further discussion of the main results}

Questions \ref{que:Mein probabilistic questions} and \ref{que:questions for algebraic waring problem}
(excluding the $L^{1}$-Waring type problems) have been tackled only
quite recently, and mainly in the case of finite simple groups. In
\cite{LaS12} a bound of the form $\left|G\right|^{r-\varepsilon(w)}$
was given on the size of the fibers of $\varphi_{w}:G^{r}\rightarrow G$
for large enough finite simple groups. In \cite{LST19}, the probabilistic
Waring problem for finite simple groups received thorough treatment.
We say that $w\in F_{r}$ is \textsl{almost }$\mathcal{G}_{\mathrm{fin}}$-\textsl{uniform
in $L^{a}$} if $\underset{G\in\mathcal{G}_{\mathrm{fin}}:\left|G\right|\rightarrow\infty}{\lim}\left\Vert \tau_{w,G}-\pi_{G}\right\Vert _{a}=0$.
It was shown in \cite{LST19}, that any word $w\in F_{r}$ has an
almost $\mathcal{G}_{\mathrm{fin}}$-uniform $L^{1}$-mixing time
$t_{1,\mathrm{fin}}(w)\leq2$ and an almost $\mathcal{G}_{\mathrm{fin}}$-uniform
$L^{\infty}$-mixing time $t_{\infty,\mathrm{fin}}(w)\leq C\cdot l^{4}$,
for some explicit $C\gg1$ (see \cite[Theorems 1 and 4]{LST19}).
As a corollary, the authors deduced that for any $w\in F_{r}$ and
any semisimple algebraic group $\underline{G}$, the map $\varphi_{w}^{*t}:\underline{G}^{r}\rightarrow\underline{G}$
is flat for any $t\geq C\cdot l^{4}$ (see \cite[Theorem 5(i)]{LST19}).

There are two substantial difficulties that arise in the compact $p$-adic
probabilistic Waring problem. The first one is that, unlike in the
case of finite simple groups, their representation theory is much
more involved and far less understood. In particular, there is no
developed theory of character bounds as in the case of finite simple
groups (e.g.~\cite{BLS18,GLTa,GLTb}), and the use of other representation
theoretic techniques is limited as well. The second main difficulty
is that, unlike in the case of $\mathcal{G}_{\mathrm{fin}}$, the
singularities of the word maps control the size of the fibers. This
requires the use of more advanced algebro-geometric tools.

Theorems \ref{thm D}, \ref{thm E} and \ref{Theorem I: bounds on fibers}
can be seen as an algebro-geometric and compact $p$-adic versions
of \cite{LaS12}. Theorem \ref{thm A} can be seen as a Lie algebra
analogue of \cite[Theorem 5(i)]{LST19}, which is richer in the sense
that it also gives information on the singularities of the fibers.
Several interesting questions arise:
\begin{question}
~\label{que:conjecture on (FRS) and flatness}
\begin{enumerate}
\item Can we find $\alpha,C>0$ such that for every $w\in F_{r}$ and every
simple algebraic group $\underline{G}$, the word map $\varphi_{w}^{*C\ell(w)^{\alpha}}$
is (FRS)? 
\item If such $\alpha$ exists, what is its optimal value? What is the optimal
value of such $\alpha$ if we just demand that our word map becomes
flat? 
\item What is the answer to (2) in the case of simple Lie algebras, where
$\ell(w)$ is replaced by $\mathrm{deg}(w)$? 
\end{enumerate}
\end{question}

Theorems \ref{thm A} and \ref{thm C} provide a great deal of evidence
towards a positive answer to (1). In fact, in \cite[Theorem IV]{AA16}
it was shown that the commutator map is (FRS) after $t$ convolutions
if and only if it is (FRS) at $(e,\dots,e)\in\underline{G}^{2t}$
after $t$ convolutions. This leads to the following question: 
\begin{question}
\label{que:Are words worst at (e,...,e)}Can we characterize or find
nice families of words $w\in F_{r}$, whose worst singularities are
concentrated at $(e,\ldots,e)$? That is, which words $w\in F_{r}$
satisfy the property that for every simple algebraic group $\underline{G}$,
$\varphi_{w}^{*t}$ is (FRS)/flat if and only if $\varphi_{w}^{*t}$
is (FRS)/flat at $(e,\ldots,e)\in\underline{G}^{rt}$. 
\end{question}

\begin{example}
\label{exa:commutator and power word}~
\begin{enumerate}
\item Let $w=x^{m}$. Then $\varphi_{w}:\underline{G}\rightarrow\underline{G}$
is smooth at $(e,\ldots,e)$ for every simple algebraic group $\underline{G}$.
On the other hand, by analyzing the fiber over $e\in\underline{G}$,
one sees that $\varphi_{w}$ is not even flat, and that in fact $\varphi_{w}^{*(m-1)}$
is not flat. 
\item Let $w_{n}:=[[\dots[[X,Y],Y],\dots],Y]\in\mathcal{L}(X,Y)$ be the
Engel word (where $Y$ appears $n$ times). Then $\varphi_{w,n}^{*(n-1)}:\mathfrak{g}^{2n-2}\rightarrow\mathfrak{g}$
is not flat. 
\end{enumerate}
We therefore see that not all words satisfy the condition of Question
\ref{que:Are words worst at (e,...,e)} and moreover, that a potentially
optimal $\alpha$ for Items (2) and (3) of Question \ref{que:conjecture on (FRS) and flatness}
is at least $1$. There are certain low-rank phenomena for word maps
(discussed in Sections \ref{sec:Lie-algebra-word} and \ref{sec:Proof-of-Theorems- Lie algebra}),
that is, when the length (resp.~degree) of the word is large compared
to the rank of $\underline{G}$ (resp.~$\mathfrak{g}$). We therefore
expect different answers for Items (2) and (3) of Question \ref{que:conjecture on (FRS) and flatness}
when considering only groups or Lie algebras of high enough rank. 
\end{example}

\subsection{\label{subsec:Methods-and-main ideas}Methods and main ideas of the
proofs}

\subsubsection{\label{subsec:Singularities-of-word further discussion}Singularities
of word maps: Theorems \ref{thm A}-\ref{thm E}}

The proof of Theorem \ref{thm A} was inspired by \cite{AA16}, where
the special case of the commutator word map was analyzed. We generalize
some of the techniques in \cite{AA16} as well as introduce new ones.
The main principal employed is that many singularity properties, such
as flatness and the (FRS) property, are preserved under smooth deformations.
In other words, given a Lie algebra word map $\varphi_{w}:\mathfrak{g}^{r}\rightarrow\mathfrak{g}$,
we may degenerate it to a map $\psi_{w}$ which is easier to handle,
and prove that $\psi_{w}$ is flat or (FRS) instead. The idea is to
apply in each time one of the following operations to a given morphism: 
\begin{enumerate}
\item Taking a self-convolution of our morphism; 
\item Applying a degeneration to our morphism,
\end{enumerate}
such that:
\begin{enumerate}
\item The use of Operation (1) will be as economical as possible. Explicitly,
we would like the number of convolutions to be bounded by a polynomial
in the degree of $w\in\mathcal{L}_{r}$, regardless of the simple
Lie algebra under consideration. 
\item Every time we apply a degeneration, the resulting morphism should
stay in a certain nice predefined class of morphisms, and in particular
it should be generating. 
\item After finitely many steps as above we want to arrive at a morphism
we can easily show is flat or (FRS). 
\end{enumerate}
A sketch of the key steps of the proof is as follows: 
\begin{enumerate}
\item In Section \ref{subsec:Reduction to self convolutions+homogeneous},
we reduce to the case of self-convolutions of homogeneous word maps. 
\item Generalizing \cite{AA16}, we attach to each $d$-homogeneous word
map $\varphi_{w}:\mathfrak{g}^{r}\rightarrow\mathfrak{g}$ a certain
``combinatorial gadget'' which we call a $d$-polyhypergraph, and
is a graph whose hyperedges consist of multisets of $d$ vertices
and have types associated to them. Each hyperedge corresponds to a
sum of certain monomials in $\varphi_{w}$, and the types correspond
to a choice of coordinates on the target. We make sense of what it
means for such an object to be flat or (FRS). 
\item We encode the two operations above combinatorially. The convolution
operation corresponds to adding another copy of the polyhypergraph
with hyperedges of identical types, and degeneration corresponds to
eliminating certain hyperedges which are of the same type. 
\item In Sections \ref{subsec:Proof-for-SLn} and \ref{subsec:Proof-for-the general case},
we deal with $d$-homogeneous word maps on high rank classical Lie
algebras, and show that by applying an economical series of convolutions
and degenerations to the corresponding $d$-polyhypergraph, one gets
a hypergraph with a certain tame combinatorial structure (for example,
a disjoint union of hypergraphs each consisting of a single hyperedge). 
\item In Sections \ref{subsec:Proof-for low rank and few edges} we deal
with the case of low rank simple Lie algebras (when $\mathrm{rk}(\mathfrak{g})=O(d)$),
and with the case of a tame combinatorial object (as in (4)). The
case of a hypergraph consisting of a single hyperedge corresponds
to a $d$-homogeneous polynomial map $\varphi:\mathbb{A}^{n}\rightarrow\mathbb{A}^{1}$.
For such maps, one can give an effective lower bound on the log canonical
threshold of the fibers. Using a Thom-Sebastiani type result, we can
deduce that $\varphi^{*(d+1)}$ is (FRS). 
\end{enumerate}
There are certain difficulties that we deal with. In general, a generating
morphism might become non-generating after degeneration, so we need
to choose our degenerations carefully; for example, $\psi_{t}(x,y)=(x,x+ty^{2})$
is a degeneration of a dominant map $\psi_{1}(x,y)=(x,x+y^{2})$ to
a non-generating map $\psi_{0}(x,y)=(x,x)$. In addition, there are
certain low rank phenomena in Lie algebra word maps that make the
choice of degenerations especially difficult. Firstly, words of degree
$d$ can be trivial on Lie algebras of comparable rank (see Example
\ref{exa:trivial word maps}). Moreover, certain degenerations can
be proved to be effective only when the rank is large enough ($\sim d$
). Since the degree $d$ of the words could be arbitrary large, Step
(5) is crucial.

The idea of the proof of Theorem \ref{thm E} is to reduce to the
case of matrix word maps, which is Theorem \ref{thm D}(1). The main
observation (Lemma \ref{lem:large generating unipotent subgroups})
is that $\mathrm{SL}_{2n}$ is boundedly generated by matrices of
the form 
\[
\left\{ \left(\begin{array}{cc}
I_{n} & A\\
0 & I_{n}
\end{array}\right):A\in M_{n}\right\} \text{ and }\left\{ \left(\begin{array}{cc}
I_{n} & 0\\
B & I_{n}
\end{array}\right):B\in M_{n}\right\} ,
\]
and that $\left(\begin{array}{cc}
I_{n} & A\\
0 & I_{n}
\end{array}\right)^{-1}=\left(\begin{array}{cc}
I_{n} & -A\\
0 & I_{n}
\end{array}\right)$. This allows us to write $\varphi_{w}$ in block matrix form, where
each block is a polynomial of matrices (in $M_{n}$), which we know
how to deal with. This is done in Section \ref{sec:flatness-and-(FRS) of word maps}.

\subsubsection{\label{subsec:The-commutator-word map}The commutator word map}

Given a $K$-scheme $X=\spec K[x_{1},\dots,x_{n}]/(f_{1},\dots,f_{k})$
and $m\in\nats$, one may define its \textsl{$m$-th jet scheme} $J_{m}(X)$,
which is the scheme having the following coordinate ring:
\[
K[x_{1},\dots,x_{n},x_{1}^{(1)},\dots,x_{n}^{(1)},\dots,x_{1}^{(m)},\dots,x_{n}^{(m)}]/(\{f_{j}^{(u)}\}_{j=1,u=1}^{k,m}),
\]
where $f_{i}^{(u)}$ is the $u$-th formal derivative of $f_{i}$.
Given a morphism $\varphi:X\rightarrow Y$ of affine $K$-schemes,
it induces a morphism $J_{m}(\varphi):J_{m}(X)\rightarrow J_{m}(Y)$
on the corresponding jet schemes called the\textsl{ $m$-th jet morphism},
which is given by $J_{m}(\varphi)=(\varphi,\varphi^{(1)},\dots,\varphi^{(m)})$,
with $\varphi^{(u)}$ again being the $u$-th formal derivative of
$\varphi$ (as a polynomial map). See Section \ref{subsec:Basic-properties-of}
for a more general (and precise) definition. Using this language,
Musta\c{t}\u{a} introduced the following criterion for rational singularities: 
\begin{thm}[\cite{Mus01}]
Let $X$ be a geometrically irreducible, local complete intersection
variety. Then the $m$-th jet scheme $J_{m}(X)$ of $X$ is geometrically
irreducible for all $m\in\nats$ if and only if $X$ has rational
singularities. 
\end{thm}

The relative formulation of this statement is as follows: 
\begin{cor}[Corollary \ref{cor: singularity properties through dim of jets}]
\label{cor:relative criterion for (FRS)}Let $X$ and $Y$ be smooth,
geometrically irreducible $K$-varieties, and let $\varphi:X\rightarrow Y$
be a flat morphism with geometrically irreducible fibers. Then $\varphi$
is (FRS) if and only if the $m$-th jet map $J_{m}(\varphi):J_{m}(X)\rightarrow J_{m}(Y)$
is flat, with geometrically irreducible fibers for each $m\in\nats$. 
\end{cor}

We study the commutator map through its jets. By \cite[Theorem IV]{AA16},
to prove Theorem \ref{thm F:commutator map is (FRS)} it is enough
to consider the Lie algebra case, and to show that $\varphi_{w}^{*4}:\mathfrak{g}^{8}\rightarrow\mathfrak{g}$
is (FRS) at $(0,\ldots,0)\in\mathfrak{g}^{8}$, where $\varphi_{w}(X,Y)=[X,Y]$.
Using Corollary \ref{cor:relative criterion for (FRS)}, we may consider
the $m$-th jet $J_{m}(\varphi_{w}):J_{m}(\mathfrak{g}^{2})\rightarrow J_{m}(\mathfrak{g})$
of the commutator map, and it is enough to show that each $J_{m}(\varphi_{w})$
is flat at $(0,\ldots,0)\in\mathfrak{g}^{2m}$, and that $J_{m}(\varphi_{w}^{-1}(0))$
is integral at $(0,\ldots,0)$. Since the property of being flat and
having locally integral fibers is preserved under deformations (Proposition
\ref{Prop:properties preserved under deformations}), we may use degeneration
methods as discussed in the previous subsection.

Under the identification $J_{m}(\mathfrak{g})\simeq\mathfrak{g}^{m+1}$,
the map $J_{m}(\varphi_{w})$ is of the form $(\varphi_{w},\varphi_{w}^{(1)},\ldots,\varphi_{w}^{(m)}):\mathfrak{g}^{2(m+1)}\rightarrow\mathfrak{g}^{m+1}$,
where each $\varphi_{w}^{(u)}$ is the $u$-th formal derivative of
$\varphi_{w}$, and in particular a Lie algebra word map of degree
$2$. By applying a self-convolution, and a series of degenerations,
the map $J_{m}(\varphi_{w}^{*2})$ is degenerated to a map comprised
of commutators in disjoint variables, 
\[
\phi_{m}(Y_{1},Y_{2},\ldots,Y_{2m-1},Y_{2m})=([Y_{1},Y_{2}],[Y_{3},Y_{4}],\ldots,[Y_{2m-1},Y_{2m}]),
\]
whose fibers are a product of fibers of commutator maps. Hence, it
is enough to show that $\varphi_{w}^{*2}$ is flat, with reduced,
geometrically irreducible fibers, which is done in Theorem \ref{Thm:commutator is flat}
using Fourier analysis.

It is worth noting that flatness of the commutator map is known in
the group case (\cite{Li93} and \cite[Theorem 2.8]{Sha09}). In the
Lie algebra case, except of $\mathfrak{sl}_{n}$, where a stronger
claim is proved in \cite{Bud}, no other proof was known to the authors.

\subsubsection{\label{Number theoretic characterization of the (FRS)}Number theoretic
interpretation of the (FRS) and $\varepsilon$-jet flatness properties
and applications}

Let $X$ be a finite type $\mathbb{Z}$-scheme. The Lang-Weil estimates
\cite{LW54} give the following for almost any prime $p$: 
\[
\left|X(\mathbb{F}_{p})\right|=p^{\mathrm{dim}X_{\mathbb{Q}}}\left(C_{X,p}+O(p^{-1/2})\right),
\]
where $C_{X,p}$ is the number of top dimensional irreducible components
of $X_{\overline{\mathbb{F}}_{p}}$ which are defined over $\mathbb{F}_{p}$.
In particular we see that $\left\{ \frac{\left|X(\mathbb{F}_{p})\right|}{p^{\mathrm{dim}X_{\rats}}}\right\} _{p\in\text{primes}}$
is uniformly bounded. Given a morphism $\varphi:X\rightarrow Y$ between
finite type smooth schemes, the flatness of $\varphi_{\rats}$ can
be shown to be equivalent to certain uniform bounds on the fibers
of $\varphi:X(\mathbb{F}_{q})\rightarrow Y(\mathbb{F}_{q})$ over
all finite fields (see Theorem \ref{thm: flatness and counting points}).

For a finite ring $A$, set $h_{X}(A):=\frac{\left|X(A)\right|}{\left|A\right|^{\mathrm{dim}X_{\rats}}}$
if $X_{\rats}$ is non-empty, and $h_{X}(A)=0$ otherwise. Unlike
the case of finite fields, when considering points over finite rings,
the singularities of $X$ come into play. For example, for $X=\mathrm{Spec}\mathbb{Z}[x]/(x^{2})$,
we have $h_{X}(\mathbb{Z}/p^{2k}\mathbb{Z})=p^{k}$, which is not
uniformly bounded in neither $p$ nor $k$. On the other end of the
spectrum, if $X$ is smooth, a generalization of Hensel's Lemma implies
that $\left\{ h_{X}(\ints/p^{k}\ints)\right\} _{p,k}$ is uniformly
bounded. The following natural question arises:
\begin{question}
Can we characterize finite type $\ints$-schemes $X$, such that $\left\{ h_{X}(\ints/p^{k}\ints)\right\} _{p,k}$
is uniformly bounded?
\end{question}

In \cite{AA18}, Aizenbud and Avni characterized this class of schemes
in the case where $X_{\rats}$ is a local complete intersection. This
was done relying on results of Musta\c{t}\u{a} \cite{Mus01} and
Denef \cite{Den87}, and was improved in \cite{Gla19}.
\begin{thm}[{cf.~\cite[Theorem A]{AA18}, \cite[Theorem 1.4]{Gla19}}]
\label{thm: rational singularities and point count}Let $X$ be a
finite type $\ints$-scheme, where $X_{\rats}$ is equidimensional
and a local complete intersection. 
\begin{enumerate}
\item $X_{\rats}$ is reduced and has rational singularities if and only
if there exists $C>0$ such that $h_{X}(\ints/p^{k}\ints)<C$, for
every prime $p$, and every $k\in\nats$. 
\item $X_{\rats}$ is geometrically irreducible, reduced and has rational
singularities if and only if there exists $C>0$ such that $\left|h_{X}(\ints/p^{k}\ints)-1\right|<Cp^{-1/2}$,
for every prime $p$, and every $k\in\nats$. 
\end{enumerate}
\end{thm}

Notice that an (FRS) morphism $\varphi_{\rats}$ is in particular
flat, so its fibers are local complete intersections (assuming the
base and target are smooth). The condition of rational singularities
puts each individual $\ints$-fiber of $\varphi$ in the framework
of Theorem \ref{thm: rational singularities and point count}. Hence,
under some mild assumptions, the (FRS) property of $\varphi_{\rats}$
can be interpreted by certain uniform bounds on the size of the fibers
of $\varphi:X(\ints/p^{k}\ints)\rightarrow Y(\ints/p^{k}\ints)$.
We make this relation more precise in an upcoming paper joint with
Raf Cluckers \cite{CGH}.

Similarly to the (FRS) property, Theorem \ref{thm H-epsilon jet flat and counting points}
characterizes the $\varepsilon$-jet flatness property in terms of
certain uniform bounds on the size of the fibers over finite rings
of the form $\ints/p^{k}\ints$. The idea of the proof is to first
reduce to the study of the $\mathbb{F}_{p}[t]/t^{k}$-fibers of $\varphi$,
or equivalently, the fibers of $J_{k-1}(\varphi):J_{k-1}(X)(\mathbb{F}_{p})\rightarrow J_{k-1}(Y)(\mathbb{F}_{p})$.
The Lang-Weil bounds on the $k$-th jet maps $J_{k}(\varphi)$, connect
between the $\varepsilon$-flatness of $J_{k-1}(\varphi)$ and the
desired upper bound on the $\mathbb{F}_{p}$-fibers of $J_{k-1}(\varphi)$,
for each fixed $k$. Since the family of morphisms $\{J_{k}(\varphi)\}_{k\in\nats}$
is of unbounded complexity, a finer treatment is needed.

It can be shown that the function $g:Y(\Zp)\times\nats\rightarrow\reals$
defined by $g(y,k):=\left|\varphi^{-1}(r_{k}(y))\right|$, where $r_{k}:Y(\Zp)\rightarrow Y(\ints/p^{k}\ints)$
is the reduction map, belongs to a certain class of functions with
a well behaved structure and integration theory, called \textsl{motivic
functions }(in the sense of \cite{CL08,CL10}, see Section \ref{subsec:The-Denef-Pas-language}
for more details). We then use results in motivic integration, to
deal with the unbounded complexity of $\{J_{k}(\varphi)\}_{k\in\nats}$
and finish the proof.

Most of the probabilistic applications follow from the algebro-geometric
results (Theorems \ref{thm A}-\ref{thm F:commutator map is (FRS)})
and the number theoretic interpretations of the flatness, $\varepsilon$-jet
flatness and the (FRS) properties. To prove Theorem \ref{thm:K- dictionary for L1}
we generalize \cite[Theorem 2]{LST19}. The main observation of \cite{LST19}
is that the property that a word $w\in F_{r}$ is almost $\{\mathbb{F}_{q}\}$-uniform
in $L^{1}$ has a geometric characterization, namely, that the generic
fiber of $\varphi_{w}:\underline{G}^{r}\rightarrow\underline{G}$
is geometrically irreducible for every simple algebraic group $\underline{G}$.
We show that if this is the case, then the $m$-th jet morphism $J_{m}(\varphi_{w}):J_{m}(\underline{G}^{r})\rightarrow J_{m}(\underline{G})$
also has a geometrically irreducible generic fiber (Lemma \ref{lem:irreducible generic fiber in jets}),
which in turn implies that $w\in F_{r}$ is almost $\{\ints_{q}\}$-uniform.

\subsection{Conventions}
\begin{itemize}
\item By $\nats$ we mean the set $\{0,1,2,...\}$.
\item We denote the complement of a set $S$ by $S^{c}$. 
\item Unless explicitly stated otherwise, $K$ is a field of characteristic
$0$ and $F$ is a non-Archimedean local field of characteristic $0$
whose ring of integers is $\mathcal{O}_{F}$. 
\item For any scheme $X$ and an element $x\in X$ we denote by $\kappa(x)$
its function field. 
\item For a morphism $\varphi:X\rightarrow Y$ of schemes, we denote by
$X_{y,\varphi}$ the scheme theoretic fiber $\mathrm{Spec}(\kappa(y))\times_{Y}X$
of $\varphi$ over $y\in Y$. We sometimes denote it by $X_{y}$ if
the map $\varphi$ is understood. 
\item For a $K$-morphism $\varphi:X\rightarrow Y$ between $K$-varieties
$X$ and $Y$, we denote by $X^{\mathrm{sm}}$ (resp.~$X^{\mathrm{sing}}$)
the smooth (resp.~non-smooth) locus of $X$, and we denote by $X^{\mathrm{sm,\varphi}}$
(resp.~$X^{\mathrm{sing,\varphi}}$) the smooth (resp.~non-smooth)
locus of $\varphi$ in $X$. 
\item For an $S$-scheme $X$, we denote the base change with respect to
$S'\rightarrow S$ by $X_{S'}$. 
\end{itemize}
\begin{acknowledgement*}
We thank Uri Bader, Itai Benjamini, Joseph Bernstein, Raf Cluckers,
Persi Diaconis, Ronen Eldan, Tsachik Gelander, Dmitry Gourevitch,
Boris Kunyavskii, Gady Kozma, Erez Lapid, Martin Liebeck, Dan Mikulincer,
Doron Puder, Aner Shalev, Shai Shechter, Stephan Snegirov and Pham
Huu Tiep for useful conversations and for commenting on various preliminary
versions of this project. We thank Nir Avni for many useful conversations
and discussions, and for answering various questions. Finally, we
wish to thank Rami Aizenbud for countless useful conversations, and
for his constant guidance and support throughout this project. Part
of this work was done during the ``Group Representation Theory and
Applications'' program at MSRI. Both authors wish to thank MSRI for
its excellent hospitality and learning environment. Both authors were
partially supported by ISF grant 249/17, BSF grant 2018201 and by
a Minerva foundation grant. 
\end{acknowledgement*}

\section{Preliminaries}

\subsection{\label{subsec:Basic-properties-of}Basic properties of jet schemes }

For a thorough discussion of jet schemes and their properties see
\cite[Chapter 3]{CLNS18} and \cite{LM09}. 
\begin{defn}[{cf.~\cite[Section 3.2]{CLNS18}}]
\label{def:basic definition jet schemes}Let $S$ be a scheme and
let $X$ be a scheme over $S$. 
\begin{enumerate}
\item For any $m\in\nats$, we define the $m$-th jet scheme of $X$, denoted
$J_{m}(X/S)$ as the $S$-scheme representing the functor 
\[
\mathcal{J}_{m}(X/S):W\longmapsto\mathrm{Hom}_{S\text{-schemes}}(W\times_{\spec\ints}\spec(\ints[t]/(t^{m+1})),X),
\]
where $W$ is an $S$-scheme. When the scheme $S$ is understood,
we omit it from our notation. 
\item Given an $S$-morphism $\varphi:X\rightarrow Y$ between two $S$-schemes
$X$ and $Y$, and given an $S$-scheme $W$, the composition with
$\varphi$ yields a map 
\[
\mathcal{J}_{m}(X/S)(W)\rightarrow\mathcal{J}_{m}(Y/S)(W),
\]
which corresponds to a morphism 
\[
J_{m}(\varphi):J_{m}(X/S)\rightarrow J_{m}(Y/S),
\]
and is called the \textsl{$m$-th jet} of $\varphi$. 
\end{enumerate}
\end{defn}

From now on we restrict ourselves to finite type $K$-schemes (with
$S=\mathrm{Spec}K$).
\begin{defn}
\label{def:truncation maps}~
\begin{enumerate}
\item Let $X$ be a finite type $K$-scheme. For any $m\geq n\in\nats$
and any finite type $K$-algebra $A$, we have a natural map $A[t]/(t^{m+1})\rightarrow A[t]/(t^{n+1})$
which induces, via $\mathcal{J}_{m}(X)(\mathrm{Spec}A)\rightarrow\mathcal{J}_{n}(X)(\mathrm{Spec}A)$,
a natural morphism $\pi_{n}^{m}:J_{m}(X)\rightarrow J_{n}(X)$. The
maps $\{\pi_{n}^{m}\}_{m\geq n}$ are called \textit{truncation maps}.
Note that any $K$-morphism $\varphi:X\rightarrow Y$ gives rise to
a collection of morphisms $\{J_{m}(\varphi):J_{m}(X)\rightarrow J_{m}(Y)\}_{m\in\nats}$
between the corresponding $m$-th jet schemes, which commute with
$\{\pi_{n}^{m}\}_{m\geq n}$. 
\item The natural map $A\rightarrow A[t]/(t^{m+1})$ induces a \textsl{zero
section} $s_{m}:X\hookrightarrow J_{m}(X)$. 
\end{enumerate}
\end{defn}

Let $X\subseteq\mathbb{A}^{n}$ be an affine $K$-scheme whose coordinate
ring is $K[x_{1},\dots,x_{n}]/(f_{1},\dots,f_{k})$. Then the coordinate
ring of $J_{m}(X)$ is 
\[
K[x_{1},\dots,x_{n},x_{1}^{(1)},\dots,x_{n}^{(1)},\dots,x_{1}^{(m)},\dots,x_{n}^{(m)}]/(\{f_{j}^{(u)}\}_{j=1,u=1}^{k,m}),
\]
where $f_{i}^{(u)}$ is the $u$-th formal derivative of $f_{i}$.
Here, the zero section is $s_{m}(x_{1},\dots,x_{n})=(x_{1},\dots,x_{n},0,\dots,0)\in J_{m}(X)$. 
\begin{example}
~\label{exa:some examples}
\begin{enumerate}
\item $J_{m}(\mathbb{A}^{n})\simeq\mathbb{A}^{n(m+1)}$. 
\item Let $X=\mathrm{Spec}\left(K[x]/(x^{n})\right)$. Then 
\[
J_{2}(X)=\mathrm{Spec}\left(K[x,x^{(1)},x^{(2)}]/(x^{n},nx^{n-1}x^{(1)},n(n-1)x^{n-2}(x^{(1)})^{2}+nx^{n-1}x^{(2)})\right).
\]
\item For any smooth $K$-scheme $X$, the first jet scheme $J_{1}(X)$
can be identified with the tangent bundle $TX$. 
\end{enumerate}
\end{example}

Let $\varphi:\mathbb{A}^{n_{1}}\rightarrow\mathbb{A}^{n_{2}}$ be
a morphism between affine spaces. Then $J_{m}(\varphi):\mathbb{A}^{n_{1}(m+1)}\rightarrow\mathbb{A}^{n_{2}(m+1)}$
is given by formally deriving $\varphi$, $J_{m}(\varphi)=(\varphi,\varphi^{(1)},\dots,\varphi^{(m)})$.
By functoriality, the $m$-th jet $J_{m}(\varphi)$ of a morphism
$\varphi:X\rightarrow Y$ of affine $K$-schemes, is given by the
formal derivative of $\varphi$. 
\begin{example}
\label{exa:example for jet of the commutator map}Let $\phi:\mathfrak{gl}_{n}^{2}\rightarrow\mathfrak{gl}_{n}$
be the commutator map $\phi(X,Y)=[X,Y]$. Then 
\begin{align*}
\phi_{1}(X,Y,X^{(1)},Y^{(1)}) & =\left([X,Y],XY^{(1)}+X^{(1)}Y-YX^{(1)}-Y^{(1)}X\right)\\
 & =\left([X,Y],[X,Y^{(1)}]+[X^{(1)},Y]\right),
\end{align*}
where $X^{(1)}$ is the formal derivative of $X$ (by deriving each
entry). We will use these kind of computations in Sections \ref{sec:Lie-algebra-word},
\ref{sec:Proof-of-Theorems- Lie algebra} and \ref{sec:The-commutator-map-revisited}. 
\end{example}

Here are two useful facts: 
\begin{lem}
\label{lem:useful facts on jet schemes}~
\begin{enumerate}
\item If $X$ is a smooth $K$-variety, then $J_{m}(X)$ is smooth of dimension
$(m+1)\mathrm{dim}X$ for any $m\geq1$, and all the truncation maps
$\text{\ensuremath{\pi_{n}^{m}}}:J_{m}(X)\rightarrow J_{n}(X)$ are
locally trivial fibrations with fiber $\mathbb{A}^{\mathrm{dim}X(m-n)}$. 
\item If $\varphi:X\rightarrow Y$ is an \'etale morphism then the commutative
diagram 
\[
\begin{array}{ccc}
J_{m}(X) & \stackrel{J_{m}(\varphi)}{\longrightarrow} & J_{m}(Y)\\
\downarrow\pi_{0}^{m} & \, & \downarrow\pi_{0}^{m}\\
X & \stackrel{\varphi}{\longrightarrow} & Y
\end{array}
\]
is a fibered diagram (and in particular $J_{m}(\varphi)$ is \'etale).
\end{enumerate}
\end{lem}

\subsection{\label{subsec:Log-cannonical-threshold}Log canonical threshold}

Let $X$ be a smooth $K$-variety and let $Z\subseteq X$ be a closed
subscheme. A \textsl{log resolution} $\pi:Y\rightarrow X$ of $(X,Z)$
is a proper, birational morphism from a smooth variety $Y$, such
that $D:=\pi^{-1}(Z)$ is a Cartier divisor and $D+K_{Y/X}$ is a
divisor with simple normal crossings, where $K_{Y/X}$ is the relative
canonical divisor defined by the determinant of the Jacobian of $\pi$
(see \cite[Section 1.1]{Mus12}). Such a resolution exists by Hironaka
\cite{Hir64}. In particular $K_{Y/X}$ and $D$ can be written as
$K_{Y/X}=\sum\limits _{i=1}^{N}k_{i}E_{i}$ and $D=\sum\limits _{i=1}^{N}a_{i}E_{i}$,
where the $E_{i}$ are distinct smooth irreducible divisors.

The \textsl{log canonical threshold} of the pair $(X,Z)$ is defined
as $\mathrm{lct}(X,Z):=\min\limits _{x\in X}\mathrm{lct}_{x}(X,Z)$,
where $\mathrm{lct}_{x}(X,Z):=\underset{i:x\in\pi(E_{i})}{\mathrm{min}}\frac{k_{i}+1}{a_{i}}$.
By convention, if $x\notin Z$, we set $\mathrm{lct}_{x}(X,Z)=\infty$.
For $X$ affine and smooth, with coordinate ring $K[X]$, and an ideal
$\mathfrak{a}\subseteq K[X]$, we write $\mathrm{lct}(\mathfrak{a}):=\mathrm{lct}(X,V(\mathfrak{a}))$,
where $V(\mathfrak{a})$ is the vanishing set of $\mathfrak{a}$.
In general, lower values of $\mathrm{lct}(X,Z)$ correspond to "wilder"
singularities of the pair $(X,Z)$. If $\mathrm{lct}(X,Z)\geq1$ then
the pair $(X,Z)$ is \emph{log canonical}. If $Z$ is a smooth subscheme,
then $\mathrm{lct}(X,Z)=\dim X-\dim Z$, which is the largest possible
value. 

Musta\c{t}\u{a} showed that the log canonical threshold can be characterized
in terms of the growth rate of the dimensions of the jet schemes of
$Z$: 
\begin{thm}[{\cite[Corollary 0.2]{Mus02}, \cite[Corollary 7.2.4.2]{CLNS18}}]
\label{thm:log canonical threshold and jet schemes}Let $X$ be a
smooth, geometrically irreducible $K$-variety, and let $Z\subsetneq X$
be a closed subscheme. Then 
\[
\mathrm{lct}(X,Z)=\mathrm{dim}X-\underset{m\geq0}{\mathrm{sup}}\frac{\mathrm{dim}J_{m}(Z)}{m+1}.
\]
\end{thm}

Note that $\mathrm{lct}(X,Z)$ does not depend on the ambient space
$X$ (only on $\mathrm{dim}X$) or the embedding of $Z$ in $X$. 
\begin{example}
Let $\varphi_{i}(x)=x_{i}^{n_{i}}$ for $1\leq i\leq k$. Then $\mathrm{lct}(\mathbb{A}_{\complex}^{1},\varphi_{i}^{-1}(0))=\mathrm{lct}_{0}(\mathbb{A}_{\complex}^{1},\varphi_{i}^{-1}(0))=\frac{1}{n_{i}}$.
By a Thom-Sebastiani type result \cite[Corollary 1]{MSS18} we have
\[
\mathrm{lct}(\mathbb{A}_{\complex}^{k},(\varphi_{1}*\dots*\varphi_{k})^{-1}(0))=\mathrm{min}\{1,\sum\limits _{i=1}^{k}\mathrm{lct}(\mathbb{A}_{\complex}^{1},\varphi_{i}^{-1}(0))\}=\mathrm{min}\{1,\sum\limits _{i=1}^{k}\frac{1}{n_{i}}\}.
\]
We therefore see that the convolution operation improves the log canonical
threshold. 
\end{example}

The following properties of the log canonical threshold will be useful
for us. 
\begin{fact}[{see \cite[Section 1]{Mus12}}]
\label{fact:properties of log cannonical threshold}Let $X$ be a
smooth, geometrically irreducible, affine $K$-variety. 
\begin{enumerate}
\item Let $0\neq\mathfrak{a}_{1}\subseteq\mathfrak{a}_{2}$ be two ideals
in $K[X]$. Then $\mathrm{lct}_{x}(\mathfrak{a}_{1})\leq\mathrm{lct}_{x}(\mathfrak{a}_{2})$
for any $x\in X$. 
\item Let $\mathfrak{a}\subseteq K[X]$ be an ideal. Then for any $x\in X$
we have $\mathrm{lct}_{x}(\mathfrak{a})\geq\frac{1}{\mathrm{ord}_{x}(\mathfrak{a})}$,
where 
\[
\mathrm{ord}_{x}(\mathfrak{a})=\underset{f\in\mathfrak{a}}{\mathrm{min}}~\mathrm{ord}_{x}(f),
\]
and $\mathrm{ord}_{x}(f)$ is the minimal $\left|\alpha\right|$ such
that $\frac{\partial^{\alpha}f}{\partial x^{\alpha}}(x)\neq0$. 
\end{enumerate}
\end{fact}

\subsection{\label{subsec:Degeneations}A method of degeneration}
\begin{defn}
\label{def:S-degeneration}Let $X$ and $S$ be $K$-schemes. An \textit{$S$-deformation}
of $X$ is a flat $K$-morphism $\pi:\widetilde{X}\rightarrow S$
such that $X_{s}\simeq X$ for some $s\in S$. Assume we are given
a deformation $\pi:\widetilde{X}\rightarrow S$ such that $\widetilde{X}_{s}\simeq X$
for any $s$ in a punctured neighborhood $U$ of $s_{0}\in S$. In
this situation we say that $\widetilde{X}_{s_{0}}$ is an\textit{
$S$-degeneration} of $X$. 
\end{defn}

\begin{prop}
\label{Prop:properties preserved under deformations}Let $S$ be a
$K$-variety, let $X$ and $Y$ be smooth $S$-varieties, with structure
maps $\pi_{X}:X\rightarrow S$ and $\pi_{Y}:Y\rightarrow S$, and
let $\varphi:X\rightarrow Y$ be an $S$-morphism. Let $P$ be one
of the following properties of morphisms: 
\begin{enumerate}
\item Flatness. 
\item Flatness with reduced fibers. 
\item Flatness with locally integral (i.e.~locally irreducible and reduced)
fibers. 
\item Flatness with normal fibers. 
\item (FRS). 
\item Smoothness. 
\end{enumerate}
Then the set of points $x\in X$ for which the fiber map $\varphi_{\pi_{X}(x)}:X_{\pi_{X}(x)}\longrightarrow Y_{\pi_{X}(x)}$
is $P$ at $x$ is open. 
\end{prop}

\begin{proof}
Let $x\in X$ and $s_{x}:=\pi_{X}(x)$. Assume that $\varphi_{s_{x}}:X_{s_{x}}$$\rightarrow Y_{s_{x}}$
is $P$ at $x$. Since $X$ and $Y$ are $S$-smooth, by miracle flatness
it follows for any $x'\in X$, that $\varphi_{s_{x'}}$ is flat at
$x'$ if and only if $\varphi$ is flat at $x'$. This implies $(1)$.
By restricting to an open neighborhood, we may assume that $\varphi$
is flat. Note that $(X_{s_{x}})_{\varphi_{s_{x}}(x),\varphi_{s_{x}}}\simeq X_{\varphi(x),\varphi}$.
Now, since $\varphi$ is flat it follows by \cite[Theorems 12.1.1, 12.1.6]{Gro66}
and by \cite{Elk78} that the set 
\[
\{x\in X:\varphi\text{ is }P\text{ at }x\}=\{x\in X:\varphi_{s_{x}}\text{ is }P\text{ at }x\}
\]
is open. 
\end{proof}
\begin{notation}
We denote by $\mathcal{P}$ the collection of singularity properties
$P$ appearing in Proposition \ref{Prop:properties preserved under deformations}.
\end{notation}

Given an $S$-degeneration $\widetilde{\varphi}:\widetilde{X}\rightarrow\widetilde{Y}$
of a morphism $\varphi:X\rightarrow Y$ to a morphism $\varphi_{0}:=\widetilde{\varphi}_{s_{0}}:\widetilde{X}_{s_{0}}\rightarrow\widetilde{Y}_{s_{0}}$,
and a property $P\in\mathcal{P}$, one can obtain information on the
$P$-locus of $\varphi$ using the $P$-locus of $\varphi_{0}$. This
is manifested in the following corollary: 
\begin{cor}[{cf.~\cite[Corollary 2.3]{AA16}}]
\label{cor:reduction to degeneration}Let $\widetilde{X}$ and $\widetilde{Y}$
be smooth $\mathbb{A}^{1}$-varieties with a $\mathbb{G}_{m}$ action
which is compatible with the action of $\mathbb{G}_{m}$ on $\mathbb{A}^{1}$.
Let $\widetilde{\varphi}:\widetilde{X}\rightarrow\widetilde{Y}$ be
a $\mathbb{G}_{m}$-equivariant $\mathbb{A}^{1}$-morphism, let $s:\mathbb{A}^{1}\rightarrow\widetilde{X}$
be a $\mathbb{G}_{m}$-equivariant section of the structure map $\widetilde{X}\rightarrow\mathbb{A}^{1}$
and let $P\in\mathcal{P}$. Assume that $\widetilde{\varphi}_{0}$
is $P$ at $s(0)$. Then $\widetilde{\varphi}_{1}$ is $P$ at $s(1)$. 
\end{cor}

Let $X$ be an algebraic $K$-variety and let $x\in X(\overline{K})$.
The local ring $\mathcal{O}_{X,x}$ has natural descending filtration
$\mathfrak{m}_{X,x}^{i}$, which induces a graded ring $\mathrm{gr}(\mathcal{O}_{X,x})=\bigoplus\limits _{i\geq0}\mathfrak{m}_{X,x}^{i}/\mathfrak{m}_{X,x}^{i+1}$.
We define the \textit{geometric tangent cone} $C_{x}(X)$ of $X$
at $x$ to be $\mathrm{Spec}(\mathrm{gr}(\mathcal{O}_{X,x})_{\mathrm{red}})$,
where $\mathrm{gr}(\mathcal{O}_{X,x})_{\mathrm{red}}$ is the quotient
of $\mathrm{gr}(\mathcal{O}_{X,x})$ by its nilradical. Any morphism
$\varphi:X\rightarrow Y$ induces a map $\varphi^{*}:\mathcal{O}_{Y,\varphi(x)}\rightarrow\mathcal{O}_{X,x}$,
which yields a morphism of graded rings $\mathrm{gr}(\varphi^{*}):\mathrm{gr}(\mathcal{O}_{Y,\varphi(x)})\rightarrow\mathrm{gr}(\mathcal{O}_{X,x})$
and therefore induces a map $\mathrm{gr}(\varphi):C_{x}X\rightarrow C_{\varphi(x)}Y$
between the tangent cones.

Since the map $\mathrm{gr}(\varphi)$ may be trivial, we would like
to consider a slightly more general construction. Let $l\in\nats$
satisfy $\varphi^{*}(\mathfrak{m}_{Y,\varphi(x)})\subseteq\mathfrak{m}_{X,x}^{l}$,
and define the following filtrations on $\mathcal{O}_{X,x}$ and $\mathcal{O}_{Y,\varphi(x)}$.
For $\mathcal{O}_{X,x}$ we set $F^{i}\mathcal{O}_{X,x}=\mathfrak{m}_{X,x}^{i}$
for $i>0$ and $F^{0}\mathcal{O}_{X,x}=\mathcal{O}_{X,x}$. For $\mathcal{O}_{Y,\varphi(x)}$
we define $F^{i}\mathcal{O}_{Y,\varphi(x)}=\mathfrak{m}_{Y,\varphi(x)}^{\left\lfloor i/l\right\rfloor }$
for $i>0$ and $F^{0}\mathcal{O}_{Y,\varphi(x)}=\mathcal{O}_{Y,\varphi(x)}$.
Our choice of $l$ guarantees that the map $\varphi^{*}$ is filtration-preserving,
so $\varphi$ induces a map on the tangent cones, denoted by
\[
D_{x}^{l-1}:=\mathrm{gr}_{F}(\varphi):C_{x}X\rightarrow C_{\varphi(x)}Y.
\]
By an appropriate choice of $l$ we may guaranty that $D_{x}^{l-1}$
is non-zero. In \cite[Section 2.1.1]{AA16} it is shown that any morphism
$\varphi:X\rightarrow Y$ can be degenerated to the morphism $D_{x}^{l}\varphi:C_{x}X\rightarrow C_{\varphi(x)}Y$
for a suitable $l$. Corollary \ref{cor:reduction to degeneration}
gives the following: 
\begin{cor}[{cf.~\cite[Corollary 2.14]{AA16}, the linearization method}]
\label{cor:(cf.---Linearization)}Let $\varphi:X\rightarrow Y$ be
a morphism of affine varieties, let $x\in X(\overline{K})$, and let
$P\in\mathcal{P}$. Suppose that $D_{x}^{l}\varphi$ is $P$ at $x$
for some $l$ for which $D_{x}^{l}\varphi$ is defined. Then $\varphi$
is $P$ at $x$. 
\end{cor}

\begin{rem}
Let $\varphi:\mathbb{A}^{n}\rightarrow\mathbb{A}^{m}$ be a morphism
and assume that all derivatives at $(a_{1},\dots,a_{n})\in K^{n}$
of $\varphi$ of order $\leq l$ vanish. Then $D_{x}^{l}\varphi$
is just the Taylor expansion of order $(l+1)$ at $(a_{1},\dots,a_{n})$
of the polynomial map $\varphi$. 
\end{rem}

The linearization method allows us to reduce the problem of showing
that a morphism satisfies a certain singularity property $P\in\mathcal{P}$
to showing that some map between affine spaces satisfies $P$. We
now treat the latter; for any $\omega=(\omega_{1},\dots,\omega_{n})\in\mathbb{Z}^{n}$
(referred to as a \textsl{weight}), we define the \textit{$\omega$-degree}
of a monomial $x_{1}^{a_{1}}\cdot\ldots\cdot x_{n}^{a_{n}}\in K[x_{1},\dots,x_{n}]$
to be $\mathrm{deg}_{\omega}(x_{1}^{a_{1}}\cdot\ldots\cdot x_{n}^{a_{n}}):=\sum\limits _{i=1}^{n}a_{i}\omega_{i}$.
For any polynomial $f\in K[x_{1},\dots,x_{n}]$, we define the \textit{symbol}
of $f$, denoted $\sigma_{\omega}(f)$, to be the sum of the monomials
of $f$ with the lowest $\omega$-degree. Notice that the degree induced
from the weight $\omega=(1,\dots,1)$ corresponds to the standard
degree of monomials $\mathrm{deg}(x_{1}^{a_{1}}\cdot\ldots\cdot x_{n}^{a_{n}})=\sum\limits _{i=1}^{n}a_{i}$.
By choosing appropriate filtrations, and applying Corollary \ref{cor:reduction to degeneration},
one can obtain the following: 
\begin{cor}[{cf.~\cite[Corollary 2.17]{AA16}, the elimination method}]
\label{cor:(cf.---Elimination)}Let $\varphi=(\varphi_{1},\dots,\varphi_{m}):\mathbb{A}^{n}\rightarrow\mathbb{A}^{m}$
be a morphism, $\omega\in\mathbb{Z}^{n}$ be a weight, $P\in\mathcal{P}$
and let $\sigma_{\omega}(\varphi)=(\sigma_{\omega}(\varphi_{1}),\dots,\sigma_{\omega}(\varphi_{m}))$
be the symbol of $\varphi$. If $\sigma_{\omega}(\varphi)$ is $P$
at $(0,\dots,0)$, then so is $\varphi$. 
\end{cor}

The above degeneration methods can be applied to jet morphisms. Given
a $K$-variety $X$, there is a natural $\mathbb{G}_{m}$-action $\phi:\mathbb{G}_{m}\times J_{m}(X)\rightarrow J_{m}(X)$
defined on the level of $A$-valued points for any commutative ring
$A$; for any $a\in A^{\times}$ and $f:\mathrm{Spec}A[t]/(t^{m+1})\rightarrow X$
we set
\[
\phi(a,f)=f\circ g_{a},
\]
where $g_{a}:\mathrm{Spec}A[t]/(t^{m+1})\rightarrow\mathrm{Spec}A[t]/(t^{m+1})$
is induced from $t\mapsto at$.

Given a morphism $\varphi:X\rightarrow Y$, the map $J_{m}(\varphi):J_{m}(X)\rightarrow J_{m}(Y)$
is $\mathbb{G}_{m}$-equivariant. 
\begin{example}
Let $X=\mathrm{Spec}\left(\complex[x_{1},\dots,x_{n}]/(f_{1},\dots,f_{k})\right)$.
Then any point $b\in J_{m}(X)(\complex)$ is of the form $b=(b_{1},\dots,b_{n},\dots,b_{1}^{(m)},\dots,b_{n}^{(m)})$,
where $f_{j}^{(u)}(b)=0$ for $1\leq j\leq k$ and $0\leq u\leq m$.
The above $\mathbb{G}_{m}$-action is then given by 
\[
s.b=(b_{1},\dots,b_{n},sb_{1}^{(1)},\dots,sb_{n}^{(1)}\dots,s^{m}b_{1}^{(m)},\dots,s^{m}b_{n}^{(m)}).
\]
Note that $f_{j}^{(u)}(s.b)=s^{u}f_{j}^{(u)}(b)=0$ as expected. 
\end{example}

\begin{cor}
\label{cor:degeneration on jet schemes}Let $\varphi:X\rightarrow Y$
be a morphism between smooth $K$-varieties and let $P\in\mathcal{P}$.
Assume that $J_{m}(\varphi):J_{m}(X)\rightarrow J_{m}(Y)$ is $P$
at $(x,0)\in J_{m}(X)$ for every $x\in X$. Then $J_{m}(\varphi)$
is $P$. 
\end{cor}

\begin{proof}
It is enough to prove the claim for $X$ and $Y$ affine. Set $\widetilde{J_{m}}(X)=J_{m}(X)\times\mathbb{A}^{1}$,
$\widetilde{J_{m}}(Y)=J_{m}(Y)\times\mathbb{A}^{1}$ and $\widetilde{J_{m}}(\varphi):\widetilde{J_{m}}(X)\rightarrow\widetilde{J_{m}}(Y)$
by 
\[
\widetilde{J_{m}}(\varphi)\left((x,x^{(1)},\dots,x^{(m)}),t\right)=\left(J_{m}(\varphi)(x,x^{(1)},\dots,x^{(m)}),t\right).
\]
Note that $\widetilde{J_{m}}(\varphi)$ is $\mathbb{G}_{m}$-equivariant.
Taking the $\mathbb{G}_{m}$-equivariant section $s(t)=(x,tx^{(1)},\dots,t^{m}x^{(m)})$,
and using the assumption that $J_{m}(\varphi)$ is $P$ at $s(0)=(x,0)$,
by Corollary \ref{cor:reduction to degeneration} we get that $J_{m}(\varphi)$
is $P$ at $s(1)=(x,x^{(1)},\dots,x^{(m)})$.
\end{proof}

\section{\label{sec:Singulrities-of-word}Some properties of the algebraic
convolution operation}

In this section we introduce various singularity properties and discuss
their behavior with respect to the algebraic convolution operation.
We first recall some results from \cite{GH} and \cite{GH19} about
morphisms from a smooth $K$-scheme to a $K$-algebraic group. Throughout
this section we assume that $K$ is a field of characteristic zero.

\subsection{The algebraic convolution operation: the case of general morphisms}
\begin{defn}
\label{def:normality and (FAI)}Let $\varphi:X\to Y$ be a morphism
between $K$-schemes. 
\begin{enumerate}
\item $\varphi$ is called \textit{normal at $x\in X$} if $\varphi$ is
flat at $x$ and the fiber $X_{\varphi(x)}$ is geometrically normal
at $x$. $\varphi$ is called \textit{normal }if it is normal at every
$x\in X$. 
\item $\varphi$ is called \textit{(FGI)} if it is flat with geometrically
irreducible fibers. 
\end{enumerate}
\end{defn}

\begin{rem}
Since by convention the empty scheme is reducible, an (FGI) morphism
is in particular surjective. 
\end{rem}

In \cite{GH19} and \cite{GH}, we have shown that the convolution
operation improves various singularity properties of morphisms in
the following sense: 
\begin{prop}[{\cite[Proposition 3.1, Lemma 3.4]{GH}}]
\label{prop: properties preserved under convolution}Let $X$ and
$Y$ be smooth geometrically irreducible $K$-varieties, $\underline{G}$
be an algebraic $K$-group and let $S$ be any of the following properties
of morphisms: 
\begin{enumerate}
\item Dominance. 
\item Flatness (or flatness with reduced fibers). 
\item Normality. 
\item Smoothness. 
\item (FGI). 
\item (FRS). 
\end{enumerate}
Suppose $\varphi:X\to\underline{G}$ has property $S$ and $\psi:Y\to\underline{G}$
is any morphism. Then the maps $\varphi*\psi:X\times Y\to\underline{G}$
and $\psi*\varphi:Y\times X\to\underline{G}$ have the property $S$. 
\end{prop}

By taking enough self-convolutions of a given dominant morphism, one
can obtain increasingly better singularity properties, and eventually
obtain an (FRS) morphism. 
\begin{prop}[{cf.~\cite[Theorem B]{GH}}]
\label{Prop: singularity properties obtained after convolution}Let
$m\in\nats$, let $X_{1},\ldots,X_{m}$ be geometrically irreducible
smooth $K$-varieties, $\underline{G}$ be a connected algebraic $K$-group
and let $\{\varphi_{i}:X_{i}\rightarrow\underline{G}\}_{i=1}^{m}$
be a collection of dominant morphisms. 
\begin{enumerate}
\item For any $1\leq i,j\leq m$ the morphism $\varphi_{i}*\varphi_{j}$
is surjective. 
\item If $m\geq\mathrm{dim}\underline{G}$ then $\varphi_{1}*\dots*\varphi_{m}$
is flat. 
\item If $m\geq\mathrm{dim}\underline{G}+1$ then $\varphi_{1}*\dots*\varphi_{m}$
is flat with reduced fibers. 
\item If $m\geq\mathrm{dim}\underline{G}+k$ then $\varphi_{1}*\dots*\varphi_{m}$
is flat with normal fibers which are regular in codimension $k-1$. 
\item The bounds in $(2)-(4)$ are tight (e.g.~consider $\varphi(x_{1},\dots,x_{m})=(x_{1}^{2},\left(x_{1}x_{2}\right)^{2},\left(x_{1}x_{3}\right)^{2},\left(x_{1}x_{m}\right)^{2})$,
see \cite[Proposition 3.6]{GH}). 
\end{enumerate}
\end{prop}

\begin{thm}[{\cite[Theorem A]{GH}}]
\label{Main result}Let $X$ be a geometrically irreducible smooth
$K$-variety, $\underline{G}$ be a connected $K$-algebraic group
and let $\varphi:X\to\underline{G}$ be a dominant morphism. Then
there exists $N\in\mathbb{N}$ such that for any $t>N$, the $t$-th
convolution power $\varphi^{*t}$ is (FRS). 
\end{thm}

For the next sections we need several more auxiliary results. 
\begin{prop}
\label{prop:extended properties of convolutions}Let $\{X_{i}\}_{i\in\nats}$
be a collection of geometrically irreducible smooth $K$-varieties,
let $\underline{G}$ be a connected algebraic $K$-group, and let
$\{\varphi_{i}:X_{i}\rightarrow\underline{G}\}_{i\in\nats}$ be a
collection of dominant morphisms. Assume that $\varphi_{1}$ is flat.
Then 
\begin{enumerate}
\item $\varphi_{1}*\varphi_{2}$ is flat with reduced fibers. 
\item If $m\geq3$ then $\varphi_{1}*\dots*\varphi_{m}$ is flat with normal
fibers which are regular in codimension $m-2$. 
\end{enumerate}
\end{prop}

\begin{proof}
Let $X=\stackrel[i=1]{k}{\prod}X_{i}$. Notice that by Proposition
\ref{prop: properties preserved under convolution}, $\varphi:=\varphi_{1}*\varphi_{2}*\dots*\varphi_{k}:X\rightarrow\underline{G}$
is a flat morphism between smooth schemes, and therefore its fibers
are local complete intersections and in particular Cohen-Macaulay.
By Serre's criterion for normality and reducedness \cite[Proposition 5.8.5, Theorem 5.8.6]{Gro66}
it is enough to show that for any $k\geq2$ the fibers are regular
in codimension $k-2$.

Since $\varphi$ is flat, $x=(x_{1},\dots,x_{k})$ is a smooth point
of the fiber $X_{\varphi(x),\varphi}$ if and only if $\varphi$ is
smooth at $x$. For each $1\leq i\leq k$ denote by $\underline{G}^{\mathrm{sing},\varphi_{i}}$
the set of $g\in\underline{G}$ with $(X_{i})_{g,\varphi_{i}}$ non-smooth.
By generic smoothness, $\mathrm{dim}\varphi_{i}^{-1}(\underline{G}^{\mathrm{sing},\varphi_{i}})\leq\mathrm{dim}X_{i}-1$.
Note that by Proposition \ref{prop: properties preserved under convolution},
$\varphi$ is non-smooth at $x$ only if $\varphi_{i}$ is non-smooth
at $x_{i}$ for each $i$. Thus, set theoretically, we have the following:
\[
(\varphi^{-1}(g))^{\mathrm{sing}}\subseteq\bigcup_{(g_{2},\dots,g_{k})\in\underline{G}^{\mathrm{sing},\varphi_{2}}\times\ldots\times\underline{G}^{\mathrm{sing},\varphi_{k}}}\varphi_{1}^{-1}(gg_{k}^{-1}\dots g_{2}^{-1})\times\varphi_{2}^{-1}(g_{2})\times\dots\times\varphi_{k}^{-1}(g_{k}).
\]
In particular,
\begin{align*}
\mathrm{dim}(X_{g,\varphi}^{\mathrm{sing}}) & \leq\mathrm{dim}X_{1}-\mathrm{dim}\underline{G}+\sum_{i=2}^{k}\mathrm{dim}\varphi_{i}^{-1}(\underline{G}^{\mathrm{sing},\varphi_{i}})\\
 & \leq\sum_{i=1}^{k}\mathrm{dim}X_{i}-\mathrm{dim}\underline{G}-(k-1)=\mathrm{dim}X_{g,\varphi}-(k-1).
\end{align*}
\end{proof}
\begin{defn}
\label{def:generating morphism}Let $\varphi:X\rightarrow\underline{G}$
be a morphism between a geometrically irreducible $K$-variety $X$
and a connected algebraic $K$-group $\underline{G}$. Then $\varphi$
is called \textsl{generating}, if for any algebraic $\overline{K}$-subgroup
$\underline{H}\lneq\underline{G}_{\bar{K}}$, and any $g\in\underline{G}(\overline{K})$
we have $\varphi(X)\nsubseteq g\underline{H}$. 
\end{defn}

\begin{lem}[{cf.~\cite[Proposition 3.5]{GH19}}]
\label{lem:generating becomes dominant}Let $X$ be a geometrically
irreducible $K$-variety, $\underline{G}$ be a connected algebraic
$K$-group and $\varphi:X\rightarrow\underline{G}$ a generating morphism
whose image contains $e\in\underline{G}$. Then $\varphi^{*\mathrm{dim}\underline{G}}$
is dominant, and thus $\varphi^{*(\mathrm{dim}\underline{G})^{2}}$
is flat. 
\end{lem}

\begin{proof}
The proof is the same as in \cite[Proposition 3.5]{GH19}. It is only
left to observe that $\overline{\varphi(X)^{n}}$ must stabilize for
some $n\leq\mathrm{dim}\underline{G}$, and then using \cite[Proposition 3.5]{GH19}
we deduce that $\overline{\varphi(X)^{n}}=\underline{G}$, so that
$\varphi^{*n}$ is dominant. 
\end{proof}
The following will be useful in Section \ref{subsec:Matrix-word-maps}. 
\begin{lem}
\label{lem:Auxilery lemma for matrix word maps}Let $X$ and $Y$
be smooth geometrically irreducible $K$-varieties, $\underline{G}$
a  connected $K$-algebraic group, and $k\in\nats$. Let $\psi:X\rightarrow\underline{G}$
and $\varphi:Y\rightarrow\underline{G}$ be dominant morphisms such
that $\mathrm{dim}Y_{g,\varphi}\leq\mathrm{dim}Y-\mathrm{dim}\underline{G}+k$,
for each $g\in\underline{G}$. Then $\varphi*\psi^{*k}$ is flat. 
\end{lem}

\begin{proof}
By induction, it suffices to show that $\varphi*\psi$ satisfies 
\[
\mathrm{dim}\left((\varphi*\psi)^{-1}(g)\right)\leq\mathrm{dim}X+\mathrm{dim}Y-\mathrm{dim}\underline{G}+k-1.
\]
Set $S_{l}:=\{g\in\underline{G}:\mathrm{dim}X_{g,\psi}=\mathrm{dim}X-\mathrm{dim}\underline{G}+l\}$,
and notice that set theoretically 
\begin{equation}
(\varphi*\psi)^{-1}(g)=\bigcup_{l=0}^{\mathrm{\mathrm{dim}\underline{G}}}\bigcup_{g_{0}\in S_{l}}\varphi^{-1}(gg_{0}^{-1})\times\psi^{-1}(g_{0}),\label{eq:(3.1)}
\end{equation}
for any $g\in\varphi*\psi(X)$. Note that $\underline{G}=\stackrel[l=0]{\mathrm{dim}\underline{G}}{\bigcup}S_{l}$,~
$X=\stackrel[l=0]{\mathrm{dim}\underline{G}}{\bigcup}\psi^{-1}(S_{l})$
and that $S_{0}\subseteq\underline{G}$ and $X_{0}:=\psi^{-1}(S_{0})\subseteq X$
are open subsets, by generic flatness of $\psi$. Hence for every
$g\in G$ we have
\begin{equation}
\mathrm{dim}\left(\bigcup_{g_{0}\in S_{0}}\varphi^{-1}(gg_{0}^{-1})\times\psi^{-1}(g_{0})\right)=\mathrm{dim}X+\mathrm{dim}Y-\mathrm{dim}\underline{G}.\label{eq:(3.2)}
\end{equation}
By the geometric irreducibility of $X$ we have $\mathrm{dim}\psi^{-1}(S_{l})\leq\mathrm{dim}X-1$
for every $l>0$, and therefore by (\ref{eq:(3.1)}) and (\ref{eq:(3.2)})
\[
\mathrm{dim}(\varphi*\psi)^{-1}(g)\leq\mathrm{dim}X+\mathrm{dim}Y-\mathrm{dim}\underline{G}+k-1.
\]
\end{proof}

\subsection{Singularity properties of schemes and morphisms via their jets }

In \cite{Mus01} and \cite{EM04}, it was shown that various singularity
properties of a scheme $X$ can be detected in terms of its jet schemes: 
\begin{thm}
\label{thm: singularities and jets}If $X$ is a geometrically irreducible,
local complete intersection $K$-variety then the following hold: 
\begin{enumerate}
\item If $X$ is normal, then $J_{m}(X)$ is equidimensional for every $m\geq1$
(and hence $\mathrm{dim}J_{m}(X)=(m+1)\mathrm{dim}X$) if and only
if $X$ has log canonical singularities. 
\item $J_{m}(X)$ is geometrically irreducible for all $m\geq1$ if and
only if $X$ has rational singularities. 
\item $J_{m}(X)$ is normal for all $m\geq1$ if and only if $X$ has terminal
singularities. 
\end{enumerate}
\end{thm}

This motivates us to introduce the following definitions: 
\begin{defn}
\label{def:jet flat}Let $\varphi:X\rightarrow Y$ be a $K$-morphism
of smooth varieties. 
\begin{enumerate}
\item $\varphi$ is called\textit{ jet-flat} (resp.\textit{~jet-normal}\textsl{/jet-(FGI)})
if $J_{m}(\varphi):J_{m}(X)\rightarrow J_{m}(Y)$ is flat (resp.~normal/(FGI))
for every $m\geq0$. 
\item $\varphi$ is called \textit{(FLCS) }(resp\textit{.~(FTS)}) if $\varphi$
is flat, with normal fibers of log canonical singularities (resp.~terminal
singularities). 
\end{enumerate}
\end{defn}

We now prove the following relative version of Theorem \ref{thm: singularities and jets}. 
\begin{cor}
\label{cor: singularity properties through dim of jets}Let $\varphi:X\rightarrow Y$
be an (FGI)-morphism of smooth geometrically irreducible $K$-varieties
$X$ and $Y$. 
\begin{enumerate}
\item If $\varphi$ is normal then $\varphi$ is (FLCS) if and only if $\varphi$
is jet-flat. 
\item $\varphi$ is (FRS) if and only if $\varphi$ is jet-(FGI). 
\item $\varphi$ is (FTS) if and only if $\varphi$ is jet-normal. 
\end{enumerate}
\end{cor}

\begin{rem}
We have $(FTS)\Rightarrow(FRS)\Rightarrow(FLCS)$.

For the proof of Corollary \ref{cor: singularity properties through dim of jets}
we need the following useful lemma. 
\end{rem}

\begin{lem}
\label{lem:irreducible generic fiber in jets}Let $\varphi:X\rightarrow Y$
be a dominant morphism between smooth geometrically irreducible $K$-varieties. 
\begin{enumerate}
\item Assume that $\varphi$ is smooth and (FGI). Then $\varphi$ is also
jet-(FGI). 
\item Assume that $\varphi$ has a geometrically irreducible generic fiber.
Then for any $m\in\nats$ the jet morphism $J_{m}(\varphi):J_{m}(X)\rightarrow J_{m}(Y)$
has a geometrically irreducible generic fiber. 
\end{enumerate}
\end{lem}

\begin{proof}
By \cite[Lemmas 37.25.5, 37.25.6]{Sta} and by generic flatness, the
generic fiber of $J_{m}(\varphi)$ is geometrically irreducible if
and only if $J_{m}(\varphi)$ is (FGI) over an open set in $J_{m}(Y)$.
By generic smoothness, it is enough to prove $(1)$. Write $d=\mathrm{dim}X-\mathrm{dim}Y$.
We may assume that $\varphi$ factors through $X\overset{\psi}{\rightarrow}Y\times\mathbb{A}^{d}\overset{\mathrm{pr}}{\rightarrow}Y$,
where $\psi$ an \'etale map, and $\mathrm{pr}$ the projection.
We can then write $J_{m}(\varphi)=J_{m}(\mathrm{pr})\circ J_{m}(\psi)$,
with $J_{m}(\psi):J_{m}(X)\rightarrow J_{m}(Y\times\mathbb{A}^{d})$
an \'etale map and $J_{m}(\mathrm{pr}):J_{m}(Y\times\mathbb{A}^{d})\rightarrow J_{m}(Y)$
is the projection. 

Write $x\in J_{m}(\mathbb{A}^{d})\simeq\mathbb{A}^{d(m+1)}$ as $x=(t,t')$
where $t\in\mathbb{A}^{d}$ and where $t'\in\mathbb{A}^{dm}$ corresponds
to the various derivatives. Since $\psi$ is \'etale, by Lemma \ref{lem:useful facts on jet schemes},
the following is a base change diagram: 
\[
\begin{array}{ccccc}
J_{m}(X) & \overset{J_{m}(\psi)}{\longrightarrow} & J_{m}(Y)\times J_{m}(\mathbb{A}^{d}) & \overset{\Phi}{\longrightarrow} & J_{m}(Y)\times\mathbb{A}^{dm}\\
\pi_{0}^{m}\downarrow & \, & \pi_{0}^{m}\downarrow & \, & \pi\downarrow\\
X & \overset{\psi}{\longrightarrow} & Y\times\mathbb{A}^{d} & \overset{\mathrm{pr}}{\longrightarrow} & Y
\end{array},
\]
where $\Phi$ is defined by $\Phi(y,y',t,t')=(y,y',t')$ (for $(y,y')\in J_{m}(Y)$),
and $\pi(y,y',t')=\pi_{0}^{m}(y,y')=y$. Notice that for any $(y,y',t')\in J_{m}(Y)\times\mathbb{A}^{dm}$,
the scheme theoretic fiber $J_{m}(X){}_{(y,y',t'),\Phi\circ J_{m}(\psi)}$
is a base change of the fiber $X_{y,\varphi}$ by a field extension,
so it is geometrically irreducible.

For any $(y,y')\in J_{m}(Y)$ consider the map $f:J_{m}(X)_{(y,y'),J_{m}(\varphi)}\rightarrow\{(y,y')\}\times\mathbb{A}^{dm}$
by $(x,x')\mapsto\Phi\circ J_{m}(\psi)(x,x')$. Then it is a flat
morphism with geometrically irreducible target, and such that all
fibers are geometrically irreducible, as they can be identified with
$J_{m}(X){}_{(y,y',t'),\Phi\circ J_{m}(\psi)}$. Applying the following
lemma to $f$, we deduce that $J_{m}(X)_{(y,y'),J_{m}(\varphi)}$
is geometrically irreducible for any $(y,y')\in J_{m}(Y)$, as required. 
\begin{lem}[{see \cite[Lemma 5.8.12]{Sta}}]
\label{Lemma 3.13-Auxilary}Let $f:X_{1}\rightarrow X_{2}$ be a
continuous map between topological spaces, such that $f$ is open,
$X_{2}$ is irreducible, and there exists a dense collection of points
$y\in X_{2}$ such that $f^{-1}(y)$ is irreducible. Then $X_{1}$
is irreducible. 
\end{lem}

\end{proof}
\begin{proof}[Proof of Corollary \ref{cor: singularity properties through dim of jets}]
$(1)$ and $(3)$ follow from Corollary \ref{cor:degeneration on jet schemes}
and Theorem \ref{thm: singularities and jets}. One direction of $(2)$
follows from Theorem \ref{thm: singularities and jets}. Assume that
$\varphi$ is (FRS). Then it follows from Theorem \ref{thm: singularities and jets}
that $J_{m}(\varphi):J_{m}(X)\rightarrow J_{m}(Y)$ is (FGI) over
$Y\subseteq J_{m}(Y)$. It is left to prove that it is (FGI).

By Corollary \ref{cor:degeneration on jet schemes} $J_{m}(\varphi)$
is flat. By \cite[Proposition 1.5]{Mus01} $J_{m}(\varphi)$ is flat
with fibers which are locally integral at each $x\in X\subseteq X_{m}$.
By Corollary \ref{cor:degeneration on jet schemes}, all of the fibers
of $J_{m}(\varphi)$ are locally integral, and in particular all geometrically
irreducible components of the any fiber of $J_{m}(\varphi)$ are disjoint.

Since $X,Y$ are smooth and $J_{m}(\varphi)$ is flat, we have $J_{m}(X^{\mathrm{sm,\varphi}})\simeq J_{m}(X)^{\mathrm{sm},J_{m}(\varphi)}$.
Indeed, by writing the Jacobian of $J_{m}(\varphi)$, one can see
that $J_{m}(\varphi)$ is a submersion at $(x,x')\in J_{m}(X)$ only
if $\varphi$ is a submersion at $x\in X$, and in the other direction,
$\varphi|_{X^{\mathrm{sm,\varphi}}}$ is smooth and thus $J_{m}(\varphi)|_{J_{m}(X^{\mathrm{sm,\varphi}})}$
is smooth.

Let $(x,x')\in J_{m}(X)$ and set $(y,y'):=J_{m}(\varphi)(x,x')$.
Let $Z$ be any geometrically irreducible component of $J_{m}(X)_{(y,y'),J_{m}(\varphi)}$.
By Lemma \ref{lem:irreducible generic fiber in jets}, $J_{m}(\varphi)|_{X^{\mathrm{sm,\varphi}}}$
has geometrically irreducible fibers, and thus $J_{m}(X)_{(y,y'),J_{m}(\varphi)}\cap J_{m}(X)^{\mathrm{sm},J_{m}(\varphi)}$
is geometrically irreducible. But since $J_{m}(\varphi)$ is flat
with reduced fibers, it follows that $J_{m}(X)_{(y,y'),J_{m}(\varphi)}\cap J_{m}(X)^{\mathrm{sm},J_{m}(\varphi)}$
is dense in $J_{m}(X)_{(y,y'),J_{m}(\varphi)}$ which implies that
$J_{m}(X)_{(y,y'),J_{m}(\varphi)}$ is geometrically irreducible. 
\end{proof}
We would like to prove a version of Corollary \ref{cor: singularity properties through dim of jets}
for non-(FGI) morphisms: 
\begin{cor}
\label{cor:(FRS) implies expected number of componnents}Let $\varphi:X\rightarrow Y$
be an (FRS) morphism between smooth geometrically irreducible $K$-varieties
$X$ and $Y$. Then for any $m\in\nats$, $y\in\varphi(X)$ and $\tilde{y}\in J_{m}(Y)$
with $\pi_{0}^{m}(\tilde{y})=y$, the schemes $\ensuremath{J_{m}(X)_{\tilde{y},J_{m}(\varphi)}}$
and $X_{y,\varphi}$ have the same number of geometrically irreducible
components. 
\end{cor}

\begin{proof}
We may assume that $Y$ is affine. Let us first assume that $\varphi$
is smooth. For any $(y,y^{(1)},...,y^{(m)})\in J_{m}(Y)$, the truncation
map $\pi_{m-1}^{m}:J_{m}(X)\rightarrow J_{m-1}(X)$ induces a morphism
\[
\psi_{m}:J_{m}(X)_{(y,y^{(1)},...,y^{(m)}),J_{m}(\varphi)}\rightarrow J_{m-1}(X)_{(y,y^{(1)},\dots,y^{(m-1)}),J_{m-1}(\varphi)}.
\]
Note that the fibers of $\psi_{m}$ are isomorphic to $\mathbb{A}^{\mathrm{dim}X-\mathrm{dim}Y}$.
Indeed, we may assume that $X$ is affine, and write $J_{m}(\varphi)=(\varphi,\varphi^{(1)},\dots,\varphi^{(m)})$.
For any $(x,x^{(1)},\dots,x^{(m-1)})\in J_{m-1}(X)_{(y,y^{(1)},\dots,y^{(m-1)}),J_{m-1}(\varphi)}$,
we have 
\[
\psi_{m}^{-1}(x,x^{(1)},\dots,x^{(m-1)})\simeq\{z:\varphi^{(m)}(x,x^{(1)},\dots,x^{(m-1)},z)=y^{(m)}\}.
\]
Since $\varphi^{(m)}$ is linear in $x^{(m)}$ where $x,x^{(1)},\dots,x^{(m-1)}$
are parameters, we are done. Restricting to each geometrically irreducible
component of $X_{y,\varphi}$, and applying Lemma \ref{Lemma 3.13-Auxilary}
to $\psi_{m}$, we deduce by induction on $m$ that $J_{m}(X)_{(y,y^{(1)},...,y^{(m)}),J_{m}(\varphi)}$
and $X_{y,\varphi}$ have the same number of geometrically irreducible
components.

Since we proved the corollary for smooth morphisms, we deduce that
$J_{m}(X)_{\tilde{y},J_{m}(\varphi)}\cap J_{m}(X^{\mathrm{sm,\varphi}})$
has the same number of geometrically irreducible components as $\text{\ensuremath{X_{y,\varphi}}\ensuremath{\ensuremath{\cap}}}X^{\mathrm{sm,\varphi}}$,
which is the same as in $X_{y,\varphi}$. But we have seen that $J_{m}(X^{\mathrm{sm,\varphi}})$
intersects any irreducible component of $J_{m}(X)_{(y,y'),J_{m}(\varphi)}$
and they are all disjoint, so we are done. 
\end{proof}
The $m$-th jet scheme $J_{m}(\underline{G})$ of an algebraic $K$-group
$\underline{G}$ has a natural structure of an algebraic $K$-group,
where the multiplication map $\mathrm{m}_{J_{m}(\underline{G})}:J_{m}(\underline{G})\times J_{m}(\underline{G})\rightarrow J_{m}(\underline{G})$
is just the $m$-th jet $J_{m}(\mathrm{m}_{\underline{G}})$ of the
usual multiplication map $\mathrm{m}_{\underline{G}}$. In particular,
we have the following lemma: 
\begin{lem}
\label{lem:convolution behaves well with jets}Let $X_{1},X_{2}$
be $K$-schemes, $\underline{G}$ an algebraic $K$-group, and let
$\{\varphi_{i}:X_{i}\rightarrow\underline{G}\}_{i=1}^{2}$ be morphisms.
Then $J_{m}(\varphi_{1})*J_{m}(\varphi_{1})=J_{m}(\varphi_{1}*\varphi_{2})$. 
\end{lem}

\begin{cor}
\label{cor:jet-flat implies FTS after two convolusions}Let $X_{1},X_{2},X_{3}$
be smooth geometrically irreducible $K$-varieties, and let $\{\varphi_{i}:X_{i}\rightarrow\underline{G}\}_{i=1}^{3}$
be dominant morphisms such that $\varphi_{1}:X_{1}\rightarrow\underline{G}$
is jet-flat. Then $\varphi_{1}*\varphi_{2}$ is (FRS) and $\varphi_{1}*\varphi_{2}*\varphi_{3}$
is (FTS). 
\end{cor}

\begin{proof}
By Proposition \ref{prop:extended properties of convolutions} and
Lemma \ref{lem:convolution behaves well with jets} it follows that
$J_{m}(\varphi_{1}*\varphi_{2})$ is flat with reduced fibers. In
particular, $\left(J_{m}(X_{1}\times X_{2})_{(g,0)}\right)^{\mathrm{sm}}\simeq\left(J_{m}((X_{1}\times X_{2})_{g})\right)^{\mathrm{sm}}$
is dense in $J_{m}((X_{1}\times X_{2})_{g})$ for any $g\in\underline{G}$.
We have also seen that $J_{m}(X_{1}\times X_{2})^{\mathrm{sm},J_{m}(\varphi_{1}*\varphi_{2})}\simeq J_{m}((X_{1}\times X_{2})^{\mathrm{sm},\varphi_{1}*\varphi_{2}})$
so
\begin{align*}
\left(J_{m}(X_{1}\times X_{2})_{(g,0)}\right)^{\mathrm{sm}} & =J_{m}(X_{1}\times X_{2})_{(g,0)}\cap J_{m}(X_{1}\times X_{2})^{\mathrm{sm},J_{m}(\varphi_{1}*\varphi_{2})}\\
 & =J_{m}((X_{1}\times X_{2})_{g}\cap(X_{1}\times X_{2})^{\mathrm{sm},\varphi_{1}*\varphi_{2}})=J_{m}((X_{1}\times X_{2})_{g}^{\mathrm{sm}}).
\end{align*}
By \cite[Proposition 1.4]{Mus01}, we deduce that $(X_{1}\times X_{2})_{g}$
has rational singularities. Proposition \ref{prop:extended properties of convolutions}
and Corollary \ref{cor: singularity properties through dim of jets}
imply that $\varphi_{1}*\varphi_{2}*\varphi_{3}$ is (FTS).
\end{proof}

\subsection{Singularity properties of word maps: first observations}

We have seen in Proposition \ref{Prop: singularity properties obtained after convolution}
that a dominant morphism $\varphi:X\rightarrow\underline{G}$ becomes
flat after $\mathrm{dim}\underline{G}$ self-convolutions, and one
cannot hope for better bounds in this generality. Since word maps
furnish an abundant of symmetries, one would expect their singularity
properties improve much faster under self-convolution than those of
a general morphism. For example, since word maps are conjugate invariant,
the locus of any given singularity property is conjugate invariant
as well. Indeed, in \cite{LST19} it was shown that any word map $\varphi_{w}:\underline{G}^{r}\rightarrow\underline{G}$
becomes flat after $O(\ell(w)^{4})$ self-convolutions, for every
semisimple group $\underline{G}$, where $\ell(w)$ is the length
of $w$. Similarly, one would expect improved bounds for Lie algebra
word maps.

In Sections \ref{sec:Lie-algebra-word} and \ref{sec:Proof-of-Theorems- Lie algebra}
we provide bounds on the number of self-convolutions of Lie algebra
word maps required to obtain flat and (FRS) morphisms. While the task
of providing effective bounds for the (FRS) property of group word
maps is not yet finished, we can still say something on the singularities
of the fibers of word maps (see Section \ref{sec:flatness-and-(FRS) of word maps}).
For this, we would like to introduce quantitative ways to measure
how bad are the singularities of a word map, or how far is a word
map from being an (FRS) morphism. We further discuss the (FGI) property
which plays an important role in the probabilistic applications of
this paper. We first recall a result of \cite{LST19}: 
\begin{thm}[{\cite[Lemma 2.4]{LST19}}]
\label{thm:-convolution of two word maps is generically absolutely irreducible}Let
$\underline{G}$ be a simply connected, semisimple $K$-algebraic
group and let $w_{1}\in F_{r_{1}}$ and $w_{2}\in F_{r_{2}}$ be words.
Then $\varphi_{w_{1}}*\varphi_{w_{2}}$ has a geometrically irreducible
generic fiber. 
\end{thm}

From Theorem \ref{thm:-convolution of two word maps is generically absolutely irreducible}
we deduce the following corollary:
\begin{cor}
\label{cor: word maps are (FAI) }Let $\underline{G}$ be a simply
connected, semisimple $K$-algebraic group. Let $\{w_{i}\in F_{r_{i}}\}_{i=1}^{3}$
and assume that $\varphi_{w_{1}}*\varphi_{w_{2}}$ is flat. Then $\varphi_{w_{1}}*\varphi_{w_{2}}*\varphi_{w_{3}}$
is (FGI). 
\end{cor}

\begin{proof}
Set $r=\stackrel[i=1]{3}{\sum}r_{i}$ and $w=w_{1}*w_{2}*w_{3}$.
Given $g\in\underline{G}$, let $\pi_{3}:(\underline{G}^{r}){}_{g,\varphi_{w}}\rightarrow\underline{G}^{r_{3}}$
be the restriction to $(\underline{G}^{r}){}_{g,\varphi_{w}}$ of
the projection $\underline{G}^{r}\rightarrow\underline{G}^{r_{3}}$.
Note that $\varphi_{w}$ is flat (via Proposition \ref{prop: properties preserved under convolution}),
so $(\underline{G}^{r}){}_{g,\varphi_{w}}$ is Cohen-Macaulay. Moreover,
for any $g_{3}\in\underline{G}^{r_{3}}$ we have 
\[
((\underline{G}^{r}){}_{g,\varphi_{w}})_{g_{3},\pi_{3}}\simeq\left(\underline{G}^{r_{1}+r_{2}}\right)_{g\cdot\varphi_{w_{3}}(g_{3})^{-1},\varphi_{w_{1}}*\varphi_{w_{2}}},
\]
so the fibers of $\pi_{3}$ are of dimension $\mathrm{dim}\underline{G}(r_{1}+r_{2}-1)=\mathrm{dim}(\underline{G}^{r}){}_{g,\varphi_{w}}-\mathrm{dim}\underline{G}^{r_{3}}$.
By miracle flatness, $\pi_{3}$ is flat, and thus open. By Theorem
\ref{thm:-convolution of two word maps is generically absolutely irreducible},
there is a dense collection of fibers of $\pi_{3}$ which are geometrically
irreducible, so by Lemma \ref{Lemma 3.13-Auxilary} we deduce that
$(\underline{G}^{r}){}_{g,\varphi_{w}}$ is geometrically irreducible,
as required. 
\end{proof}
The following Corollary is an analogue of Theorem \ref{thm:-convolution of two word maps is generically absolutely irreducible}
for jet schemes, and is a direct application of Lemma \ref{lem:irreducible generic fiber in jets}
combined with Theorem \ref{thm:-convolution of two word maps is generically absolutely irreducible}. 
\begin{cor}
\label{cor:convolution of word maps has jets with geometrically irreducible generic fiber}Let
$\underline{G}$ be a simply connected, semisimple $K$-algebraic
group and let $w_{1}\in F_{r_{1}}$ and $w_{2}\in F_{r_{2}}$ be words.
Then for each $m\in\nats$, the jet map $J_{m}(\varphi_{w_{1}*w_{2}}):J_{m}(\underline{G})^{r_{1}+r_{2}}\rightarrow J_{m}(\underline{G})$
has a geometrically irreducible generic fiber. 
\end{cor}

Finally we would like to introduce quantitative versions of flatness
and jet-flatness, which we will use in Sections \ref{sec:Proof-of-Theorems- Lie algebra}
and \ref{sec:flatness-and-(FRS) of word maps}. 
\begin{defn}
\label{def:epsilon jet flat}Let $\varphi:X\rightarrow Y$ be a morphism
between geometrically irreducible smooth $K$-varieties. 
\begin{enumerate}
\item $\varphi$ is called \textit{$\varepsilon$-flat} if for every $x\in X$
we have $\mathrm{dim}X_{\varphi(x),\varphi}\leq\mathrm{dim}X-\varepsilon\mathrm{dim}Y$. 
\item $\varphi$ is called \textit{$\varepsilon$-jet flat} if $J_{m}(\varphi)$
is $\varepsilon$-flat for every $m\in\nats$. 
\end{enumerate}
Note that a $1$-flat (resp.~$1$-jet flat) morphism is just flat
(resp.~jet-flat). 
\end{defn}

\begin{lem}
\label{lem:epsilon jet flat and log canonical threshold}Let $\varphi:X\rightarrow Y$
be a morphism between geometrically irreducible smooth $K$-varieties.
Then $\varphi$ is $\varepsilon$-jet flat if and only if $\mathrm{lct}(X,X_{\varphi(x),\varphi})\geq\varepsilon\mathrm{dim}Y$
for all $x\in X$. 
\end{lem}

\begin{proof}
From Theorem \ref{thm:log canonical threshold and jet schemes} we
see that if $\varphi$ is $\varepsilon$-jet flat then $\mathrm{lct}(X,X_{\varphi(x),\varphi})\geq\varepsilon\mathrm{dim}Y$
for all $x\in X$, and that if the latter holds then $J_{m}(\varphi)$
is $\varepsilon$-flat at $X\subseteq J_{m}(X)$. By upper semi-continuity
of the dimension of the fiber on the source (see \cite[Theorem 13.1.3]{Gro66})
and using the $\mathbb{G}_{m}$-action on $J_{m}(X)$, we get that
$J_{m}(\varphi)$ is $\varepsilon$-flat. 
\end{proof}
Motivated by Definition \ref{def:epsilon jet flat}, we give the following
definition for word maps: 
\begin{defn}
\label{def:epsilon jet flat for words}A word $w\in F_{r}$ is called
\textit{$\varepsilon$-flat} (resp.\textit{~$\varepsilon$-jet flat})
if for every simple algebraic $K$-group $\underline{G}$, the map
$\varphi_{w}:\underline{G}^{r}\rightarrow\underline{G}$ is $\varepsilon$-flat
(resp.~$\varepsilon$-jet flat).
\end{defn}

\begin{rem}
\label{rem:A-similar-notion of epsilon jet flat}A similar notion
of $\varepsilon$-(jet) flatness can be defined for Lie algebra words
$w\in\mathcal{L}_{r}$ (resp.~algebra words $w\in\mathcal{A}_{r}$),
by varying over all simple $K$-Lie algebras $\mathfrak{g}$ (resp.~$\{M_{n}\}_{n\in\nats}$)
for which $\varphi_{w}:\mathfrak{g}^{r}\rightarrow\mathfrak{g}$ (resp.~$\varphi_{w}:M_{n}^{r}\rightarrow M_{n}$)
is non-trivial. 
\end{rem}

The following observation is useful for showing that a morphism is
$\varepsilon$-flat (resp.~$\varepsilon$-jet flat): 
\begin{lem}
\label{lem:convolution gives bounds on epsilon flatness}Let $X$
be a geometrically irreducible smooth $K$-variety, $\underline{G}$
be a  connected algebraic $K$-group and let $\varphi:X\rightarrow\underline{G}$
be a morphism. Assume that $\varphi^{*t}$ is flat (resp.~jet-flat).
Then $\varphi$ is $\frac{1}{t}$-flat (resp.~$\frac{1}{t}$-jet
flat). 
\end{lem}

\begin{proof}
It is enough to prove the claim for flatness. Assume that $\varphi$
is not $\frac{1}{t}$-flat. Then there exists $g\in\underline{G}$
such that $\mathrm{dim}X_{g,\varphi}>\mathrm{dim}X-\frac{1}{t}\mathrm{dim}\underline{G}$.
In particular, $X_{g^{t},\varphi^{*t}}^{t}$ contains $(X_{g,\varphi})^{t}$
so we have 
\[
\mathrm{dim}X_{g^{t},\varphi^{*t}}^{t}>\mathrm{dim}X^{t}-\mathrm{dim}\underline{G},
\]
and hence $\varphi^{*t}$ is not flat. 
\end{proof}

\section{\label{sec:Lie-algebra-word}Lie algebra word maps}

\subsection{Definitions and discussion of Lie algebra word maps}
\begin{defn}
Let $\mathcal{A}_{r}$ (resp.~$\mathcal{L}_{r}$) be the free associative
$K$-algebra (resp.~the free $K$-Lie algebra) on a finite set $\{X_{1},\ldots,X_{r}\}$. 
\begin{enumerate}
\item We refer to an element $w$ of $\mathcal{A}_{r}$ (resp.~of $\mathcal{L}_{r}$)
as an \textit{algebra word} (resp.~a \textit{Lie algebra word}). 
\item For a matrix algebra $M_{n}:=\mathrm{Mat}_{n\times n}$ (resp.~a
Lie algebra $\frak{\frak{\mathfrak{g}}}$), the word $w$ in $\mathcal{A}_{r}$
(resp.~in $\mathcal{L}_{r}$) induces a \textit{matrix word map}
(resp.~a \textit{Lie word map}) $\varphi_{w}:M_{n}^{r}\rightarrow M_{n}$
(resp.~$\varphi_{w}:\frak{\mathfrak{g}}^{r}\rightarrow\frak{\mathfrak{g}}$). 
\item Given two words $w_{1}$ and $w_{2}$ (either algebra words or Lie
algebra words), we may define their (additive) convolution $w_{1}*w_{2}:=w_{1}+w_{2}$. 
\item Note that both $\mathcal{A}_{r}$ and $\mathcal{L}_{r}$ have natural
gradations. We define the \textit{degree} of a word $w$ as the maximal
grade $d\in\nats$ in which the image of $w$ is non-trivial. A word
$w$ is said to be \textit{$d$-homogeneous} if it lies in the $d$-th
grade. 
\end{enumerate}
\end{defn}

\begin{rem}
We can also define a multiplicative convolution operation. Unlike
the addition map, the multiplication map $M_{n}\times M_{n}\rightarrow M_{n}$
is not smooth, so one does not expect the multiplicative convolution
operation to improve singularity properties of morphisms. 
\end{rem}

In this section we would like to study the smoothing effect of the
additive convolution operation on Lie (and matrix) algebra word maps,
and show that under very mild assumptions, any Lie algebra word map
$\varphi_{w}:\mathfrak{g}^{r}\rightarrow\mathfrak{g}$ where $\mathfrak{g}$
is a semisimple Lie algebra, becomes flat and eventually (FRS) after
a number of convolutions $N(w)$ which is independent of $\mathfrak{g}$.

Having no assumptions on $w$ turns out to be too optimistic. It is
a classical theorem of Borel, that given a non-trivial word $w\in F_{r}$,
the corresponding word map $\varphi_{w}:\underline{G}^{r}\rightarrow\underline{G}$
on any semisimple algebraic group $\underline{G}$ is dominant \cite[Theorem B]{Bor83}.
Unlike the group case, a non-trivial Lie algebra word might induce
a trivial word map on certain simple Lie algebras.
\begin{example}[{\cite[Example 3.8]{BGKP12}, \cite{Raz73}}]
\label{exa:trivial word maps}The word 
\[
w(X,Y,Z)=[[[[Z,Y],Y],X],Y]-[[[[Z,Y],X],Y],Y]
\]
is non-trivial, but induces the zero map on $\mathfrak{sl}_{2}$. 
\end{example}

Aside of this pathology, the situation is not too bad. In \cite{BGKP12},
a Lie algebra analogue of Borel's theorem was proven: 
\begin{thm}[{\cite[Theorem 3.2]{BGKP12}}]
\label{thm:Lie algebra analogue of Borel theorem}Let $w\in\mathcal{L}_{r}$
be a Lie algebra word, and assume $\varphi_{w}:\mathfrak{sl}_{2}^{r}\rightarrow\mathfrak{sl}_{2}$
is non-trivial. Then $\varphi_{w}:\mathfrak{g}^{r}\rightarrow\mathfrak{g}$
is dominant for any semisimple $K$-Lie algebra $\mathfrak{g}$. 
\end{thm}

In addition, one can observe the following: 
\begin{prop}
\label{prop:Lie algebra word maps are generating}Let $\mathfrak{g}$
be a semisimple $K$-Lie algebra, and $w\in\mathcal{L}_{r}$ a Lie
algebra word. Assume $\varphi_{w}:\mathfrak{g}^{r}\rightarrow\mathfrak{g}$
is non-zero. Then $\varphi_{w}$ is generating. 
\end{prop}

\begin{proof}
It is enough to prove the statement when $\mathfrak{g}$ is simple
and when $K=\complex$. Let $W:=\mathrm{span}_{\complex}\varphi_{w}(\mathfrak{g}(\complex)^{r})$
and let $\underline{G}$ be a simply connected group with $\mathrm{Lie}(\underline{G})=\mathfrak{g}$.
Notice that for any $g\in\underline{G}(\complex)$ and $X_{1},\ldots,X_{r}\in\mathfrak{g}(\complex)$,
\[
\varphi_{w}(\mathrm{Ad}_{g}X_{1},\dots,\mathrm{Ad}_{g}X_{r})=\mathrm{Ad}_{g}\varphi_{w}(X_{1},\dots,X_{r}),
\]
so $W$ is $\underline{G}(\complex)$-invariant. By the simplicity
of $\mathfrak{g}(\complex)$ we must have $W=\mathfrak{g}(\complex)$. 
\end{proof}
\begin{conjecture}[{\cite[Question 3.1]{BGKP12}}]
Let $\mathfrak{g}$ be a semisimple $K$-Lie algebra, and $w\in\mathcal{L}_{r}$.
Assume $\varphi_{w}:\mathfrak{g}^{r}\rightarrow\mathfrak{g}$ is non-zero
(i.e.~$\varphi_{w}(\mathfrak{g}(\overline{K})^{r})\neq\{0\}$). Then
$\varphi_{w}$ is dominant.
\end{conjecture}

By Theorem \ref{Main result} and Lemma \ref{lem:generating becomes dominant}
(or by \cite[Theorem 1.7]{GH19}), any generating map into a vector
space becomes dominant, and eventually (FRS), after sufficiently many
self-convolutions. Thus, the assumption that $\varphi_{w}:\mathfrak{g}^{r}\rightarrow\mathfrak{g}$
is non-zero is sufficient for our purposes. Now, from the Amitsur-Levitzki
theorem, we can deduce that any Lie algebra word map is non-trivial
if the rank of $\mathfrak{g}$ is large enough: 
\begin{thm}[Amitsur-Levitzki, \cite{AL50}]
\label{thm:(Amitsur-Levitzky)-Let-}Let $R$ be a commutative ring.
Any polynomial identity on $M_{d}(R)$ is of degree at least $2d$. 
\end{thm}

\begin{cor}
~\label{cor: commutator relations on simple Lie algebras} 
\begin{enumerate}
\item Any $0\neq w\in\mathcal{L}_{r}$ such that $\varphi_{w}:\mathfrak{sl}_{d}^{r}\rightarrow\mathfrak{sl}_{d}$
is trivial is of degree at least $2d$. 
\item Any $0\neq w\in\mathcal{L}_{r}$ such that $\varphi_{w}:\mathfrak{g}^{r}\rightarrow\mathfrak{g}$
is trivial on a simple classical Lie algebra $\mathfrak{g}$ of rank
$d$ is of degree at least $2d$. 
\end{enumerate}
\end{cor}

\begin{proof}
(1) follows easily from Amitsur-Levitzki and (2) follows from (1)
since each of $\mathfrak{so}_{2d},\mathfrak{so}_{2d+1}$ and $\mathfrak{sp}_{2d}$
contains $\mathfrak{sl}_{d}$ as a Lie subalgebra. 
\end{proof}

\subsection{Formulation of the main theorems}

We turn to formulate some of our main results. 
\begin{thm}
\label{thm: main thm Lie algebra word maps}Let $\{w_{i}\in\mathcal{L}_{r_{i}}\}_{i\in\nats}$
be a collection of Lie algebra words of degree at most $d\in\nats$.
Then there exists $0<C<10^{6}$,\textcolor{red}{{} }such that for any
semisimple $K$-Lie algebra $\mathfrak{g}$ for which $\{\varphi_{w_{i}}\}_{i\in\nats}$
are non-trivial when restricted to each simple constituent of $\mathfrak{g}$,
we have the following: 
\begin{enumerate}
\item If $m\geq Cd^{3}$ then $\varphi_{w_{1}}*\dots*\varphi_{w_{m}}$ is
flat. 
\item If $m\geq Cd^{3}+2$ then $\varphi_{w_{1}}*\dots*\varphi_{w_{m}}$
is flat with normal fibers. 
\item If $m\geq Cd^{4}$ then $\varphi_{w_{1}}*\dots*\varphi_{w_{m}}$ is
(FRS). 
\item If $m\geq Cd^{4}+2$ then $\varphi_{w_{1}}*\dots*\varphi_{w_{m}}$
is (FTS). 
\end{enumerate}
\end{thm}

Note that by Corollary \ref{cor: commutator relations on simple Lie algebras},
if the rank of every simple constituent of a semi-simple Lie algebra
$\mathfrak{g}$ is larger than $d/2$, then every Lie algebra word
of degree at most $d$ induces a non-trivial Lie algebra word map
on $\mathfrak{g}$.

An analogue of Theorem \ref{thm: main thm Lie algebra word maps}
holds also for matrix word maps: 
\begin{thm}
\label{thm: main theorem for matrix word map}Let $\{w_{i}\in\mathcal{A}_{r_{i}}\}_{i\in\nats}$
be a collection of matrix words of degree at most $d\in\nats$. Then
for any $n\in\nats$ such that $\{\varphi_{w_{i}}:M_{n}^{r}\rightarrow M_{n}\}_{i\in\nats}$
are generating, the assertions of Theorem \ref{thm: main thm Lie algebra word maps}
hold. 
\end{thm}

Theorems \ref{thm: main thm Lie algebra word maps} and \ref{thm: main theorem for matrix word map}
combined with Lemma \ref{lem:convolution gives bounds on epsilon flatness},
imply that (Lie) algebra word maps are $\varepsilon$-jet flat, for
$\varepsilon\sim d^{-4}$. In fact, we can achieve slightly better
lower bounds on the $\varepsilon$. 
\begin{thm}
\label{thm:lower bounds on epsilon jet flatness}Let $w_{1}\in\mathcal{A}_{r}$
and $w_{2}\in\mathcal{L}_{r}$ be words of degree at most $d$. Then
$w_{1}$ is $\frac{1}{25d^{3}}$-jet flat and $w_{2}$ is $\frac{1}{2\cdot10^{5}d^{3}}$-jet
flat. 
\end{thm}

The arguments appearing in the proofs of Theorems \ref{thm: main thm Lie algebra word maps}
and \ref{thm: main theorem for matrix word map} generalize the arguments
used in \cite[Theorem 2.1]{AA16}, in which a version of Theorem \ref{thm: main thm Lie algebra word maps}
was proved for the specific case of the commutator Lie algebra word.

The main approach for proving Theorem \ref{thm: main thm Lie algebra word maps}
is as follows; we first reduce to the case of self-convolutions of
$d$-homogeneous Lie algebra word maps (of pure type). Given a $d$-homogeneous
word $w$ and Lie algebra $\mathfrak{g}$, we apply a series of $\mathbb{G}_{m}$-degenerations
to the word map $\varphi_{w}:\mathfrak{g}^{r}\rightarrow\mathfrak{g}$,
such that the resulting map $\psi_{w}$ is much easier to analyze.
We then use Corollary \ref{cor:reduction to degeneration} to conclude
$\varphi_{w}^{*t}$ is flat (resp.~(FRS)), given we managed to show
$\psi_{w}^{*t}$ is flat (resp.~(FRS)), for some $t\in\nats$. The
price of degenerating a map is that its regularity properties might
diminish. Our main goal is to apply a series of carefully chosen degenerations
to $\varphi_{w}$, such that the number of convolutions of $\psi_{w}$
needed to get a flat (resp.~(FRS)) morphism is uniformly bounded
over all semisimple Lie algebras $\frak{\mathfrak{g}}$, and behaves
nicely with respect to $d$.

To achieve this goal, we construct a certain 'combinatorial gadget'
which encodes the word map $\varphi_{w}:\mathfrak{g}^{r}\rightarrow\mathfrak{g}$,
and generalizes the construction given in \cite{AA16} in the case
of the commutator map. We then translate the flatness and (FRS) properties,
as well as the convolution operation and the degeneration methods
we use to properties of and operations on this gadget.

Let us first explain the combinatorial construction in the case of
the commutator map (see also \cite{AA16}).

\subsection{\label{subsec:Combinatirical-picture-for commutator}Combinatorial
picture in the case of the commutator word}

As before, let $[\cdot,\cdot]$ be the commutator Lie algebra word,
and by abusing notation, given a Lie algebra $\frak{\mathfrak{g}}$
we denote by $[\cdot,\cdot]:\frak{\mathfrak{g}}\times\frak{\frak{\mathfrak{g}}}\to\frak{\mathfrak{g}}$
the corresponding Lie algebra word map. For simplicity, we work over
the field of complex numbers. The proof that there exists a number
$t$ such that $[\cdot,\cdot]^{*t}:\frak{\mathfrak{g}}^{2t}\to\frak{\mathfrak{g}}$
is (FRS) for every semisimple Lie algebra $\frak{\mathfrak{g}}$ can
be explained in six main steps: 
\begin{enumerate}
\item \textbf{Construct a polygraph from a classical Lie algebra $\mathfrak{g}$;}
We attach to each simple Lie algebra $\mathfrak{g}$ a polygraph $\Gamma_{\mathfrak{g}}=(\mathcal{I},\mathcal{J},\mathcal{S})$,
which is a triplet consisting of finite sets $\mathcal{I},\mathcal{J}$
and a subset $\mathcal{S}\subseteq\mathcal{I}^{\{2\}}\times\mathcal{J}$,
where $\mathcal{I}^{\{2\}}$ denotes the set of subsets of $\mathcal{I}$
of size $2$. The difference between a polygraph and a graph is that
a polygraph has edges of different \textit{types}, where each type
is an element of $\mathcal{J}$. Here, the vertices $\{e_{i}\}_{i\in\mathcal{I}}$
of the polygraph correspond to a Chevalley basis of $\mathfrak{g}$
and the types $\mathcal{J}$ correspond to elements of the dual basis
$\{\widehat{e}_{l}\}_{l\in\mathcal{J}}$ of $\mathfrak{g}^{*}$. The
edges $\mathcal{S}$ are pairs $(\{e_{i},e_{j}\},\widehat{e}_{l})$,
such that $\widehat{e}_{l}([e_{i},e_{j}])\neq0$.

For example, taking $\mathfrak{g}=\mathfrak{sl}_{d}$, one can choose
the basis $\mathcal{B}=\{E_{i,j}\}_{i\neq j}\cup\{H_{i,i}\}$, where
$(E_{i,j})_{i',j'}=\delta_{i,i'}\delta_{j,j'}$ (here $\delta_{x,y}$
stands for the Kronecker delta function), and $H_{i,i}=[E_{i,i+1},E_{i+1,i}]$.
The set of types $\mathcal{J}$ corresponds to the basis dual to $\mathcal{B}$,
denoted $\widehat{\mathcal{B}}=\{\hat{E}_{i,j}\}\cup\{\hat{H}_{i,i}\}$.
The vertices $E_{i,j}$ and $E_{j,k}$ satisfy $\hat{E}_{i,k}([E_{i,j},E_{j,k}])=\hat{E}_{i,k}(E_{i,k})\neq0$,
so $(\{E_{i,j},E_{j,k}\},\hat{E}_{i,k})\in\mathcal{S}$ represents
an edge of type $\hat{E}_{i,k}$. 
\item \textbf{Describe $[\cdot,\cdot]^{*t}:\mathfrak{g}^{2t}\rightarrow\mathfrak{g}$
in terms of the constructed polygraph $\Gamma_{\mathfrak{g}}$; }Let
$\Gamma_{\mathfrak{g}}=(\mathcal{I},\mathcal{J},\mathcal{S})$ as
above. We have the following identification of vector spaces $\mathfrak{g}^{2t}\simeq\mathfrak{g}\otimes V\simeq V^{\mathcal{I}}$,
where $V$ is a $2t$-dimensional vector space. Let $(a_{i,j,l}:=\hat{e}_{l}([e_{i},e_{j}]))_{i,j\in\mathcal{I},l\in\mathcal{J}}$
be the structure coefficients of $\mathfrak{g}$ with respect to $\mathcal{I}$
and $\mathcal{J}$. Then under the identifications above, $[\cdot,\cdot]^{*t}$
corresponds to the map $\Phi_{\mathrm{comm}}^{*t}:V^{\mathcal{I}}\rightarrow\complex^{\mathcal{J}}$
which is given by 
\begin{align*}
\Phi_{\mathrm{comm}}^{*t}((v_{i})_{i\in\mathcal{I}}) & =\left(\sum_{i,j\in\mathcal{I}}a_{i,j,l}\cdot\left(\sum_{k=1}^{t}v_{i,2k-1}v_{j,2k}\right)\right)_{l\in\mathcal{J}}\\
 & =\left(\sum_{i<j\in\mathcal{I}}a_{i,j,l}\cdot\left(\sum_{k=1}^{t}v_{i,2k-1}v_{j,2k}-v_{j,2k-1}v_{i,2k}\right)\right)_{l\in\mathcal{J}},
\end{align*}
where each $v_{i}\in V$ is of the form $(v_{i,1},\dots,v_{i,2t})$.
Note that the sum $\sum\limits _{k=1}^{t}v_{i,2k-1}v_{j,2k}-v_{j,2k-1}v_{i,2k}$
is just a standard symplectic form on a $2t$-dimensional symplectic
space $(V,\langle\,\cdot,\cdot\,\rangle)$. The map $\Phi_{\mathrm{comm}}^{*t}:V^{\mathcal{I}}\rightarrow\complex^{\mathcal{J}}$
can then be written as 
\[
\Phi_{\mathrm{comm}}^{*t}((v_{i})_{i\in\mathcal{I}})=\left(\sum_{i<j\in\mathcal{I}}a_{i,j,l}\cdot\langle v_{i},v_{j}\rangle\right)_{l\in\mathcal{J}}.
\]
This operation can be considered more abstractly. 
\begin{defn}[{\cite[Definition 2.22]{AA16}}]
Given a polygraph $\Gamma=(\mathcal{I},\mathcal{J},\mathcal{S})$,
an ordering $<$ of $\mathcal{I}$, a function $a:\mathcal{S}\rightarrow\complex^{\times}$
and a symplectic space $(V,\langle\,\cdot,\cdot\,\rangle)$, we define
a map $\Phi_{\Gamma,<,a,V}:V^{\mathcal{I}}\rightarrow\complex^{\mathcal{J}}$
by 
\[
\Phi_{\Gamma,<,a,V}((v_{i})_{i\in\mathcal{I}})=\left(\sum_{i<j\in\mathcal{I}}a(\{i,j\},l)\cdot\langle v_{i},v_{j}\rangle\right)_{l\in\mathcal{J}}.
\]
Let $\Gamma$ be a polygraph and $V$ be a symplectic space. We say
that the pair $(\Gamma,V)$ is \textit{(FRS)} if for all $<$ and
$a$ as above, the map $\Phi_{\Gamma,<,a,V}$ is (FRS). 
\end{defn}

Notice that the (FRS) property of $[\cdot,\cdot]^{*t}:\mathfrak{g}^{2t}\rightarrow\mathfrak{g}$
is equivalent to the (FRS) property of the polygraph $\Gamma_{\mathfrak{g}}=(\mathcal{I},\mathcal{J},\mathcal{S})$.
It is also important to note that taking convolution powers of $[\cdot,\cdot]$
(or $\Phi_{\mathrm{comm}}$), translates to choosing a larger symplectic
space, where $t$ self-convolutions, correspond to assigning a $2t$-dimensional
symplectic space to each element of $\mathcal{I}$. 
\item \textbf{Relate the elimination method to certain combinatorial operations
on the polygraph;} Let $\Phi_{\Gamma,<,a,V}$ be the map induced by
a polygraph $\Gamma$. Let $\omega\in\mathbb{Z}^{\mathcal{I}}$ be
a vector, which we call an \textit{$\mathcal{I}$-weight}. An $\mathcal{I}$-weight
$\omega$ induces an $\mathcal{I}^{\{2\}}$-weight $\widetilde{\omega}:\mathcal{I}^{\{2\}}\rightarrow\mathbb{Z}$
by $\widetilde{\omega}(\{i_{1},i_{2}\})=\omega(i_{1})+\omega(i_{2})$.
We furthermore define $\mathrm{gr}_{\omega}(\Gamma):=(\mathcal{I},\mathcal{J},\mathrm{gr}_{\omega}\mathcal{S})$
by 
\[
\mathrm{gr}_{\omega}\mathcal{S}=\{(\{i_{1},i_{2}\},l)\in\mathcal{S}:\forall(\{j_{1},j_{2}\},l)\in\mathcal{S}\text{ we have }\widetilde{\omega}(i_{1},i_{2}\})\leq\widetilde{\omega}(j_{1},j_{2}\})\}.
\]
An application of the elimination method (Corollary \ref{cor:(cf.---Elimination)})
gives the following: 
\begin{cor}[{\cite[Corollary 2.28]{AA16}, elimination for polygraphs}]
\label{cor:(-elimination-for-polygraphs)}Let $\Gamma:=(\mathcal{I},\mathcal{J},\mathcal{S})$
be a polygraph and $V$ be a symplectic space. Let $\omega\in\mathbb{Z}^{\mathcal{I}}$
be an $\mathcal{I}$-weight. If $(\mathrm{gr}_{\omega}(\Gamma),V)$
is (FRS), then so is $(\Gamma,V)$. 
\end{cor}

Notice that this elimination procedure introduces a new polygraph
$\mathrm{gr}_{\omega}(\Gamma)$ which has the same vertices but eliminates
some of the edges, leaving only the edges $\mathrm{gr}_{\omega}\mathcal{S}$.
Furthermore, this operation does not increase $\mathrm{dim}V$, as
no convolutions are involved. 
\item \textbf{Reduction from a polygraph to a graph; }It turns out that
for the special case of polygraphs induced from a classical Lie algebra
$\mathfrak{g}$, one can use the elimination method twice, and obtain
a new polygraph, $\Gamma_{2}$, which has a unique edge of each type
$\widehat{e}_{l}\in\mathcal{J}$ (see \cite[Section 2.6.2]{AA16}).
In this situation, the polygraph $\Gamma_{2}$ is just a graph. Since
graphs $\Gamma=(\mathcal{V},\mathcal{E})$, where $\mathcal{V}$ denotes
the set of vertices, and $\mathcal{E}$ the edges, are special cases
of polygraphs, we can define a map $\Phi_{\Gamma,V}:V^{\mathcal{V}}\rightarrow\complex^{\mathcal{E}}$
as before. The graph $\Gamma=(\mathcal{V},\mathcal{E})$ is then (FRS)
if $\Phi_{\Gamma,V}$ is. 
\item \textbf{A method of coloring and reduction to an edge; }Let $\Gamma=(\mathcal{V},\mathcal{E})$
be a graph, let $M$ be a finite set, which we think of as our set
of \textit{colors}, and let $\omega:\mathcal{V}\rightarrow\ints^{M}$
be a choice of an $M$-tuple of weights to each vertex. We call such
$\omega$ a \textit{coloring function.} As in the case of polygraphs,
$\omega$ induces a weight function on $\mathcal{V}^{\{2\}}$ by $\widetilde{\omega}(\{v_{1},v_{2}\}):=\omega(v_{1})+\omega(v_{2})$.
We say that $\omega$ is \textit{admissible} if for each edge $\varepsilon=\{v_{1},v_{2}\}\in\mathcal{E}$,
the tuple $\widetilde{\omega}(\varepsilon)\in\mathbb{Z}^{M}$ has
a unique minimum. We say that an edge $\varepsilon$ has color $i\in M$
if for each $j\in M$ we have $\widetilde{\omega}(\varepsilon)_{i}\leq\widetilde{\omega}(\varepsilon)_{j}$.
Notice that if $\omega$ is admissible, then each edge is colored
in a unique color. This means we can define a new graph $\mathrm{gr}_{\omega,i}\Gamma:=(\mathcal{V},\mathrm{gr}_{\omega,i}\mathcal{E})$
where $\mathrm{gr}_{\omega,i}\mathcal{E}=\{\varepsilon\in\mathcal{E}:\varepsilon\text{ has color }i\}$.
The following corollary relates the (FRS) property of $\Gamma$ to
the (FRS) property of each of the $\mathrm{gr}_{\omega,i}\Gamma$: 
\begin{cor}[{\cite[Corollary 2.33]{AA16}, coloring method}]
\label{cor:(-Coloring-method)-Let}Let $\Gamma:=(\mathcal{V},\mathcal{E})$
be a graph, $M$ a finite set and $\omega:\mathcal{V}\rightarrow\mathbb{Z}^{M}$
an admissible coloring. Let $\{V_{i}\}_{i\in M}$ be symplectic spaces
and $V=\bigoplus\limits _{i\in M}V_{i}$. Suppose that $(\mathrm{gr}_{\omega,i}\Gamma,V_{i})$
is (FRS). Then so is $(\Gamma,V)$. 
\end{cor}

It is important to note that unlike the elimination method for polygraphs
(Corollary \ref{cor:(-elimination-for-polygraphs)}), the coloring
method comes with a price. Namely, we may divide our graph $\Gamma:=(\mathcal{V},\mathcal{E})$
to simpler graphs $\mathrm{gr}_{\omega,i}\Gamma$ with fewer edges,
and prove each is (FRS) separately, but we ``pay'' by enlarging
our symplectic space, which means more convolutions of $[\cdot,\cdot]:\mathfrak{g}^{2}\rightarrow\mathfrak{g}$.
Thus, the coloring method can be seen as applying $M$ self-convolutions,
followed by an elimination of edges.

In \cite[Section 2.6.3]{AA16}, Avni and Aizenbud showed that using
3 colors one can reduce the graph $\Gamma_{2}$ from Step 4 to a disjoint
union of trees (i.e.~a forest), each with degree bounded by a number
$D$. The bound on the degree $D$ depends on the type of $\mathfrak{g}$
(e.g.~for $\mathfrak{sl}_{n}$, taking $D=3$ suffices).

After reducing to a forest of maximal degree $D$, one can easily
reduce to a disjoint union of edges with the use of $2(D-1)$ colors
(see \cite[Theorem 2.35]{AA16}). 
\item \textbf{Proof for the case of an edge;} Given a pair $(\Gamma,V),$where
$\Gamma$ is a graph consisting of a single edge, the map $\Phi_{\Gamma,V}$
has a very simple description. It corresponds to the map $\Psi^{*\frac{\mathrm{dim}V}{2}}:\mathbb{A}^{2\dim V}\to\mathbb{A}^{1}$
where $\Psi(x_{1},x_{2},x_{3},x_{4})=x_{1}x_{2}-x_{3}x_{4}$. Since
$\Psi$ can be easily shown to be (FRS), it is enough to take $\mathrm{dim}V=2$.
\end{enumerate}
Taking $\mathfrak{sl}_{n}$ for example, the above steps are as follows.
We first construct the polygraph $\Gamma=(\mathcal{I},\mathcal{J},\mathcal{S})$
corresponding to $[\cdot,\cdot]:\mathfrak{sl}_{n}^{2}\to\mathfrak{sl}_{n}$.
We then reduce the polygraph $\Gamma$ to a graph $\Gamma_{2}$. Three
colors are needed to reduce $\Gamma_{2}$ to a forest $\Gamma_{3}$
of degree$\leq3$. Then, using four more colors, one can reduce to
an edge, which is (FRS) if $V$ is at least two dimensional. Altogether,
we deduce that if $\mathrm{dim}V=2\cdot4\cdot3=24$ then $(\Gamma,V)$
is (FRS) and hence also $[\cdot,\cdot]^{*12}:\mathfrak{sl}_{n}^{24}\rightarrow\mathfrak{sl}_{n}$
is (FRS), regardless of $n$.

\subsubsection{Generalizing to word maps: motivation}

We would like to prove Theorems \ref{thm: main thm Lie algebra word maps}
and \ref{thm: main theorem for matrix word map} using similar tools
as above. It would be convenient to reduce our problem to word maps
of homogeneous degree $d$, and to the special case of self-convolutions
of word maps (i.e.~to consider $\varphi_{w_{1}}^{*m}$ rather then
$\varphi_{w_{1}}*\dots*\varphi_{w_{m}}$).

A different combinatorial gadget is needed in order to deal with words
$w$ of homogeneous degree $d>2$; the edges in the polygraph in the
case of the commutator word, correspond to commutator relations between
various pairs of elements in a chosen basis $\mathcal{B}$ of the
Lie algebra $\mathfrak{g}$. In the case of a $d$-homogeneous word,
the edges should encode relations between $d$ elements in $\mathcal{B}$,
so one should consider hypergraphs instead of graphs, where edges
now consist of sets (or multisets) of $d$ vertices.

As discussed in the introduction, the main challenge is to degenerate
the hypergraph such that the resulting map will be simpler to analyze,
with similar singularity properties (or only slightly worse), and
generating. This is much more involved in the general case than the
commutator case. Certain pathologies such as in Example \ref{exa:trivial word maps}
appear when the rank of the Lie algebra is comparable to the degree
$d$ of $w$. Such situations will require a more careful treatment.

Before we construct the combinatorial gadget, we first make some reductions.

\subsection{\label{subsec:Reduction to self convolutions+homogeneous}Reduction
to self-convolutions of homogeneous word maps of pure type}
\begin{defn}
\label{def:simple commutator}An element in $\mathcal{L}_{r}(X_{1},\dots,X_{r})$
of the form $[[\dots[[X_{s_{1}},X_{s_{2}}],X_{s_{3}}],\dots],X_{s_{d}}]$
is called a \textit{(left-normed) Lie monomial of degree $d$,} and
it is denoted $X_{S}:=[X_{s_{1}},\ldots,X_{s_{d}}]$, for $S=(s_{1},\ldots,s_{d})\in[r]^{d}$.
Note that here, $X_{s}$ counts as a Lie monomial of degree 1. 
\end{defn}

The Jacobi identity and the anti-commutativity imply the following: 
\begin{fact}[{e.g.~\cite[Lemma 3.3]{Chi06}}]
\label{fact:sum of left normed monomials}Any $d$-homogeneous Lie
algebra word $w\in\mathcal{L}_{r}$ can be written as a linear combination
of left-normed Lie monomials of degree $d$. 
\end{fact}

\begin{defn}
\label{def:pure word}A word $w\in\mathcal{L}_{r}(X_{1},\dots,X_{r})$
is called \textit{$d$-homogeneous of pure type}, if it is of the
form 
\[
w(X_{1},\dots,X_{r})=\sum\limits _{S\in\mathrm{MS}^{-1}(S_{0})}c_{S}X_{S},
\]
where $c_{S}\in K$, $S_{0}$ is an element of $[r]^{(d)}$, the set
of multisets of $[r]$ of size $d$, and $\mathrm{MS}:[r]^{d}\rightarrow[r]^{(d)}$
is the forgetful map from $d$-tuples to $d$-multisets. In this case
we say that $w$ is \textsl{of pure type $S_{0}$}. 
\end{defn}

\begin{prop}
\label{prop:reduction to balanced homogeneous maps}It is enough to
prove Theorem \ref{thm: main thm Lie algebra word maps} for $d$-homogeneous
word maps of pure type. 
\end{prop}

The proof of Proposition \ref{prop:reduction to balanced homogeneous maps}
relies on the following lemma: 
\begin{lem}
\label{lem: reduction to largest degree}Let $\{\varphi_{i}=(\varphi_{i,1},\dots,\varphi_{i,m}):\mathbb{A}^{n_{i}}\rightarrow\mathbb{A}^{m}\}_{i=1}^{N}$
be polynomial maps, and let $\widetilde{\varphi}_{i,j}\in K[x_{1},\ldots,x_{n_{i}}]$
be the homogeneous part of $\varphi_{i,j}$ of maximal degree. Assume
that $\widetilde{\varphi}_{1}*\ldots*\widetilde{\varphi}_{N}:\mathbb{A}^{n_{1}}\times\ldots\times\mathbb{A}^{n_{N}}\to\mathbb{A}^{m}$
is flat (resp.~(FRS)) at $(0,\dots,0)$. Then $\varphi_{1}*\ldots*\varphi_{N}$
is flat (resp.~(FRS)). 
\end{lem}

\begin{proof}
We prove the lemma in the case $N=1$, the general case is similar.
Write $\varphi_{1,i}(x_{1},\dots,x_{n})=\sum\limits _{k=0}^{l_{i}}P_{i,k}(x_{1},\ldots,x_{n})$,
where $P_{i,k}$ is a homogeneous polynomial of degree $k$, and $P_{i,l_{i}}$
is non-zero. Now, define a map $\psi=(t,\psi_{1},\dots,\psi_{m}):\mathbb{A}^{1}\times\mathbb{A}^{n}\rightarrow\mathbb{A}^{1}\times\mathbb{A}^{m}$
by
\[
\psi_{i}(t,x_{1},\dots,x_{n})=\sum\limits _{k=0}^{l_{i}}t^{l_{i}-k}P_{i,k}(x_{1},\dots,x_{n}).
\]
Then $\psi$ is an $\mathbb{A}^{1}$-family of morphisms parameterized
by $t$, where $\psi(1)=\varphi$ and $\psi(0)=\widetilde{\varphi}$.
Consider the following action of $\mathbb{G}_{m}$; for $s\in K^{\times}$
set $s.(t,x_{1},\dots,x_{n})=(st,sx_{1},\dots,sx_{n})$ and $s.(t,y_{1},\dots,y_{m}))=(st,s^{l_{1}}y_{1},\dots,s^{l_{m}}y_{m})$.
Note that $\psi$ is $\mathbb{G}_{m}$-equivariant, as for any $s\in K^{\times}$,
\[
\psi_{i}(st,sx_{1},\dots,sx_{n})=\sum\limits _{k=0}^{l_{i}}s^{l_{i}-k}t^{l_{i}-k}P_{i,k}(sx_{1},\dots,sx_{n})=s^{l_{i}}\sum\limits _{k=0}^{l_{i}}t^{l_{i}-k}P_{i,k}(x_{1},\dots,x_{n})=s.\psi_{i}(t,x_{1},\dots,x_{n}).
\]
Now, take $(a_{1},\dots,a_{n})\in K^{n}$ and choose the section $f:\mathbb{A}^{1}\rightarrow\mathbb{A}^{1}\times\mathbb{A}^{n}$
defined by $f(t)=(t,ta_{1},\dots,ta_{n})$. Clearly $f$ is $\mathbb{G}_{m}$-equivariant,
so by Corollary \ref{cor:reduction to degeneration} $\varphi$ is
flat (resp.~(FRS)) at $(a_{1},\dots,a_{n})$. 
\end{proof}
\begin{rem}
Note that if $\varphi:\mathbb{A}^{n}\to\mathbb{A}^{m}$ is a homogeneous
map which is flat (resp.~(FRS)) at $(0,\ldots,0)$, then it is flat
(resp.~(FRS)). 
\end{rem}

\begin{proof}[Proof of Proposition \ref{prop:reduction to balanced homogeneous maps}]
Let $\{w_{i}\in\mathcal{L}_{r_{i}}\}_{i\in\nats}$ be a collection
of Lie algebra words of degree at most $d$, and set $w=w_{i_{1}}*\dots*w_{i_{m}}$.
For each $i$, let $\widetilde{w}_{i}$ be the homogeneous part of
maximal degree of $w_{i}$. By the above lemma, in order to prove
that $\varphi_{w}=\varphi_{w_{i_{1}}}*\dots*\varphi_{w_{i_{m}}}$
is flat (resp.~(FRS)) it is enough to show that $\varphi_{\widetilde{w}_{i_{1}}}*\dots*\varphi_{\widetilde{w}_{i_{m}}}$
is flat (resp.~(FRS)). We have therefore reduced Theorem \ref{thm: main thm Lie algebra word maps}
to the case of homogeneous words.

To reduce to the case of homogeneous words of pure type, we use the
elimination method (Corollary \ref{cor:(cf.---Elimination)}). Given
a homogeneous word $v\in\mathcal{L}(X_{1},\dots,X_{r})$ of degree
$d$, we choose a weight function $\omega(X_{i})=(d+1)^{i}$, that
is, we assign the same weight $(d+1)^{i}$ to all of the variables
in the $i$-th copy of $\mathfrak{g}$ which is represented by $X_{i}$.
This degenerates $v$ to a homogeneous word of pure type, since two
Lie monomials $X_{I}$ and $X_{J}$ of degree $d$ satisfy $\mathrm{deg}_{\omega}(X_{I})=\mathrm{deg}_{\omega}(X_{J})$
if and only if $\mathrm{MS}(I)=\mathrm{MS}(J)$. 
\end{proof}
We would now like to further reduce to the case of self-convolutions: 
\begin{prop}
\label{prop:reduction to self convolutions}It is enough to prove
Theorem \ref{thm: main thm Lie algebra word maps} in the case of
self-convolutions of homogeneous word maps of pure type. 
\end{prop}

The proof of Proposition \ref{prop:reduction to self convolutions}
follows directly from the following proposition: 
\begin{prop}
\label{prop:reduction to self convolution- the general case}Let $\{\varphi_{i}:X_{i}\rightarrow V\}_{i\in I}$
be a collection of morphisms, with $X_{i}$ smooth and geometrically
irreducible for each $i$, and $V$ a $K$-vector space. Assume that
there exists $t\in\nats$ such that $\varphi_{i}^{*t}$ is flat (resp.~(FRS))
for every $i$. Then $\varphi_{1}*\dots*\varphi_{2t}$ is flat (resp.~(FRS)). 
\end{prop}

\begin{proof}
Write $\psi:=\varphi_{1}*\dots*\varphi_{2t}$ and $X:=\prod\limits _{i=1}^{2t}X_{i}$.
The case of (FRS) was proved in \cite[Lemma 7.3]{GH19}. For flatness,
we prove the claim in the case $K=\rats$, and the reduction to $\rats$
is similar to \cite[Section 6]{GH19}.

By choosing a $\ints$-model, we may assume that $\psi$ is defined
over $\ints$. Denote by $\mu_{X(\mathbb{F}_{p})}$ and $\mu_{V(\mathbb{F}_{p})}$
the uniform probability measures on $X(\mathbb{F}_{p})$ and $V(\mathbb{F}_{p})$,
and let $\tau_{\psi,p}:=\psi_{*}(\mu_{X(\mathbb{F}_{p})})$, and $\tau_{i,p}:=\varphi_{i*}(\mu_{X_{i}(\mathbb{F}_{p})})$.
There exists $p_{0}$ such that for every $p>p_{0}$, the map $\varphi_{i}^{*t}$
is flat over $\mathbb{F}_{p}$ for every $i$. By Corollary \ref{thm: flatness and counting points},
there exists $C>0$ such that for every $p>p_{0}$ and every $i$,
we have $\left\Vert \tau_{i,p}^{*t}\right\Vert _{\infty}<C$. Therefore
also 
\[
\left\Vert \mathcal{F}(\tau_{i,p})\right\Vert _{2t}^{t}=\left\Vert \mathcal{F}(\tau_{i,p})^{t}\right\Vert _{2}=\left\Vert \tau_{i,p}^{*t}\right\Vert _{2}\leq\left\Vert \tau_{i,p}^{*t}\right\Vert _{\infty}<C.
\]
By a generalization of H\"older's inequality:
\[
\left\Vert \tau_{\psi,p}\right\Vert _{\infty}\leq\left\Vert \mathcal{F}(\tau_{\psi,p})\right\Vert _{1}=\left\Vert \prod_{i=1}^{2t}\mathcal{F}(\tau_{i,p})\right\Vert _{1}\leq\prod_{i=1}^{2t}\left\Vert \mathcal{F}(\tau_{i,p})\right\Vert _{2t}<C^{2}.
\]
By Corollary \ref{thm: flatness and counting points}, this implies
that $\psi$ is flat so we are done. 
\end{proof}
\begin{rem}
\label{rem:also works for algebraic groups}Proposition \ref{prop:reduction to self convolution- the general case}
is also true if we replace $V$ with a connected algebraic $K$-group.
Instead of the commutative Fourier transform, one has to work with
the non-commutative one. For a very similar argument, see \cite[Proposition 5.3]{GH}.
\end{rem}

\subsection{Combinatorial description of homogeneous word maps}

We would like to generalize the construction discussed in Section
\ref{subsec:Combinatirical-picture-for commutator} to the case of
homogeneous Lie algebra word maps of any degree.

\subsubsection{The combinatorial gadget}
\begin{defn}
~\label{def:polyhypergraph}
\begin{enumerate}
\item A\textit{ $d$-polyhypergraph (or $d$-PH)} is a triplet $\Gamma:=(\mathcal{I},\mathcal{J},\mathcal{S})$
consisting of finite sets $\mathcal{I},\mathcal{J}$ and a subset
$\mathcal{S}\subseteq\mathcal{I}^{(d)}\times\mathcal{J}$, where $\mathcal{I}^{(d)}$
is the set of multisets of $\mathcal{I}$ of size $d$. An element
$i\in\mathcal{I}$ is called a \textsl{vertex} of $\Gamma$. An element
$\{i_{1},\dots,i_{d}\}\in\mathcal{I}^{(d)}$ is called a \textsl{hyperedge
of }$\Gamma$, if there exists $l\in\mathcal{J}$ such that $(\{i_{1},\dots,i_{d}\},l)\in\mathcal{S}$.
A hyperedge $\{i_{1},\dots,i_{d}\}\in\mathcal{I}^{(d)}$ is \textsl{of
type $l\in\mathcal{J}$} if $(\{i_{1},\dots,i_{d}\},l)\in\mathcal{S}$.
In particular, $\mathcal{S}$ is the collection of all tuples consisting
of a hyperedge together with a type. 
\item Given a $d$-polyhypergraph $\Gamma=(\mathcal{I},\mathcal{J},\mathcal{S})$,
an assignment $\mathrm{V}=(V(i))_{i\in\mathcal{I}}$ of a vector space
$V(i)$ to each vertex $i\in\mathcal{I}$, and a collection $\{\tau_{\gamma}\}_{\gamma\in\mathcal{S}}$
of $d$-homogeneous multilinear forms $\tau_{\{i_{1},\dots,i_{d}\},l}:\bigoplus\limits _{i\in\{i_{1},\dots,i_{d}\}}V(i)\rightarrow K$,
we define a map $\Phi_{\Gamma,\mathrm{V},\{\tau_{\gamma}\}}:\bigoplus\limits _{i\in\mathcal{I}}V(i)\rightarrow K^{\mathcal{J}}$
by 
\begin{equation}
\Phi_{\Gamma,\mathrm{V},\{\tau_{\gamma}\}}(v)=\left(\sum\limits _{(\{i_{1},\dots,i_{d}\},l)\in\mathcal{S}}\tau_{\{i_{1},\dots,i_{d}\},l}(v)\right)_{l\in\mathcal{J}}.\label{eq:(4.1)}
\end{equation}
\item Let $\Gamma$ be a $d$-PH. A triplet $(\Gamma,\mathrm{V},\{\tau_{\gamma}\})$
is called an \textit{assignment} of $\Gamma$. If all of the $V(i)$
are of the same dimension $N$, we write $(\Gamma,V,\{\tau_{\gamma}\})$,
where $V$ is just an $N$-dimensional vector space. 
\item We say that the assignment $(\Gamma,\mathrm{V},\{\tau_{\gamma}\})$
is flat (resp.~(FRS)) if the map $\Phi_{\Gamma,\mathrm{V},\{\tau_{\gamma}\}}$
is flat (resp.~(FRS)). 
\item Let $\Gamma=(\mathcal{I},\mathcal{J},\mathcal{S})$ be a $d$-PH.
Assume that $\mathcal{I}=\bigsqcup\limits _{u=1}^{N}\mathcal{I}_{u}$,
$\mathcal{J}=\bigsqcup\limits _{u=1}^{N}\mathcal{J}_{u}$ and that
$\mathcal{S}=\bigsqcup\limits _{u=1}^{N}\mathcal{S}_{u}$, where $\mathcal{S}_{u}:=\mathcal{S}\cap\left(\mathcal{I}_{u}^{(d)}\times\mathcal{J}_{u}\right)$.
In this case we say that $\Gamma$ is a \textsl{disjoint union of}
$\{\Gamma_{u}\}_{u\in[N]}$, where $\Gamma_{u}=(\mathcal{I}_{u},\mathcal{J}_{u},\mathcal{S}_{u})$. 
\end{enumerate}
\end{defn}

\begin{rem}
\label{rem:The forms =00005Ctau behave nicely}We further make the
assumption that
\[
\tau_{\{i_{1},\dots,i_{d}\},l}((\lambda_{i}v_{i})_{i\in\mathcal{I}})=\left(\prod\limits _{i=1}^{d}\lambda_{i}^{k_{i}}\right)\tau_{\{i_{1},\dots,i_{d}\},l}((v_{i})_{i\in\mathcal{I}}),
\]
where $k_{i}:=\left|\{j\in[d]:i_{j}=i\}\right|$. 
\end{rem}

Let us give a concrete example for a $d$-PH induced from a $d$-homogeneous
word map $\varphi_{w}:\mathfrak{g}^{r}\rightarrow\mathfrak{g}$, which
is the main object of study in this section and in Section \ref{sec:Proof-of-Theorems- Lie algebra}.

\subsubsection{\label{subsec:Attaching-a-polyhypergraph to a word map}Attaching
a polyhypergraph to a homogeneous Lie algebra word map}

Assume that $K$ is algebraically closed. Let $\mathfrak{g}$ be a
simple Lie algebra, $\mathfrak{t}$ a Cartan subalgebra, and write
$\mathfrak{g}=\mathfrak{t}\oplus\bigoplus\limits _{\alpha\in\Sigma(\mathfrak{g},\mathfrak{t})}\mathfrak{g}_{\alpha}$,
where $\Sigma(\mathfrak{g},\mathfrak{t})$ is the corresponding root
system. Let $\triangle\subseteq\Sigma(\mathfrak{g},\mathfrak{t})$
be the set of simple roots. We choose a Chevalley basis $\mathcal{B}=\{h_{\alpha}\}_{\alpha\in\triangle}\cup\{e_{\beta}\}_{\beta\in\Sigma(\mathfrak{g},\mathfrak{t})}$,
such that $e_{\beta}\in\mathfrak{g}_{\beta}$, $h_{\alpha}=[e_{\alpha},e_{-\alpha}]\in\mathfrak{t}$,
and $[e_{\beta},e_{-\beta}]\in\sum\limits _{\gamma\in\triangle}\ints h_{\gamma}$,
for any $\alpha\in\triangle$ and $\beta\in\Sigma(\mathfrak{g},\mathfrak{t})$.
Now let $w\in\mathcal{L}_{r}(X_{1},\dots,X_{r})$ be a $d$-homogeneous
word and let $\varphi_{w}:\mathfrak{g}^{r}\rightarrow\mathfrak{g}$
be its corresponding Lie algebra word map. Recall from Fact \ref{fact:sum of left normed monomials}
that $w$ can be written as 
\[
w(X_{1},\dots,X_{r})=\sum\limits _{S\in[r]^{d}}c_{S}X_{S},
\]
for constants $c_{S}$. Set $\mathcal{I}=\mathcal{B}$, and let $\mathcal{J}=\widehat{\mathcal{B}}$
be the dual basis of $\mathcal{B}$. Consider $\varphi_{w}^{*t}:\mathfrak{g}^{rt}\rightarrow\mathfrak{g}$,
let $\{X_{s,k}=\sum\limits _{i\in\mathcal{I}}a_{i,s,k}e_{i}\}_{(s,k)\in[r]\times[t]}$
be $tr$ elements in $\mathfrak{g}$, and write $\mathrm{MS}:\mathcal{I}^{d}\rightarrow\mathcal{I}^{(d)}$
for the forgetful map from $d$-tuples to $d$-multisets. Then 
\begin{align*}
\varphi_{w}^{*t}(X_{1,1},\dots,X_{r,t}) & =\sum\limits _{k=1}^{t}\varphi_{w}(X_{1,k},\dots,X_{r,k})=\sum\limits _{k=1}^{t}\sum\limits _{S\in[r]^{d}}c_{S}[X_{s_{1},k},\ldots,X_{s_{r},k}]\\
 & =\sum\limits _{k=1}^{t}\sum\limits _{S\in[r]^{d}}c_{S}[\sum\limits _{i\in\mathcal{I}}a_{i,s_{1},k}e_{i},\dots,\sum\limits _{i\in\mathcal{I}}a_{i,s_{d},k}e_{i}]\\
 & =\sum\limits _{k=1}^{t}\sum\limits _{I\in\mathcal{I}^{(d)}}\sum\limits _{(i_{1},\dots,i_{d})\in\mathrm{MS}^{-1}(I)}\sum\limits _{S\in[r]^{d}}c_{S}\prod\limits _{u=1}^{d}a_{i_{u},s_{u},k}[e_{i_{1}},\dots,e_{i_{d}}].
\end{align*}
Rewriting this equation gives the following: 
\begin{lem}
\label{lem:combinatorial word map}Let $V$ be an $rt$-dimensional
vector space. Under the identification $\mathfrak{g}^{rt}\simeq\mathfrak{g}\otimes V\simeq V^{\mathcal{I}}$,
the map $\varphi_{w}^{*t}:\mathfrak{g}^{rt}\simeq V^{\mathcal{I}}\rightarrow K^{\mathcal{J}}\simeq\mathfrak{g}$
is given by 
\[
\varphi_{w}^{*t}((v_{i})_{i\in\mathcal{I}})=\left(\sum\limits _{k=1}^{t}\sum\limits _{\{i_{1},\dots,i_{d}\}\in\mathcal{I}^{(d)}}\tau_{\{i_{1},\dots,i_{d}\},l}((v_{i})_{i\in\mathcal{I}})\right)_{l\in\mathcal{J}}
\]
where $v_{i}=(v_{i,s,k})_{s\in[r],k\in[t]}$ and 
\[
\tau_{\{i_{1},\dots,i_{d}\},l}((v_{i})_{i\in\mathcal{I}})=\sum\limits _{(i'_{1},\dots,i'_{d})\in\mathrm{MS}^{-1}(\{i_{1},\dots,i_{d}\})}\sum\limits _{S\in[r]^{d}}c_{S}\prod\limits _{u=1}^{d}v_{i'_{u},s_{u},k}\widehat{e}_{l}([e_{i'_{1}},\dots,e_{i'_{d}}]).
\]
\end{lem}

\begin{rem}
\label{rem:different description for Lie algebra word maps}Given
an embedding $\mathfrak{g}\hookrightarrow\mathfrak{gl}_{N}$, we may
replace $[X,Y]$ with $X\cdot Y-Y\cdot X$ and write any Lie algebra
word map $\varphi_{w}:\mathfrak{g}^{r}\rightarrow\mathfrak{g}$ as
a restriction to $\mathfrak{g}$ of a matrix word map $\varphi_{\widetilde{w}}:M_{N}^{r}\rightarrow M_{N}$,
with $\widetilde{w}=\sum\limits _{S\in[r]^{d}}d_{S}X_{s_{1}}\cdot...\cdot X_{s_{d}}$
for some constants $d_{S}$. 
\end{rem}

\begin{example}
\label{exa:for a D-PH of a word}Let $\mathfrak{g}$ be a simple Lie
algebra, let $w(X_{1},X_{2})=[[X_{1},X_{2}],X_{2}]$, let $t=1$,
and consider $\varphi_{w}^{*t}=\varphi_{w}$. In particular $\mathfrak{g}^{rt}=\mathfrak{g}^{2}\simeq\mathfrak{g}\otimes V\simeq V^{\mathcal{I}}$
for $V$ a two dimensional vector space. We have the following: 
\begin{align*}
\left(\varphi_{w}((v_{i})_{i\in\mathcal{I}})\right)_{l} & =\widehat{e}_{l}\left(\varphi_{w}\left(\sum\limits _{i\in\mathcal{I}}v_{i,1}e_{i},\sum\limits _{i\in\mathcal{I}}v_{i,2}e_{i}\right)\right)\\
 & =\widehat{e}_{l}\left([\sum\limits _{i\in\mathcal{I}}v_{i,1}e_{i},\sum\limits _{i\in\mathcal{I}}v_{i,2}e_{i},\sum\limits _{i\in\mathcal{I}}v_{i,2}e_{i})]\right)\\
 & =\sum\limits _{\{i'_{1},i'_{2},i'_{3}\}\in\mathcal{I}^{(3)}}\sum\limits _{(i_{1},i_{2},i_{3})\in\mathrm{MS}^{-1}(\{i'_{1},i'_{2},i'_{3}\})}v_{i_{1},1}v_{i_{1},2}v_{i_{3},2}\left(\hat{e}_{l}([e_{i_{1}},e_{i_{2}},e_{i_{3}}])\right).
\end{align*}
For example, taking $\{i'_{1},i'_{2},i'_{3}\}=\{1,2,2\}$, we have
$\mathrm{MS}^{-1}(\{i'_{1},i'_{2},i'_{3}\})=\{(1,2,2),(2,1,2),(2,2,1)\}$,
and
\begin{align*}
 & \tau_{\{1,2,2\},l}(v_{1,1},v_{1,2},v_{2,1},v_{2,2})\\
 & =v_{1,1}v_{2,2}^{2}\hat{e}_{l}([e_{1},e_{2},e_{2}])+v_{2,1}v_{1,2}v_{2,2}\hat{e}_{l}([e_{2},e_{1},e_{2}])+v_{2,1}v_{2,2}v_{1,2}\hat{e}_{l}([e_{2},e_{2},e_{1}])\\
 & =\left(v_{1,1}v_{2,2}^{2}-v_{2,1}v_{1,2}v_{2,2}\right)\hat{e}_{l}([e_{1},e_{2},e_{2}]).
\end{align*}
\end{example}

\begin{defn}
\label{def:D-PH assigned to a Lie algebra word map }Let $w\in\mathcal{L}_{r}$
be a $d$-homogeneous word, and $\mathfrak{g}$ a simple Lie algebra.
Define the $d$-PH corresponding to $\mathfrak{g}$ and $w$ denoted
$\Gamma_{\mathfrak{g},w}=(\mathcal{I},\mathcal{J},\mathcal{S})$ as
follows. Take $\mathcal{I}$ and $\mathcal{J}$, to be a chosen Chevalley
basis $\mathcal{B}=\{e_{i}\}_{i\in\mathcal{I}}$, and its dual $\widehat{\mathcal{B}}=\{\widehat{e}_{j}\}_{j\in\mathcal{J}}$.
Define $\mathcal{S}\subseteq\mathcal{I}^{(d)}\times\mathcal{J}$ by
\[
\mathcal{S}:=\{(\{i_{1},\dots,i_{d}\},l)\in\mathcal{I}^{(d)}\times\mathcal{J}:\tau_{\{i_{1},\dots,i_{d}\},l}\neq0\},
\]
with $\tau_{\{i_{1},\dots,i_{d}\},l}$ as in Lemma \ref{lem:combinatorial word map}.
Notice that for each $l\in\mathcal{J}$ there exists at least one
$\{i_{1},\dots,i_{d}\}$ such that $\tau_{\{i_{1},\dots,i_{d}\},l}$
is non-zero, as otherwise the map $\varphi_{w}$ is not generating,
contrary to our assumption. Let $V$ be an $tr$-dimensional vector
space. Then $(\Gamma_{\mathfrak{g},w},V,\{\tau_{\gamma}\}_{\gamma\in\mathcal{S}})$
is an assignment of $\Gamma_{\mathfrak{g},w}$. Notice that the collection
$\{\tau_{\gamma}\}_{\gamma\in\mathcal{S}}$ satisfies the demand of
Remark \ref{rem:The forms =00005Ctau behave nicely}.
\end{defn}

\subsubsection{\label{subsec:Coloring-and-elimination for d-PH}Coloring and elimination
for polyhypergraphs}

We next describe two combinatorial procedures on a given assignment
of a $d$-PH: the elimination method and the coloring method. The
elimination method is a direct generalization of Corollary \ref{cor:(-elimination-for-polygraphs)}
(c.f. \cite[Corollary 2.28]{AA16}) to the case of polyhypergraphs.
The coloring method was described in \cite[Corollary 2.33]{AA16}
in the case of graphs, but can be easily generalized to polygraphs
and polyhypergraphs.

Let $\Gamma=(\mathcal{I},\mathcal{J},\mathcal{S})$ be a $d$-PH.
Let $\omega\in\mathbb{Z}^{\mathcal{I}}$ be a vector (termed an $\mathcal{I}$-weight).
Then $\omega$ induces an $\mathcal{I}^{(d)}$-weight $\widetilde{\omega}:\mathcal{I}^{(d)}\rightarrow\mathbb{Z}$
by $\widetilde{\omega}(\{i_{1},\dots,i_{d}\})=\sum\limits _{u=1}^{d}\omega(i_{u})$.
We define $\mathrm{gr}_{\omega}(\Gamma):=(\mathcal{I},\mathcal{J},\mathrm{gr}_{\omega}\mathcal{S})$
by 
\[
\mathrm{gr}_{\omega}\mathcal{S}=\{(I,l)\in\mathcal{S}:\forall(I',l)\in\mathcal{S}\text{ we have }\widetilde{\omega}(I)\leq\widetilde{\omega}(I')\},
\]
where $I,I'\in\mathcal{I}^{(d)}$ and $l\in\mathcal{J}$. 
\begin{cor}[{cf.~\cite[Corollary 2.28]{AA16}}]
\label{cor:elimination for dPH}Let $(\Gamma=(\mathcal{I},\mathcal{J},\mathcal{S}),\mathrm{V},\{\tau_{\gamma}\}_{\gamma\in\mathcal{S}})$
be an assignment of a $d$-PH, and let $\omega\in\mathbb{Z}^{\mathcal{I}}$
be an $\mathcal{I}$-weight. If $(\mathrm{gr}_{\omega}(\Gamma),\mathrm{V},\{\tau_{\gamma}\})$
is flat (resp.~(FRS)), then so is $(\Gamma,\mathrm{V},\{\tau_{\gamma}\})$. 
\end{cor}

\begin{proof}
By Corollary \ref{cor:(cf.---Elimination)}, it is enough to show
that the symbol $\sigma_{\omega}(\Phi_{\Gamma,\mathrm{V},\{\tau_{\gamma}\}})$
is equal to $\Phi_{\mathrm{gr}_{\omega}\Gamma,\mathrm{V},\{\tau_{\gamma}\}}$.
By our assumption on the forms $\{\tau_{\gamma}\}_{\gamma\in\mathcal{S}}$
(Remark \ref{rem:The forms =00005Ctau behave nicely}) and by Equation
(\ref{eq:(4.1)}) it is easy to see that the symbol $\sigma_{\omega}(\Phi_{\Gamma,\mathrm{V},\{\tau_{\gamma}\}})$
is indeed $\Phi_{\mathrm{gr}_{\omega}\Gamma,\mathrm{V},\{\tau_{\gamma}\}}$. 
\end{proof}
\begin{rem}[Warning!]
\label{rem:warning-elimination}Consider the map $\varphi:\mathbb{A}^{2}\rightarrow\mathbb{A}^{2}$
defined by $\varphi(x,y)=(y^{2}+x,x)$, and let $\omega=(1,1)$ be
the weight which induces the standard degree of monomials. Note that
$\varphi$ is dominant, while $\sigma_{\omega}(\varphi)(x,y)=(x,x)$
is not even generating, so $\sigma_{\omega}(\varphi)$ will never
become (FRS) after any number of convolutions. This phenomenon might
occur when applying the elimination method, so we must choose our
weight functions carefully. 
\end{rem}

Let $\Gamma:=(\mathcal{I},\mathcal{J},\mathcal{S})$ be a $d$-PH.
We describe the coloring method for $\Gamma$. Let $M$ be a finite
set, and let $\omega:\mathcal{I}\rightarrow\mathbb{Z}^{M}$ be a map
which we call a \textit{coloring function}. Define $\widetilde{\omega}:\mathcal{I}^{(d)}\rightarrow\mathbb{Z}^{M}$
by $\widetilde{\omega}(\{i_{1},\dots,i_{d}\}):=\sum\limits _{u=1}^{d}\omega(i_{u})$,
and for any $m\in M$ set $\mathrm{gr}_{\omega,m}\Gamma:=(\mathcal{I},\mathrm{gr}_{\omega,m}\mathcal{J},\mathrm{gr}_{\omega,m}\mathcal{S})$,
where 
\begin{equation}
\mathrm{gr}_{\omega,m}\mathcal{S}=\{(I,l)\in\mathcal{S}:\forall m'\in M\text{ we have }\widetilde{\omega}(I)_{m}\leq\widetilde{\omega}(I)_{m'}\}\subseteq\mathcal{S},\label{eq:coloring for d-PH}
\end{equation}
\[
\mathrm{gr}_{\omega,m}\mathcal{J}=\{l\in\mathcal{J}:\exists I\in\mathcal{I}^{(d)}\text{ such that }(I,l)\in\mathrm{gr}_{\omega,m}\mathcal{S}\}.
\]
We say that the coloring function $\omega$ is \textit{admissible}
if for any $l\in\mathcal{J}$, there exists $m\in M$ such that for
any hyperedge $(I,l)\in\mathcal{S}$, the number $\widetilde{\omega}(I)_{m}$
is the unique minimum of the tuple $\widetilde{\omega}(I)\in\mathbb{Z}^{M}$.
We have the following analogue of Corollary \ref{cor:(-Coloring-method)-Let}: 
\begin{cor}[{cf.~\cite[Corollary 2.33]{AA16}, the coloring method}]
\label{cor:Coloring for for dPH}Let $\Gamma:=(\mathcal{I},\mathcal{J},\mathcal{S})$
be a $d$-PH, let $M$ be a finite set, let $\omega:\mathcal{I}\rightarrow\mathbb{Z}^{M}$
be an admissible coloring function and let $(\mathrm{gr}_{\omega,m}\Gamma,\mathrm{V}_{m},\{\tau_{\gamma}^{m}\}_{\gamma})$
be assignments of $\mathrm{gr}_{\omega,m}\Gamma$ for each $m\in M$.
Suppose that $(\mathrm{gr}_{\omega,m}\Gamma,\mathrm{V}_{m},\{\tau_{\gamma}^{m}\}_{\gamma})$
is flat (resp.~(FRS)) for each $m\in M$. Then so is $(\Gamma,\mathrm{V},\{\tau_{\gamma}\}_{\gamma})$,
with $\mathrm{V}(i):=\bigoplus\limits _{m=1}^{M}\mathrm{V}_{m}(i)$
and $\tau_{\gamma}=\bigoplus\limits _{i=1}^{M}\tau_{\gamma}^{m}$
for each $\gamma\in\mathcal{S}$. 
\end{cor}

\begin{proof}
Set $\Gamma_{M}:=(\mathcal{I}\times M,\mathcal{J},\mathcal{S}_{M})$,
where $\mathcal{S}_{M}=\{((I,m),l):(I,l)\in\mathcal{S},m\in M\}$,
to be the $d$-PH consisting of $M$ copies of $\Gamma=(\mathcal{I},\mathcal{J},\mathcal{S})$.
Then we define a weight vector $\omega'\in\mathbb{Z}^{\mathcal{I}\times M}$,
by $\omega'(i,m):=\omega(i)_{m}$. Consider the assignment $(\Gamma_{M},\widetilde{\mathrm{V}},\{\widetilde{\tau}_{(\gamma,m)}\})$
of $\Gamma_{M}$, with $\widetilde{\mathrm{V}}(i,m)=\mathrm{V}_{m}(i)$
and $\widetilde{\tau}_{(\gamma,m)}:=\tau_{\gamma}^{m}$ for any $m\in M$
and $\gamma\in\mathcal{S}$. Notice that $(\Gamma_{M},\widetilde{\mathrm{V}},\{\widetilde{\tau}_{(\gamma,m)}\})$
is flat (resp.~(FRS)) if and only if $(\Gamma,\mathrm{V},\{\tau_{\gamma}\})$
is flat (resp.~(FRS)), as both polygraphs induce the same morphism.
Notice that by the admissibility of $\omega$, we get that $\mathrm{gr}_{\omega'}(\Gamma_{M})$
is a disjoint union (in the sense of Definition \ref{def:polyhypergraph})
of the polyhypergraphs $\mathrm{gr}_{\omega,m}\Gamma$ for $1\leq m\leq M$.
In particular, the fiber $\Phi_{\mathrm{gr}_{\omega'}(\Gamma_{M}),\widetilde{\mathrm{V}},\{\widetilde{\tau}_{(\gamma,m)}\}}^{-1}(0)$
is isomorphic to the product $\prod\limits _{m\in M}\Phi_{\mathrm{gr}_{\omega,m}\Gamma,\mathrm{V}_{m},\{\tau_{\gamma}^{m}\}}^{-1}(0)$.
The claim now follows from Corollary \ref{cor:elimination for dPH}. 
\end{proof}
\begin{rem}
\label{rem:The-coloring-method is convolution+elimination}The coloring
method is just a different manifestation of the elimination method,
after first applying convolutions; let $M$ be a finite set, $(t_{m})_{m\in M}\in\ints_{>0}^{|M|}$,
with $t=\sum\limits _{m=1}^{\left|M\right|}t_{m}$, let $\varphi:\mathbb{A}^{n_{2}}\rightarrow\mathbb{A}^{n_{1}}$
be a morphism, and let $\omega:[n_{2}]\rightarrow\ints^{M}$ be a
(coloring) function. Then the coloring method can be seen as first
considering $\varphi^{*t}:\mathbb{A}^{n_{2}t}\rightarrow\mathbb{A}^{n_{1}}$,
and then applying the elimination method, by assigning the weight
$\nu_{\omega}:[n_{2}]\times[t]\rightarrow\ints$ by $\nu_{\omega}(i,l)=\omega(i)_{m(l)}$,
where $m(l)=m'$ if $\sum\limits _{j=1}^{m'-1}t_{j}<l\leq\sum\limits _{j=1}^{m'}t_{j}$.
The admissibility of the coloring $\omega$ guarantees that $\sigma_{\nu_{\omega}}(\varphi^{*t})$
is a disjoint union of morphisms $\sigma_{\omega_{m}}(\varphi^{*t_{m}})$.
In particular, if $\sigma_{\omega_{m}}(\varphi^{*t_{m}})$ is flat
or (FRS) for every $m\in M$, then so is $\varphi^{*t}$. 
\end{rem}

\begin{defn}
\label{def:convolutions of a d-PH}Let $\Gamma:=(\mathcal{I},\mathcal{J},\mathcal{S})$
be a $d$-PH and let $(\Gamma,\mathrm{V},\{\tau_{\gamma}\})$ be an
assignment. Then $(\Gamma,\mathrm{V}^{\oplus t},\{\tau_{\gamma}^{\oplus t}\})$
is called the\textsl{ $t$-th self-convolution} of $(\Gamma,\mathrm{V},\{\tau_{\gamma}\})$.
In particular, $(\Gamma,\mathrm{V},\{\tau_{\gamma}\})$ is said to
be \emph{flat} (resp.~\emph{(FRS)}) \emph{after $t$ convolutions}
if $(\Gamma,\mathrm{V}^{\oplus t},\{\tau_{\gamma}^{\oplus t}\})$
is flat (resp.~(FRS)). 
\end{defn}

The following is a special case of a $d$-PH which will be particularly
important to us, and is easier to analyze: 
\begin{defn}
A \textit{$d$-hypergraph $\Gamma$} is a pair $\Gamma=(\mathcal{V},\mathcal{E})$,
where $\mathcal{V}$ a set of vertices and $\mathcal{E}\subseteq\mathcal{V}^{(d)}$
is the set of $d$-hyperedges, where $\mathcal{V}^{(d)}$ is the set
of multisubsets of $\mathcal{V}$ of size $d$. Any $d$-hypergraph
induces a natural $d$-PH 
\[
\mathcal{P}(\Gamma):=(\mathcal{V},\mathcal{E},\triangle\mathcal{E})
\]
where $\triangle\mathcal{E}\subseteq\mathcal{V}^{(d)}\times\mathcal{E}$
embedded diagonally. An \textsl{assignment} of $\Gamma$ is just an
assignment of $\mathcal{P}(\Gamma)$. Similarly as for polyhypergraphs,
we say that an assignment $(\Gamma,\mathrm{V},\{\tau_{I}\}_{I\in\mathcal{E}})$
of $\Gamma$ is (FRS) if $(\mathcal{P}(\Gamma),V,\{\tau_{I}\}_{I\in\mathcal{E}})$
is (FRS). 
\end{defn}

Summarizing the above two methods, the elimination method helps us
eliminate edges of a given $d$-PH, and results in a $d$-PH with
fewer edges. The coloring method consists of first constructing a
larger $d$-PH $\Gamma_{M}$ which corresponds to taking $M$ self-convolutions
and consists of $M$ copies of our $d$-PH $\Gamma$, and then using
the elimination method on $\Gamma_{M}$. The use of the coloring method
is essential in following two cases: 
\begin{enumerate}
\item Once we manage to reduce a polyhypergraph to a hypergraph, the elimination
method will usually not be effective. 
\item In case we are given a $d$-PH, denoted $\Gamma$, but we cannot guarantee
that the elimination method will result in new polygraph $\mathrm{gr}_{\omega}(\Gamma)$
whose corresponding map $\Phi_{\mathrm{gr}_{\omega}\Gamma,\mathrm{V},\{\tau_{\gamma}\}}$
is generating. 
\end{enumerate}

\section{\label{sec:Proof-of-Theorems- Lie algebra}Lie algebra word maps:
proof of the main theorems}

In this section we prove Theorems \ref{thm: main thm Lie algebra word maps},
\ref{thm: main theorem for matrix word map} and \ref{thm:lower bounds on epsilon jet flatness}.
By Proposition \ref{prop:reduction to self convolutions} we may consider
self-convolutions of a single $d$-homogeneous word $w\in\mathcal{L}_{r}$
of pure type. Our main methodology is to encode $\varphi_{w}^{*t}:\mathfrak{g}^{r}\rightarrow\mathfrak{g}$
in terms of an assignment $(\Gamma_{\mathfrak{g},w},V,\{\tau_{\gamma}\}_{\gamma\in\mathcal{S}})$
of a $d$-PH $\Gamma_{\mathfrak{g},w}$, as in Definition \ref{def:D-PH assigned to a Lie algebra word map },
with $V$ a $tr$-dimensional vector space. Using the elimination
and coloring methods, we reduce $\Gamma_{\mathfrak{g},w}$ to a simpler
object, at the price of enlarging $\mathrm{dim}V$. Here, a simpler
object is either a $d$-PH consisting of (disjoint unions of) a single
hyperedge, or a $d$-PH with a small number of types (relative to
$d$). A complication arises when $\mathrm{rk}(\mathfrak{g})$ is
small. This requires special attention. The structure of this section
is as follows: 
\begin{itemize}
\item In Subsection \ref{subsec:Discussion-on-the methods} we discuss the
main methods we use - the elimination and coloring methods. These
methods are used repeatedly throughout Section \ref{sec:Proof-of-Theorems- Lie algebra}.
\item In Subsection \ref{subsec:Proof-for low rank and few edges} we deal
with low rank simple Lie algebras (in particular with all exceptional
Lie algebras) and with polyhypergraphs with a few types of hyperedges. 
\item Subsections \ref{subsec:Proof-for-SLn} and \ref{subsec:Proof-for-the general case}
deal with the case of $\text{\ensuremath{\mathfrak{g}}=}\mathfrak{sl}_{n}$
and the general case, respectively. 
\item In Subsection \ref{subsec:Matrix-word-maps} we discuss matrix word
maps, and prove Theorem \ref{thm: main theorem for matrix word map}. 
\end{itemize}

\subsection{\label{subsec:Discussion-on-the methods}Discussion of the main methods}

Throughout Section \ref{sec:Proof-of-Theorems- Lie algebra} we use
the elimination and coloring methods abundantly, for polyhypergraphs,
hypergraphs and general polynomial maps. Our applications of the coloring
and elimination methods can be clustered into three types: 
\begin{enumerate}
\item The \textbf{\textit{averaging method}}, using $\omega_{\mathrm{av}}$. 
\item The \textbf{\textit{monomialization method}}, using $\omega_{\mathrm{mon}}$. 
\item The \textbf{\textit{level separation method}}, using $\omega_{\mathrm{LS}}$. 
\end{enumerate}
Let us explain these procedures using a simple example, which is very
similar to the analysis of the jets of the commutator map considered
in Section \ref{sec:The-commutator-map-revisited}. Consider the following
map (which is the $2t$-th jet map of $(x,y)\mapsto xy$): 
\[
\varphi(x_{0},y_{0},x_{1},y_{1},\dots,x_{2t},y_{2t})=(x_{0}y_{0},x_{0}y_{1}+x_{1}y_{0},\sum\limits _{j=0}^{2}x_{j}y_{2-j},\dots,\sum\limits _{j=0}^{2t}x_{j}y_{2t-j}).
\]
We first apply the weight $\omega_{\mathrm{av}}(x_{u})=\omega_{\mathrm{av}}(y_{u})=3^{u}$.
Notice that $\mathrm{deg}_{\omega_{\mathrm{av}}}(x_{j}y_{u-j})=3^{j}+3^{u-j}$
so that $\mathrm{deg}_{\omega_{\mathrm{av}}}$ is minimized when $j$
is as close as possible to $\frac{u}{2}$. More precisely, 
\[
\sigma_{\omega_{\mathrm{av}}}(\varphi)(x_{0},y_{0},x_{1},y_{1},\dots,x_{2t},y_{2t})=(x_{0}y_{0},x_{0}y_{1}+x_{1}y_{0},x_{1}y_{1},\dots,x_{t-1}y_{t}+x_{t}y_{t-1},x_{t}y_{t}).
\]
We see that $\omega_{\mathrm{av}}$ not only eliminated most of the
monomials, but also that for any $u_{1},u_{2}$ such that $\left|u_{1}-u_{2}\right|>1$,
the polynomials $\sigma_{\omega_{\mathrm{av}}}(\varphi)_{u_{1}}$
and $\sigma_{\omega_{\mathrm{av}}}(\varphi)_{u_{2}}$ have disjoint
variables, while for example, in $\varphi_{2t}$ all of the variables
$x_{0},y_{0},\dots,x_{t},y_{t}$ are involved. We call the use of
$\omega_{\mathrm{av}}$ and its variations the \textbf{\textit{averaging
method}}.

Let us now apply a different type of weight. We first order the variables
by $\mathrm{ord}(x_{u})=2u+1$, and $\mathrm{ord}(y_{u})=2u+2$ such
that $x_{0}<y_{0}<\ldots<x_{2t}<y_{2t}$. Set 
\[
\omega_{\mathrm{mon}}(x_{u}):=3^{\mathrm{ord}(x_{u})}\text{ and }\omega_{\mathrm{mon}}(y_{u}):=3^{\mathrm{ord}(y_{u})}.
\]
Notice that $\mathrm{deg}_{\omega_{\mathrm{mon}}}(x_{u-1}y_{u})>\mathrm{deg}_{\omega_{\mathrm{mon}}}(x_{u}y_{u-1})$
and therefore: 
\[
\psi:=\sigma_{\omega_{\mathrm{mon}}}\left(\sigma_{\omega_{\mathrm{av}}}(\varphi)\right)=(x_{0}y_{0},x_{1}y_{0},x_{1}y_{1},\dots,x_{t}y_{t-1},x_{t}y_{t})
\]
is a monomial map. The use of $\omega_{\mathrm{mon}}$ is designed
to degenerate a polynomial map to a monomial map, so we refer to the
use of $\omega_{\mathrm{mon}}$ and its variations as the \textbf{\textit{monomialization
method.}} In the case of polyhypergraphs, we will use $\omega_{\mathrm{mon}}$
to turn a polyhypergraph into a hypergraph.

Since the map is already monomial, we cannot use the elimination method
(at least without a change of variables). Instead, we apply the coloring
method, by applying one convolution, and then use the elimination
method. The fact that the various coordinates of $\psi$ have small
interactions with their neighbors (a feature we earned by using $\omega_{\mathrm{av}}$),
makes this method effective. Explicitly, we consider 
\[
\psi^{*2}(x_{0},y_{0},z_{0},w_{0},\dots,x_{2t},y_{2t},z_{2t},w_{2t})=\psi(x_{0},y_{0},\dots,x_{2t},y_{2t})+\psi(z_{0},w_{0},\dots.,z_{2t},w_{2t})
\]
and then use the weights 
\[
\omega_{\mathrm{LS}}(x_{u})=2\cdot10^{u},~\omega_{\mathrm{LS}}(y_{u})=3\cdot10^{u},~\omega_{\mathrm{LS}}(z_{u})=3\cdot10^{u},~\text{ and }\omega_{\mathrm{LS}}(w_{u})=10^{u}.
\]
Notice that $\mathrm{deg}_{\omega_{\mathrm{LS}}}(x_{u}y_{u})>\mathrm{deg}_{\omega_{\mathrm{LS}}}(z_{u}w_{u})$
but $\mathrm{deg}_{\omega_{\mathrm{LS}}}(x_{u+1}y_{u})<\mathrm{deg}_{\omega_{\mathrm{LS}}}(z_{u+1}w_{u})$
for any $u$. Therefore we get 
\[
\phi:=\sigma_{\omega_{\mathrm{LS}}}(\psi*\psi)=(z_{0}w_{0},x_{1}y_{0},z_{1}w_{1},\dots,x_{t}y_{t-1},z_{t}w_{t}).
\]
Notice that each pair of coordinates of $\phi$ has disjoint variables,
and hence $\phi^{-1}(0)$ is isomorphic to the product $\prod\limits _{u=0}^{t}(z_{u}w_{u})^{-1}(0)\times(x_{u+1}y_{u})^{-1}(0)$.
We have thus reduced the problem to a single monomial $x_{u}y_{u}$.
Note that $x_{1}x_{2}+x_{3}x_{4}=0$ has rational singularities, so
$\phi^{*2}=\sigma_{\omega_{\mathrm{LS}}}(\psi^{*4})$ is (FRS), which
implies by Corollary \ref{cor:(cf.---Elimination)} that $\varphi^{*4}$
is (FRS). The last method we used is a special case of the coloring
method (as it involves an application of convolution followed by elimination)
which we call \textbf{\textit{level separation}}, as we use $\omega_{\mathrm{LS}}$
to separate a given map into 'levels' which have pairwise disjoint
variables.

\subsubsection{\label{subsec:Degenerations-of-jets of word maps}Degenerations of
jets of Lie algebra word maps}

We next describe a certain manifestation of the degeneration techniques
presented above, applied to the jets of Lie algebra word maps. Since
a formal derivation of a Lie algebra word map is a Lie algebra word
map (Lemma \ref{lem: jets of word maps are towers of word maps}),
and since the (FRS) property of a morphism is closely related to the
flatness of its $m$-th jets (Corollary \ref{cor: singularity properties through dim of jets}),
this allows us to reduce certain questions on the (FRS) property of
Lie algebra word maps to question about flatness of word maps or maps
built out of bounded ``layers'' of word maps. For example, in Section
\ref{sec:The-commutator-map-revisited} we use this method to reduce
the question about the (FRS) property of the commutator map to a question
about flatness of the commutator map. This is further used in Section
\ref{subsec:Matrix-word-maps}.

Let $w\in\mathcal{L}_{r}$ be a word of degree $d\in\nats$, and $\mathfrak{g}$
be a simple $K$-Lie algebra. Following the notation from Section
\ref{subsec:Basic-properties-of}, for any $m\in\nats$, we consider
the $m$-th jet maps $J_{m}(\varphi_{w}):J_{m}(\mathfrak{g})^{r}\rightarrow J_{m}(\mathfrak{g})$
where $J_{m}(\mathfrak{g})$ is identified with $\mathfrak{g}^{m+1}$
as vector spaces. We can then write 
\[
J_{m}(\varphi_{w})^{*t}(X_{1},\ldots,X_{r,t},\ldots,X_{1}^{(m)},\ldots,X_{r,t}^{(m)})=\left(\varphi_{w}^{*t},\left(\varphi_{w}^{*t}\right)^{(1)},\ldots,\left(\varphi_{w}^{*t}\right)^{(m)}\right),
\]
where $\left(\varphi_{w}^{*t}\right)^{(u)}$ is the $u$-th formal
derivative of $\varphi_{w}^{*t}$. 
\begin{lem}
\label{lem: jets of word maps are towers of word maps}For each $t,u\in\mathbb{N}$,
the map $\left(\varphi_{w}^{*t}\right)^{(u)}$ is a $d$-homogeneous
word map. 
\end{lem}

\begin{proof}
Embed $\mathfrak{g}\hookrightarrow\mathfrak{gl}_{n}$ and then by
functoriality of $J_{m}$, it is enough to prove the claim for $\mathfrak{g}=\mathfrak{gl}_{n}$.
Note that $\left(\varphi_{w}^{*t}\right)^{(u)}=\left(\varphi_{w}^{(u)}\right)^{*t}$
so we may assume $t=1$. By Fact \ref{fact:sum of left normed monomials}
we may assume that $w=X_{S}$ is a left-normed Lie monomial.

The formal derivative of any monomial of matrices $X_{s_{1}}\cdot X_{s_{2}}\cdot\ldots\cdot X_{s_{k}}$
is $\sum\limits _{j=1}^{k}X_{s_{1}}\cdot\ldots\cdot X_{s_{j}}^{(1)}\cdot\ldots\cdot X_{s_{k}}$.
In particular, by induction on $d$ it can be shown that 
\begin{equation}
X_{S}^{(1)}=\sum\limits _{k=1}^{d}[X_{s_{1}},\ldots,X_{s_{k-1}},X_{s_{k}}^{(1)},X_{s_{k+1}},X_{s_{d}}].\label{eq:(5.1)}
\end{equation}
By induction on $u$, we deduce that $\varphi_{w}^{(u)}$ is a $d$-homogeneous
word map. 
\end{proof}
\begin{defn}
\label{def:pure summands, equivalence class}Recall that a word is
of pure type $S_{0}\in[r]^{(d)}$, if it can be written as $\sum\limits _{S\in\mathrm{MS}^{-1}(S_{0})}c_{S}X_{S}$,
and note that any word $w$ can be written as a sum of words of pure
type 
\[
w=\sum\limits _{S_{0}\in[r]^{(d)}}w_{S_{0}}=\sum\limits _{S_{0}\in[r]^{(d)}}\left(\sum\limits _{S\in\mathrm{MS}^{-1}(S_{0})}c_{S}X_{S}\right).
\]
We call each $w_{S_{0}}$ a \textsl{pure summand} of $w$.

For $J\in\nats^{d}$ and $S\in\mathrm{MS}^{-1}(S_{0})$ define $X_{S}^{J}:=[X_{s_{1}}^{(j_{1})},\ldots,X_{s_{d}}^{(j_{d})}]$.
We say that $(S,J)\sim(S',J')$ are equivalent if there exists $\sigma\in S_{d}$
such that $(\sigma S,\sigma J)=(S',J')$. We denote equivalence classes
with respect to this action by $[(S,J)]$. Notice that $X_{S}^{J}$
and $X_{S'}^{J'}$ have the same multiset of variables if and only
if $(S,J)\sim(S',J')$. We therefore say that $X_{S}^{J}$ is of type
$[(S,J)]$. 
\end{defn}

Using Formula (\ref{eq:(5.1)}), by induction, we get the following
technical lemma. Note that $\binom{u}{j_{1},\ldots,j_{d}}=\frac{u!}{j_{1}!\ldots j_{d}!}$
are the multinomial coefficients counting the number of ways to put
$u$ distinct objects in $d$ distinct bins, with $j_{k}$ objects
in the $k$-th bin. 
\begin{lem}
\label{lem:formal derivatives of word maps}Let $S\in[r]^{d}$ and
$w=X_{S}$. Then $w^{(u)}=\sum\limits _{J\in\nats^{d}:\left|J\right|=u}\binom{u}{j_{1},\ldots,j_{d}}X_{S}^{J}$. 
\end{lem}

\begin{cor}
\label{cor:all parial derivatives survive in the formal derivative of a word}Let
$\mathfrak{g}$ be a simple $K$-Lie algebra, and let $w$ be a $d$-homogeneous
word of pure type $S_{0}$ such that $\varphi_{w}:\mathfrak{g}^{r}\to\mathfrak{g}$
is not trivial. Then for any $u\in\nats$, $\tilde{J}\in\nats^{d}$
with $|\tilde{J}|=u$, and any $\tilde{S}\in\mathrm{MS}^{-1}(S_{0})$,
the map $\varphi_{w}^{(u)}$ contains a non-trivial pure summand $w'$
of type $[(\tilde{S},\tilde{J})]$. 
\end{cor}

\begin{proof}
Write $w=\sum\limits _{S\in\mathrm{MS}^{-1}(S_{0})}c_{S}X_{S}$. Then
by Lemma \ref{lem:formal derivatives of word maps}, 
\[
\varphi_{w}^{(u)}=\sum\limits _{S\in\mathrm{MS}^{-1}(S_{0})}\sum\limits _{J:\left|J\right|=u}c_{S}\binom{u}{j_{1},\ldots,j_{d}}X_{S}^{J}.
\]
In particular, $\varphi_{[(\tilde{S},\tilde{J})]}:=\sum\limits _{S\in\mathrm{MS}^{-1}(S_{0})}\sum\limits _{J:(S,J)\in[(\tilde{S},\tilde{J})]}c_{S}\binom{u}{j_{1},\ldots,j_{d}}X_{S}^{J}$
is the pure summand of $\varphi_{w}^{(u)}$ of type $[(\tilde{S},\tilde{J})]$.
It is left to verify it is not trivial. Let $W$ be the subspace of
$J_{m}(\mathfrak{g}^{r})$ defined by $X_{s}^{(u)}=X_{s}$ for all
$s$ and $u$. Note that
\[
\left|\{J:(S,J)\in[(\tilde{S},\tilde{J})]\}\right|=\frac{\left|\mathrm{Stab}_{S_{d}}(\tilde{S})\right|}{\left|\mathrm{Stab}_{S_{d}}((\tilde{S},\tilde{J}))\right|}=:C([\tilde{S},\tilde{J}]),
\]
and hence there exists some $C'([\tilde{S},\tilde{J}])\in\ints_{>0}$
such that 
\[
\varphi_{[(\tilde{S},\tilde{J})]}|_{W}=\sum\limits _{S\in\mathrm{MS}^{-1}(S_{0})}C(\tilde{S},\tilde{J})c_{S}\binom{u}{j_{1},\ldots,j_{d}}X_{S}=C'([\tilde{S},\tilde{J}])\cdot\varphi_{w},
\]
which is non-trivial by our assumption.
\end{proof}
We next use the methods discussed in Subsection \ref{subsec:Discussion-on-the methods}
to reduce $J_{m}(\varphi_{w})$ to a simpler map.
\begin{prop}
\label{prop:reduction of the jet map to a simpler map}There exists
a degeneration of $J_{m}(\varphi_{w})$ into a generating map $\phi:\mathfrak{g}^{r(m+1)}\rightarrow\mathfrak{g}^{m+1}$
of the form $\phi=(\phi_{0},...,\phi_{m})$, where each $\phi_{u}:\mathfrak{g}^{r(m+1)}\rightarrow\mathfrak{g}$
is a Lie algebra word map. Furthermore there is a partition of the
interval $[0,m]$ into intervals $\{[i_{j},i_{j+1}]\}_{j}$ of length
at most $2d+1$, such that the collection $\{\psi_{j}\}_{j}$, where
$\psi_{j}:=(\phi_{i_{j}},\phi_{i_{j}+1},...,\phi_{i_{j+1}})$, has
mutually disjoint variables. In particular, $\phi^{-1}(0)=\prod_{j}\psi_{j}^{-1}(0)$.
\end{prop}

\begin{proof}
We choose an averaging weight 
\[
\omega_{\mathrm{av}}(X_{s}^{(u)}):=(d+1)^{u},
\]
and observe that 
\[
\sigma_{\omega_{\mathrm{av}}}(J_{m}(\varphi_{w}))=\left(\sigma_{\omega_{\mathrm{av}}}(\varphi_{w}),\sigma_{\omega_{\mathrm{av}}}(\varphi_{w}^{(1)}),\ldots,\sigma_{\omega_{\mathrm{av}}}(\varphi_{w}^{(m)})\right),
\]
where each $\sigma_{\omega_{\mathrm{av}}}(\varphi_{w}^{(u)})$ is
a sum of pure summands of $\varphi_{w}$, and each pure summand is
of type $[(S,J)]$, for some $J=(j_{1},\ldots,j_{d})$ such that $\left|j_{k}-\frac{u}{d}\right|\leq1$
for any $k\in[d]$. Note that each of these summands does not vanish
by Corollary \ref{cor:all parial derivatives survive in the formal derivative of a word},
and they cannot cancel each other as they consist of different monomials.
We now order the variables $X_{1},\ldots,X_{r},\ldots,X_{1}^{(m)},\ldots,X_{r}^{(m)}$
from $1$ to $r(m+1)$, by $\mathrm{ord}(X_{s}^{(u)})=ur+s$, and
set a monomialization weight 
\[
\omega_{\mathrm{mon}}(X_{s}^{(u)}):=(d+1)^{\mathrm{ord}(X_{s}^{(u)})}.
\]
Observe that for any $\sigma_{\omega_{\mathrm{av}}}(\varphi_{w}^{(u)})$,
there is a unique pure summand of type $[(S,J(u))]$, with $J(u)=(j_{u,1},\ldots,j_{u,d})$,
such that $\mathrm{deg}_{\omega_{\mathrm{mon}}}([X_{s_{1}}^{(j_{u,1})},\ldots,X_{s_{d}}^{(j_{u,d})}])$
is minimal. In particular, we have the following.
\begin{cor}
\label{Cor:after monomialization of jets of words}The map $\widetilde{w}_{u}:=\sigma_{\omega_{\mathrm{mon}}}\left(\sigma_{\omega_{\mathrm{av}}}(\varphi_{w}^{(u)})\right)$
is a non-trivial $d$-homogeneous word map of pure type $[(S,J(u))]$
with $\left|J(u)\right|=u$ and $\left|j_{u,k}-\frac{u}{d}\right|\leq1$
for any $k\in[d]$. 
\end{cor}

Write $X^{(u)}$ (resp.~$Y^{(u)}$) for the set of variables $X_{1}^{(u)},\ldots,X_{r}^{(u)}$
(resp.~$Y_{1}^{(u)},\ldots,Y_{r}^{(u)}$). We now use the level separation
method; we apply one convolution to $J_{m}(\varphi_{w})$ to create
a new set of $r(m+1)$ variables $Y,\ldots,Y^{(m)}$, and set 
\[
\omega_{\mathrm{LS}}(X^{(u)})=\begin{cases}
0 & \text{ if }u\mathrm{\,mod\,}4=0,1\\
d^{2}+1 & \text{ if }u\mathrm{\,mod\,}4=2,3,
\end{cases}
\]
\[
\omega_{\mathrm{LS}}(Y^{(u)})=\begin{cases}
d^{2} & \text{ if }u\mathrm{\,mod\,}4=0,1\\
0 & \text{ if }u\mathrm{\,mod\,}4=2,3.
\end{cases}
\]
By Corollary \ref{Cor:after monomialization of jets of words} and
by the definition of $\omega_{\mathrm{LS}}$, the map $\phi:=\sigma_{\omega_{\mathrm{LS}}}\left(\sigma_{\omega_{\mathrm{mon}}}\left(\sigma_{\omega_{\mathrm{av}}}(J_{m}(\varphi_{w})^{*2})\right)\right)$
is of the form $\phi=(\phi_{0},\ldots,\phi_{m})$ where the dependence
of $\phi_{j}$ on $\{X^{(u)},Y^{(u)}\}_{u=0}^{m}$ is as follows: 
\begin{enumerate}
\item The polynomials $\phi_{0},\ldots,\phi_{\left\lceil \frac{3d}{2}\right\rceil -1}$
depend only on the variables $X,X^{(1)},X^{(2)}$. 
\item For every $j\in\nats$, the polynomials $\phi_{\left\lceil \frac{(3+8j)d}{2}\right\rceil },\ldots,\phi_{\left\lfloor \frac{(7+8j)d}{2}\right\rfloor }$
depend only on the variables $Y^{(4j+1)},Y^{(4j+2)},Y^{(4j+3)},Y^{(4j+4)}$. 
\item For every $j\in\nats$, the polynomials $\phi_{\left\lfloor \frac{(7+8j)d}{2}\right\rfloor +1},\ldots,\phi_{\left\lceil \frac{(11+8j)d}{2}\right\rceil -1}$
depend only on the variables $X^{(4j+3)},X^{(4j+4)},X^{(4j+5)},X^{(4j+6)}$. 
\end{enumerate}
Using the level separation method, we separated $\sigma_{\omega_{\mathrm{mon}}}\left(\sigma_{\omega_{\mathrm{av}}}(J_{m}(\varphi_{w})^{*2}\right)$
into levels with pairwise disjoint variables, where each level is
of size at most $(2d+1)\mathrm{dim}\mathfrak{g}$. Notice that $\phi_{u}=\widetilde{w}_{u}$
for any $u$, in either the variables $X,\ldots,X^{(u)}$ or $Y,\ldots,Y^{(u)}$,
so every $\phi_{u}$ is non-trivial and thus generating. Since for
each $u$ the polynomial $\phi_{u}$ involves different monomials,
it follows that the map $(\phi_{s_{1}},\ldots,\phi_{s_{2}})$ where
$(s_{1},s_{2})\in\left\{ (0,\left\lceil \frac{3d}{2}\right\rceil -1),(\left\lceil \frac{(3+8j)d}{2}\right\rceil ,\left\lfloor \frac{(7+8j)d}{2}\right\rfloor ),(\left\lfloor \frac{(7+8j)d}{2}\right\rfloor +1,\left\lceil \frac{(11+8j)d}{2}\right\rceil -1)\right\} $
is generating as well. This finishes the proof of the proposition. 
\end{proof}

\subsection{\label{subsec:Proof-for low rank and few edges}Lie algebras of low
rank and small polyhypergraphs}

In order to treat word maps on low rank Lie algebras and polyhypergraphs
with few types of edges, we need the following Proposition \ref{prop:(FRS) for generating morphisms}.
For a polynomial $f(x_{1},\ldots,x_{n})=\sum\limits _{I\in\nats^{n}}a_{I}x^{I}$
set
\[
d_{\mathrm{mon}}(f):=\max\{i_{j}:I=(i_{1},\ldots,i_{n})\in\nats^{n}\text{ and }a_{I}\neq0\}.
\]
For a polynomial map $\varphi=(\varphi_{1},...,\varphi_{m}):\mathbb{A}_{K}^{n}\rightarrow\mathbb{A}_{K}^{m}$
set $d_{\mathrm{mon}}(\varphi):=\underset{i}{\max}\{d_{\mathrm{mon}}(\varphi_{i})\}$.
\begin{prop}
\label{prop:(FRS) for generating morphisms}Let $\varphi=(\varphi_{1},...,\varphi_{m}):\mathbb{A}_{K}^{n}\rightarrow\mathbb{A}_{K}^{m}$
be a generating morphism. Then $\varphi^{*(md_{\mathrm{mon}}(\varphi)+1)}$
is (FRS). 
\end{prop}

\begin{proof}
By \cite[Section 6]{GH19} we may assume that $K=\rats$. Write $\mu_{\ints_{p}}$
for the normalized Haar measure on $\ints_{p}$ and set $\tau_{p,\varphi}=\varphi_{*}(\mu_{\Zp}^{n})$,
for any large enough prime $p$. By \cite[Theorem 6.1]{Clu04} and
\cite{Chu76}, it follows that $\mathcal{F}(\tau_{p,f})\in L^{md_{\mathrm{mon}}(\varphi)+\epsilon}$
for every $\epsilon>0$. In particular, for any $t>md_{\mathrm{mon}}(\varphi)$
we have:
\[
\left\Vert \mathcal{F}(\tau_{p,\varphi^{*t}})\right\Vert _{1}=\left\Vert \mathcal{F}(\tau_{p,\varphi}^{*t})\right\Vert _{1}=\left\Vert \left(\mathcal{F}(\tau_{p,\varphi})\right)^{t}\right\Vert _{1}=\left\Vert \mathcal{F}(\tau_{p,\varphi})\right\Vert _{t}^{t}<C,
\]
so $\tau_{p,\varphi}^{*t}$ has continuous density. It now follows
from \cite[Proposition 3.16]{GH19} that $\varphi^{*t}$ is (FRS)
(this is a consequence of an analytic criterion for the (FRS) property,
see Theorem \ref{thm:analytic criterion of the (FRS) property}). 
\end{proof}
Note that $d_{\mathrm{mon}}(\varphi_{w})\leq d-1$ for any word map
$\varphi_{w}:\mathfrak{g}^{r}\rightarrow\mathfrak{g}$ of degree $d\geq2$.
Corollaries \ref{cor:low dimensional Lie algebras} and \ref{cor:small polyhypergraphs are (FRS)}
now follow directly from Proposition \ref{prop:(FRS) for generating morphisms}
and Lemma \ref{lem:generating becomes dominant}:
\begin{cor}
\label{cor:low dimensional Lie algebras}Let $w\in\mathcal{L}_{r}$
be a word of degree $d\in\nats$ and let $\mathfrak{g}$ be a simple
$K$-Lie algebra such that $\varphi_{w}$ is non-trivial on $\mathfrak{g}$.
Then, 
\begin{enumerate}
\item If $t\geq\left(\mathrm{dim}\mathfrak{g}\right)^{2}$ then $\varphi_{w}^{*t}$
is flat. 
\item If $t\geq d\left(\mathrm{dim}\mathfrak{g}\right)$ then $\varphi_{w}^{*t}$
is (FRS). 
\end{enumerate}
\end{cor}

Let $\Gamma:=(\mathcal{I},\mathcal{J},\mathcal{S})$ be a $d$-PH
and let $(\Gamma,\mathrm{V},\{\tau_{\gamma}\})$ be an assignment,
where $\mathrm{V}=(V(i))_{i\in\mathcal{I}}$ and $\{\tau_{\gamma}\}_{\gamma\in\mathcal{S}}$
is a collection of $d$-homogeneous multilinear forms $\tau_{\{i_{1},\dots,i_{d}\},l}:\bigoplus\limits _{i\in\{i_{1},\dots,i_{d}\}}V(i)\rightarrow K$.
Write $\Phi_{\Gamma,\mathrm{V},\{\tau_{\gamma}\}}:\bigoplus\limits _{i\in\mathcal{I}}V(i)\rightarrow K^{\mathcal{J}}$
for the corresponding map (as in (\ref{eq:(4.1)})). 
\begin{cor}
\label{cor:small polyhypergraphs are (FRS)}The $t$-th self-convolution
$(\Gamma,\mathrm{V}^{\oplus t},\{\tau_{\gamma}^{\oplus t}\})$ of
$(\Gamma,\mathrm{V},\{\tau_{\gamma}\})$ (Definition \ref{def:convolutions of a d-PH})
is
\begin{enumerate}
\item Flat if $t\geq\left|\mathcal{J}\right|^{2}$. 
\item (FRS) if $t\geq d\left|\mathcal{J}\right|+1$. 
\end{enumerate}
Moreover, if $\Gamma$ is a hyperedge consisting of a single hyperedge,
with an assignment $(\Gamma,V,\tau)$, then $(\Gamma,V,\tau)$ is
flat, and $(\Gamma,V^{\oplus t},\{\tau^{*t}\})$ is (FRS) if $t\geq d_{\mathrm{mon}}(\tau)+1$. 
\end{cor}

\begin{rem}
Note that the map $\Phi_{\Gamma,\mathrm{V}^{\oplus t},\{\tau_{\gamma}^{*t}\}}$
can be naturally identified with $\Phi_{\Gamma,\mathrm{V},\{\tau_{\gamma}\}}^{*t}$. 
\end{rem}

\subsection{\label{subsec:Proof-for-SLn}Proof for $\mathfrak{sl}_{n}$ }

We assume that $n>d$, as otherwise we may use the results of Subsection
\ref{subsec:Proof-for low rank and few edges}. Let $\mathfrak{t}$
be the Cartan subalgebra of $\mathfrak{sl}_{n}$ consisting of diagonal
matrices, write $\mathfrak{sl}_{n}=\mathfrak{t}\oplus\bigoplus\limits _{1\leq i\neq j\leq n}\mathfrak{g}_{i,j}$
and consider the basis $\mathcal{B}=\{E_{i,j}:1\leq i\neq j\leq n\}\cup\{H_{i}\}_{i=1}^{n-1}$
with $0\neq E_{i,j}\in\mathfrak{g}_{i,j}$ and $H_{i}=E_{i,i}-E_{i+1,i+1}\in\mathfrak{t}$.
Set $e_{i,j}=E_{i,j}$ if $i\neq j$ and $e_{i,i}=H_{i}$ if $i=j<n$.
We denote by $\Sigma(\mathfrak{sl}_{n},\mathfrak{t})$ the corresponding
root system, consisting of the roots $\{\epsilon_{i}-\epsilon_{j}\}_{i\neq j}$,
where $\epsilon_{i}\in\mathfrak{t}^{*}$ is the projection map to
the $i$-th coordinate.

Let $\mathcal{B}^{*}=\{\hat{e}_{i,j}\}$ be the dual basis of $\mathcal{B}$.
Let $w$ be a $d$-homogeneous word of pure type and consider the
$d$-PH $\Gamma_{\mathfrak{sl}_{n},w}$ and its assignment $(\Gamma_{\mathfrak{sl}_{n},w},V,\{\tau_{\gamma}\}_{\gamma\in\mathcal{S}})$
as defined in Subsection \ref{subsec:Attaching-a-polyhypergraph to a word map},
for $\mathcal{I}=\mathcal{J}=[n]^{2}\backslash\{(n,n)\}$. Recall
that $\Gamma_{\mathfrak{sl}_{n},w}$ has vertices $\{e_{i,j}\}_{(i,j)\in\mathcal{I}}$
and hyperedges $\mathcal{S}$ of the form $(\{e_{i_{1},j_{1}},\ldots,e_{i_{d},j_{d}}\},\hat{e}_{k,l})$.
Here, $(k,l)\in\mathcal{J}$ is the type of the hyperedge. 
\begin{rem}
\label{rem:edges with more than one type}Note that if $[e_{i_{1},j_{1}},\ldots,e_{i_{d},j_{d}}]\in\mathfrak{g}_{k,l}$
with $k\neq l$, then for each permutation $\sigma\in S_{d}$, we
have $[e_{i_{\sigma(1)},j_{\sigma(1)}},\ldots,e_{i_{\sigma(d)},j_{\sigma(d)}}]\in\mathfrak{g}_{k,l}$.
In particular, each edge $\{e_{i_{1},j_{1}},\ldots,e_{i_{d},j_{d}}\}$
is of unique type $\hat{e}_{k,l}$. On the other hand, if $[e_{i_{1},j_{1}},\ldots,e_{i_{d},j_{d}}]\in\mathfrak{t}$,
the hyperedge $\{e_{i_{1},j_{1}},\ldots,e_{i_{d},j_{d}}\}$ might
be of several types. For example, take $w=[X_{1},X_{2}]$, then 
\[
[e_{1,4},e_{4,1}]=E_{1,1}-E_{4,4}=e_{1,1}+e_{2,2}+e_{3,3},
\]
so the edge $\{e_{1,4},e_{4,1}\}$ is of types $\hat{e}_{11},\hat{e}_{22}$
and $\hat{e}_{33}$. 
\end{rem}

Our goal in this subsection is to prove the following theorem: 
\begin{thm}
\label{thm:main theorem on d-ph for SLn}Let $w\in\mathcal{L}_{r}$
be a $d$-homogeneous word of pure type and let $n>d$. 
\begin{enumerate}
\item The $d$-PH $(\Gamma_{\mathfrak{\mathfrak{sl}_{n}},w},V,\{\tau_{\gamma}\}_{\gamma\in\mathcal{S}})$
is flat if $\mathrm{dim}V\geq20d^{3}r$. 
\item The $d$-PH $(\Gamma_{\mathfrak{\mathfrak{sl}_{n}},w},V,\{\tau_{\gamma}\}_{\gamma\in\mathcal{S}})$
is (FRS) if $\mathrm{dim}V\geq80d^{3}r$. 
\end{enumerate}
\end{thm}

We would first like to reduce $\Gamma_{\mathfrak{\mathfrak{sl}_{n}},w}$
to three simpler polyhypergraphs. Let $M=[3]$ and set a coloring
function $\nu_{0}(e_{i,j})=(2d(i-j),1,2d(j-i))$. Notice that given
a hyperedge $(\{e_{i_{1},j_{1}},\ldots,e_{i_{d},j_{d}}\},\hat{e}_{k,l})\in\mathcal{S}$,
\[
\widetilde{\nu}_{0}(\{e_{i_{1},j_{1}},\ldots,e_{i_{d},j_{d}}\})=\sum\limits _{u=1}^{d}\nu_{0}(e_{i_{u},j_{u}})=(2\sum\limits _{u=1}^{d}d(i_{u}-j_{u}),d,2\sum\limits _{u=1}^{d}d(j_{u}-i_{u})).\tag{\ensuremath{\star}}
\]
Recall that in any simple Lie algebra $\mathfrak{g}$, for any $\alpha,\beta\in\Sigma(\mathfrak{g},\mathfrak{t})$
we have $[\mathfrak{g}_{\alpha},\mathfrak{g}_{\beta}]\subseteq\mathfrak{g}_{\alpha+\beta}$
if $\alpha+\beta\in\Sigma(\mathfrak{g},\mathfrak{t})$, and zero otherwise.
Since the hyperedge above is of type $(k,l)$, it follows that $\sum\limits _{u=1}^{d}\epsilon_{i_{u}}-\epsilon_{j_{u}}=\epsilon_{k}-\epsilon_{l}$.
In particular, we get 
\[
(\star)=(2d(k-l),d,2d(l-k)),
\]
so $\nu_{0}$ is admissible, and hence by Corollary \ref{cor:Coloring for for dPH},
we reduce to the following $d$-PHs: 
\begin{enumerate}
\item $\Gamma_{+}:=\mathrm{gr}_{\nu_{0},1}(\Gamma_{\mathfrak{\mathfrak{sl}_{n}},w})$
consisting of hyperedges of type $\hat{e}_{k,l}$ with $k<l$. 
\item $\Gamma_{0}:=\mathrm{gr}_{\nu_{0},2}(\Gamma_{\mathfrak{\mathfrak{sl}_{n}},w})$
consisting of hyperedges of type $\hat{e}_{k,k}$ with $k\in[n-1]$. 
\item $\Gamma_{-}:=\mathrm{gr}_{\nu_{0},3}(\Gamma_{\mathfrak{\mathfrak{sl}_{n}},w})$
consisting of hyperedges of type $\hat{e}_{k,l}$ with $k>l$. 
\end{enumerate}
As we will see soon, the $d$-PHs $\Gamma_{+}$ and $\Gamma_{-}$
can rather easily be reduced to hypergraphs. $\Gamma_{0}$ requires
a special treatment, since by Remark \ref{rem:edges with more than one type},
the hyperedges of $\Gamma_{0}$ may not have a unique type, so a non-careful
enough use of the elimination method might result in a non-generating
map (see Remark \ref{rem:warning-elimination}). We start by analyzing
$\Gamma_{+}$ and $\Gamma_{-}$.

\subsubsection{\label{subsec:Reduction-of-=00005Cgamma +-}Reduction of $\Gamma_{+}$
and $\Gamma_{-}$ to $d$-hypergraphs}

We consider $\Gamma_{+}$, the proof for $\Gamma_{-}$ is similar.
Recall $\Gamma_{+}:=(\mathcal{I},\mathcal{J}_{+},\mathcal{S}_{+})$
with $\mathcal{J}_{+}=\mathrm{gr}_{\nu_{0},1}(\mathcal{J})=\{(k,l)\in\mathcal{J}:k<l\}$
and $\mathcal{S}_{+}=\mathrm{gr}_{\nu_{0},1}(\mathcal{S})$. The reduction
to a hypergraph is done by applying the averaging method followed
by the monomialization method.

Set a weight function $\omega_{\mathrm{av}}(e_{i,j})=(d+1)^{\left|i-j\right|}$
for any $i,j$. Following the discussion in Subsection \ref{subsec:Discussion-on-the methods},
$\omega_{\mathrm{av}}$ leaves hyperedges $(\{e_{i_{1},j_{1}},\ldots,e_{i_{d},j_{d}}\},\hat{e}_{k,l})$
such that each of $e_{i_{1},j_{1}},\ldots,e_{i_{d},j_{d}}$ has length
$\left|i_{m}-j_{m}\right|$ which is as close as possible to the average,
i.e.~to $\frac{\left|k-l\right|}{d}$. Denote $\Gamma_{1}=\mathrm{gr}_{\omega_{\mathrm{av}}}(\Gamma_{+})=(\mathcal{I},\mathcal{J}_{+},\mathcal{S}_{1})$,
where 
\[
\mathcal{S}_{1}:=\{(\{e_{i_{1},j_{1}},\ldots,e_{i_{d},j_{d}}\},\hat{e}_{k,l})\in\mathcal{S}_{+}\text{ such that }\widetilde{\omega}_{\mathrm{av}}(\{e_{i_{1},j_{1}},\ldots,e_{i_{d},j_{d}}\})\text{ is minimal}\}.
\]
Order $\{e_{i,j}\}$ by $\mathrm{ord}(e_{i,j})=(i-1)n+j$ and set
a monomialization weight 
\[
\omega_{\mathrm{mon}}(e_{i,j}):=(d+1)^{\mathrm{ord}(e_{i,j})}.
\]
Note that for any $(k,l)\in\mathcal{J}_{+}$, there is a unique edge
$(\{e_{i_{1},j_{1}},\dots,e_{i_{d},j_{d}}\},\hat{e}_{k,l})\in\mathcal{S}_{1}$
with minimal weight $\widetilde{\omega}_{\mathrm{mon}}(\{e_{i_{1},j_{1}},\dots,e_{i_{d},j_{d}}\})$.
Indeed, 
\begin{equation}
\widetilde{\omega}_{\mathrm{mon}}(\{e_{i_{1},j_{1}},\ldots,e_{i_{d},j_{d}}\})=\sum\limits _{u=1}^{d}(d+1)^{\mathrm{ord}(e_{i_{u},j_{u}})}=\sum\limits _{p=1}^{n^{2}-1}\left(\#\{u:\mathrm{ord}(e_{i_{u},j_{u}})=p\}\right)\cdot(d+1)^{p},\label{eq:uniqueness of monomialization}
\end{equation}
so this decomposition is unique. Denote the resulting $d$-PH by $\Gamma_{2}=\mathrm{gr}_{\omega_{\mathrm{mon}}}(\Gamma_{1}):=(\mathcal{I},\mathcal{J}_{+},\mathcal{S}_{2})$. 
\begin{prop}
\label{prop:reduction to a hypergraph}For any $(k,l)\in\mathcal{J}_{+}$
we have 
\begin{enumerate}
\item $\Gamma_{2}$ contains exactly one hyperedge $\{e_{i_{1},j_{1}},\dots,e_{i_{d},j_{d}}\}$
of type $(k,l)$ and it is of unique type. 
\item Any hyperedge $(\{e_{i_{1},j_{1}},\ldots,e_{i_{d},j_{d}}\},\hat{e}_{k,l})\in\mathcal{S}_{2}$
with $\left|k-l\right|\leq d$ consists of vertices $e_{i_{u},j_{u}}$
which are either simple roots $e_{i,i+1}$ or elements $e_{i,i}$
of the Cartan $\mathfrak{t}$. 
\end{enumerate}
\end{prop}

\begin{proof}
Indeed, (\ref{eq:uniqueness of monomialization}) and Remark \ref{rem:edges with more than one type}
imply (1). It is left to prove (2). Assume $k<l\leq k+d$. Since $\mathcal{S}_{2}\subseteq\mathcal{S}_{1}$,
it is enough to show that any $(\{e_{i_{1},j_{1}},\ldots,e_{i_{d},j_{d}}\},\hat{e}_{k,l})\in\mathcal{S}_{1}$
satisfies (2), which is equivalent to 
\[
\widetilde{\omega}_{\mathrm{av}}(\{e_{i_{1},j_{1}},\ldots,e_{i_{d},j_{d}}\})\leq(d+1)d\tag{\ensuremath{\star}}.
\]
It is left to find a hyperedge of type $\hat{e}_{k,l}$ in $\Gamma_{+}$
that satisfies $(\star)$. Without loss of generality assume that
$k>\frac{d}{2}$ (a similar construction can be done for $k\leq\frac{d}{2}$).
Further recall that $n>d$. Let $I_{0}=((i_{0,1},j_{0,1}),\ldots,(i_{0,d},j_{0,d}))\in\mathcal{I}^{d}$
be the $d$-tuple defined by 
\[
\begin{cases}
(i_{0,u},j_{0,u}):=(k-u+1,k-u) & \text{if }u\leq\frac{d-(l-k)}{2}\\
(i_{0,u},j_{0,u}):=(l+u-d-1,l+u-d) & \text{if }u>\frac{d-(l-k)}{2}
\end{cases}
\]
if $d-(l-k)$ is even, and as follows if $d-(l-k)$ is odd: 
\[
\begin{cases}
(i_{0,u},j_{0,u}):=(k-u+1,k-u) & \text{if }u<\frac{d-(l-k)+1}{2}\\
(i_{0,u},j_{0,u}):=(k-u+1,k-u+1) & \text{if }u=\frac{d-(l-k)+1}{2}\\
(i_{0,u},j_{0,u}):=(l+u-d-1,l+u-d) & \text{if }u>\frac{d-(l-k)+1}{2}.
\end{cases}
\]
In view of Remark \ref{rem:different description for Lie algebra word maps},
$\varphi_{w}:\mathfrak{sl}_{n}^{r}\rightarrow\mathfrak{sl}_{n}$ can
be written as a restriction of a matrix word map $\varphi_{\widetilde{w}}:M_{n}^{r}\rightarrow M_{n}$,
with $\widetilde{w}=\sum\limits _{S\in\mathrm{MS}^{-1}(S_{0})}d_{S}\cdot X_{s_{1}}\cdot...\cdot X_{s_{d}}$
for some $S_{0}\in[r]^{(d)}$ and constants $d_{S}$. Write $\widetilde{I}_{0}:=\mathrm{MS}(I_{0})$.
By Lemma \ref{lem:combinatorial word map} we have 
\[
\tau_{\widetilde{I}_{0},(k,l)}(X_{1},\ldots,X_{r})=\sum\limits _{((i_{1},j_{1}),\ldots,(i_{d},j_{d}))\in\mathrm{MS}^{-1}(\widetilde{I}_{0})}\sum\limits _{S\in\mathrm{MS}^{-1}(S_{0})}d_{S}\prod\limits _{u=1}^{d}a_{i_{u},j_{u},s_{u},}\hat{e}_{k,l}\left(e_{i_{1},j_{1}}\cdot\ldots\cdot e_{i_{d},j_{d}}\right),
\]
where $X_{s}=\sum\limits _{(i,j)\in\mathcal{I}}a_{i,j,s}e_{i,j}$.
Now, note that $e_{(i_{0,1},j_{0,1})}\cdot\ldots\cdot e_{(i_{0,d},j_{0,d})}\neq0$,
but any reordering of these terms is the zero matrix, namely $e_{(i_{0,\sigma(1)},j_{0,\sigma(1)})}\cdot\ldots\cdot e_{(i_{0,\sigma(d)},j_{0,\sigma(d)})}=0$
for any $1\neq\sigma\in S_{d}$. We get, 
\begin{align*}
\tau_{\widetilde{I}_{0},(k,l)}(X_{1},\ldots,X_{r}) & =\sum\limits _{S\in\mathrm{MS}^{-1}(S_{0})}d_{S}\prod\limits _{u=1}^{d}a_{i_{0,u},j_{0,u},s_{u}}\hat{e}_{k,l}\left(e_{i_{0,1},j_{0,1}}\cdot\ldots\cdot e_{i_{0,d},j_{0,d}}\right)\\
 & =\sum\limits _{S\in\mathrm{MS}^{-1}(S_{0})}d_{S}\prod\limits _{u=1}^{d}a_{i_{0,u},j_{0,u},s_{u}.}
\end{align*}
Note that the monomials $\{\prod\limits _{u=1}^{d}a_{i_{0,u},j_{0,u},s_{u}}\}_{S\in\mathrm{MS}^{-1}(S_{0})}$
are mutually different (since the pairs $\{(i_{0,u},j_{0,u})\}_{u=1}^{d}$
are mutually different), and $d_{S}\neq0$ for some $S\in\mathrm{MS}^{-1}(S_{0})$,
so $\tau_{\widetilde{I}_{0},(k,l)}(X_{1},\ldots,X_{r})\neq0$, but
by our construction of $\Gamma_{\mathfrak{sl}_{n},w}$, we have $(\widetilde{I}_{0},\hat{e}_{k,l})\in\mathcal{S}$
and as it satisfies $(\star)$ we are done. 
\end{proof}
Proposition \ref{prop:reduction to a hypergraph} implies the following: 
\begin{cor}
The $d$-PH $\Gamma_{2}$ is induced from a $d$-hypergraph $\mathcal{G}=(\mathcal{V},\mathcal{E})$,
with $\mathcal{V}=\mathcal{I}$ and 
\[
\mathcal{E}=\{\text{the unique hyperedges \ensuremath{(\{e_{i_{1},j_{1}},\dots,e_{i_{d},j_{d}}\},\hat{e}_{k,l})\in\mathcal{S}_{2}} of type }\hat{e}_{k,l},\text{ for any }(k,l)\in\mathcal{J}_{+}\}.
\]
\end{cor}

Before we start to simplify $\mathcal{G}$ (or $\Gamma_{2}$) we would
like to have a better understanding of its structure. 
\begin{lem}
\label{lemma:not so far away from average}Let $\Gamma_{2}=(\mathcal{I},\mathcal{J}_{+},\mathcal{S}_{2})$
as above, and let $k,l$ with $d<l-k$. Write $l-k=pd+q$, where $p,q\in\nats$
and $0\leq q\leq d-1$. Then the hyperedge $(\{e_{i_{1},j_{1}},\ldots,e_{i_{d},j_{d}}\},\hat{e}_{k,l})\in\mathcal{S}_{2}$
satisfies $(i_{u},j_{u})=(k+(u-1)p,k+up)$ for any $1\leq u\leq d-q$,
and 
\[
(i_{d-q+g},j_{d-q+g})=(k+(d-q)p+(g-1)(p+1),k+(d-q)p+g(p+1)),
\]
for $g\in\{1,\ldots,q\}$. 
\end{lem}

\begin{proof}
We first show that for $(i_{1},j_{1}),\ldots,(i_{d},j_{d})$ as in
the lemma, we have $(\{e_{i_{1},j_{1}},\ldots,e_{i_{d},j_{d}}\},\hat{e}_{k,l})\in\mathcal{S}$.
Indeed, since $i_{1}=k,j_{d}=l$, $i_{u+1}=j_{u}$, and $\{i_{u}\}_{u=1}^{d}\cup\{j_{d}\}$
are mutually different, it follows that $e_{i_{\sigma(1)}j_{\sigma(1)}}\cdot\ldots\cdot e_{i_{\sigma(d)}j_{\sigma(d)}}\neq0$
if and only if $\sigma\in S_{d}$ is the trivial permutation. Using
an argument similar to that in Proposition \ref{prop:reduction to a hypergraph},
we have $\tau_{\{(i_{1},j_{1}),\ldots,(i_{d},j_{d})\},(k,l)}\neq0$.\textcolor{red}{{}
}It is clear from the definitions of $\omega_{\mathrm{av}}$ and $\omega_{\mathrm{mon}}$,
that $\widetilde{\omega}_{\mathrm{av}}(\{e_{i_{1},j_{1}},\dots,e_{i_{d},j_{d}}\})$
is minimal over all hyperedges of type $\hat{e}_{k,l}$, and that
$\widetilde{\omega}_{\mathrm{\mathrm{mon}}}(\{e_{i_{1},j_{1}},\dots,e_{i_{d},j_{d}}\})$
is minimal over all hyperedges in $\mathcal{S}_{1}$ of type $\hat{e}_{k,l}$.
This implies that $(\{e_{i_{1},j_{1}},\dots,e_{i_{d},j_{d}}\},\hat{e}_{k,l})\in\mathcal{S}_{2}$
as required. 
\end{proof}

\subsubsection{Proof for the $d$-hypergraph\textup{ $\mathcal{G}=(\mathcal{V},\mathcal{E})$}}

We now reduce our hypergraph $\mathcal{G}=(\mathcal{V},\mathcal{E})$
to two simpler hypergraphs. We set $M=\{1,2\}$ and 
\[
\omega_{2}(e_{i,j})=\begin{cases}
(d,0) & \text{if }i<j\\
(0,d^{2}) & \text{if }i\geq j.
\end{cases}
\]
This coloring results in two sub-hypergraphs: 
\begin{enumerate}
\item $\mathcal{G}_{+}=(\mathcal{V}_{+},\mathcal{E}_{+})$, consisting of
vertices $\mathcal{V}_{+}=\{e_{i,j}\}_{i<j}$ corresponding to positive
roots, and whose edges are 
\[
\mathcal{E}_{+}=\{(\{e_{i_{1},j_{1}},\dots,e_{i_{d},j_{d}}\},\hat{e}_{k,l})\in\mathcal{E}:j_{u}>i_{u}\forall u\}=\{\hat{e}_{k,l}:l-k\geq d\}.
\]
\item $\mathcal{G}_{\mathrm{small}}=(\mathcal{V}_{\mathrm{small}},\mathcal{E}_{\mathrm{small}})$,
consisting of $\mathcal{E}_{\mathrm{small}}=\mathcal{E}\backslash\mathcal{E}_{+}$.
By Proposition \ref{prop:reduction to a hypergraph} it follows that
$\mathcal{V}_{\mathrm{small}}$ consists of vertices $e_{i,j}$ with
$\left|i-j\right|\leq1$. 
\end{enumerate}
We treat each of these hypergraphs separately.

\textbf{Proof for }$\mathcal{G}_{+}$ \textbf{:}

The idea is to apply a series of level separations (see Subsection
\ref{subsec:Discussion-on-the methods}), to reduce $\mathcal{G}_{+}$
to the case of a hyperedge. We set the following coloring function
$\omega_{\mathrm{LS},1}:\mathcal{V}_{+}\rightarrow\mathbb{N}^{\{0,1\}}$,
\[
\omega_{\mathrm{LS},1}(e_{i,j})=\begin{cases}
(0,2d) & \text{if }(j-i)\mathrm{\,mod\,}4=0,1\\
(2d+1,0) & \text{if }(j-i)\mathrm{\,mod\,}4=2,3.
\end{cases}
\]
Notice that an edge $\hat{e}_{k,l}\in\mathcal{E}_{+}$ (i.e.~an edge
where $l-k\geq d$) 
\begin{enumerate}
\item Has color $0$ if and only if $|k-l|\mathrm{\,mod\,}4d\in\left[0,\left\lceil \frac{3d}{2}\right\rceil -1\right]\cup\left[\left\lfloor \frac{7d}{2}\right\rfloor +1,4d-1\right]$. 
\item Has color $1$ if and only if $|k-l|\mathrm{\,mod\,}4d\in\left[\left\lceil \frac{3d}{2}\right\rceil ,\left\lfloor \frac{7d}{2}\right\rfloor \right]$. 
\end{enumerate}
As a result, $\omega_{\mathrm{LS},1}$ splits the hypergraph $\mathcal{G}_{+}$
into levels (of size $\sim2dn$): 
\begin{enumerate}
\item $\mathrm{gr}_{\omega_{\mathrm{LS},1},0}\mathcal{G}_{+}$ consisting
of a disjoint union of hypergraphs $\mathcal{G}_{+,0,m}=(\mathcal{V}_{+,0,m},\mathcal{E}_{+,0,m})$,
for $m\in\nats$ such that 
\begin{enumerate}
\item $\mathcal{E}_{+,0,0}$ consists of edges $\hat{e}_{k,l}$ where $l-k\in\left[0,\left\lceil \frac{3d}{2}\right\rceil -1\right]$, 
\item $\mathcal{E}_{+,0,m}$ consists of edges $\hat{e}_{k,l}$ where $l-k\in\left[4d(m-1)+\left\lfloor \frac{7d}{2}\right\rfloor +1,4dm+\left\lceil \frac{3d}{2}\right\rceil -1\right]$. 
\end{enumerate}
\item $\mathrm{gr}_{\omega_{\mathrm{LS},1},1}\mathcal{G}_{+}$ consisting
of a disjoint union of hypergraphs $\mathcal{G}_{+,1,m}=(\mathcal{V}_{+,1,m},\mathcal{E}_{+,1,m})$,
for $m\in\nats$, where $\mathcal{E}_{+,1,m}$ consists of edges $\hat{e}_{k,l}$
where $l-k\in\left[4dm+\left\lceil \frac{3d}{2}\right\rceil ,4dm+\left\lfloor \frac{7d}{2}\right\rfloor \right]$. 
\end{enumerate}
It is left now to deal with each of the hypergraphs $\mathcal{G}_{+,1,m}$
and $\mathcal{G}_{+,0,m}$. We prove our statement for $\mathcal{G}_{+,1,m}$,
the proof for $\mathcal{G}_{+,0,m}$ is similar. Notice that any edge
$\hat{e}_{k,l}$ consists of vertices $e_{i_{1},j_{1}},\ldots,e_{i_{d},j_{d}}\in\mathcal{V}_{+,1,m}$
such that $j_{u}-i_{u}$ are as in Lemma \ref{lemma:not so far away from average}.
In particular, $j_{u}-i_{u}\in[4m+1,4m+4]$, implying the following: 
\begin{lem}
\label{lem:A-sufficient-condition for disjointness}A sufficient condition
for two edges $\hat{e}_{k,l},\hat{e}_{k',l'}\in\mathcal{E}_{+,1,m}$
not to share a common vertex is $\left|k-k'\right|>(4m+4)d$. 
\end{lem}

\begin{example}
\label{exa:close and far interactions}Take $n=1000$, $d=10$ and
consider $\mathcal{G}_{+,1,20}$. Then $\mathcal{E}_{+,1,20}$ consists
of the edges $\hat{e}_{k,l}$ such that $l-k\in\left[815,835\right]$.
Note that $\hat{e}_{1,821}$ shares common vertices with some of its
neighbors such as $\hat{e}_{1,822},\ldots,\hat{e}_{1,830}$ and $\hat{e}_{2,821},\hat{e}_{3,821},\dots$.
Apart from these close interactions, note that $\hat{e}_{1,821}$
and $\hat{e}_{83,903}$, which are far from each other, share the
vertex $e_{83,165}$. Both types of interactions between edges will
be taken into account. 
\end{example}

We would therefore like to use level separation to minimize such interactions
between edges. We start with taking care of far interactions. For
any $i$ let $\bar{i}=\left\lfloor \frac{i}{4m}\right\rfloor \mathrm{\,mod\,}2d$,
and set $\omega_{\mathrm{LS},2}:\mathcal{V}_{+,1,m}\rightarrow\mathbb{N}^{\{0,\ldots,2d-1\}}$
as follows: 
\[
\omega_{\mathrm{LS},2}(e_{i,j})(s):=n+(n-j)\cdot\delta_{s\bar{j}}-(n-i)\cdot\delta_{s\bar{i}}.
\]
Note that 
\[
\widetilde{\omega}_{\mathrm{LS},2}(\hat{e}_{k,l})(s)=dn+(n-l)\cdot\delta_{s\bar{l}}-(n-k)\cdot\delta_{s\bar{k}},
\]
and in particular, the color of the edge $\hat{e}_{k,l}$ is always
$\bar{k}$. As a consequence, for any $t\in\{0,\ldots,2d-1\}$ the
hypergraph $\mathrm{gr}_{\omega_{\mathrm{LS},2},t}(\mathcal{G}_{+,1,m})$
consists of edges $\hat{e}_{k,l}$ with $\bar{k}=t$. By Lemma \ref{lem:A-sufficient-condition for disjointness},
each $\mathrm{gr}_{\omega_{\mathrm{LS},2},t}(\mathcal{G}_{+,1,m})$
is a disjoint union of hypergraphs $\{\mathcal{G}_{+,1,m}^{t,e}\}_{e\in\nats}$
with edges $\hat{e}_{k,l}$ where $k\in\left[4m(t+2ed),4m(t+2ed)+4m-1\right]$.
Restricting to each of the smaller hypergraphs $\{G_{+,1,m}^{t,e}\}_{e\in\nats}$,
we obtain the following: 
\begin{lem}
\label{lem:Another sufficient condition for disjointness }A sufficient
condition for two edges $\hat{e}_{k,l},\hat{e}_{k',l'}\in\mathcal{E}_{+,1,m}^{t,e}$
not to share a common vertex is that either $\left|k-k'\right|>d$
or $\left|l-l'\right|>d$. 
\end{lem}

Following Lemma \ref{lem:Another sufficient condition for disjointness },
using two more level separations $\omega_{\mathrm{LS},3}$ and $\omega_{\mathrm{LS},4}$
(with $d$ colors each), which are defined in a way similar to $\omega_{\mathrm{LS},2}$,
only with $\bar{i}:=i\mathrm{\,mod\,}d$ instead of $\bar{i}:=\left\lfloor \frac{i}{4m}\right\rfloor \mathrm{\,mod\,}2d$,
we get that 
\begin{enumerate}
\item $\mathrm{gr}_{\omega_{\mathrm{LS},3},t_{1}}\mathcal{G}_{+,1,m}^{t,e}$
is a disjoint union of hypergraphs, each supported on the hyperedges
$\hat{e}_{k_{0},l_{0}},\ldots,\hat{e}_{k_{0},l_{0}+2d}$ with $l_{0}=k_{0}+4dm+\left\lceil \frac{3d}{2}\right\rceil $,
for a fixed $k_{0}$. 
\item $\mathrm{gr}_{\omega_{\mathrm{LS},4},t_{2}}\mathrm{gr}_{\omega_{\mathrm{LS},3},t_{1}}\mathcal{G}_{+,1,m}^{t,e}$
is a disjoint union of hypergraphs, each consisting of a single hyperedge
$\hat{e}_{k_{0},l_{0}}$. 
\end{enumerate}
Write $(\{e_{i_{1},j_{1}},\ldots,e_{i_{d},j_{d}}\},\hat{e}_{k_{0},l_{0}})$
for any hyperedge of $\mathrm{gr}_{\omega_{\mathrm{LS},4},t_{2}}\mathrm{gr}_{\omega_{\mathrm{LS},3},t_{1}}\mathcal{G}_{+,1,u}^{t,e}$.
Note that by Lemma \ref{lem:combinatorial word map}, since each hyperedge
of $\mathrm{gr}_{\omega_{\mathrm{LS},4},t_{2}}\mathrm{gr}_{\omega_{\mathrm{LS},3},t_{1}}\mathcal{G}_{+,1,u}^{t,e}$
is supported on distinct vertices, it follows that $\tau_{\{(i_{1},j_{1}),\ldots,(i_{d},j_{d})\},(k_{0},l_{0})}$
is a sum of monomials of the form $x_{1}x_{2}\dots x_{d}$. By Corollary
\ref{cor:small polyhypergraphs are (FRS)}, since $d_{\mathrm{mon}}(x_{1}\cdot...\cdot x_{d})=1$,
it follows that $\mathrm{gr}_{\omega_{\mathrm{LS},4},t_{2}}\mathrm{gr}_{\omega_{\mathrm{LS},3},t_{1}}\mathcal{G}_{+,1,u}^{t,e}$
is flat if $\mathrm{dim}V\geq r$ and is (FRS) if $\mathrm{dim}V\geq2r$.

\textbf{Summary for }$\mathcal{G}_{+}$ and $\mathcal{G}_{-}$:

We have used the level separation method four times, with coloring
functions $\omega_{\mathrm{LS},1},\omega_{\mathrm{LS},2},\omega_{\mathrm{LS},3}$
and $\omega_{\mathrm{LS},4}$, with $2,2d,d$ and $d$ colors respectively,
to reduce to a case of a hyperedge. For a hyperedge we need $\mathrm{dim}V\geq r$
to guarantee flatness, and $\mathrm{dim}V\geq2r$ for the (FRS) property.
Thus the hypergraph $\mathcal{G}_{+}$ (and similarly $\mathcal{G}_{-}$)
is flat after $4d^{3}=2\cdot2d\cdot d\cdot d$ convolutions, and is
(FRS) after $8d^{3}$ convolutions.

\textbf{Proof for }$\mathcal{G}_{\mathrm{small}}$:

Set $\bar{i}=i\mathrm{\,mod\,}2d$ and $\omega_{\mathrm{LS},5}:\mathcal{V}_{\mathrm{small}}\rightarrow\mathbb{N}^{\{0,\ldots,2d-1\}}$
by
\[
\omega_{\mathrm{LS},5}(e_{i,j})(s):=n+(n-j)\cdot\delta_{s\bar{j}}-(n-i)\cdot\delta_{s\bar{i}},
\]
and note that $\hat{e}_{k.l}$ is of color $\bar{k}$ (recall we still
have $l>k$).

By Proposition \ref{prop:reduction to a hypergraph}, since each of
the vertices $e_{i_{u}.j_{u}}$ on which the hyperedge $\hat{e}_{k.l}$
is supported is either a simple root or an element of the Cartan,
it follows that for each $s\in\{0,\ldots,2d-1\}$, $\mathrm{gr}_{\omega_{\mathrm{LS},5},s}(\mathcal{G}_{\mathrm{small}})$
is a disjoint union of hypergraphs, each with at most $d$ edges.
Using another level separation $\omega_{\mathrm{LS},6}(e_{i,j})(s):=n-j\cdot\delta_{s\bar{j}}+i\cdot\delta_{s\bar{i}}$
of $d$ colors (with $\bar{i}=i\mathrm{\,mod\,}d$), we reduce to
the case of a single hyperedge.

\textbf{Summary for }$\mathcal{G}_{\mathrm{small}}$: Since we used
the coloring method twice, with $2d$ and $d$ colors, by Corollary
\ref{cor:small polyhypergraphs are (FRS)}, we need $\mathrm{dim}V\geq2d^{2}r$
for $\mathcal{G}_{\mathrm{small}}$ to be flat, and $\mathrm{dim}V\geq2d^{3}r$
for $\mathcal{G}_{\mathrm{small}}$ to be (FRS).

We have therefore finished the cases of $\Gamma_{+}$ and $\Gamma_{-}$,
so it is left to deal with $\Gamma_{0}$.

\subsubsection{\label{subsec:The-case-of gamma zero}The case of $\Gamma_{0}$}

Recall that $\Gamma_{0}:=\mathrm{gr}_{\nu_{0},2}(\Gamma_{\mathfrak{\mathfrak{sl}_{n}},w})=(\mathcal{I},\mathcal{J}_{0},\mathcal{S}_{0})$.
In order to show that $(\Gamma_{0},V,\{\tau_{\gamma}\}_{\gamma\in\mathcal{S}_{0}})$
is (FRS) for $V$ sufficiently large, we can use the coloring and
elimination methods, to reduce to a $d$-PH with few types of hyperedges,
and then use Corollary \ref{cor:small polyhypergraphs are (FRS)}.
Instead, we prove that $\Phi_{\Gamma_{0},V,\{\tau_{\gamma}\}}$ is
jet-flat when $V$ is large enough, using similar coloring and elimination
techniques, only on the $m$-th jet map $J_{m}(\Phi_{\Gamma_{0},V,\{\tau_{\gamma}\}})$.
This approach will put us in a better position to treat the case of
matrix word maps in Section \ref{subsec:Matrix-word-maps}.

Similarly to as done in Subsection \ref{subsec:Degenerations-of-jets of word maps},
using an averaging weight, a monomialization weight, and the level
separation method, we reduce to proving the flatness property of a
map $\varphi_{m}=(\varphi_{m}^{(k_{0})},\ldots,\varphi_{m}^{(k_{0}+2d)})$
whose target is a vector space of dimension $\mathrm{dim}\mathfrak{t}\cdot(2d+1)=(2d+1)(n-1)$,
and each $\varphi_{m}^{(u)}$ is a certain pure summand of $\varphi_{w}^{(u)}$
(composed with $\{\hat{e}_{l,l}^{(u)}\}_{l\in[n-1]}$). The variables
of this map are $\{e_{i,j}^{(u')}\}$ for $0\leq u'\leq k_{0}+2d$
and $(i,j)\in[n]^{2}\backslash\{(n,n)\}$. We now set the following
weight 
\[
\nu_{1}(e_{i,j}^{(u')})=\begin{cases}
d & \text{if }\left|i-j\right|>1\text{ and }(i>d\text{ or }j>d)\\
0 & \text{if }\left|i-j\right|\leq1\text{ or }i,j\leq d.
\end{cases}
\]

\begin{lem}
\label{lem:still generating}The map $\mathrm{gr}_{\nu_{1}}(\varphi_{m})$
is a generating polynomial map on variables $\{e_{i,j}^{(u')}\}$
with either $i,j\leq d$ or $\left|i-j\right|\leq1$. 
\end{lem}

\begin{proof}
Note that $\mathrm{gr}_{\nu_{1}}(\varphi_{m})=\left(\mathrm{gr}_{\nu_{1}}\varphi_{m}^{(k_{0})},\ldots,\mathrm{gr}_{\nu_{1}}\varphi_{m}^{(k_{0}+2d)}\right)$.
Since $\varphi_{m}^{(u_{1})}$ and $\varphi_{m}^{(u_{2})}$ involve
distinct monomials if $u_{1}\neq u_{2}$, the maps $\mathrm{gr}_{\nu_{1}}\varphi_{m}^{(u_{1})}$
and $\mathrm{gr}_{\nu_{1}}\varphi_{m}^{(u_{2})}$ involve distinct
monomials as well. Therefore, $\mathrm{gr}_{\nu_{1}}(\varphi_{m})$
is generating if and only if $\mathrm{gr}_{\nu_{1}}\varphi_{m}^{(u)}$
is generating for every $k_{0}\leq u\leq k_{0}+2d$. Let $W$ be the
subspace of $J_{m}(\mathfrak{sl}_{n}^{r})$ defined by the relations
$X_{s}^{(u')}=X_{s}$ for every $u'$ and $s\in[r]$. By Corollary
\ref{cor:all parial derivatives survive in the formal derivative of a word},
$\varphi_{m}^{(u)}|_{W}=\lambda\cdot\varphi_{w}^{(0)}$ for some $\lambda>0$.
Thus we may assume that $u=0$ and it is enough to show that $\Phi_{(\mathrm{gr}_{\nu_{1}}\Gamma_{0},V,\{\tau_{\gamma}\}_{\gamma\in\mathcal{S}_{0}})}$
is generating. 

Consider the Lie subalgebra $\mathfrak{h}\subseteq\mathfrak{sl}_{n}$
generated by the simple roots $e_{1,2},\dots,e_{d-1,d}$. Since $\mathfrak{h}\simeq\mathfrak{sl}_{d}$,
Corollary \ref{cor: commutator relations on simple Lie algebras}
implies that $\mathrm{gr}_{\nu_{1}}\varphi_{w}|_{\mathfrak{h}}=\varphi_{w}|_{\mathfrak{h}}\neq0$,
and Proposition \ref{prop:Lie algebra word maps are generating} implies
that $\Phi_{(\mathrm{gr}_{\nu_{1}}\Gamma_{0},V,\{\tau_{\gamma}\}_{\gamma\in\mathcal{S}_{0}})}$
is generating on the types $\{\hat{e}_{l,l}\}_{l\in[d-1]}$ using
only the variables of $\mathfrak{h}$. 

We use a similar construction as in the proof of Proposition \ref{prop:reduction to a hypergraph}(2),
only now for hyperedges of type $\hat{e}_{k,k}$. For simplicity assume
that $d$ is even (the case of an odd $d$ is similar). For any $d+1\leq k\leq n$,
consider the $d$-tuple $I_{0}(k)=((i_{0,1},j_{0,1}),\ldots,(i_{0,d},j_{0,d}))\in\mathcal{I}^{d}$
defined by 
\[
\begin{cases}
(i_{0,c},j_{0,c}):=(k-c+1,k-c) & \text{if }1\leq c\leq\frac{d}{2}\\
(i_{0,c},j_{0,c}):=(k+c-d-1,k+c-d) & \text{if }\frac{d}{2}<c\leq d,
\end{cases}
\]
and let $\widetilde{I_{0}}(k)=\mathrm{MS}(I_{0}(k))$. Note that here,
unlike in the case of $\Gamma_{+}$ and $\Gamma_{-}$, the hyperedge
$\widetilde{I_{0}}(k)$ is supported on several types, i.e.~on the
types $\{\hat{e}_{l,l}\}_{l\in[k-\frac{d}{2},k-1]}\in\mathcal{J}_{0}$.
By a similar computation as in Proposition \ref{prop:reduction to a hypergraph},
and by induction on $k$, it can be shown that the forms $\{\tau_{\gamma}\}$
supported on $\mathfrak{h}$ and on the hyperedges $\{\widetilde{I_{0}}(l)\}_{l\leq k}$
are generating on the types $\{\hat{e}_{l,l}\}_{l\in[k-1]}$, so we
are done.
\end{proof}
Back to our morphism $\mathrm{gr}_{\nu_{1}}(\varphi_{m})$. Our choice
of $\nu_{1}$ guarantees that $\hat{e}_{k,k}^{(u)}\circ\mathrm{gr}_{\nu_{1}}(\varphi_{m})$
is supported on variables $e_{i,j}^{(u')}$ with $\left|i-k\right|\leq2d$.
We set the following (monomialization) weight 
\[
\nu_{2}(e_{i,j}^{(u')})=\begin{cases}
(d+1)^{i} & \text{if }i>2d\\
0 & \text{if }i\leq2d.
\end{cases}
\]
Similarly to the proof of Lemma \ref{lem:still generating}, it can
be shown that $\mathrm{gr}_{\nu_{2}}(\mathrm{gr}_{\nu_{1}}(\varphi_{m}))$
is generating. We now use the level separation method. Set $\bar{i}=\left\lfloor \frac{i}{4d}\right\rfloor \mathrm{\,mod\,}2$,
and define a coloring function $\nu_{\mathrm{LS}}$ with two colors
$s\in\{0,1\}$ by 
\[
\nu_{\mathrm{LS}}(e_{i,j}^{(u)})(s)=(d^{2}+1)\cdot\delta_{s\bar{i}}+d^{2}\cdot\delta_{s(1-\bar{i})}.
\]
In view of Remark \ref{rem:The-coloring-method is convolution+elimination},
our choice of $\nu_{2}$ makes $\nu_{\mathrm{LS}}$ an admissible
coloring for the map $\mathrm{gr}_{\nu_{2}}(\mathrm{gr}_{\nu_{1}}(\varphi_{m}))$.
Then each of $\{\mathrm{gr}_{\nu_{\mathrm{LS}},s}(\mathrm{gr}_{\nu_{2}}(\mathrm{gr}_{\nu_{1}}(\varphi_{m})))\}_{s=1}^{2}$
is a collection of morphisms in distinct variables into affine spaces
of dimension at most $6d\cdot(2d+1)\leq15d^{2}$. By Corollary \ref{prop:(FRS) for generating morphisms},
the map $\mathrm{gr}_{\nu_{\mathrm{LS}},s}(\mathrm{gr}_{\nu_{2}}(\mathrm{gr}_{\nu_{1}}(\varphi_{m})))^{*15d^{3}}$
is normal (as $d_{\mathrm{mon}}\leq d-1$). By Corollaries \ref{cor:(cf.---Elimination)}
and \ref{cor: singularity properties through dim of jets}, and since
we have used four colors, the assignment $(\Gamma_{0},V,\{\tau_{\gamma}\}_{\gamma\in\mathcal{S}_{0}})$
is (FRS) if $\mathrm{dim}V\geq60d^{3}r=4\cdot15d^{3}r$. In particular,
following the same proof for $m=0$, we reduce to a disjoint set of
morphisms into a $6d$-dimensional vector space, so $6d^{2}$ convolutions
are enough to ensure each map is flat, and $(\Gamma_{0},V,\{\tau_{\gamma}\}_{\gamma\in\mathcal{S}_{0}})$
is flat if $\mathrm{dim}V\geq12d^{2}r$.

\subsubsection{Proof of Theorem \ref{thm:main theorem on d-ph for SLn}}

Assume $n>d\geq2$. We have seen that the hypergraphs $\mathcal{G}_{+}$
and $\mathcal{G}_{-}$ are flat after $4d^{3}$ convolutions, and
(FRS) after $8d^{3}$ convolutions. $\mathcal{G}_{\mathrm{small}}$
is flat after $2d^{2}$ convolutions and (FRS) after $2d^{3}$ convolutions.
Therefore, each of the $d$-PHs $\Gamma_{+}$ and $\Gamma_{-}$ are
flat after $5d^{3}$ convolution and (FRS) after $10d^{3}$ convolutions.
We have seen that $\Gamma_{0}$ is flat after $12d^{2}$ and (FRS)
after $60d^{3}$ convolutions. Theorem \ref{thm:main theorem on d-ph for SLn}
now follows by combining the results for $\Gamma_{+}$, $\Gamma_{-}$
and $\Gamma_{0}$ and using Corollary \ref{cor:Coloring for for dPH}
repeatedly.

\subsection{Matrix word maps\label{subsec:Matrix-word-maps} }

The proof for matrix word maps (Theorem \ref{thm: main theorem for matrix word map})
is similar to the $\mathfrak{sl}_{n}$ case. With the notations of
Section \ref{subsec:Proof-for-SLn}, let $e_{\mathrm{cent}}=H_{\mathrm{cent}}=\stackrel[i=1]{n}{\sum}\frac{E_{i,i}}{n}$
and $\hat{e}_{\mathrm{cent}}=\mathrm{Tr}(\cdot)$ be such that $\{\hat{e}_{i,j}\}\cup\{\hat{e}_{\mathrm{cent}}\}$
is the basis dual to $\{e_{i,j}\}\cup\{e_{\mathrm{cent}}\}$. Let
$w\in\mathcal{A}_{r}$ be a matrix word of degree $d$. We identify
$\varphi_{w}$ with a map $(\widetilde{\varphi}_{w},\mathrm{Tr}(\varphi_{w})):M_{n}^{r}\rightarrow\mathbb{A}^{n^{2}}$,
where $(\widetilde{\varphi}_{w})_{i,j}=\hat{e}_{i,j}\ensuremath{\circ\varphi_{w}}$
and $\mathrm{Tr}(\varphi_{w})=\hat{e}_{\mathrm{cent}}\circ\varphi_{w}$.
Let $\varphi_{w_{0}}$ be the highest degree homogeneous part of $\varphi_{w}$.
Since the image of $\varphi_{w_{0}}$ is conjugate invariant, it must
generate either $M_{n}$ or $\mathfrak{sl}_{n}$. In the first case,
using Lemma \ref{lem: reduction to largest degree} we reduce to a
generating $d$-homogeneous matrix word map $\varphi_{w_{0}}$. In
the second case, we reduce $(\widetilde{\varphi}_{w},\mathrm{Tr}(\varphi_{w}))$
to a map $(\widetilde{\varphi}_{w_{0}},\mathrm{Tr}(\varphi_{w_{1}}))$,
where $\varphi_{w_{1}}$ is the highest degree homogeneous part of
$\varphi_{w}$ whose trace is non-zero. Notice that the pair $(\widetilde{\varphi}_{w_{0}},\mathrm{Tr}(\varphi_{w_{1}}))$
is not necessarily a matrix word map, but it is generating. Write
the resulting morphism by $\psi_{w_{0},w_{1}}$. 
\begin{example}
Let $\varphi_{w}(X,Y)=X^{3}+[[[X,Y],Y],Y]$. Then $\varphi_{w_{0}}=[[[X,Y],Y],Y]$
and $\varphi_{w_{1}}=X^{3}$. 
\end{example}

\begin{defn}
\label{def:pure type for Matrix word}A $d$-homogeneous word $w\in\mathcal{A}_{r}$
is said to be of \textsl{pure type} $S_{0}\in[r]^{(d)}$ if it is
of the form 
\[
w=\sum\limits _{S\in\mathrm{MS}^{-1}(S_{0})}c_{S}X_{s_{1}}\cdot X_{s_{2}}\cdot\ldots\cdot X_{s_{d}}.
\]
\end{defn}

\begin{lem}
\label{lem:reduction to pure Matrix words}We may assume that the
words $w_{0}$ and $w_{1}$ appearing in $\psi_{w_{0},w_{1}}=(\widetilde{\varphi}_{w_{0}},\mathrm{Tr}(\varphi_{w_{1}}))$
are homogeneous matrix words of pure type. 
\end{lem}

\begin{proof}
We use the elimination method similarly as in Proposition \ref{prop:reduction to balanced homogeneous maps}.
We set a monomialization weight $\omega_{\mathrm{mon}}(X_{i})=(d+1)^{i}$
and then $\mathrm{gr}_{\omega_{\mathrm{mon}}}(\psi_{w_{0},w_{1}})$
is of the required form. 
\end{proof}
\begin{proof}[Proof of Theorem \ref{thm: main theorem for matrix word map}]
If $n\leq d$, then Proposition \ref{prop:(FRS) for generating morphisms}
implies that for any $w\in\mathcal{A}_{r}$ with $\varphi_{w}:M_{n}^{r}\rightarrow M_{n}$
generating, the map $\varphi_{w}^{*d^{3}+1}$ is (FRS). 

Now assume that $n>d$. We use an idea similar to the one used in
the reduction to the $d$-PHs $\Gamma_{+},\Gamma_{-}$ and $\Gamma_{0}$
in the case of $\mathrm{\mathfrak{sl}_{n}}$. We use the coloring
function $\nu_{0}(e_{i,j})=(2d(i-j),1,2d(j-i))$ for the polynomial
map $\psi_{w_{0},w_{1}}$ (see discussion after Theorem \ref{thm:main theorem on d-ph for SLn}).
We therefore reduce to a disjoint union of three morphisms 
\begin{gather*}
\Psi_{+}:=(\hat{e}_{k,l})_{k<l}\circ(\widetilde{\varphi}_{w_{0}}):M_{n}^{r}\rightarrow\mathbb{A}^{\frac{n(n-1)}{2}},\\
\Psi_{-}:=(\hat{e}_{k,l})_{k>l}\circ(\widetilde{\varphi}_{w_{0}}):M_{n}^{r}\rightarrow\mathbb{A}^{\frac{n(n-1)}{2}},\\
\Psi_{0}:=((\hat{e}_{1,1},\ldots,\hat{e}_{n-1,n-1})\circ\widetilde{\varphi}_{w_{0}},\mathrm{Tr}(\varphi_{w_{1}})):M_{n}^{r}\rightarrow\mathbb{A}^{n}.
\end{gather*}
$\Psi_{+}$ and $\Psi_{-}$ can be encoded in terms of assignments
of $d$-PHs named $\widetilde{\Gamma}_{+}$ and $\widetilde{\Gamma}_{-}$
as in the $\mathfrak{sl}_{n}$ case, only now the forms in play will
encode matrix word maps instead of Lie algebra word maps. The proof
for $\widetilde{\Gamma}_{+}$ and $\widetilde{\Gamma}_{-}$ is identical
to the proof for $\Gamma_{+}$ and $\Gamma_{-}$ in the $\mathfrak{sl}_{n}$
case. In particular, the maps $\Psi_{+}^{*t}$ and $\Psi_{-}^{*t}$
are flat if $t\geq5d^{3}$ and (FRS) if $t\geq10d^{3}$. 

The proof for $\Psi_{0}$ is done by first dealing with $\psi_{w_{0}}:=(\hat{e}_{1,1},\ldots,\hat{e}_{n-1,n-1})\circ\widetilde{\varphi}_{w_{0}}$,
similarly as in $\Gamma_{0}$ for $\mathfrak{sl}_{n}$, so that $\psi_{w_{0}}^{*12d^{2}}$
is flat and $\psi_{w_{0}}^{*60d^{3}}$ is (FRS). In particular, we
get that the dimensions of the fibers of $\Psi_{0}^{*12d^{2}}$ are
at most $12d^{2}rn^{2}-n+1$, so by Lemma \ref{lem:Auxilery lemma for matrix word maps},
$\Psi_{0}^{*24d^{2}}$ is flat. Using a similar argument (Lemma \ref{lem:Auxilery lemma for matrix word maps})
on a certain degeneration of the $m$-th jets of $\Psi_{0}$ (as in
the case of $\Gamma_{0}$ for $\mathfrak{sl}_{n}$), we deduce that
$\Psi_{0}^{*200d^{4}}$ is (FRS). 

We therefore get that $\varphi_{w}^{*t}$ is flat if $t\geq30d^{3}>10d^{3}+24d^{2}$
and is (FRS) if $t\geq300d^{4}$. Proposition \ref{prop:reduction to self convolutions}
allows us to reduce to the case of self-convolutions of matrix word
maps (recall that when doing this the required number of convolutions
is multiplied by $2$). Combining this with the above arguments concludes
the proof. 
\end{proof}

\subsection{Proof of the general case\label{subsec:Proof-for-the general case}}

Let $\mathfrak{g}$ be a simple $K$-Lie algebra. We may assume that
$K$ is algebraically closed. Since we already dealt with Lie algebras
of small rank we may assume that $\mathfrak{g}$ is either $\mathfrak{so}_{2n+1},\mathfrak{sp}_{2n}$
or $\mathfrak{so}_{2n}$, and $n>4d$. Write $\mathfrak{g}=\mathfrak{t}\oplus\bigoplus\limits _{\alpha\in\Sigma(\mathfrak{g},\mathfrak{t})}\mathfrak{g}_{\alpha}$
with $\mathfrak{t}$ a Cartan subalgebra, $\Sigma(\mathfrak{g},\mathfrak{t})$
the corresponding root system and $\triangle$ a basis of $\Sigma(\mathfrak{g},\mathfrak{t})$.
Choose a Chevalley basis $\{h_{\alpha}\}_{\alpha\in\triangle}\cup\{e_{\alpha}\}_{\alpha\in\Sigma(\mathfrak{g},\mathfrak{t})}$.
We define the \textsl{length }$\ell(\alpha)$ of a root $\alpha\in\Sigma^{+}(\mathfrak{g},\mathfrak{t})$
to be the number of simple roots $\alpha_{i_{1}},\dots,\alpha_{i_{\ell(\alpha)}}$
such that $\alpha=\sum\limits _{j=1}^{\ell(\alpha)}\alpha_{i_{j}}$.

We index this basis and its dual as $\{e_{i}\}_{i\in\mathcal{I}}$
and $\{\widehat{e}_{j}\}_{j\in\mathcal{J}}$ respectively. Let $w\in\mathcal{L}_{r}$
be a $d$-homogeneous word of pure type. We consider the assignment
$(\Gamma_{\mathfrak{g},w},V,\{\tau_{\gamma}\}_{\gamma\in\mathcal{S}})$
of the $d$-PH $\Gamma_{\mathfrak{g},w}=(\mathcal{I},\mathcal{J},\mathcal{S})$.
Recall that $\Gamma_{\mathfrak{g},w}$ has vertices $\{e_{i}\}_{i\in\mathcal{I}}$
and hyperedges of the form $(\{e_{i_{1}},\dots,e_{i_{d}}\},\widehat{e}_{j})$.
Such an edge has type $j$. We would like to prove the following: 
\begin{thm}
\label{thm:main theorem for general Lie algebra}Let $w\in\mathcal{L}_{r}$
be a $d$-homogeneous word of pure type and let $n>4d$. Then the
$d$-PH $(\Gamma_{\mathfrak{g},w},V,\{\tau_{\gamma}\}_{\gamma\in\mathcal{S}})$
is flat if $\mathrm{dim}V\geq4\cdot10^{4}d^{3}r$, and (FRS) if $\mathrm{dim}V\geq8\cdot10^{4}d^{4}r$. 
\end{thm}

As in the case of $\mathfrak{sl}_{n}$, in order to split $\Gamma_{\mathfrak{g},w}$
into simpler $d$-PHs we use the coloring function 
\[
\nu_{0}(\alpha)=\begin{cases}
(-2d\cdot\ell(\alpha),1,2d\cdot\ell(\alpha)) & \text{if }\alpha\in\Sigma^{+}(\mathfrak{g},\mathfrak{t})\\
(2d\cdot\ell(\alpha),1,-2d\cdot\ell(\alpha)) & \text{if }\alpha\in\Sigma^{-}(\mathfrak{g},\mathfrak{t})\\
(0,1,0) & \text{else}.
\end{cases}
\]
We get three simpler $d$-PHs: 
\begin{enumerate}
\item $\Gamma_{+}:=\mathrm{gr}_{\nu_{0},1}(\Gamma_{\mathfrak{g},w})=(\mathcal{I},\mathcal{J}_{+},\mathcal{S}_{+})$
consisting of hyperedges of type $\widehat{e}_{\delta}$ with $\delta\in\Sigma^{+}(\mathfrak{g},\mathfrak{t})$. 
\item $\Gamma_{0}:=\mathrm{gr}_{\nu_{0},2}(\Gamma_{\mathfrak{g},w})=(\mathcal{I},\mathcal{J}_{0},\mathcal{S}_{0})$
consisting of hyperedges of type $\widehat{h}_{\alpha}$ with $\alpha\in\triangle$. 
\item $\Gamma_{-}:=\mathrm{gr}_{\nu_{0},3}(\Gamma_{\mathfrak{g},w})=(\mathcal{I},\mathcal{J}_{-},\mathcal{S}_{-})$
consisting of hyperedges of type $\widehat{e}_{\delta}$ with $\delta\in\Sigma^{-}(\mathfrak{g},\mathfrak{t})$. 
\end{enumerate}

\subsubsection{Reduction of $\Gamma_{+}$ and $\Gamma_{-}$ to $d$-hypergraphs}

Let $\mathfrak{g}$ be either $\mathfrak{so}_{2n},\mathfrak{so}_{2n+1}$
or $\mathfrak{sp}_{2n}$ and consider the following bilinear forms:
\[
\mathrm{B}_{2n}:=\left(\begin{array}{cc}
0 & I_{n}\\
I_{n} & 0
\end{array}\right),~~\mathrm{B}_{2n+1}:=\left(\begin{array}{ccc}
0 & I_{n} & 0\\
I_{n} & 0 & 0\\
0 & 0 & 1
\end{array}\right)~\text{ and }~\mathrm{J}_{n}:=\left(\begin{array}{cc}
0 & I_{n}\\
-I_{n} & 0
\end{array}\right).
\]
These forms induce embeddings $\mathfrak{so}_{2n}\hookrightarrow\mathfrak{so}_{2n+1}\hookrightarrow\mathfrak{gl}_{2n+1}$
and $\mathfrak{sp}_{2n}\hookrightarrow\mathfrak{gl}_{2n}$ such that
$\mathrm{diag}_{2n+1}\cap\mathfrak{g}$ is a Cartan subalgebra of
$\mathfrak{g}$, where $\mathrm{diag}_{2n+1}\subseteq\mathfrak{gl}_{2n+1}$
is the space of diagonal matrices. Let $\epsilon_{i}\in\mathfrak{t}^{*}$
be the restriction of the projection to $\mathrm{diag}_{n}$ to its
$i$-th coordinate. 
\begin{defn}
\label{def:roots of different types}Let $\mathfrak{g}$ be a simple
classical Lie algebra. A root $\alpha\in\Sigma(\mathfrak{g},\mathfrak{t})$
is said to be of 
\begin{enumerate}
\item Type $a$ if it is of the form $a_{i,j}:=\epsilon_{i}-\epsilon_{j}$,
for $i\neq j$. 
\item Type $b$ if it is of the form $b_{i,j}:=\epsilon_{i}+\epsilon_{j}$,
for any $i,j$ (including $2\epsilon_{i}$ in the case of $\mathfrak{sp}_{2n}$). 
\item Type $c$ if it is of the form $c_{i}:=\epsilon_{i}$. 
\end{enumerate}
Notice that $\mathfrak{sl}_{n}$ contains only roots of type $a$,
$\mathfrak{so}_{2n}$ and $\mathfrak{sp}_{2n}$ contain only roots
of type $a$ and $b$ and $\mathfrak{so}_{2n+1}$ contains roots of
type $a,b$ and $c$. 
\end{defn}

\begin{lem}
\label{lem:structure of roots of various types}Let $\delta\in\Sigma^{+}(\mathfrak{g},\mathfrak{t})$
and assume that there exist $\alpha_{i_{1}},\dots,\alpha_{i_{d}}\in\Sigma^{+}(\mathfrak{g},\mathfrak{t})$
such that $\widehat{e}_{\delta}([e_{\alpha_{i_{1}}},\dots,e_{\alpha_{i_{d}}}])\neq0$.
Then 
\begin{enumerate}
\item $\delta$ is of type $a$ if and only if all $\{\alpha_{i_{j}}\}$
are of type $a$. 
\item $\delta$ is of type $b$ if and only if either, 
\begin{enumerate}
\item exactly $d-1$ of $\{\alpha_{i_{j}}\}$ are of type $a$, and $\alpha_{i_{k}}$
is of type $b$ for a single $k\in[d]$, or; 
\item exactly $d-2$ of $\{\alpha_{i_{j}}\}$ are of type $a$, and $\alpha_{i_{k_{1}}}$
and $\alpha_{i_{k_{2}}}$ are of type $c$ for $k_{1},k_{2}\in[d]$. 
\end{enumerate}
\item $\delta$ is of type $c$ if and only if exactly $d-1$ of $\{\alpha_{i_{j}}\}$
are of type $a$, and for a single $k\in[d]$ the root $\alpha_{i_{k}}$
is of type $c$. 
\end{enumerate}
\end{lem}

We set an averaging weight (as in Subsection \ref{subsec:Reduction-of-=00005Cgamma +-})
\[
\omega_{\mathrm{av}}(e_{\alpha})=(d+1)^{\ell(\alpha)}\text{ and }\omega_{\mathrm{av}}(h_{\alpha})=1,
\]
and denote by $\Gamma_{1}=\mathrm{gr}_{\omega_{\mathrm{av}}}(\Gamma_{+})=(\mathcal{I},\mathcal{J}_{+},\mathcal{S}_{1})$.
Order $\{e_{i}\}_{i\in\mathcal{I}}$ as follows; choose an arbitrary
order $\mathrm{ord}(h_{\alpha})\in[n]$ for $\{h_{\alpha}\}_{\alpha\in\triangle}$,
and, 
\[
\mathrm{ord}(e_{\alpha}):=\begin{cases}
n+(i-1)n+j & \text{if }\alpha=a_{i,j}\\
n+4n^{2}+(i-1)n+j & \text{if }\alpha=b_{i,j}\\
n+8n^{2}+(i-1)n & \text{if }\alpha=c_{i}.
\end{cases}
\]
Set a monomialization weight 
\[
\omega_{\mathrm{mon}}(e_{i}):=(d+1)^{\mathrm{ord}(e_{i})},
\]
and denote the resulting $d$-PH by $\Gamma_{2}=\mathrm{gr}_{\omega_{\mathrm{mon}}}(\Gamma_{1})=(\mathcal{I},\mathcal{J}_{+},\mathcal{S}_{2})$. 
\begin{prop}
\label{prop:reduction to hypergraph-general case}Let $\delta\in\Sigma^{+}(\frak{g},\frak{t})$
be a positive root. Then we have the following: 
\begin{enumerate}
\item $\Gamma_{2}$ contains exactly one hyperedge $\{e_{i_{1}},\dots,e_{i_{d}}\}$
of type $\widehat{e}_{\delta}$, and it is unique of this type. 
\item If $\ell(\delta)\leq4d$, then the edge $\left(\{e_{i_{1}},\dots,e_{i_{d}}\},\widehat{e}_{\delta}\right)\in\mathcal{S}_{2}$
consists of vertices $e_{i_{j}}=e_{\alpha_{i_{j}}}$ with $\ell(\alpha_{i_{j}})\leq8d$
or $e_{\alpha_{i_{j}}}=h_{\alpha_{i_{j}}}$. 
\end{enumerate}
\end{prop}

\begin{proof}
The proof of (1) is similar to (1) of Proposition \ref{prop:reduction to a hypergraph}.
To prove Item (2), it is enough by the definition of the weight $\omega_{\mathrm{av}}$,
to prove the existence of $(\{e_{i_{1}},\dots,e_{i_{d}}\},\hat{e}_{\delta})\in\mathcal{S}_{+}$
with $e_{i_{j}}=e_{\alpha_{i_{j}}}$ and $\ell(\alpha_{i_{j}})\leq8d$,
or $e_{i_{j}}=h_{\alpha_{i_{j}}}$. Let $\beta_{1},\dots,\beta_{t}$
be the simple roots appearing when writing $\delta$ as a sum of simple
roots. Complete it to any connected subdiagram that consists of $4d$
simple roots (recall our assumption that $\mathrm{rk}(\mathfrak{g})>4d$).
Such a subdiagram is isomorphic to a Dynkin diagram of some classical
simple Lie subalgebra $\mathfrak{h}$ of rank $4d$. By Corollary
\ref{cor: commutator relations on simple Lie algebras} and Proposition
\ref{prop:Lie algebra word maps are generating}, $\varphi_{w}|_{\mathfrak{h}^{r}}:\mathfrak{h}^{r}\rightarrow\mathfrak{h}$
is generating, so we may find $(\{e_{i_{1}},\dots,e_{i_{d}}\},\hat{e}_{\delta})\in\mathcal{S}_{+}$
with $e_{i_{j}}\in\mathfrak{h}$. Since $\mathrm{rk}\mathfrak{h}=4d$
it follows that $\ell(e_{i_{j}})\leq8d$, so we are done. 
\end{proof}
Proposition \ref{prop:reduction to hypergraph-general case} implies
the following: 
\begin{cor}
\label{cor:reduction to d-PH for Gamma plus}The $d$-PH $\Gamma_{2}$
is induced from a $d$-hypergraph $\mathcal{G}=(\mathcal{V},\mathcal{E})$,
with $\mathcal{V}=\mathcal{I}$ and 
\[
\mathcal{E}=\{\text{the unique hyperedges }\{e_{i_{1}},\dots,e_{i_{d}}\}\text{ in }\Gamma_{2}\text{ of type }\widehat{e}_{\delta},\text{ for any }\delta\in\mathcal{J}_{+}\}.
\]
\end{cor}

We would now like to reduce $\mathcal{G}$ (or $\Gamma_{2}$) into
simpler hypergraphs. For this we need a more general analogue of Lemma
\ref{lemma:not so far away from average}. 
\begin{lem}
\label{lem:nice description of the hypergraph}Let $\Gamma_{2}=(\mathcal{I},\mathcal{J}_{+},\mathcal{S}_{2})$
be as above, and let $\hat{e}_{\delta}\in\mathcal{J}_{+}$ be such
that $\ell(\delta)\geq4d$. Then there exist $\alpha_{i_{1}},\alpha_{i_{2}},\dots,\alpha_{i_{d}}\in\Sigma^{+}(\mathfrak{g},\mathfrak{t})$
such that $(\{e_{\alpha_{i_{1}}},e_{\alpha_{i_{2}}},\dots,e_{\alpha_{i_{d}}}\},\widehat{e}_{\delta})\in\mathcal{S}_{2}$
and $\ell(\alpha_{i_{j}})\in[\left\lfloor \frac{\ell(\delta)}{d}\right\rfloor -1,\left\lceil \frac{\ell(\delta)}{d}\right\rceil +1]$
for every $j\in[d]$. 
\end{lem}

\begin{proof}
Let $\delta\in\Sigma^{+}(\mathfrak{g},\mathfrak{t})$. If $\delta$
is of type $a$, this is just Lemma \ref{lemma:not so far away from average}.
If $\delta$ is of type $b$, we may always find a tuple $(\alpha_{i_{1}},\dots,\alpha_{i_{d}})$
of positive roots $\alpha_{i_{j}}=\epsilon_{k_{j}}\pm\epsilon_{k_{j+1}}$
where $d-1$ roots are of type $a$, and one root is of type $b$,
whose lengths are in the interval $[\left\lfloor \frac{\ell(\delta)}{d}\right\rfloor -1,\left\lceil \frac{\ell(\delta)}{d}\right\rceil +1]$,\textcolor{red}{{}
}such that $\hat{e}_{\delta}([e_{\alpha_{i_{1}}},\dots,e_{\alpha_{i_{d}}}])\neq0$
and such that the $\{k_{j}\}_{j\in[d]}$ are distinct (the flexible
choice of their lengths allows us to make sure that the $\{k_{j}\}$
are indeed distinct). A similar argument as in the proof of Proposition
\ref{prop:reduction to a hypergraph} gives $(\{e_{\alpha_{i_{1}}},\dots,e_{\alpha_{i_{d}}}\},\widehat{e}_{\delta})\in\mathcal{S}_{+}$.
By definition of $\omega_{\mathrm{av}}$, any hyperedge $(\{e_{\beta_{j_{1}}},\dots,e_{\beta_{i_{d}}}\},\widehat{e}_{\delta})\in\mathcal{S}_{1}$
must satisfy that $\ell(\beta_{j_{u}})\in[\left\lfloor \frac{\ell(\delta)}{d}\right\rfloor -1,\left\lceil \frac{\ell(\delta)}{d}\right\rceil +1]$.
Since $\mathcal{S}_{2}\subseteq\mathcal{S}_{1}$ we are done. Similar
analysis can be done for $\delta$ of type $c$. 
\end{proof}

\subsubsection{\textsl{\label{subsec:Proof-for-the hypergraph G+- and Gsmall}Proof
for the $d$-hypergraph $\mathcal{G}=(\mathcal{V},\mathcal{E})$ in
the general case}}

We now reduce our hypergraph $\mathcal{G}=(\mathcal{V},\mathcal{E})$
into two simpler hypergraphs $\mathcal{G}_{+}$ and $\mathcal{G}_{\mathrm{small}}$,
as was done in the case of $\mathfrak{sl}_{n}$. Let $M=\{0,1\}$
and set 
\[
\omega_{2}(e_{\alpha})=\begin{cases}
(d,0) & \text{if }\alpha\in\Sigma^{+}(\mathfrak{g},\mathfrak{t})\\
(0,d^{2}) & \text{if }\alpha\in\Sigma^{-}(\mathfrak{g},\mathfrak{t})
\end{cases}\text{~ and ~}\omega_{2}(h_{\alpha})=(0,d^{2}).
\]

This coloring results in two sub-hypergraphs: 
\begin{enumerate}
\item $\mathcal{G}_{+}=(\mathcal{V}_{+},\mathcal{E}_{+})$ consisting of
vertices $\mathcal{V}_{+}=\{e_{\alpha}\}_{\alpha\in\Sigma^{+}(\mathfrak{g},\mathfrak{t})}$
and whose edges are 
\[
\mathcal{E}_{+}=\{(\{e_{\alpha_{i_{1}}},\dots,e_{\alpha_{i_{d}}}\},\widehat{e}_{\delta})\in\mathcal{E}:\alpha_{i_{k}}\in\Sigma^{+}(\mathfrak{g},\mathfrak{t})\text{ for any }k\in[d]\}.
\]
\item $\mathcal{G}_{\mathrm{small}}=(\mathcal{V}_{\mathrm{small}},\mathcal{E}_{\mathrm{small}})$
consisting of $\mathcal{E}_{\mathrm{small}}=\mathcal{E}\backslash\mathcal{E}_{+}$.
By Proposition \ref{prop:reduction to hypergraph-general case} and
Lemma \ref{lem:nice description of the hypergraph}, it follows that
all of the hyperedges in $\mathcal{E}_{\mathrm{small}}$ are of type
$\widehat{e}_{\delta}$ with $\ell(\delta)<4d$, and that $\mathcal{V}_{\mathrm{small}}$
consists of vertices $e_{i}\in\{h_{\alpha}\}_{\alpha\in\triangle}\cup\{e_{\alpha}\}_{\ell(\alpha)\leq8d}$. 
\end{enumerate}
\textbf{Proof for $\mathcal{G}_{+}$:}

We apply level separation,
\[
\omega_{\mathrm{LS},1}(e_{\alpha})=\begin{cases}
(2d,0) & \text{if }\left\lfloor \frac{\ell(\alpha)}{8}\right\rfloor \mathrm{\,mod\,}2=0\\
(0,2d+1) & \text{if }\left\lfloor \frac{\ell(\alpha)}{8}\right\rfloor \mathrm{\,mod\,}2=1.
\end{cases}
\]
The difference from the coloring $\omega_{\mathrm{LS},1}$ in the
case of $\mathfrak{sl}_{n}$ is that now the vertices $e_{\alpha_{i_{1}}},\dots,e_{\alpha_{i_{d}}}$
of each hyperedge $(\{e_{\alpha_{i_{1}}},\dots,e_{\alpha_{i_{d}}}\},\widehat{e}_{\delta})$
are not uniquely determined as in the $\mathfrak{sl}_{n}$ case (Lemma
\ref{lem:nice description of the hypergraph} versus Lemma \ref{lemma:not so far away from average}).
Note that each of $\mathrm{gr}_{\omega_{\mathrm{LS},1},0}\mathcal{G}_{+}$
and $\mathrm{gr}_{\omega_{\mathrm{LS},1},1}\mathcal{G}_{+}$ is a
disjoint union of hypergraphs 
\[
\left\{ \mathcal{G}_{+,u}=(\mathcal{V}_{+,u},\mathcal{E}_{+,u})\right\} _{u\in\nats},
\]
where $\mathcal{E}_{+,u}$ consists of certain edges $\widehat{e}_{\delta}$
with $\ell(\delta)\in[L(u)+1,L(u)+12d]$, for some $L(u)\in\nats$
depending on $u$.

Denote $\mathcal{G}_{+,u}$ by $\widetilde{\mathcal{G}}$ (as we deal
with each of $\mathcal{G}_{+,u}$ in the same method). Using two level
separations, we can separate these hypergraphs to ones that contain
hyperedges $\widehat{e}_{\delta}$ with all $\delta$ of the same
type (either $a,b$ or $c$). Explicitly, set 
\[
\omega_{\mathrm{LS},2}(e_{\alpha})=\begin{cases}
(-2,-1,0) & \text{if }\alpha\text{ is of type \ensuremath{a}}\\
(0,-d,2-2d) & \text{if }\alpha\text{ is of type \ensuremath{b}}\\
(0,0,-3d) & \text{if }\alpha\text{ is of type \ensuremath{c}},
\end{cases}
\]
and notice that by Lemma \ref{lem:structure of roots of various types},
$\omega_{\mathrm{LS},2}$ splits $\widetilde{\mathcal{G}}$ into three
hypergraphs: 
\begin{enumerate}
\item $\widetilde{\mathcal{G}}_{a}:=\mathrm{gr}_{\omega_{\mathrm{LS},2},1}(\widetilde{\mathcal{G}})$
which contains only edges $(\{e_{\alpha_{i_{1}}},\dots,e_{\alpha_{i_{d}}}\},\widehat{e}_{\delta})$
with $\delta$ of type $a$. 
\item $\widetilde{\mathcal{G}}_{b}:=\mathrm{gr}_{\omega_{\mathrm{LS},2},2}(\widetilde{\mathcal{G}})$
which contains only edges $(\{e_{\alpha_{i_{1}}},\dots,e_{\alpha_{i_{d}}}\},\widehat{e}_{\delta})$
with $\delta$ of type $b$, and such that $\{\alpha_{i_{1}},\dots,\alpha_{i_{d}}\}$
contains $d-1$ roots of type $a$ and one root of type $b$. 
\item $\widetilde{\mathcal{G}}_{c,b}:=\mathrm{gr}_{\omega_{\mathrm{LS},2},3}(\widetilde{\mathcal{G}})$
which contains only edges with $\delta$ of type $c$, and edges of
type $b$ such that $\{\alpha_{i_{1}},\dots,\alpha_{i_{d}}\}$ contains
$d-2$ roots of type $a$ and two roots of type $c$ (this only happens
in the case of $\mathfrak{g}=\mathfrak{so}_{2n+1}$). 
\end{enumerate}
Using another coloring function (with two colors) on the hypergraph
$\widetilde{\mathcal{G}}_{c,b}$ we can differentiate between the
edges of type $b$ and edges of type $c$, so that $\widetilde{\mathcal{G}}_{c,b}$
splits into $\widetilde{\mathcal{G}}'_{b}$ and $\widetilde{\mathcal{G}}_{c}$.
Notice that $\widetilde{\mathcal{G}}_{c}$ is a hypergraph with at
most $12d$ hyperedges (as there are at most $12d$ edges $\widehat{e}_{\delta}$
of type $c$ with $\ell(\delta)\in[L(u)+1,L(u)+12d]$), so by Corollary
\ref{cor:small polyhypergraphs are (FRS)}, the graph $\widetilde{\mathcal{G}}_{c}$
is (FRS) after $12d^{2}$ convolutions. The graph $\widetilde{\mathcal{G}}_{a}$
can be solved as in the $\mathfrak{sl}_{n}$ case.

It is left to deal with $\widetilde{\mathcal{G}}_{b}$ and $\widetilde{\mathcal{G}}'_{b}$.
We consider $\widetilde{\mathcal{G}}_{b}$ (the proof for $\widetilde{\mathcal{G}}'_{b}$
is similar). For any $i$ let $\bar{i}=i\mathrm{\,mod\,}12d$ and
for $s\in\{0,\dots,12d-1\}$ set 
\[
\omega_{\mathrm{LS},3}(\epsilon_{i}\pm\epsilon_{j})(s):=-i\cdot\delta_{s\bar{i}}\mp j\cdot\delta_{s\bar{j}},
\]
and note that 
\[
\widetilde{\omega}_{\mathrm{LS},3}(\hat{e}_{b_{k,l}})(s)=-k\cdot\delta_{s\bar{k}}-l\cdot\delta_{s\bar{l}},
\]
and in particular, the color of the hyperedge $\hat{e}_{b_{k,l}}$
is always $\bar{l}$ (with $k\leq l$). As a consequence, for any
$s\in\{0,\dots,12d-1\}$ the hypergraph $\mathrm{gr}_{\omega_{\mathrm{LS},3},s}(\widetilde{\mathcal{G}}_{b})$
consists only of hyperedges $\hat{e}_{b_{k,l}}$ with $l\,\mathrm{mod}\,12d=s$.
Using a similar weight, 
\[
\omega_{\mathrm{LS},4}(\epsilon_{i}\pm\epsilon_{j})(s):=n+(i-n)\delta_{s\bar{i}}\pm(j-n)\cdot\delta_{s\bar{j}},
\]
we get 
\[
\widetilde{\omega}_{\mathrm{LS},4}(\hat{e}_{b_{k,l}})(s)=nd+(k-n)\cdot\delta_{s\bar{k}}+(l-n)\delta_{s\bar{l}},
\]
and then for any $s_{1},s_{2}\in\{0,\dots,12d-1\}$ we get that $\widetilde{\mathcal{G}}_{b,s_{1},s_{2}}:=\mathrm{gr}_{\omega_{\mathrm{LS},4},s_{2}}\left(\mathrm{gr}_{\omega_{\mathrm{LS},3},s_{1}}(\widetilde{\mathcal{G}}_{b})\right)$
consists only of hyperedges $\hat{e}_{b_{k,l}}$ such that
\[
l\,\mathrm{mod}\,12d=s_{1}\text{ and }k\,\mathrm{mod}\,12d=s_{2}.
\]

\begin{lem}
\label{lem:Unique length}Fix $s_{1},s_{2}\in\{0,\dots,12d-1\}$.
Then all the hyperedges in $\widetilde{\mathcal{G}}_{b,s_{1},s_{2}}$
are of types $b_{k,l}$ of the same length. 
\end{lem}

\begin{proof}
Note that $\ell(b_{k,l})=2n-l-k$ (resp.~$2n-l-k+1$, $2n-l-k+2$)
for $\text{\ensuremath{\mathfrak{g}}}=\mathfrak{so}_{2n}$ (resp.~$\mathfrak{sp}_{2n},\mathfrak{so}_{2n+1}$).
But by the definition of $\mathcal{G}_{+,u}$, its edges $\mathcal{E}_{+,u}$
consists of $\widehat{e}_{\delta}$ with $\ell(\delta)\in[L(u)+1,L(u)+12d]$.
Since $\hat{e}_{b_{k,l}}\in\widetilde{\mathcal{G}}_{b,s_{1},s_{2}}$
it follows that $\ell(b_{k,l})$ has a unique value inside $[L(u)+1,L(u)+12d]$.
\end{proof}
Now for a fixed $s_{1},s_{2}\in\{0,\dots,12d-1\}$, let $L$ be the
unique length of the edges of $\widetilde{\mathcal{G}}_{b,s_{1},s_{2}}$.
Recall from Lemma \ref{lem:nice description of the hypergraph}, that
the vertices of $\widetilde{\mathcal{G}}_{b,s_{1},s_{2}}$ have lengths
inside the interval $[\left\lfloor \frac{L}{d}\right\rfloor -1,\left\lceil \frac{L}{d}\right\rceil +1]$.
Set $t:=\left\lfloor \frac{L}{d}\right\rfloor -2$. For any $i$ let
$\bar{i}=\left\lfloor \frac{i}{t}\right\rfloor \,\mathrm{mod}\,8d$.
Set
\[
\omega_{\mathrm{LS},5}(\epsilon_{i}\pm\epsilon_{j})(s):=n+(i-n)\delta_{s\bar{i}}\pm(j-n)\cdot\delta_{s\bar{j}},
\]
for $s\in\{0,\dots,8d-1\}$. Again we have
\[
\widetilde{\omega}_{\mathrm{LS},5}(\hat{e}_{b_{k,l}})(s)=nd+(k-n)\cdot\delta_{s\bar{k}}+(l-n)\delta_{s\bar{l}},
\]
and thus for any $s_{3}\in\{0,\dots,8d-1\}$ we get that $\widetilde{\mathcal{G}}_{b,s_{1},s_{2},s_{3}}:=\mathrm{gr}_{\omega_{\mathrm{LS},5},s_{3}}(\widetilde{\mathcal{G}}_{b,s_{1},s_{2}})$
consists only of hyperedges $\hat{e}_{b_{k,l}}$ such that
\[
l\,\mathrm{mod}\,12d=s_{1}\text{ and }k\,\mathrm{mod}\,12d=s_{2}\text{ and }\left\lfloor \frac{k}{t}\right\rfloor \,\mathrm{mod}\,8d=s_{3}.
\]

\begin{lem}
For any $s_{1},s_{2}\in\{0,\dots,12d-1\}$ and $s_{3}\in\{0,\dots,8d-1\}$,
the hypergraph $\widetilde{\mathcal{G}}_{b,s_{1},s_{2},s_{3}}$ is
a disjoint union of hypergraphs, each consisting of at most two hyperedges. 
\end{lem}

\begin{proof}
Let $k_{0}$ be minimal such that $\hat{e}_{b_{k_{0},l_{0}}}$ is
a hyperedge in $\widetilde{\mathcal{G}}_{b,s_{1},s_{2},s_{3}}$. Then
the types of the hyperedges in $\widetilde{\mathcal{G}}_{b,s_{1},s_{2},s_{3}}$
are a subset of the types $\{b_{k_{0}+12jd,l_{0}-12jd}\}_{j\in\nats}$,
where $k_{0}+12jd\leq l_{0}-12jd$ for any $j$. Furthermore, for
any hyperedge $\hat{e}_{b_{k,l}}$ in $\widetilde{\mathcal{G}}_{b,s_{1},s_{2},s_{3}}$,
$k$ belongs to a specific interval $[t_{0},t_{0}+\left\lfloor \frac{L}{d}\right\rfloor -1]$,
which is roughly the length of a vertex of $\hat{e}_{b_{k,l}}$. Lemma
\ref{lem:nice description of the hypergraph} and the weight $\omega_{\mathrm{LS},5}$
guarantee that two edges $\hat{e}_{b_{k,l}}$ and $\hat{e}_{b_{k',l'}}$
share a common vertex only if $\left|k-k'\right|<4d$, or $\left|k-l'-e\frac{L}{d}\right|<4d$
for some $e\in\ints$. The assumption that $k=k'\mathrm{\,mod\,}12d$
and $\omega_{\mathrm{LS},5}$ guarantee that either $b_{k,l}=b_{k',l'}$
or there is precisely one $\hat{e}_{b_{k',l'}}\neq\hat{e}_{b_{k,l}}$
with a common vertex. 
\end{proof}
Finally, using one more coloring function $\omega_{\mathrm{LS},6}$
(involving two colors) we may split $\widetilde{\mathcal{G}}_{b,s_{1},s_{2},s_{3}}$
to two hypergraphs, each a disjoint union of hyperedges. Moreover,
each of these hyperedges is of type $b_{k,l}$ and by Lemma \ref{lem:nice description of the hypergraph},
the corresponding monomials (via Lemma \ref{lem:combinatorial word map})
are of the form $x_{1}^{n_{1}}\cdot\dots\cdot x_{k}^{n_{k}}$ where
$\sum\limits _{i=1}^{k}n_{i}=d$ and $n_{i}\leq2$. By Corollary \ref{cor:small polyhypergraphs are (FRS)},
$\mathrm{dim}V\geq r$ (resp.~$\mathrm{dim}V\geq3r$) is enough for
each hyperedge to be flat (resp.~(FRS)). In summary, we have 
\begin{enumerate}
\item The hypergraph $\widetilde{\mathcal{G}}_{a}$ is flat after $4d^{3}$
convolutions and (FRS) after $8d^{3}$ convolutions. 
\item The hypergraph $\widetilde{\mathcal{G}}_{c}$ is (FRS) (and flat)
after $12d^{2}$ convolutions. 
\item The hypergraphs $\widetilde{\mathcal{G}'}_{b}$ and $\widetilde{\mathcal{G}}_{b}$
are flat after $12d\cdot12d\cdot8d\cdot2\cdot1=2304d^{3}$ convolutions,
and (FRS) after $12d\cdot12d\cdot8d\cdot2\cdot3=6912d^{3}$ convolutions. 
\end{enumerate}
Altogether we get that $\mathcal{G}_{+}$ is flat after $10^{4}d^{3}>2\cdot\left(4d^{3}+12d^{2}+2\cdot2304d^{3}\right)$
convolutions and (FRS) after $3\cdot10^{4}d^{3}$ convolutions.

\textbf{Proof for $\mathcal{G}_{\mathrm{small}}$:}

Since in the general case we have a weaker version of Proposition
\ref{prop:reduction to a hypergraph} (namely Proposition \ref{prop:reduction to hypergraph-general case}),
we need to apply a more careful analysis than in the $\mathfrak{sl}_{n}$
case. Set $\bar{i}=i\,\mathrm{mod\,}8d^{2}$ and $\omega_{\mathrm{LS},7}:\mathcal{V}_{\mathrm{small}}\rightarrow\mathbb{N}^{\{0,\dots,8d^{2}-1\}}$
by
\[
\omega_{\mathrm{LS},7}(\pm\epsilon_{i}\pm\epsilon_{j})(s):=n\pm(i-n)\cdot\delta_{s\bar{i}}\pm(j-n)\cdot\delta_{s\bar{j}},
\]
for $s\in\{0,1\}$. Note that both $\hat{e}_{a_{k,l}},\hat{e}_{b_{k,l}}$
or $\hat{e}_{c_{k}}$ has the color $\bar{k}$ (with $k<l$). By Proposition
\ref{prop:reduction to hypergraph-general case} (and by the definition
of $\mathcal{E}_{\mathrm{small}}$), we have that $\mathrm{gr}_{\omega_{\mathrm{LS},s}}(\mathcal{G}_{\mathrm{small}})$
is a disjoint union of hypergraphs with at most $8d$ hyperedges.
Using a similar coloring function with $8d$ colors, we reduce to
a disjoint union of hyperedges (i.e., each hypergraph consists of
hyperedges $\hat{e}_{a_{k,l}},\hat{e}_{b_{k,l}}$ or $\hat{e}_{c_{k}}$
with $k=k_{0}\,\mathrm{mod\,}8d^{2}$ and $l=l_{0}\,\mathrm{mod\,}8d$
for some $k_{0},l_{0}$). By Corollary \ref{cor:small polyhypergraphs are (FRS)},
$d$ convolutions are enough for the (FRS) property of a hyperedge.

\textbf{Summary for $\mathcal{G}_{\mathrm{small}}$ and for $\Gamma_{+}$
and $\Gamma_{-}$:}

$\mathcal{G}_{\mathrm{small}}$ is flat after $64d^{3}=8d^{2}\cdot8d\cdot1$
convolutions and (FRS) after $64d^{4}=8d^{2}\cdot8d\cdot d$ convolutions.
Combining with the case of $\mathcal{G}_{+}$ we get that: 
\begin{enumerate}
\item The hypergraphs $\Gamma_{+}$ and $\Gamma_{-}$ are flat after $2\cdot10^{4}d^{3}>10^{4}d^{3}+64d^{3}$. 
\item The hypergraphs $\Gamma_{+}$ and $\Gamma_{-}$ are (FRS) after $4\cdot10^{4}d^{4}>3\cdot10^{4}d^{3}+64d^{4}$. 
\end{enumerate}

\subsubsection{Proof for $\Gamma_{0}$: the general case}

The proof is almost identical to the case of $\mathfrak{sl}_{n}$
(see Subsection \ref{subsec:The-case-of gamma zero}). Let $\mathfrak{g}$
be either $\mathfrak{so}_{2n}$, $\mathfrak{so}_{2n+1}$ or $\mathfrak{sp}_{2n}$.
By eliminating edges containing vertices of type $c$, we may assume
that $\mathfrak{g}$ is $\mathfrak{so}_{2n}$ or $\mathfrak{sp}_{2n}$.
As in the case of $\mathfrak{sl}_{n}$, using an averaging weight,
a monomialization weight, and the level separation method, we can
reduce the $m$-th jet map $J_{m}(\Phi_{\Gamma_{0},V,\{\tau_{\gamma}\}})$
to a disjoint union of morphisms to a $\mathrm{dim}\mathfrak{t}\cdot(2d+1)$-dimensional
vector space.

Let $\mathfrak{sl}_{n}\hookrightarrow\mathfrak{g}$ be the map $A\mapsto\left(\begin{array}{cc}
A & 0\\
0 & -A^{t}
\end{array}\right)$. Note that $\triangle$ contains $n-1$ simple roots $\{a_{12},...,a_{n-1,n}\}$
of type $a$, and one simple root ($b_{n}:=b_{n,n}$ or $b_{n}:=b_{n-1,n}$)
of type $b$. Let $\mathfrak{h}$ be the simple Lie algebra generated
by $e_{a_{n-d,n-d+1}},...,e_{a_{n,n-1}}$ and $e_{b_{n}}$ (of the
same type as $\mathfrak{g}$). Using an elimination weight, we may
eliminate all monomials involving variables outside of $\mathrm{span}\{\mathfrak{h},\mathfrak{sl}_{n}\}$.
From here, the proof is essentially the same as in the case of $\mathfrak{sl}_{n}$,
where $\mathfrak{h}$ takes the role of $\mathrm{span}\{e_{i,j}\}_{i,j\leq d}$.
In particular, $(\Gamma_{0},V,\{\tau_{\gamma}\})$ is flat if $\mathrm{dim}V\geq12d^{2}r$
and it is (FRS) for $\mathrm{dim}V\geq60d^{3}r$, as in the case of
$\mathfrak{sl}_{n}$.

\subsubsection{\label{subsec:Proof-of-the main theorems}Proof of Theorems \ref{thm: main thm Lie algebra word maps},
\ref{thm:lower bounds on epsilon jet flatness} and \ref{thm:main theorem for general Lie algebra}}

Theorem \ref{thm:main theorem for general Lie algebra} follows from
the analysis we did for $\Gamma_{+},\Gamma_{-}$ and $\Gamma_{0}$. 
\begin{proof}[Proof of Theorem \ref{thm: main thm Lie algebra word maps}]
By Propositions \ref{prop:reduction to self convolutions} and \ref{prop:reduction to self convolution- the general case}
it is enough to prove the theorem in the case of self-convolutions
of a single $d$-homogeneous word $w\in\mathcal{L}_{r}$ of pure type,
at the price of twice the required number of convolutions. By Proposition
\ref{prop:extended properties of convolutions} and Corollary \ref{cor:jet-flat implies FTS after two convolusions},
it is enough to prove we can obtain the flatness and (FRS) properties. 
\begin{enumerate}
\item For a simple Lie algebra $\mathfrak{g}$ with $\mathrm{rk}(\mathfrak{g})>4d$,
the map $\varphi_{w}^{*t}$ is flat (resp.~(FRS)) if $t\geq4\cdot10^{4}d^{3}$
(resp.~$t\geq8\cdot10^{4}d^{4}$). 
\item For a simple Lie algebra $\mathfrak{g}$ with $\mathrm{rk}(\mathfrak{g})\leq4d$,
we have $\mathrm{dim}\mathfrak{g}\leq100d^{2}$. By Corollary \ref{cor:low dimensional Lie algebras},
the map $\varphi_{w}^{*t}$ is (FRS) if $t\geq100d^{3}$.
\end{enumerate}
This finishes the proof. 
\end{proof}
\begin{proof}[Proof of Theorem \ref{thm:lower bounds on epsilon jet flatness}]
Let $w\in\mathcal{L}_{r}$ be a Lie algebra word of degree $d$ and
let $\mathfrak{g}$ be a simple $K$-Lie algebra. At first assume
that $\mathrm{rk}(\mathfrak{g})<50d$ (and thus $\mathrm{dim}(\mathfrak{g})<10^{4}d^{2}$).
Then for any $Y\in\mathfrak{g}$ the fiber $\varphi_{w}^{-1}(Y)$
is defined by polynomials of degree at most $d$. By Part (2) of Fact
\ref{fact:properties of log cannonical threshold}, we have $\mathrm{lct}(\mathfrak{g}^{r},\varphi_{w}^{-1}(Y))\geq\frac{1}{d}$
and therefore by Lemma \ref{lem:epsilon jet flat and log canonical threshold}
$\varphi_{w}:\mathfrak{g}^{r}\rightarrow\mathfrak{g}$ is $\frac{1}{10^{4}d^{3}}$-jet
flat.

Now assume that $\mathrm{rk}\mathfrak{g}\geq50d$. Choose a Chevalley
basis $\mathcal{B}=\{e_{\beta}\}_{\beta\in\Sigma(\mathfrak{g},\mathfrak{t})}\cup\{h_{\alpha}\}_{\alpha\in\mathfrak{g}}$
and let $\hat{\mathcal{B}}$ be its dual basis. Identify $\varphi_{w}:\mathfrak{g}^{r}\rightarrow\mathfrak{g}$
with $\psi_{w}:\mathfrak{g}^{r}\rightarrow\mathbb{A}^{\mathrm{dim}\mathfrak{g}}$
via $\mathcal{B}$ and $\hat{\mathcal{B}}$. Let $\pi_{+}:\mathbb{A}^{\mathrm{dim}\mathfrak{g}}\rightarrow\mathbb{A}^{N}$
be the projection to the root subspaces $\{e_{\beta}\}$, such that
$\beta$ is a positive root and $\ell(\beta)>4d$. By our assumption
on $\mathfrak{g}$, we have $N>\frac{1}{4}\mathrm{dim}\mathfrak{g}$.
Notice that by Subsection \ref{subsec:Proof-for-the hypergraph G+- and Gsmall}
($\pi_{+}\circ\psi_{w})^{*t}$ is (FRS) for $t\geq3\cdot10^{4}d^{3}$.
By Lemma \ref{lem:convolution gives bounds on epsilon flatness},
$\pi_{+}\circ\psi_{w}$ is $\frac{1}{3\cdot10^{4}d^{3}}$-jet flat,
so for any $Y\in\mathbb{A}^{\mathrm{dim}\mathfrak{g}}$ we have by
Fact \ref{fact:properties of log cannonical threshold},
\[
\mathrm{lct}(\mathfrak{g}^{r},\psi_{w}^{-1}(Y))\geq\mathrm{lct}(\mathfrak{g}^{r},(\pi_{+}\circ\psi_{w})^{-1}(\pi_{+}(Y)))\geq\frac{N}{3\cdot10^{4}d^{3}}\geq\frac{\mathrm{dim}\mathfrak{g}}{2\cdot10^{5}d^{3}}.
\]
Altogether we get that $w$ is $\frac{1}{2\cdot10^{5}d^{3}}$-jet
flat. 

For matrix word map, the same analysis yields that $w\in\mathcal{A}_{r}$
is $\frac{1}{25d^{3}}$-jet flat. 
\end{proof}

\section{\label{sec:The-commutator-map-revisited}The commutator map revisited}

In this section we use the methods of Subsection \ref{subsec:Degenerations-of-jets of word maps}
in order to provide a simple proof to the following theorem: 
\begin{thm}
\label{thm:Commutator is (FRS) after 4 convolutions}Let $w_{0}=xyx^{-1}y^{-1}\in F_{r}$
and $w=[X,Y]\in\mathcal{L}_{r}$ be the commutator words. Then for
any semisimple algebraic $K$-group $\underline{G}$, with Lie algebra
$\mathfrak{g}$, the corresponding word maps $\varphi_{w_{0}}(x,y)=xyx^{-1}y^{-1}$
and $\psi_{w}(X,Y)=[X,Y]$ satisfy 
\begin{enumerate}
\item $\varphi_{w_{0}}^{*4}:\underline{G}^{8}\rightarrow\underline{G}$
is (FRS). 
\item $\psi_{w}^{*4}:\mathfrak{g}^{8}\rightarrow\mathfrak{g}$ is (FRS). 
\end{enumerate}
\end{thm}

It was shown in \cite{AA16}, that $\varphi_{w_{0}}^{*21}$ is (FRS)
for all semisimple algebraic groups $\underline{G}$, and that $\psi_{w}^{*21}$
is (FRS) for all classical Lie algebras. By using a more efficient
degenerations of graphs, these bounds were improved to $11$ in \cite{Kap}.
It is conjectured that $\varphi_{w_{0}}^{*2}$ and $\psi_{w}^{*2}$
are (FRS), and this conjecture was verified for $\mathrm{SL}_{n}$
in \cite{Bud}. Theorem \ref{thm:Commutator is (FRS) after 4 convolutions}
improves these bounds (outside of the $\mathrm{SL}_{n}$ and $\mathfrak{sl}_{n}$
cases), and its proof is significantly simpler. It also treats the
case of exceptional Lie algebras, which was untreated before. Furthermore,
in \cite{AA16,AA18}, the (FRS) property of $\varphi_{w_{0}}^{*21}$
was used to bound the representation growth of compact $p$-adic and
arithmetic groups by a polynomial independent of the rank of the group.
Theorem \ref{thm:Commutator is (FRS) after 4 convolutions} improves
these bounds (see Corollary \ref{cor: representation growth}).

By \cite[Theorem IV]{AA16}, in order to prove Theorem \ref{thm:Commutator is (FRS) after 4 convolutions},
it is enough to show that $\psi_{w}^{*4}:\mathfrak{g}^{4}\rightarrow\mathfrak{g}$
is (FRS) at $(0,\ldots,0)$. We start by proving the following theorem,
which is interesting on its own. 
\begin{thm}
\label{Thm:commutator is flat}The map $\psi_{w}^{*2}$ is flat with
geometrically irreducible, reduced fibers. 
\end{thm}

The flatness of $\varphi_{w_{0}}^{*2}:\underline{G}^{4}\rightarrow\underline{G}$
follows from \cite{Li93} and \cite[Theorem 2.8]{Sha09}. For Lie
algebras, as far as we know, only flatness of $\mathfrak{sl}_{n}$
was previously known (see \cite{CB01} and also \cite{Bud}). 
\begin{proof}
We may assume that $K=\complex$. Let $\mathfrak{g}$ be a simple
Lie algebra and $\langle\cdot,\cdot\rangle$ be the Killing form on
$\mathfrak{g}$. We can endow $\mathfrak{g}$ with a $\ints$-structure,
by choosing a Chevalley basis. For almost any prime $p$, the form
$\langle\cdot,\cdot\rangle$ is non-degenerate on $\mathfrak{g}(\mathbb{F}_{p})$
(see e.g.~\cite{Sel67}), and thus we have an identification $Z\mapsto\Psi_{Z}$
between $\mathfrak{g}(\mathbb{F}_{p})$ and $\mathfrak{g}(\mathbb{F}_{p})^{\vee}$,
where $\Psi_{Z}(W):=\Psi(\langle Z,W\rangle)$ and $\Psi$ is a fixed
non-trivial additive character of $\mathbb{F}_{p}$. Set $\tau_{w}(X):=\frac{\left|\psi_{w}^{-1}(X)\right|}{\left|\mathfrak{g}(\mathbb{F}_{p})\right|^{2}}$.
In order to prove that $\psi_{w}^{*2}$ is (FGI) it is enough to show
\[
\left|\mathfrak{g}(\mathbb{F}_{p})\right|\cdot\left|\tau_{w}^{*2}(0)-\frac{1}{\left|\mathfrak{g}(\mathbb{F}_{p})\right|}\right|=O(p^{-1/2}),\tag{\ensuremath{\star}}
\]
for almost any prime $p$. Indeed, the Lang-Weil bounds and $(\star)$
imply that $\psi_{w}^{*2}$ is (FGI) over $0$. Note that $\tau_{w}$
is a symmetric measure (i.e.~$\tau_{w}(X)=\tau_{w}(-X)$ for every
$X\in\frak{g}(\mathbb{F}_{p})$), so $\left|(\psi_{w}^{*2})^{-1}(0)\right|\geq\left|(\psi_{w}^{*2})^{-1}(X)\right|$
for every $X\in\mathfrak{g}(\mathbb{F}_{p})$. This will imply that
$\psi_{w}^{*2}$ is (FGI), as otherwise, using Chebotarev's density
theorem, we may find infinitely many primes $p$ that contradict $(\star)$
(see e.g.~Theorem \ref{thm: flatness and counting points}). Using
Fourier transform $\mathcal{F}$, we may write, 
\[
\tau_{w}(X)=\frac{1}{\left|\mathfrak{g}(\mathbb{F}_{p})\right|}\sum_{Z\in\mathfrak{g}(\mathbb{F}_{p})}\mathcal{F}(\tau_{w})(Z)\cdot\Psi_{Z}(X).
\]
We then have,
\begin{align*}
\mathcal{F}(\tau_{w})(Z) & =\sum_{W\in\mathfrak{g}(\mathbb{F}_{p})}\tau_{w}(W)\cdot\Psi_{Z}(W)=\frac{1}{\left|\mathfrak{g}(\mathbb{F}_{p})\right|^{2}}\sum_{W\in\mathfrak{g}(\mathbb{F}_{p})}\left|\psi_{w}^{-1}(W)\right|\cdot\Psi_{Z}(W)\\
 & =\frac{1}{\left|\mathfrak{g}(\mathbb{F}_{p})\right|^{2}}\sum_{W\in\mathfrak{g}(\mathbb{F}_{p})}\sum_{\begin{array}{c}
X,Y\in\mathfrak{g}(\mathbb{F}_{p})\\{}
[X,Y]=W
\end{array}}\Psi_{Z}(W)=\frac{1}{\left|\mathfrak{g}(\mathbb{F}_{p})\right|^{2}}\sum_{X,Y\in\mathfrak{g}(\mathbb{F}_{p})}\Psi_{Z}([X,Y])\\
 & =\frac{1}{\left|\mathfrak{g}(\mathbb{F}_{p})\right|^{2}}\sum_{X,Y\in\mathfrak{g}(\mathbb{F}_{p})}\Psi(\langle Z,[X,Y]\rangle)\\
 & =\frac{1}{\left|\mathfrak{g}(\mathbb{F}_{p})\right|^{2}}\sum_{X,Y\in\mathfrak{g}(\mathbb{F}_{p})}\Psi(\langle X,[Y,Z]\rangle)=\frac{\left|\mathrm{Cent}_{\mathfrak{g}(\mathbb{F}_{p})}(Z)\right|}{\left|\mathfrak{g}(\mathbb{F}_{p})\right|},
\end{align*}
where the last equality follows from the Schur orthogonality relations
(i.e.~averaging any non-trivial character over $\mathfrak{g}(\mathbb{F}_{p})$
yields $0$). In particular, we have the following:
\[
\tau_{w}^{*2}(0)=\frac{1}{\left|\mathfrak{g}(\mathbb{F}_{p})\right|}\sum_{Z\in\mathfrak{g}(\mathbb{F}_{p})}\left(\mathcal{F}(\tau_{w})(Z)\right)^{2}\cdot\Psi_{Z}(0)=\sum_{Z\in\mathfrak{g}(\mathbb{F}_{p})}\frac{\left|\mathrm{Cent}_{\mathfrak{g}(\mathbb{F}_{p})}(Z)\right|^{2}}{\left|\mathfrak{g}(\mathbb{F}_{p})\right|^{3}}.\tag{\ensuremath{\star\star}}
\]
Notice that ($\star\star$) has a geometric description. It equals
$\frac{\left|\Upsilon(\mathbb{F}_{p})\right|}{\left|\mathfrak{g}(\mathbb{F}_{p})\right|^{3}}$,
where 
\[
\Upsilon:=\{(X,Y,Z)\in\mathfrak{g}^{3}:[X,Z]=[Y,Z]=0\}.
\]
\begin{lem}
\label{lem:6.3}For any $\mathrm{dim}\mathfrak{g}\geq k\in\nats$,
the set 
\[
S_{k}:=\{Z\in\mathfrak{g}:\mathrm{dim}\mathrm{Cent}_{\mathfrak{g}}(Z)=k\}
\]
satisfies the following if $k<\mathrm{dim}\mathfrak{g}$, 
\[
\mathrm{dim}S_{k}\leq\mathrm{dim}\mathfrak{g}+\mathrm{rk}(\mathfrak{g})-k,
\]
and furthermore $\mathrm{dim}S_{\mathrm{dim}\mathfrak{g}}=0$. 
\end{lem}

\begin{proof}
We have $\psi_{w}^{-1}(0)\supseteq\{(Z,W)\in S_{k}\times\mathfrak{g}:[Z,W]=0\}$,
so 
\[
\mathrm{dim}S_{k}\leq\mathrm{dim}\psi_{w}^{-1}(0)-k=\mathrm{dim}\mathfrak{g}+\mathrm{rk}(\mathfrak{g})-k,
\]
where the last equality is well known (e.g.~\cite[Theorem A]{Rich79}). 
\end{proof}
\begin{lem}
\label{lem:6.4}For any $0\neq Z\in\mathfrak{g}$, we have, 
\[
\mathrm{dim}\mathrm{Cent}_{\mathfrak{g}}(Z)\leq\mathrm{dim}\mathfrak{g}-\mathrm{rk}(\mathfrak{g})-1.
\]
\end{lem}

\begin{proof}
We use the fact that any proper subalgebra of a simple Lie algebra
is of codimension at least $\mathrm{\mathrm{rk}(\mathfrak{g})}$ where
equality is achieved only for type $A$, in which the only subalgebra
of codimension $\mathrm{rk}(\mathfrak{g})$ is a parabolic subalgebra
(e.g.~\cite[p.47, Table 1]{CX18}, \cite{Stu91}). For $Z$ semisimple,
$\mathrm{Cent}_{\mathfrak{g}}(Z)$ is a Levi subalgebra so it is properly
contained in a parabolic subalgebra. For non-nilpotent elements, their
centralizer is contained in the centralizer of their semisimple part.
For nilpotent elements, their centralizer cannot contain a maximal
torus, so we are done. 
\end{proof}
\begin{prop}
\label{prop:good point count}We have $\mathrm{dim}\Upsilon=2\mathrm{dim}\mathfrak{g}$
and furthermore, 
\[
\left|\Upsilon(\mathbb{F}_{p})\right|=\left|\mathfrak{g}(\mathbb{F}_{p})\right|^{2}(1+O(p^{-1/2})).
\]
\end{prop}

\begin{proof}
Notice that 
\[
\Upsilon=\bigcup_{k=0}^{\mathrm{dim}\mathfrak{g}}\{(X,Y,Z)\in\Upsilon:Z\in S_{k}\}
\]
so 
\begin{align*}
\mathrm{dim}\Upsilon & =\max\limits _{k}\mathrm{dim}\{(X,Y,Z)\in\Upsilon:Z\in S_{k}\}=\max_{k}(2k+\mathrm{dim}S_{k})\\
 & =\mathrm{max}(2\mathrm{dim}\mathfrak{g},\max_{k<\mathrm{dim}\mathfrak{g}}(2k+\mathrm{dim}S_{k}))\leq2\mathrm{dim}\mathfrak{g}
\end{align*}
where the inequality follows by Lemmas \ref{lem:6.3} and \ref{lem:6.4}.
Furthermore, Lemma \ref{lem:6.4} also implies that $\mathrm{dim}\{(X,Y,Z)\in\Upsilon:Z\in S_{k}\}\leq2\mathrm{dim}\mathfrak{g}-1$
for $k\neq\mathrm{dim}\mathfrak{g}$. Thus $\mathfrak{g}\times\mathfrak{g}\times\{0\}$
is the only top dimensional irreducible component of $\Upsilon$ so
the proposition follows. 
\end{proof}
Proposition \ref{prop:good point count} implies that ($\star\star$)
can be written as 
\[
\sum_{Z\in\mathfrak{g}(\mathbb{F}_{p})}\frac{\left|\mathrm{Cent}_{\mathfrak{g}(\mathbb{F}_{p})}(Z)\right|^{2}}{\left|\mathfrak{g}(\mathbb{F}_{p})\right|^{3}}=\frac{\left|\Upsilon(\mathbb{F}_{p})\right|}{\left|\mathfrak{g}(\mathbb{F}_{p})\right|^{3}}=\frac{1}{\left|\mathfrak{g}(\mathbb{F}_{p})\right|}(1+O(p^{-\frac{1}{2}})),
\]
so $(\star)$ holds. This proves that $\psi_{w}^{*2}$ is flat and
has geometrically irreducible fibers.

Note that the fibers of $\psi_{w}^{*2}$ are local complete intersections
and in particular Cohen-Macaulay, so by the $R_{0}+S_{1}$-criterion
(e.g.~\cite[Proposition 5.8.5]{Gro66}), in order to show that the
fibers are reduced it is enough to show that they are generically
reduced. Since the fibers are irreducible, it is enough to find a
smooth point in $(\psi_{w}^{*2})^{-1}(Y)$ for every $Y\in\mathfrak{g}$.
Note that $\psi_{w}$ is generically smooth over $\mathfrak{g}$,
i.e.~smooth over an open (dense) set $U\subseteq\mathfrak{g}$. Since
$U+U=\mathfrak{g}$, we may write $Y=Y_{1}+Y_{2}$, with $Y_{1},Y_{2}\in U$.
Take any points $(X_{1},X_{2})\in\psi_{w}^{-1}(Y_{1})$ and $(X_{3},X_{4})\in\psi_{w}^{-1}(Y_{2})$
and then $\psi_{w}^{*2}$ is smooth at $(X_{1},X_{2},X_{3},X_{4})$.
This implies that $(X_{1},X_{2},X_{3},X_{4})$ is a smooth point of
$(\psi_{w}^{*2})^{-1}(Y)$ and we are done. We have finished proving
Theorem \ref{Thm:commutator is flat}. 
\end{proof}
We now turn to the proof of Theorem \ref{thm:Commutator is (FRS) after 4 convolutions}.

Recall from Section \ref{subsec:Basic-properties-of} (and Example
\ref{exa:example for jet of the commutator map}) that the $m$-th
jet map $J_{m}(\psi_{w}):J_{m}(\mathfrak{g}^{2})\rightarrow J_{m}(\mathfrak{g})$
is given as follows (where $a_{i}$ are some non-zero numbers)
\begin{flalign*}
J_{m}(\psi_{w}) & (X_{1},X_{2},X_{1}^{(1)},X_{2}^{(1)},\ldots,X_{1}^{(m)},X_{2}^{(m)})\\
 & =([X_{1},X_{2}],[X_{1}^{(1)},X_{2}^{(0)}]+[X_{1}^{(0)},X_{2}^{(1)}],\ldots,\sum_{i+j=m}a_{i}[X_{1}^{(i)},X_{2}^{(j)}]).
\end{flalign*}
Using an averaging weight $\omega_{\mathrm{av}}$ and a monomialization
weight $\omega_{\mathrm{mon}}$ as in Subsection \ref{subsec:Degenerations-of-jets of word maps},
we reduce to a map of the following form: 
\[
\widetilde{\phi}_{m}(X_{1},X_{2},X_{1}^{(1)},X_{2}^{(1)},\ldots,X_{1}^{(m)},X_{2}^{(m)}):=([X_{1},X_{2}],[X_{1}^{(1)},X_{2}],[X_{1}^{(1)},X_{2}^{(1)}],\ldots,[X_{1}^{\left(\left\lceil \frac{m}{2}\right\rceil \right)},X_{2}^{\left(\left\lfloor \frac{m}{2}\right\rfloor \right)}]).
\]
Using the level separation method, similarly to Subsection \ref{subsec:Discussion-on-the methods},
we set $\omega_{\mathrm{LS}}(X_{1}^{(u)})=(10^{u}\cdot2,10^{u}\cdot3)$
and $\omega_{\mathrm{LS}}(X_{2}^{(u)})=(10^{u}\cdot3,10^{u}\cdot1)$
such that $\mathrm{gr}_{\omega_{\mathrm{LS}},1}(\widetilde{\phi}_{m})$
and $\mathrm{gr}_{\omega_{\mathrm{LS}},2}(\widetilde{\phi}_{m})$
are maps of the form 
\[
\phi_{m}(Y_{1},Y_{2},\ldots,Y_{2m-1},Y_{2m})=([Y_{1},Y_{2}],[Y_{3},Y_{4}],\ldots,[Y_{2m-1},Y_{2m}]).
\]
We therefore see that $J_{m}(\psi_{w}^{*2})$ can be degenerated to
a morphism $\phi_{m}$ which consists of a tower of commutator maps
with disjoint variables. By Theorem \ref{Thm:commutator is flat}
it follows that $\phi_{m}^{*2}$ is flat, with geometrically irreducible,
reduced fibers. By Proposition \ref{Prop:properties preserved under deformations}
this property is well behaved under deformations. Since $\phi_{m}^{*2}$
is a degeneration of $J_{m}(\psi_{w}^{*4})$, we deduce that $J_{m}(\psi_{w}^{*4})$
is flat and locally integral at $(0,\ldots,0)$. Note that all irreducible
components of $(J_{m}(\psi_{w}^{*4}))^{-1}(0)$ touch $(0,..,0)$,
so $(J_{m}(\psi_{w}^{*4}))^{-1}(0)$ is irreducible. By Corollary
\ref{cor: singularity properties through dim of jets} this implies
that $\psi_{w}^{*4}$ is (FRS) as required.

\section{\label{sec:flatness-and-(FRS) of word maps}Local behavior of word
maps and a lower bound on the log canonical threshold of their fibers}

In this section we study certain singularity properties of word maps,
in terms of the complexity class of $w$. The most natural and commonly
used way to measure the complexity of a word $w$ is its length $\ell(w)$
(e.g.~$\ell(xyx^{3}y^{-1})=6$). An alternative way is to consider
its degree $\mathrm{deg}(w)$, which we will soon define. We will
see that in order to capture singularity properties of word maps,
locally around $(e,\ldots,e)$, it is best to consider the degree
of word maps, while if we are interested in global behavior, it is
better to consider the length of the word. Before we state the main
theorems of this section, we define the degree and symbol of a word
and of a word map.

\subsection{\label{sec:The-degree-of a word}The degree of a word and of a word
map}
\begin{defn}
Let $G$ be a group. The \textit{lower central series} of $G$ is
the sequence 
\[
G=\Gamma_{1}(G)\supseteq\Gamma_{2}(G)\supseteq\dots\supseteq\Gamma_{k}(G)\supseteq...
\]
where $\Gamma_{k+1}(G):=[G,\Gamma_{k}(G)]$. 
\end{defn}

The associated graded abelian group $\mathrm{gr}(G)=\underset{k\geq1}{\bigoplus}\mathrm{gr}_{k}(G)$,
where $\mathrm{gr}_{k}(G)=\Gamma_{k}(G)/\Gamma_{k+1}(G)$, has the
structure of a graded Lie ring, where the Lie bracket is induced from
the group commutator (e.g.~\cite{MKS76}). In particular, iterating
this construction for $G=F_{r}$, we get an isomorphism $\mathrm{gr}(F_{r})\simeq\mathcal{L}_{r}^{\ints}$,
where $\mathcal{L}_{r}^{\ints}$ is the free Lie ring on $r$ elements.
We write $\mathrm{symb}:G\rightarrow\mathrm{gr}(G)$ for the natural
symbol map. 
\begin{defn}
\label{def: degree and symbol of a word}A word $w\in F_{r}$ is said
to be of \textit{degree $d$} if $\mathrm{symb}(w)\in\mathcal{L}_{r}^{\ints}$
is of degree $d$.
\end{defn}

\begin{rem}
\label{rem:relation between degree and length}For the $n$-th iterated
commutator word $w_{n}(x,y)=[\ldots[[x,y],y],\ldots,y]$ it is easy
to see that $\ell(w_{n})$ is exponential in the degree (which is
$n+1$). Set $\alpha(k):=\min\{\ell(w):\mathrm{deg}(w)=k\}$. It was
shown in \cite[Theorem 1.2]{MP10} that $\alpha(k)\geq k$, so $\ell(w)\geq\mathrm{deg}(w)$
for any $w\in F_{r}$. The asymptotic properties of $\alpha(k)$ have
been studied in \cite{ET15}, where it was shown that $\alpha(k)$
grows slower then $k^{1.45}$. 
\end{rem}

\subsubsection{The symbol of a word map}

Let $\underline{G}$ be a simple algebraic $K$-group, $w\in F_{r}$
a word, and let $\varphi_{w}:\underline{G}^{r}\rightarrow\underline{G}$
be the corresponding word map. The tangent cone $C_{e}\underline{G}$
at $e\in G$ is the Lie algebra $\mathfrak{g}$. Using the linearization
method (Corollary \ref{cor:(cf.---Linearization)}), we can choose
minimal $l\in\mathbb{N}$ such that $D_{(e,\ldots,e)}^{l}(\varphi_{w})$
is non-trivial, and then the word map $\varphi_{w}:\underline{G}^{r}\rightarrow\underline{G}$
degenerates to a non-trivial map $D_{(e,\ldots,e)}^{l}(\varphi_{w}):\mathfrak{g}^{r}\rightarrow\mathfrak{g}$. 
\begin{defn}
\label{def:symbol of a word map}For $l$ minimal as above, we call
$D_{(e,\ldots,e)}^{l}(\varphi_{w}):\mathfrak{g}^{r}\rightarrow\mathfrak{g}$
the \textit{symbol} of the map $\varphi_{w}:\underline{G}^{r}\rightarrow\underline{G}$
and denote it by $\widetilde{\varphi}_{w}:\mathfrak{g}^{r}\rightarrow\mathfrak{g}$. 
\end{defn}

\begin{prop}
\label{prop:symbol of a word map is a Lie algebra word map}The map
$\widetilde{\varphi}_{w}:\mathfrak{g}^{r}\rightarrow\mathfrak{g}$
is a homogeneous Lie algebra word map of degree $l+1$. 
\end{prop}

\begin{proof}
$\varphi_{w}$ is defined over a finitely generated field $K'\subset K$.
Embed $K'\hookrightarrow\complex$ and consider $\mathfrak{g}_{\complex}:=\mathfrak{g}\otimes_{K'}\complex$
and the word map $\varphi_{w}:\underline{G}(\complex)^{r}\rightarrow\underline{G}(\complex)$.
Recall that $\mathrm{exp}:\mathfrak{g}(\complex)\rightarrow\underline{G}(\complex)$
is a local diffeomorphism, whose inverse is $\mathrm{log}$. By a
generalization of the Baker-Campbell-Hausdorff formula we have, 
\[
\mathrm{log}(\mathrm{exp}(X_{1})\cdots\mathrm{exp}(X_{m}))=\sum_{n=1}^{\infty}\frac{(-1)^{n-1}}{n}\sum_{\sum_{j}r_{ij}>0}\frac{[X_{1}^{r_{11}}\cdots X_{m}^{r_{1m}}\cdots X_{1}^{r_{n1}}\cdots X_{m}^{r_{nm}}]}{\sum_{i=1}^{n}(r_{i1}+\ldots+r_{im})\cdot\prod_{i,j}^{n,m}r_{ij}!},
\]
where 
\[
[X_{1}^{r_{11}}\cdots X_{m}^{r_{1m}}\cdots X_{1}^{r_{n1}}\cdots X_{m}^{r_{nm}}]:=[\underbrace{X_{1},\cdots[X_{1}}_{r_{11}\text{ times}},\cdots[\underbrace{X_{m},\cdots[X_{m}}_{r_{1m}\text{ times}},\cdots[\underbrace{X_{1},\cdots[X_{1}}_{r_{n1}\text{ times}},\cdots[\underbrace{X_{m},\cdots,X_{m}}_{r_{nm}\text{ times}}]]\cdots].
\]
Now, given $\varphi_{w}(g_{1},\ldots,g_{r})=g_{i_{1}}^{\varepsilon_{1}}\cdot\ldots\cdot g_{i_{m}}^{\varepsilon_{m}}$
with $\varepsilon_{j}\in\{1,-1\}$, we have 
\[
\mathrm{log}\circ\varphi_{w}\circ(\mathrm{exp},\ldots,\mathrm{exp})(X_{1},\ldots,X_{r})=\mathrm{log}(\mathrm{exp}(\varepsilon_{1}X_{i_{1}})\cdot\ldots\cdot\mathrm{exp}(\varepsilon_{m}X_{i_{m}}))
\]
By the chain rule, and since $d_{(0,\ldots,0)}\mathrm{exp}=\mathrm{Id}$
and $d_{(e,\ldots,e)}\mathrm{log}=\mathrm{Id}$ we deduce that $D_{(e,\ldots,e)}^{l}\varphi_{w}$
is a homogeneous Lie algebra word map of degree $l+1$.
\end{proof}
\begin{prop}
\label{Word maps is generating and bounded degree}Let $\varphi_{w}:\underline{G}^{r}\rightarrow\underline{G}$
be a word map. Then $\widetilde{\varphi}_{w}$ is a Lie algebra word
map of degree $\leq\ell(w)$. 
\end{prop}

\begin{proof}
We may assume that $\underline{G}$ is simply connected and thus we
have an embedding $\mathrm{SL}_{2}\hookrightarrow\underline{G}$.
By Proposition \ref{prop:symbol of a word map is a Lie algebra word map},
it is enough to prove the claim for $\mathrm{SL}_{2}$, which holds
by a direct calculation. 
\end{proof}
Note that we have two notions of symbol: $\mathrm{symb}(w)$ as the
image of $w$ in $\mathrm{gr}(F_{r})$ (see Definition \ref{def: degree and symbol of a word})
and $\widetilde{\varphi}_{w}=D_{(e,\ldots,e)}^{l}(\varphi_{w})$,
the symbol of the word map. If $\mathrm{rk}(\underline{G})>\ell(w)$,
then it turns out that these two notions agree, i.e.~$\varphi_{\mathrm{symb}(w)}=\widetilde{\varphi}_{w}$.
\begin{prop}
\label{prop:symbol of word equal to sym of word map}Let $w\in F_{r}$.
Then for any simple algebraic group $\underline{G}$ with $\mathrm{rk}(\underline{G})>\ell(w)$
its symbol $\widetilde{\varphi}_{w}:\mathfrak{g}^{r}\rightarrow\mathfrak{g}$
is identical to $\varphi_{\mathrm{symb}(w)}$. In particular, in this
case $w$ is of degree $d$ if and only if $\varphi_{w}$ has symbol
of degree $d$. 
\end{prop}

\begin{proof}
Assume that $w$ has degree $d$. By the identification $\mathrm{gr}(F_{r})\simeq\mathcal{L}_{r}^{\ints}$
it follows that $w$ can be written as $w=w_{1}\cdot w_{0}$ with
$w_{0}\in\Gamma_{d+1}(F_{r})$ and $w_{1}$ a product of iterated
commutators of the form $[\ldots[[g_{i_{1}},g_{i_{2}}],g_{i_{3}}],\ldots,g_{i_{d}}]$.
By Proposition \ref{prop:Lie algebra word maps are generating}, Corollary
\ref{cor: commutator relations on simple Lie algebras}, and by the
Baker-Campbell-Hausdorff formula, it follows that $\widetilde{\varphi}_{w_{1}}=\varphi_{\mathrm{symb}(w_{1})}$.
Since $\widetilde{\varphi}_{w_{0}}$ is of degree $>d$ it follows
that 
\[
\widetilde{\varphi}_{w}=\widetilde{\varphi}_{w_{1}}=\varphi_{\mathrm{symb}(w_{1})}=\varphi_{\mathrm{symb}(w)}.
\]
\end{proof}

\subsection{\label{subsec:The-main-results for group word maps}Algebro-geometric
results on word maps}

We now state the main results of this section. In Theorem \ref{thm: (FRS) at (e,...,e)}
we discuss local behavior of fibers of word maps at $(e,\ldots,e)$
in terms of the degree of the words. In Theorem \ref{thm:jet flatness of word maps},
we bound from below the log canonical threshold of the fibers of word
maps $\varphi_{w}:\mathrm{SL}_{n}^{r}\rightarrow\mathrm{SL}_{n}$
in terms of $\ell(w)$. 
\begin{thm}
\label{thm: (FRS) at (e,...,e)}Let $d\in\nats$, let $\underline{G}$
be a connected semisimple $K$-group, and let $\{\varphi_{w_{i}}:\underline{G}^{r_{i}}\rightarrow\underline{G}\}_{i=1}^{m}$
be a collection of word maps with symbols $\widetilde{\varphi}_{w_{i}}$
of degree at most $d$. Then there exists $0<C<10^{6}$, such that
for any $m\geq C\cdot d^{6}$ the morphism $\varphi_{w_{1}}*\ldots*\varphi_{w_{m}}$
is (FRS) at $(e,\ldots,e)$. 
\end{thm}

\begin{proof}
By Corollary \ref{cor:(cf.---Linearization)} and Proposition \ref{prop:symbol of a word map is a Lie algebra word map},
it is enough to prove the theorem for homogeneous Lie algebra word
maps. Then the claim follows from Theorem \ref{thm: main thm Lie algebra word maps}. 
\end{proof}
\begin{thm}
\label{thm:jet flatness of word maps}Let $w\in F_{r}$. Then $\varphi_{w}:\mathrm{SL}_{n}^{r}\rightarrow\mathrm{SL}_{n}$
is $\frac{1}{10^{17}\ell(w)^{9}}$-jet flat. In particular, $\mathrm{lct}(\mathrm{SL}_{n}^{r},\varphi_{w}^{-1}(g))\geq\frac{\mathrm{dim}\mathrm{SL}_{n}}{10^{17}\ell(w)^{9}}$
for every $g\in\mathrm{SL}_{n}$. 
\end{thm}

We first prove Theorem \ref{thm:jet flatness of word maps} for $\mathrm{SL}_{2n}$.
Let $\varphi_{w}:\mathrm{SL}_{2n}^{r}\rightarrow\mathrm{SL}_{2n}$
be a word map. Recall that $\mathrm{SL}_{2n}$ contains two large
commutative unipotent subgroups; the $n\times n$ right upper unipotent
block $\mathrm{Mat}_{n}^{+}:=\left\{ \left(\begin{array}{cc}
I_{n} & A\\
0 & I_{n}
\end{array}\right):A\in\mathrm{Mat}_{n\times n}\right\} $, and similarly $\mathrm{Mat}_{n}^{-}:=\left\{ \left(\begin{array}{cc}
I_{n} & 0\\
B & I_{n}
\end{array}\right):B\in\mathrm{Mat}_{n\times n}\right\} $. Moreover, $\mathrm{Mat}_{n}^{\pm}$ boundedly generates $\mathrm{SL}_{2n}$: 
\begin{lem}
\label{lem:large generating unipotent subgroups}Let $\varepsilon_{i}=(-1)^{i+1}$.
The multiplication map $\psi:\prod\limits _{i=1}^{13}\mathrm{Mat}_{n}^{\varepsilon_{i}}\rightarrow\mathrm{SL}_{2n}$
is a surjective morphism. 
\end{lem}

\begin{proof}
This follows from the proof of \cite[Lemma 2.1]{Kas07}. We show that
any $g:=\left(\begin{array}{cc}
A_{11} & A_{12}\\
A_{21} & A_{22}
\end{array}\right)$ is in $\stackrel[i=1]{13}{\prod}\mathrm{Mat}_{n}^{\varepsilon_{i}}$.
Note that the row rank of $\left(\begin{array}{c}
A_{12}\\
A_{22}
\end{array}\right)$ is $n$. This means we can find $B_{1}\in\mathrm{Mat}_{n}^{+}$ such
that $B_{1}g:=\left(\begin{array}{cc}
A'_{11} & A'_{12}\\
A_{21} & A_{22}
\end{array}\right)$, where $A'_{12}=A_{12}+(B_{1})_{12}A_{22}$ is invertible. Taking
$B_{2}=\left(\begin{array}{cc}
I_{n} & 0\\
(I_{n}-A_{22})\cdot(A'_{12})^{-1} & I_{n}
\end{array}\right)\in\mathrm{Mat}_{n}^{-}$ we get $B_{2}B_{1}g=\left(\begin{array}{cc}
A'_{11} & A'_{12}\\
A'_{21} & I_{n}
\end{array}\right)$. Now, we can find $B_{3}\in\mathrm{Mat}_{n}^{+}$ and $B_{4}\in\mathrm{Mat}_{n}^{-}$
such that $B_{3}B_{2}B_{1}gB_{4}:=\left(\begin{array}{cc}
A''_{11} & 0\\
0 & I_{n}
\end{array}\right)$ where $A''_{11}\in\mathrm{SL}_{n}$. Since any element in $\mathrm{SL}_{n}(\overline{K})$
is a commutator, we can write $A''_{11}=[B,C]$ for $B,C\in\mathrm{SL}_{n}(\overline{K})$.
By \cite[p.314]{Kas07} we see that $\left(\begin{array}{cc}
[B,C] & 0\\
0 & I_{n}
\end{array}\right)\in\psi(\prod\limits _{i=1}^{9}\mathrm{Mat}_{n}^{\varepsilon_{i}})$, so we are done.
\end{proof}
Denote $\Phi_{w}:=\varphi_{w}\circ(\psi*\psi,\dots,\psi*\psi):(\prod\limits _{i=1}^{13}\mathrm{Mat}_{n}^{\varepsilon_{i}}\times\prod\limits _{i=1}^{13}\mathrm{Mat}_{n}^{\varepsilon_{i}})^{r}\rightarrow\mathrm{SL}_{2n}$.
\begin{lem}
\label{lem:reduction to matrix word maps}For any $\varepsilon>0$,
if $\Phi_{w}$ is $\varepsilon$-jet flat, then $\varphi_{w}$ is
$\varepsilon$-jet flat. 
\end{lem}

\begin{proof}
Assume that $\Phi_{w}$ is $\varepsilon$-jet flat. It is enough to
show that $J_{m}(\varphi_{w})$ is $\varepsilon$-flat over $\mathrm{SL}_{2n}\subset J_{m}(\mathrm{SL}_{2n})$.
Since $\psi$ is surjective, $J_{m}(\psi)$ is dominant, thus $J_{m}(\psi*\psi)$
is surjective for each $m\in\nats$. We get that for every $g\in\mathrm{SL}_{2n}$,
\[
\mathrm{dim}J_{m}(\mathrm{SL}_{2n}^{r})-\mathrm{dim}J_{m}\left((\varphi_{w})^{-1}(g)\right)\geq\mathrm{dim}J_{m}(\prod_{i=1}^{13}\mathrm{Mat}_{n}^{\varepsilon_{i}})^{2r}-\mathrm{dim}J_{m}\left((\Phi_{w})^{-1}(g)\right)\geq\varepsilon\mathrm{dim}J_{m}(\mathrm{SL}_{2n})
\]
so $\varphi_{w}$ is $\varepsilon$-jet flat. 
\end{proof}
\begin{prop}
\label{prop:Proposition 7.13}The map $\Phi_{w}$ is of the form 
\[
\Phi_{w}(A_{1},\ldots,A_{26r})=\left(\begin{array}{cc}
P_{11}(A_{1},\ldots,A_{26r}) & P_{12}(A_{1},\ldots,A_{26r})\\
P_{21}(A_{1},\ldots,A_{26r}) & P_{22}(A_{1},\ldots,A_{26r})
\end{array}\right),
\]
where $P_{ij}$ are matrix word maps of degree at most $26\ell(w)$. 
\end{prop}

\begin{proof}
Notice that any word map $\varphi_{w}:\underline{G}^{r}\rightarrow\underline{G}$
is of the form $\varphi_{w}(g_{1},\ldots,g_{r})=g_{i_{1}}^{\varepsilon_{1}}\ldots g_{i_{\ell(w)}}^{\varepsilon_{\ell(w)}}$,
with $\varepsilon_{i_{k}}\in\{\pm1\}$, since the inverse of $\left(\begin{array}{cc}
I & A\\
0 & I
\end{array}\right)$ is $\left(\begin{array}{cc}
I & -A\\
0 & I
\end{array}\right)$ we get that $\Phi_{w}$ has the above shape. 
\end{proof}
\begin{cor}
\label{cor:word maps on SL2n are jet flat}The map $\varphi_{w}:\mathrm{SL}_{2n}^{r}\rightarrow\mathrm{SL}_{2n}$
is $\frac{1}{2\cdot10^{6}\ell(w)^{3}}$-jet flat for any $w\in F_{r}$.
\end{cor}

\begin{proof}
$\Phi_{w}^{-1}(e)\subseteq P_{12}(A_{1},\ldots,A_{26r})^{-1}(0)$
so by Item (1) of Fact \ref{fact:properties of log cannonical threshold}
we get the following: 
\[
\mathrm{lct}((\prod\limits _{i=1}^{13}\mathrm{Mat}_{n}^{\varepsilon_{i}})^{2r},\Phi_{w}^{-1}(e))\geq\mathrm{lct(}(\prod_{i=1}^{13}\mathrm{Mat}_{n}^{\varepsilon_{i}})^{2r},P_{12}^{-1}(0)).
\]
By Proposition \ref{prop:Proposition 7.13}, $P_{12}$ is a matrix
word map of degree at most $26\ell(w)$ so by Theorem \ref{thm:lower bounds on epsilon jet flatness}
\[
\mathrm{lct(}(\prod_{i=1}^{13}\mathrm{Mat}_{n}^{\varepsilon_{i}})^{2r},P_{12}^{-1}(0))\geq\frac{\mathrm{dim}\mathrm{Mat}_{n}}{25\left(26\ell(w)\right)^{3}}\geq\frac{\mathrm{dim}\mathrm{SL}_{2n}}{2\cdot10^{6}\ell(w)^{3}}.
\]
Thus, $\Phi_{w}$ is $\frac{1}{2\cdot10^{6}\ell(w)^{3}}$-jet flat,
and by Lemma \ref{lem:reduction to matrix word maps} $\varphi_{w}$
is $\frac{1}{2\cdot10^{6}\ell(w)^{3}}$-jet flat. 
\end{proof}
\begin{proof}[Proof of Theorem \ref{thm:jet flatness of word maps}]
We may assume that $r<\ell(w)$. We have proved the theorem for $\mathrm{SL}_{2n}$.
Consider $\varphi_{w,2n-1}:\mathrm{SL}_{2n-1}^{r}\rightarrow\mathrm{SL}_{2n-1}$
and note that $\varphi_{w,2n-1}^{-1}(g)\subseteq\varphi_{w,2n}^{-1}(g)$
for any $g\in\mathrm{SL}_{2n-1}$. In particular, by Corollary \ref{cor:word maps on SL2n are jet flat},
if $n\geq10^{7}\ell(w)^{3}r$ then
\begin{align*}
\sup_{m\geq0}\frac{\dim J_{m}\left(\varphi_{w,2n-1}^{-1}(g)\right)}{m+1} & \leq\sup_{m\geq0}\frac{\dim J_{m}\left(\varphi_{w,2n}^{-1}(g)\right)}{m+1}\leq\mathrm{dim}\mathrm{SL}_{2n}\left(r-\frac{1}{2\cdot10^{6}\ell(w)^{3}}\right)\\
 & \leq\mathrm{dim}\mathrm{SL}_{2n-1}\left(r-\frac{1}{2\cdot10^{6}\ell(w)^{3}}\right)+4nr\\
 & \leq\mathrm{dim}\mathrm{SL}_{2n-1}\left(r-\frac{1}{10^{6}\ell(w)^{3}}\right).
\end{align*}

Now assume that $n<10^{7}\ell(w)^{3}r<10^{7}\ell(w)^{4}$. By \cite[Lemma 2.1]{Kas07}
(and as explained in Lemma \ref{cor:word maps on SL2n are jet flat}),
any $g\in\mathrm{SL}_{2n-1}$ can be written as a product of at most
$20$ matrices of the form $\left(\begin{array}{cc}
I_{p} & 0\\
B & I_{q}
\end{array}\right)$ or $\left(\begin{array}{cc}
I_{p} & A\\
0 & I_{q}
\end{array}\right)$, where $A\in\mathrm{Mat}_{p\times q}$ and $B\in\mathrm{Mat}_{q\times p}$.
Similarly as in the case of $\mathrm{SL}_{2n}$, and by Lemma \ref{lem:reduction to matrix word maps},
this yields a map $\Phi_{w}:\mathbb{A}^{N}\rightarrow\mathrm{SL}_{2n-1}$,
such that
\[
\mathrm{lct}(\mathrm{SL}_{2n-1}^{r},\varphi_{w,2n-1}^{-1}(g))\geq\mathrm{lct}(\mathbb{A}^{N},\Phi_{w}^{-1}(g)).
\]
Since $\Phi_{w}$ is a polynomial map of degree at most $40\ell(w)$,
we have by Item (2) of Fact \ref{fact:properties of log cannonical threshold},
\[
\mathrm{lct}(\mathrm{SL}_{2n-1}^{r},\varphi_{w,2n-1}^{-1}(g))\geq\frac{1}{40\ell(w)}\geq\frac{\mathrm{dim}\mathrm{SL}_{2n-1}}{10^{17}\ell(w)^{9}}.
\]
\end{proof}

\section{\label{sec:The-(FRS)-property and counting points}Number theoretic
interpretation of the flatness, $\varepsilon$-jet flatness and (FRS)
properties}

Let $\varphi:X\to Y$ be a dominant morphism between smooth, geometrically
irreducible $\rats$-varieties. In this section, we provide number
theoretic interpretations to the flatness, (FRS) and $\varepsilon$-jet
flatness properties of $\varphi$. Flatness (and $\varepsilon$-flatness)
can be interpreted, via the Lang-Weil bounds, as certain uniform bounds
on the fibers of $\varphi:X(\mathbb{F}_{q})\rightarrow Y(\mathbb{F}_{q})$
over all finite fields. The (FRS) and $\varepsilon$-jet flatness
properties can be interpreted by certain uniform bounds on the fibers
over finite rings of the form $\ints/p^{k}\ints$. For the $\varepsilon$-jet
flatness property, the idea is to apply the Lang-Weil bounds to the
$k$-th jet maps $J_{k}(\varphi)$, and to use results in motivic
integration to deal with the fact that the family of morphisms $\{J_{k}(\varphi)\}_{k\in\nats}$
is of unbounded complexity.

\subsection{The Lang-Weil bounds and a number theoretic characterization of flatness}
\begin{defn}[{cf.~\cite[Definition 7.7]{GH19}}]
Let $K'$ be any field.
\begin{enumerate}
\item An affine $K'$-scheme $X$ has \textit{complexity at most $M$},
if $X$ has a closed embedding $\psi_{X}:X\hookrightarrow\mathbb{A}_{K'}^{n}$
with $X=\spec(K'[x_{1},\dots,x_{n}]/(f_{1},\dots,f_{k}))$, where
$n,k$ and $\max\limits _{i}\{\mathrm{deg}(f_{i})\}$ are at most
$M$. 
\item A morphism $\varphi:X\rightarrow Y$ of affine $K'$-schemes is said
to be \textit{of complexity at most }$M$, if there exist embeddings
$\psi_{X}:X\hookrightarrow\mathbb{A}_{K'}^{n_{1}}$ and $\psi_{Y}:Y\hookrightarrow\mathbb{A}_{K'}^{n_{2}}$
of complexity at most $M$, such that the polynomial map $\widetilde{\varphi}:\mathbb{A}_{K'}^{n_{1}}\rightarrow\mathbb{A}_{K'}^{n_{2}}$
induced from $\varphi$ is of the form $\widetilde{\varphi}=(\widetilde{\varphi}_{1},\dots,\widetilde{\varphi}_{n_{2}})$,
where the degrees of $\{\widetilde{\varphi}_{i}\}_{i=1}^{n_{2}}$
are at most $M$.
\end{enumerate}
This definition can be extended to non-affine schemes, and morphisms
between them (see e.g.~\cite[Definition 7.7]{GH19}). We denote by
$\mathcal{C}_{M}$ the class of all $K'$-schemes and $K'$-morphisms
of complexity at most $M$. Since $K'$ is usually clear from the
context, we omit it from the notation. 
\end{defn}

\begin{defn}
~\label{def:constant in Lang weil}Let $Z$ be a finite type $\mathbb{F}_{q}$-scheme,
and let $\varphi:X\to Y$ be a morphism between finite type $\ints$-schemes. 
\begin{enumerate}
\item We denote by $C_{Z}$ the number of top-dimensional geometrically
irreducible components of $Z$ which are defined over $\mathbb{F}_{q}$. 
\item We write $C_{X,q}$:=$C_{X_{\mathbb{F}_{q}}}$ and $C_{\varphi,q,y}:=C_{(X_{\mathbb{F}_{q}})_{y,\varphi}}$
for any $y\in Y(\mathbb{F}_{q})$.
\end{enumerate}
\end{defn}

\begin{thm}[The Lang-Weil estimates \cite{LW54}]
\label{thm:Lang-Weil}For any $M\in\nats$, there exists $C(M)>0$,
such that for each finite type $\mathbb{F}_{q}$-scheme $X\in\mathcal{C}_{M}$
we have 
\[
\left|\frac{|X(\mathbb{F}_{q})|}{q^{\mathrm{dim}X}}-C_{X}\right|<C(M)q^{-\frac{1}{2}}.
\]
\end{thm}

For a set $S$ of primes, we set $\mathcal{P}_{S}:=\{q\in\nats:q=p^{r}\text{ for a prime }p\notin S\text{ and }r\in\nats\}$.
The following is a relative version of the Lang-Weil bounds which
will be very useful later on. 
\begin{thm}
\label{thm: flatness and counting points}Let $\varphi:X\to Y$ be
a dominant morphism between finite type $\ints$-schemes $X$ and
$Y$ such that $X_{\rats}$ and $Y_{\rats}$ are smooth and geometrically
irreducible. Then the following are equivalent: 
\begin{enumerate}
\item There exist $C_{1}>0$ and a set $S$ primes, such that for every
$p\notin S$, and every $y\in Y(\mathbb{F}_{p})$ 
\[
\left|\frac{|\varphi^{-1}(y)|}{p^{(\mathrm{dim}X_{\mathbb{Q}}-\mathrm{dim}Y_{\mathbb{Q}})}}-C_{\varphi,p,y}\right|<C_{1}p^{-\frac{1}{2}}.
\]
\item There exist $C_{2}>0$ and a set $S$ of primes, such that for every
$q\in\mathcal{P}_{S}$, and every $y\in Y(\mathbb{F}_{q})$ 
\[
\left|\frac{|\varphi^{-1}(y)|}{q^{(\mathrm{dim}X_{\mathbb{Q}}-\mathrm{dim}Y_{\mathbb{Q}})}}-C_{\varphi,q,y}\right|<C_{2}q^{-\frac{1}{2}}.
\]
\item There exist $C_{3}>0$ and a set $S$ of primes, such that for every
$q\in\mathcal{P}_{S}$, it holds that
\[
\left|\frac{\varphi_{*}(\mu_{X(\mathbb{F}_{q})})}{\mu_{Y(\mathbb{F}_{q})}}-\widetilde{C}_{\varphi,q}\right|_{L^{\infty}}<C_{3}q^{-\frac{1}{2}},
\]
where $\widetilde{C}_{\varphi,q}:Y(\mathbb{F}_{q})\rightarrow\nats$
is defined by $\widetilde{C}_{\varphi,q}(y):=C_{\varphi,q,y}$. 
\item $\varphi_{\rats}:X_{\rats}\to Y_{\rats}$ is flat. 
\end{enumerate}
\end{thm}

\begin{proof}
We would like to show $(1)\Rightarrow(4)\Rightarrow(2)\Rightarrow(1)$
and $(3)\Leftrightarrow(2)$. Clearly $(2)\Rightarrow(1)$.

$(3)\Leftrightarrow(2)$. We prove $(2)\Rightarrow(3)$, the proof
of $(3)\Rightarrow(2)$ is similar. Choose $S$ large enough such
that $\mathrm{dim}X_{\rats}=\mathrm{dim}X_{\mathbb{F}_{q}}$, $\mathrm{dim}Y_{\rats}=\mathrm{dim}Y_{\mathbb{F}_{q}}$
and such that $X_{\mathbb{F}_{q}}$ and $Y_{\mathbb{F}_{q}}$ are
geometrically irreducible for every $q\in\mathcal{P}_{S}$. By the
triangle inequality it is enough to show that 
\begin{equation}
\left|\frac{\left|\varphi^{-1}(y)\right|}{q^{(\mathrm{dim}X_{\mathbb{Q}}-\mathrm{dim}Y_{\mathbb{Q}})}}-\frac{\varphi_{*}(\mu_{X(\mathbb{F}_{q})})}{\mu_{Y(\mathbb{F}_{q})}}(y)\right|<C_{4}q^{-\frac{1}{2}},\label{eq:(8.1)}
\end{equation}
for any $y\in Y(\mathbb{F}_{q})$. The left hand side of (\ref{eq:(8.1)})
can be written as 
\[
|\varphi^{-1}(y)|\left|\frac{1}{q^{(\mathrm{dim}X_{\mathbb{Q}}-\mathrm{dim}Y_{\mathbb{Q}})}}-\frac{\left|Y(\mathbb{F}_{q})\right|}{\left|X(\mathbb{F}_{q})\right|}\right|<\frac{C_{5}\left|\varphi^{-1}(y)\right|}{q^{(\mathrm{dim}X_{\mathbb{Q}}-\mathrm{dim}Y_{\mathbb{Q}})}}q^{-\frac{1}{2}}<C_{5}q^{-\frac{1}{2}}(C_{\varphi,q,y}+C_{2}q^{-\frac{1}{2}})<C_{4}q^{-\frac{1}{2}},
\]
where the first inequality follows from our assumption on $S$ and
the Lang-Weil bounds (Theorem \ref{thm:Lang-Weil}), and the last
inequality follows from the fact that the collection $\{C_{\varphi,q,y}\}_{y,q}$
is uniformly bounded (see e.g.~\cite[Corollary 9.7.9]{Gro66}).

$(4)\Rightarrow(2)$. Assume that $\varphi_{\rats}$ is flat. Then
there exists a set $S$ of primes such that for any $q\in\mathcal{P}_{S}$,
$\varphi_{\mathbb{F}_{q}}$ is flat, $\mathrm{dim}X_{\rats}=\mathrm{dim}X_{\mathbb{F}_{q}}$
and $\mathrm{dim}Y_{\rats}=\mathrm{dim}Y_{\mathbb{F}_{q}}$. Since
all of the fibers of $\varphi_{\mathbb{F}_{q}}$ belong to $\mathcal{C}_{M}$,
for some $M\in\nats$, the Lang-Weil estimates (Theorem \ref{thm:Lang-Weil})
imply $(2)$.

$(1)\Rightarrow(4)$. Since $X_{\rats}$ and $Y_{\rats}$ are smooth,
by enlarging $S$, we may assume that $X_{\mathbb{F}_{p}}$ and $Y_{\mathbb{F}_{p}}$
are smooth for every $p\notin S$. Using miracle flatness it is enough
to show that $\mathrm{dim}X_{\varphi(x)}=\mathrm{dim}X_{\rats}-\mathrm{dim}Y_{\rats}$
for every $x\in X(\overline{\rats})$.

Given $x\in X(\overline{\rats})$, there exists a finite extension
$K/\rats$ such that $x\in X(K)$ and such that the geometrically
irreducible components of the fiber $X_{\varphi(x)}$ are defined
over $K$. Note that $\varphi(x)\in Y(S_{1}^{-1}\mathcal{O}_{K})$,
where $\mathcal{O}_{K}$ is the ring of integers of $K$, and $S_{1}\subset\spec(\mathcal{O}_{K})$
is a finite set of primes. Write $Z=(X_{S_{1}^{-1}\mathcal{O}_{K}})_{\varphi(x)}$.
Since $\spec(K)$ is the generic point of $\spec(S_{1}^{-1}\mathcal{O}_{K})$,
by \cite[Proposition 9.7.8]{Gro66} (see also \cite[Theorem 2.3.5]{AA18}),
for all but finitely many primes $[\mathfrak{p}]\in\mathrm{Spec}(S_{1}^{-1}\mathcal{O}_{K})$
it holds that all irreducible components of $Z_{[\mathfrak{p}]}$
are geometrically irreducible.

Further note that $Z_{[\mathfrak{p}]}\simeq X_{\varphi(\overline{x})}$
where $\overline{x}$ is the image of $x$ under the natural map $X(S_{1}^{-1}\mathcal{O}_{K})\to X(S_{1}^{-1}\mathcal{O}_{K}/\mathfrak{p})$.
Now, using Chebotarev's density theorem, there exist infinitely many
primes $[\mathfrak{p}]\in\mathrm{Spec}(S_{1}^{-1}\mathcal{O}_{K})$
such that $S_{1}^{-1}\mathcal{O}_{K}/\mathfrak{p}\simeq\mathbb{F}_{p}$,
for (infinitely many) prime numbers $p$. Since by our construction
$C_{\varphi,p,\varphi(\overline{x})}\neq0$, using our assumption
and the Lang-Weil estimates we conclude $\mathrm{dim}Z_{[\mathfrak{p}]}=\mathrm{dim}X_{\varphi(\overline{x})}=\mathrm{dim}X_{\rats}-\mathrm{dim}Y_{\rats}$
for infinitely many primes $[\mathfrak{p}]\in\mathrm{Spec}(S_{1}^{-1}\mathcal{O}_{K})$.
This implies $\mathrm{dim}X_{\varphi(x)}=\mathrm{dim}X_{\rats}-\mathrm{dim}Y_{\rats}$. 
\end{proof}

\subsection{Number theoretic and analytic interpretation of the (FRS) property}

In \cite{AA16}, Aizenbud and Avni gave an analytic criterion for
the (FRS) property: 
\begin{thm}[{cf.~\cite[Theorem 3.4]{AA16}}]
\label{thm:analytic criterion of the (FRS) property}Let $\varphi:X\rightarrow Y$
be a map between smooth $\rats$-varieties. Then the following are
equivalent: 
\begin{enumerate}
\item $\varphi$ is (FRS). 
\item For any non-Archimedean local field $\rats\subseteq F$ and any smooth,
compactly supported measure $\mu$ on $X(F)$, the measure $\varphi_{*}(\mu)$
has continuous density. 
\item For any $x\in X(\overline{\rats})$ and any finite extension $K/\rats$
with $x\in X(K)$, there exists a non-Archimedean local field $F\supseteq K$
and a non-negative smooth, compactly supported measure $\mu$ on $X(F)$
that does not vanish at $x$ such that $\varphi_{*}(\mu)$ has continuous
density. 
\end{enumerate}
\end{thm}

In \cite{AA18}, a number theoretic characterization was given for
the property of having rational singularities in the case of local
complete intersection schemes. A strengthening of this characterization
was later given in \cite{Gla19}. 
\begin{defn}
\label{def:expected number of points}Let $X$ be a finite type $\ints$-scheme.
For a finite ring $A$, set $h_{X}(A):=\frac{\left|X(A)\right|}{\left|A\right|^{\mathrm{dim}X_{\rats}}}$
if $X_{\rats}$ is non-empty, and $h_{X}(A)=0$ otherwise. 
\end{defn}

\begin{thm}[{see \cite[Theorem A]{AA18} and \cite[Theorem 1.4]{Gla19}}]
\label{thm:number theoretic criterion for rational singularities}Let
$X$ be a finite type $\mathbb{Z}$-scheme such that $X_{\mathbb{Q}}$
is equidimensional and a local complete intersection. Then the following
are equivalent: 
\begin{enumerate}
\item For any $k$, $\lim\limits _{p\rightarrow\infty}h_{X}(\mathbb{Z}/p^{k}\mathbb{Z})=1$. 
\item $X_{\mathbb{\overline{Q}}}$ is irreducible, and there exists a finite
set $S$ of primes, such that for every prime $p\notin S$ the sequence
$k\mapsto h_{X}(\mathbb{Z}/p^{k}\mathbb{Z})$ is bounded. 
\item There exists $C>0$ such that $\left|h_{X}(\mathbb{Z}/p^{k}\mathbb{Z})-1\right|<Cp^{-1/2}$
for every prime $p$ and every $k\in\mathbb{N}$. 
\item $X_{\mathbb{\overline{Q}}}$ is reduced, irreducible and has rational
singularities. 
\end{enumerate}
\end{thm}

\begin{rem}
\label{rem:family version for rational singularities}Notice that
an (FRS) morphism $\varphi:X\rightarrow Y$ between geometrically
irreducible, smooth $\rats$-varieties is in particular flat, so its
fibers are local complete intersections. The condition of having rational
singularities puts each individual $\ints$-fiber of $\varphi$ in
the framework of Theorem \ref{thm:number theoretic criterion for rational singularities}.
Two difficulties now arise when trying to provide a version of Theorem
\ref{thm:number theoretic criterion for rational singularities} for
families. Firstly, it is not enough to consider $\ints$-fibers (e.g.,
we might have $Y(\ints)=\varnothing$). Secondly, Theorem \ref{thm:number theoretic criterion for rational singularities}
gives us a different bound $C'(y)$ for each fiber $X_{y}$ where
$y\in Y(\ints)$, and it is a priori not clear why it should be uniform.
These difficulties are dealt with in an upcoming work \cite{CGH}.
\end{rem}

\subsection{\label{subsec:The-Denef-Pas-language}The Denef-Pas language and
motivic functions }

The \textsl{Presburger language}, denoted 
\[
\mathcal{L}_{\mathrm{Pres}}=(+,-,\leq,\{\equiv_{\mathrm{mod}~n}\}_{n>0},0,1)
\]
is a first order language of ordered abelian groups along with constants
$0,1$ and a family of relations $\{\equiv_{\mathrm{mod}~n}\}_{n>0}$
of congruences modulo $n$. The \textsl{Denef-Pas language, }denoted
\[
\Ldp=(\mathcal{L}_{\mathrm{Val}},\mathcal{L}_{\mathrm{Res}},\mathcal{L}_{\mathrm{Pres}},\mathrm{\val},\mathrm{\ac})
\]
is a first order three sorted language, where 
\begin{enumerate}
\item The valued field sort $\VF$ is endowed with the language of rings
$\mathcal{L}_{\mathrm{Val}}=(+,-,\cdot,0,1)$. 
\item The residue field sort $\RF$ is endowed with the language of rings
$\mathcal{L}_{\mathrm{Res}}=(+,-,\cdot,0,1)$. 
\item The value group sort $\VG$ (which we just call $\ints$), is endowed
with the Presburger language $\mathcal{L}_{\mathrm{Pres}}=(+,-,\leq,\{\equiv_{\mathrm{mod}~n}\}_{n>0},0,1)$. 
\item $\val:\VF\backslash\{0\}\rightarrow\ints$ and $\ac:\VF\rightarrow\RF$
are two function symbols. 
\end{enumerate}
Let $\mathrm{Loc}$ be the collection of all non-Archimedean local
fields. For $F\in\mathrm{Loc}$, we denote by $k_{F}$ its residue
field, and set $q_{F}:=\left|k_{F}\right|$. We use the notation $\mathrm{Loc}_{\gg}$
for the collection of $F\in\mathrm{Loc}$ with large enough residual
characteristic. By choosing a uniformizer $\pi$ of $\mathcal{O}_{F}$,
we can interpret $\val$ and $\ac$ as the usual valuation map $\val:F^{\times}\rightarrow\ints$
and angular component map $\ac:F\rightarrow k_{F}$. Hence, we can
interpret any formula $\phi$ in $\Ldp$ with $n_{1}$ free $\VF$-variables,
$n_{2}$ free $\RF$-variables and $n_{3}$ free $\ints$-variables,
in $F\in\mathrm{Loc}$, yielding a subset $\phi(F)\subseteq F^{n_{1}}\times k_{F}^{n_{2}}\times\mathbb{Z}^{n_{3}}$.
A collection $X=(X_{F})_{F\in\mathrm{Loc}_{\gg}}$ with $X_{F}=\phi(F)$
is called an \textsl{$\mathcal{L}_{\mathrm{DP}}$-definable set}.
Given definable sets $X$ and $Y$, an \textsl{$\mathcal{L}_{\mathrm{DP}}$-}\textit{definable
function} is a collection $f=(f_{F}:X_{F}\rightarrow Y_{F})_{F\in\mathrm{Loc}}$
of functions whose collection of graphs $(\Gamma_{f_{F}})_{F\in\mathrm{Loc}}$
is a definable set.

The class of definable functions is not preserved under integration,
which motivates us to consider a larger class of functions. 
\begin{defn}[{See \cite[definitions 2.3-2.6]{CGH16}}]
\label{def:motivic function}Let $X$ be a definable set. A collection
$h=(h_{F})_{F\in\mathrm{Loc}}$ of functions $h_{F}:X_{F}\rightarrow\mathbb{R}$
is called a \textit{motivic }function, if for $F\in\mathrm{Loc}_{\gg}$
it can be written as: 
\[
h_{F}(x)=\sum\limits _{i=1}^{N_{1}}|Y_{i,F,x}|q_{F}^{\alpha_{i,F}(x)}\left(\prod_{j=1}^{N_{2}}\beta_{ij,F}(x)\right)\left(\prod_{j=1}^{N_{3}}\frac{1}{1-q_{F}^{a_{ij}}}\right),
\]
where: 
\begin{itemize}
\item $N_{1},N_{2},N_{3}$ and $a_{il}$ are non-zero integers. 
\item $\alpha_{i}:X\rightarrow\mathbb{Z}$ and $\beta_{ij}:X\rightarrow\mathbb{Z}$
are definable functions. 
\item $Y_{i,F,x}=\{\xi\in k_{F}^{r_{i}}:(x,\xi)\in Y_{i,F}\}$ is the fiber
over $x$ of a definable set $Y_{i}\subseteq X\times\mathrm{RF}^{r_{i}}$
with $r_{i}\in\nats$. 
\end{itemize}
We denote by $\mathcal{C}(X)$ the ring of motivic functions on a
definable set $X$. Unlike the class of definable functions, this
class is preserved under integration: 
\end{defn}

\begin{thm}[{\cite[Theorem 4.3.1]{CGH14}}]
\label{integration of motivic}Let $X$ be an $\Ldp$-definable set,
and let $f\in\mathcal{C}(X\times\mathrm{VF}^{m})$. Then there exists
a function $g\in\mathcal{C}(X)$ such that for every $F\in\mathrm{Loc}_{\gg}$
and $x\in X_{F}$, if $f_{F}(x,y)\in L^{1}(F^{m})$ then 
\[
g_{F}(x)=\int_{y\in F^{m}}f_{F}(x,y)dy.
\]
\end{thm}

\begin{rem}
\label{rem:Larger generality of motivic functions}One can also define
motivic measures and functions on algebraic varieties, and show that
this class is preserved under integration. For more details, see for
example \cite[Section 3.3]{GH19}.
\end{rem}

We further need the following structure theory for $\mathcal{L}_{\mathrm{Pres}}$-definable
functions. 
\begin{defn}[{See \cite[Definition 1]{Clu03} and \cite[Section 4.1]{CL08}}]
\label{def:-Linear function}Let $X\subseteq\ints^{m}$ be an $\mathcal{L}_{\mathrm{Pres}}$-definable
set. An $\mathcal{L}_{\mathrm{Pres}}$-definable function $f:X\to\ints$
is \textit{linear} if there is $\gamma\in\ints$ and integers $a_{i}$
and $0\leq c_{i}<n_{i}$ for $1\leq i\leq m$ such that $x_{i}-c_{i}\equiv0\mathrm{\,mod\,}n_{i}$
and $f(x_{1},...,x_{m})=\sum\limits _{i=1}^{m}a_{i}(\frac{x_{i}-c_{i}}{n_{i}})+\gamma$. 
\end{defn}

It turns out that every $\mathcal{L}_{\mathrm{Pres}}$-definable function
is piece-wise linear: 
\begin{thm}[{Presburger cell decomposition \cite[Theorem 1]{Clu03}}]
\label{Presburger Cell decomposition}Let $X\subseteq\ints^{m}$
and let $f:X\to\ints$ be $\mathcal{L}_{\mathrm{Pres}}$-definable.
Then there exists a finite partition $X=\bigsqcup\limits _{i=1}^{N}A_{i}$,
such that the restriction $f|_{A_{i}}:A_{i}\to\ints$ is linear for
each $1\leq i\leq N$. 
\end{thm}

We further have, 
\begin{thm}[{Rectilinearization, see \cite[Theorem 2]{Clu03}}]
\label{thm:Rectilinearization}Let $X\subseteq\ints^{m}$ be $\mathcal{L}_{\mathrm{Pres}}$-definable
set. Then there exists a finite partition $X=\bigsqcup\limits _{i=1}^{N}A_{i}$,
and linear bijections $f_{i}:A_{i}\rightarrow\nats^{l_{i}}$ for each
$i$. 
\end{thm}

Another useful property of the model-theoretic formalism, is that
certain properties of motivic functions $g=(g_{F})_{F\in\mathrm{Loc}}$,
such as integrability and boundedness of $g_{F}$, depend only on
the isomorphism class of the residue field $k_{F}$ of $F$, if the
residual characteristic is large enough.
\begin{thm}[{Transfer principle for bounds, \cite[Theorem 3.1]{CGH16}}]
\label{thm:-transfer principle for bounds}Let $X$ be an $\mathcal{L}_{\mathrm{DP}}$-definable
set, and let $H,G\in\mathcal{C}(X)$ be motivic functions. Then the
following holds for $F\in\mathrm{Loc}_{\gg}$; if 
\[
\left|H_{F}(x)\right|\leq\left|G_{F}(x)\right|,
\]
for all $x\in X_{F}$, then for any $F'\in\mathrm{Loc}$ with the
same residue field as $F$, one also has 
\[
\left|H_{F'}(x)\right|\leq\left|G_{F'}(x)\right|,
\]
for all $x\in X_{F'}$. 
\end{thm}

\subsection{\label{subsec:Number-theoretic-characterization of epsilon jet flat}Number
theoretic characterization of $\varepsilon$-jet flatness}

Let $q=p^{r}$ be a prime power. We denote by $\mathbb{Q}_{q}$ the
unique unramified extension of $\Qp$ of degree $r$, its ring of
integers by $\mathbb{Z}_{q}$, and the maximal ideal of $\mathbb{Z}_{q}$
by $\mathfrak{m}_{q}$.

Given an $\mathcal{O}_{F}$-scheme $X$, we have a natural reduction
map $r_{k}:X(\mathcal{O}_{F})\to X(\mathcal{O}_{F}/\frak{\mathfrak{m}}_{F}^{k})$.
If $\varphi:X\rightarrow Y$ is an $\mathcal{O}_{F}$-morphism between
$\mathcal{O}_{F}$-varieties $X$ and $Y$, by abuse of notation we
denote the natural maps $X(\mathcal{O}_{F}/\frak{\mathfrak{m}}_{F}^{k})\rightarrow Y(\mathcal{O}_{F}/\frak{\mathfrak{m}}_{F}^{k})$
by $\varphi$ as well, and so for any $\overline{y}\in Y(\mathcal{O}_{F}/\frak{\mathfrak{m}}_{F}^{k})$,
$\varphi^{-1}(\overline{y})$ is a finite set in $X(\mathcal{O}_{F}/\frak{\mathfrak{m}}_{F}^{k})$. 
\begin{defn}
\label{def:Schwartz measure}Let $X$ be a variety defined over a
local field $F$. A \textit{Schwartz measure} $\mu$ on $X(F)$ is
a compactly supported, smooth measure (i.e.~$\mu$ is locally a Haar
measure in the analytic topology on $X(F)$). 
\end{defn}

\begin{lem}[{see e.g.~\cite[Theorem 2.25]{Wei82} and \cite[Proposition 3.1.1, arXiv version]{AA18}}]
\label{lem:canonical measure}If $X$ is a smooth variety over $\mathcal{O}_{F}$
of pure relative dimension $d$, then there is a unique Schwartz measure,
denoted $\mu_{X(\mathcal{O}_{F})}$, such that for every $k\in\nats$
and $\bar{x}\in X(\mathcal{O}_{F}/\frak{\mathfrak{m}}_{F}^{k})$,
\[
\mu_{X(\mathcal{O}_{F})}(r_{k}^{-1}(\bar{x}))=\left|\mathcal{O}_{F}/\frak{\mathfrak{m}}_{F}\right|^{-kd}.
\]
\end{lem}

We now prove the following theorem: 
\begin{thm}
\label{thm:epsilon jet flat and counting points}Let $\varphi:X\to Y$
be a dominant morphism between finite type $\ints$-schemes $X$ and
$Y$, where $X_{\rats}$ and $Y_{\rats}$ are smooth and geometrically
irreducible, and let $0<\varepsilon\leq1$. Then the following are
equivalent: 
\begin{enumerate}
\item For any $0<\varepsilon'<\varepsilon$ there exists a finite set $S$
of primes, such that for every $q\in\mathcal{P}_{S}$, $k\in\nats$
and $y\in Y(\ints_{q}/\frak{m}_{q}^{k})$, 
\[
\left|\varphi^{-1}(y)\right|<q^{k(\mathrm{dim}X_{\mathbb{Q}}-\varepsilon'\mathrm{dim}Y_{\mathbb{Q}})}.
\]
\item $\varphi_{\rats}:X_{\rats}\to Y_{\rats}$ is $\varepsilon$-jet flat. 
\end{enumerate}
\end{thm}

\begin{rem}
\label{rem:version for epsilon flatness}If we restrict Condition
$(1)$ of Theorem \ref{thm:epsilon jet flat and counting points}
to $k=1$ (i.e.~when $\ints_{q}/\frak{m}_{q}^{k}\simeq\mathbb{F}_{q}$),
then $\left|\varphi^{-1}(y)\right|<q^{(\mathrm{dim}X_{\mathbb{Q}}-\varepsilon'\mathrm{dim}Y_{\mathbb{Q}})}$
is equivalent to $\varphi_{\rats}:X_{\rats}\to Y_{\rats}$ being $\varepsilon$-flat.
This follows directly by observing that the fibers of $\varphi$ are
of bounded complexity, and applying the Lang-Weil bounds. Since the
complexity of the fibers of $J_{k}(\varphi_{\rats})$ is arbitrarily
large for $k\gg0$, Theorem \ref{thm:epsilon jet flat and counting points}
deserves a special treatment. 
\end{rem}

For the proof of Theorem \ref{thm:epsilon jet flat and counting points}
we need some preparation. Let $\varphi:X\to Y$ be as in the theorem,
assume that $Y$ is affine and let $\mu_{X(\mathcal{O}_{F})}$ and
$\mu_{Y(\mathcal{O}_{F})}$ be the measures as in Lemma \ref{lem:canonical measure}
for $F\in\mathrm{Loc}_{\gg}$. Since $\varphi$ is dominant it follows
that $\varphi_{*}(\mu_{X(\mathcal{O}_{F})})$ is absolutely continuous
with respect to $\mu_{Y(\mathcal{O}_{F})}$, and thus has an $L^{1}$-density
(see e.g.~\cite[Corollary 3.6]{AA16}). Set $\tau_{F}:=\varphi_{*}(\mu_{X(\mathcal{O}_{F})})$,
and write $\tau_{F}=f_{F}(y)\cdot\mu_{Y(\mathcal{O}_{F})}$. Moreover,
for any $k\in\nats$ define a function 
\[
g_{F}(y,k)=\frac{1}{\mu_{Y(\mathcal{O}_{F})}(B(y,k))}\int_{B(y,k)}f_{F}(\widetilde{y})d\mu_{Y(\mathcal{O}_{F})},
\]
where $B(y,k)=r_{k}^{-1}(r_{k}(y))$. By Theorem \ref{integration of motivic}
and Remark \ref{rem:Larger generality of motivic functions} (see
also \cite[Theorem 4.2]{GH19}), it follows that both $(f_{F}:Y(F)\rightarrow\complex)_{F\in\mathrm{Loc}}$
and $(g_{F}:Y(F)\times\nats\rightarrow\complex)_{F\in\mathrm{Loc}}$
are $\mathcal{L}_{\mathrm{DP}}$-motivic functions. Note that for
any $F\in\mathrm{Loc}_{\gg}$, any $y\in Y(\mathcal{O}_{F})$ and
any $k\in\nats$ 
\begin{equation}
g_{F}(y,k)=\frac{\left|\varphi_{*}(\mu_{X(\mathcal{O}_{F})})(B(y,k)\right|}{\mu_{Y(\mathcal{O}_{F})}(B(y,k))}=\frac{\left|\varphi^{-1}(r_{k}(y))\right|}{q^{k(\mathrm{dim}X_{\mathbb{Q}}-\mathrm{dim}Y_{\mathbb{Q}})}}.\label{eq:(8.2)}
\end{equation}

\begin{proof}[Proof of Theorem \ref{thm:epsilon jet flat and counting points}]
$(1)\Rightarrow(2)$. By Theorem \ref{thm:-transfer principle for bounds}
and by possibly enlarging $S$, we have for each $k\in\nats$ and
each $y\in Y(\mathbb{F}_{q}[t]/t^{k})$ that 
\[
\left|\varphi^{-1}(y)\right|<q^{k(\mathrm{dim}X_{\mathbb{Q}}-\varepsilon'\mathrm{dim}Y_{\mathbb{Q}})}.
\]
Identify $y$ with $\tilde{y}\in J_{k-1}(Y)(\mathbb{F}_{q})$ via
the identification $Y(\mathbb{F}_{q}[t]/t^{k})\simeq J_{k-1}(Y)(\mathbb{F}_{q})$.
Note that 
\[
\left|\varphi^{-1}(y)\right|=\left|(J_{k-1}(X))_{\tilde{y},J_{k-1}(\varphi)}(\mathbb{F}_{q})\right|.
\]
By the Lang-Weil bounds, this implies that $\varphi_{\rats}$ is $\varepsilon$-jet
flat.

$(2)\Rightarrow(1)$. Assume that $\varphi_{\rats}$ is $\varepsilon$-jet
flat. Fix $\varepsilon'<\varepsilon$. By (\ref{eq:(8.2)}) it is
enough to show that $g_{F}(y,k)q_{F}^{-k(1-\varepsilon')\mathrm{dim}Y_{\rats}}<1$
for $F\in\mathrm{Loc}_{\gg}$. By \cite[Theorem 2.1.3]{CGH18}, there
exist a motivic function $G$ on $\nats$ and $d\in\ints$, such that
for any $F\in\mathrm{Loc}_{\gg}$ 
\[
\underset{y\in Y(F)}{\mathrm{sup}}g_{F}(y,k)\leq G_{F}(k)\leq q_{F}^{d}\underset{y\in Y(F)}{\mathrm{sup}}g_{F}(y,k).
\]
It is therefore enough to show that for $F\in\mathrm{Loc}_{\gg}$
\[
H_{F}(k):=G_{F}(k)q_{F}^{-k(1-\varepsilon')\mathrm{dim}Y_{\rats}}<1.\tag{\ensuremath{\star}}
\]
\begin{lem}
\label{lem:Nice description}There exist a definable partition $\nats=\bigsqcup\limits _{j=1}^{N}A_{j}$
of $\nats$, natural numbers $L,\{a_{i}\}_{i=1}^{L}\in\nats$ and
$\{b_{i}\}_{i=1}^{L}\in\rats$, such that on each definable part 
\[
G_{F}(k)=\sum\limits _{i=1}^{L}c_{i}(F)\cdot k^{a_{i}}q_{F}^{b_{i}k},
\]
for any $F\in\mathrm{Loc}_{\gg}$ and some constants $\{c_{i}(F)\}_{i=1}^{L}$
depending on $F$. 
\end{lem}

\begin{proof}
Since $G_{F}(k)$ is a motivic function, by Definition \ref{def:motivic function},
it can be written as 
\[
G_{F}(k)=\sum\limits _{i=1}^{N_{1}}\left|Y_{i,F,k}\right|q_{F}^{\alpha_{i,F}(k)}\left(\prod_{j=1}^{N_{2}}\beta_{ij,F}(k)\right)\left(\prod_{j=1}^{N_{3}}\frac{1}{1-q_{F}^{a_{ij}}}\right),\tag{\ensuremath{\star\star}}
\]
where $\alpha_{i}:\nats\rightarrow\mathbb{Z}$ and $\beta_{ij}:\nats\rightarrow\mathbb{Z}$
are definable functions and $Y_{i}\subseteq\nats\times\mathrm{RF}^{r_{i}}$
are definable sets. By the disjointness of the $\mathrm{RF}$ and
$\VG$ sorts (see e.g.~\cite[Theorem 2.1.1]{CL08}), we may partition
$\nats$ into finitely many definable subsets, such that on each part
$Y_{i,F,k}=Y_{i,F}$ does not depend on $k$, and $\alpha_{i},\beta_{ij}$
are $\mathcal{L}_{\mathrm{Pres}}$-definable functions. By further
refining the partition, we may assume that $\alpha_{i},\beta_{ij}$
are linear (Theorem \ref{Presburger Cell decomposition}). By rearranging
$(\star\star)$, the lemma follows. 
\end{proof}
By Theorem \ref{thm:Rectilinearization}, we may refine the partition
$\nats=\bigsqcup\limits _{j=1}^{N}A_{j}$ of Lemma \ref{lem:Nice description},
so that we have linear bijections $f_{j}:A_{j}\rightarrow\nats^{l_{j}}$
for $l_{j}\in\{0,1\}$. It is enough to prove $(\star)$ for each
definable subset $A_{j}$, so by the last lemma we may assume that
$H_{F}(k)=\sum\limits _{i=1}^{L}c_{i}(F)\cdot k^{a_{i}}q_{F}^{k(b_{i}-(1-\varepsilon')\mathrm{dim}Y)}$.
Since $J_{k-1}(\varphi_{\rats})$ is $\varepsilon$-flat, it follows
from the Lang-Weil bounds that for any $\varepsilon'<\varepsilon$
and any $k\in\nats$, $\lim\limits _{q_{F}\rightarrow\infty}H_{F}(k)=0$.
It is enough to consider the case where $A_{j}$ is infinite. By changing
coordinates via $f_{j}$ we may assume that $A_{j}=\nats$ (where
now $a_{i}\in\nats$ and $b_{i}\in\ints$).

Without loss of generality we may assume that $b_{1}$ is maximal.
Set $T:=\{i\in[L]:b_{i}=b_{1}\}$ and $P_{F}(k):=\sum\limits _{i\in T}c_{i}(F)\cdot k^{a_{i}}$.
Then we can write
\[
H_{F}(k)=P_{F}(k)q_{F}^{k(b_{1}-(1-\varepsilon')\mathrm{dim}Y_{\rats})}+E_{F}(k),
\]
where $E_{F}(k)=\sum\limits _{i\in T^{c}}c_{i}(F)\cdot k^{a_{i}}q_{F}^{k(b_{i}-(1-\varepsilon')\mathrm{dim}Y_{\rats})}$.
We may assume that for any $N\in\nats$ there exists $F\in\mathrm{Loc}$
with $\mathrm{char}(k_{F})>N$ such that $P_{F}(k)$ is not identically
zero as a polynomial on $k$. Denote by $\mathrm{Loc}_{P}$ the collection
of such $F\in\mathrm{Loc}$, and set $\mathcal{S}_{P}=\{q_{F}:F\in\mathrm{Loc}_{P}\}$,
and $\mathrm{Loc}_{q,P}=\{F\in\mathrm{Loc}_{P}:q_{F}=q\}$. 
\begin{lem}
\label{lem:Auxiliary lemma}~
\begin{enumerate}
\item There exists $k_{0}\in\nats$ such that for any $k>k_{0}$, there
are infinitely many $q\in\mathcal{S}_{P}$, such that $\exists F\in\mathrm{Loc}_{q,P}$
with $P_{F}(k)\neq0$. 
\item There exists $e\in\nats$ such that if $P_{F}(k)\neq0$ for some $k\in\nats$,
then $q_{F}^{-e}<\left|P_{F}(k)\right|$ and moreover, $c_{i}(F)<q_{F}^{e}$
for each $i\in L$ and each $q_{F}>2$. 
\item For any $\delta>0$ and any $a\in\nats$ there exists $k_{1}\in\nats$
such that $k^{a}<q^{\delta k}$ for every $k>k_{1}$ and every $q\geq2$. 
\end{enumerate}
\end{lem}

\begin{proof}
$(1)$. Assume the contrary, and set $d:=\mathrm{deg}(P)$. Then for
any $k_{0}\in\nats$, there exists $k'>k_{0}$ such that $P_{F}(k')=0$
for any $F\in\mathrm{Loc}_{q,P}$ with $q\in\mathcal{S}_{P}$ large
enough. By induction, this implies that for $q_{F}$ large enough,
$P_{F}(k)$ has more than $d$ roots yielding a contradiction.

Item $(2)$ follows by analyzing the possible terms appearing in a
motivic function (e.g.~$(\star\star)$). Item $(3)$ is obvious. 
\end{proof}
\begin{cor}
\label{cor:8.22}There exists $k_{2}\in\nats$ such that for each
$k>k_{2}$,
\[
\lim\limits _{q_{F}\rightarrow\infty}P_{F}(k)q_{F}^{k(b_{1}-(1-\varepsilon')\mathrm{dim}Y_{\rats})}=0,
\]
and $P_{F}(k)\neq0$ for infinitely many $q\in\mathcal{S}_{P}$. 
\end{cor}

\begin{proof}
Let $\delta=\frac{1}{2}\min\limits _{i\in T^{c}}\left|b_{1}-b_{i}\right|$.
By the last lemma, for any $q_{F}>2$ we have
\[
\left|E_{F}(k)\right|<\underset{i\in T^{c}}{\sum}q_{F}^{e+k(\delta+b_{i}-(1-\varepsilon')\mathrm{dim}Y_{\rats})}.
\]
Hence, there exists $k_{2}\in\nats$ such that for every $k>k_{2}$
and any $F\in\mathrm{Loc}_{q,P}$ with $P_{F}(k)\neq0$, we have
\begin{equation}
\left|E_{F}(k)\right|<\frac{1}{2}\left|P_{F}(k)q_{F}^{k(b_{1}-(1-\varepsilon')\mathrm{dim}Y_{\rats})}\right|.\label{eq:(8.3)}
\end{equation}
By Lemma \ref{lem:Auxiliary lemma}(1), we may choose $k_{2}$ large
enough such that for any $k>k_{2}$, we have $P_{F}(k)\neq0$ for
infinitely many $q\in\mathcal{S}_{P}$ and $F\in\mathrm{Loc}_{q,P}$.
Since $\lim\limits _{q_{F}\rightarrow\infty}H_{F}(k)=0$ for any $k\in\nats$,
and by (\ref{eq:(8.3)}) we are done.
\end{proof}
Corollary \ref{cor:8.22} implies the following: 
\begin{cor}
We have $b_{1}\leq(1-\varepsilon)\mathrm{dim}Y_{\rats}$. 
\end{cor}

For any $\varepsilon'<\varepsilon$ we can now write $H_{F}(k)=\sum\limits _{i=1}^{L}c_{i}(F)\cdot k^{a_{i}}q_{F}^{kd_{i}(\varepsilon')}$,
with all $d_{i}(\varepsilon')<0$. Set $d(\varepsilon')=\max\limits _{1\leq i\leq L}\{d_{i}(\varepsilon')\}$,
then by Lemma \ref{lem:Auxiliary lemma} for every $\delta>0$ we
have $k_{0}$ such that for every $k>k_{0}$ 
\[
H_{F}(k)<q_{F}^{(d(\varepsilon')+\delta)k+e}.
\]
In particular, there exists $k_{1}$ such that for $k>k_{1}$ and
every $F\in\mathrm{Loc}_{\gg}$ we have $H_{F}(k)<1$. We are now
done as for each $1\leq k\leq k_{1}$ we can conclude individually
by the Lang-Weil estimates, by taking $F\in\mathrm{Loc}_{\gg}$. 
\end{proof}

\section{\label{sec:Applications-to-probabilistic}Applications to the $p$-adic
probabilistic Waring type problem }

Our goal in this section is to combine our results on word maps (Sections
\ref{sec:Singulrities-of-word}-\ref{sec:flatness-and-(FRS) of word maps}),
with the number-theoretic interpretation of the (FRS) property and
of $\varepsilon$-jet flatness (Section \ref{sec:The-(FRS)-property and counting points}),
to obtain results on random walks induced from word maps. We start
by discussing the more general case of random walks induced from algebraic
morphisms, and then descend to the case of word maps.

\subsection{\label{subsec:Algebraic-random-walks-general case}Algebraic random
walks: general case}

Let $X$ be a finite type $\ints$-scheme, with $X_{\rats}$ smooth
and geometrically irreducible, let $\underline{G}$ be a finite type
affine group scheme, with $\underline{G}_{\rats}$ connected and let
$\varphi:X\rightarrow\underline{G}$ be a dominant morphism.

Let $\mu_{X(\ints_{q}/\frak{m}_{q}^{k})}$ denote the uniform probability
measure on $X(\ints_{q}/\frak{m}_{q}^{k})$. For every prime power
$q$ and $k\in\nats$, the map $\varphi:X(\ints_{q}/\frak{m}_{q}^{k})\rightarrow\underline{G}(\ints_{q}/\frak{m}_{q}^{k})$
gives rise to a probability measure $\tau_{\varphi,q,k}:=\varphi_{*}(\mu_{X(\ints_{q}/\frak{m}_{q}^{k})})$
by considering the pushforward of $\mu_{X(\ints_{q}/\frak{m}_{q}^{k})}$
under $\varphi$. Explicitly, the probability given to $g\in\underline{G}(\ints_{q}/\frak{m}_{q}^{k})$
is
\[
\tau_{\varphi,q,k}(g)=\frac{\left|\varphi^{-1}(g)\right|}{\left|X(\ints_{q}/\frak{m}_{q}^{k})\right|}.
\]
The collection $\{\tau_{\varphi,q,k}\}_{q,k}$ induces a family of
random walks on $\{\underline{G}(\ints_{q}/\frak{m}_{q}^{k})\}_{q,k}$.
Note that $\tau_{\varphi,q,k}^{*2}=\tau_{\varphi^{*2},q,k}$ (see
e.g.~\cite[Remark 1.2]{GH19}), so the $t$-th step of the random
walk induced by $\tau_{\varphi,q,k}$ is given by
\[
\tau_{\varphi,q,k}^{*t}(g)=\frac{\left|(\varphi^{*t})^{-1}(g)\right|}{\left|X(\ints_{q}/\frak{m}_{q}^{k})\right|^{t}}.
\]
Let $\pi_{q,k}$ be the uniform probability on $\underline{G}(\ints_{q}/\frak{m}_{q}^{k})$.
Recall that if $\supp(\tau_{\varphi,q,k})$ is generating and contains
the identity element, then $\tau_{\varphi,q,k}^{*t}$ converges to
$\pi_{q,k}$ as $t\rightarrow\infty$, and the mixing time (see Definition
\ref{def:mixing time}) measures this convergence rate. Also recall
we define $L^{a}$ norms as follows: 
\begin{defn}
\label{def:Lp norms}For a real function $f$ on a finite group $G$,
set 
\[
\left\Vert f\right\Vert _{a}:=(\left|G\right|^{a-1}\sum\limits _{g\in G}\left|f(g)\right|^{a})^{1/a}\text{ ~~and~~ }\left\Vert f\right\Vert _{\infty}=\left|G\right|\cdot\max_{g\in G}\left|f(g)\right|.
\]
\end{defn}

The following is a uniform analogue of mixing time. 
\begin{defn}
\label{def:uniform mixing time- section}Let $\varphi:X\rightarrow\underline{G}$
be a dominant morphism as above. We say $\varphi$ 
\begin{enumerate}
\item is\textit{ almost $\{\mathbb{F}_{q}\}$-unifor}\textsl{m in $L^{a}$}
if there exists a finite set $S$ of primes such that 
\[
\underset{\mathcal{P}_{S}\ni q\rightarrow\infty}{\lim}\left\Vert \tau_{\varphi,q,1}-\pi_{q,1}\right\Vert _{a}=0.
\]
\item has an \textit{almost} $\{\mathbb{F}_{q}\}$\textit{-uniform $L^{a}$-mixing
time }$t_{a,\varphi,\{\mathbb{F}_{q}\}}$ if it is minimal such that
$\varphi^{*t_{a,\varphi,\{\mathbb{F}_{q}\}}}$ is almost $\{\mathbb{F}_{q}\}$-uniform
in $L^{a}$. 
\item is\textit{ almost $\{\ints_{q}\}$-unifor}\textsl{m in $L^{a}$} if
there exists a finite set $S$ of primes such that 
\[
\underset{\mathcal{P}_{S}\ni q\rightarrow\infty}{\lim}\left(\underset{k\in\nats}{\sup}\left\Vert \tau_{\varphi,q,k}-\pi_{q,k}\right\Vert _{a}\right)=0.
\]
\item has an \textit{almost} $\{\ints_{q}\}$\textit{-uniform $L^{a}$-mixing
time }$t_{a,\varphi,\{\ints_{q}\}}$ if it is minimal such that $\varphi^{*t_{a,\varphi,\{\ints_{q}\}}}$
is almost $\{\ints_{q}\}$-uniform in $L^{a}$. 
\end{enumerate}
\end{defn}

\subsubsection{\label{subsec:Algebraic interpretation of mixing times}Algebraic
interpretation of almost uniform mixing time}

We now use results from Section \ref{sec:The-(FRS)-property and counting points}
to give an algebraic interpretation of various properties of the family
of random walks induced by a morphism $\varphi:X\rightarrow\underline{G}$
as above. We first characterize properties related to $L^{\infty}$
bounds and then deal with properties related to $L^{1}$ bounds. 
\begin{thm}[Characterization of $L^{\infty}$-bounds]
\label{thm:dictionary for random walks Linfty}Let $X$ be a finite
type $\ints$-scheme, with $X_{\rats}$ smooth and geometrically irreducible,
let $\underline{G}$ be a finite type affine group scheme, with $\underline{G}_{\rats}$
connected and let $\varphi:X\rightarrow\underline{G}$ be a dominant
morphism. 
\begin{enumerate}
\item $\varphi$ is almost $\{\mathbb{F}_{q}\}_{q}$-uniform in $L^{\infty}$
if and only if $\varphi_{\rats}:X_{\rats}\rightarrow\underline{G}_{\rats}$
is (FGI). 
\item $\varphi$ is almost $\{\ints_{q}\}_{q}$-uniform in $L^{\infty}$
if and only if $\varphi_{\rats}:X_{\rats}\rightarrow\underline{G}_{\rats}$
is (FRS) and (FGI). 
\item $\varphi_{\rats}$ is $\varepsilon$-jet flat if and only if for every
$\varepsilon'<\varepsilon$ there exists a set $S(\varepsilon')$
of primes such that the family $\{\tau_{\varphi,q,k}\}_{q,k}$ satisfies
$\tau_{\varphi,q,k}(g)\leq\left|\underline{G}(\ints_{q}/\frak{m}_{q}^{k})\right|^{-\varepsilon'}$
for every $q\in\mathcal{P}_{S}$ and every $k\in\ints_{>0}$. 
\end{enumerate}
\end{thm}

\begin{proof}
~
\begin{enumerate}
\item Assume that \textit{\emph{$\varphi_{\rats}:X_{\rats}\rightarrow\underline{G}_{\rats}$
is (FGI).}} Theorem \ref{thm: flatness and counting points} implies
that $\varphi$ is almost $\{\mathbb{F}_{q}\}_{q}$-uniform in $L^{\infty}$.
For the other direction, Theorem \ref{thm: flatness and counting points}
implies that $\varphi_{\rats}$ is flat. Assume that $\varphi_{\rats}:X_{\rats}\rightarrow\underline{G}_{\rats}$
is not (FGI). By \cite[Corollary 9.7.9]{Gro66}, we may find a prime
$p$, such that for each $r\in\nats$ and $q=p^{r}$ the map $\varphi_{\mathbb{F}_{q}}:X_{\mathbb{F}_{q}}\rightarrow\underline{G}_{\mathbb{F}_{q}}$
is flat but not (FGI). Take $g\in\underline{G}(\mathbb{F}_{q})$ with
$C_{\varphi,q,g}>1$. Then also $C_{\varphi,q^{r'},g}>1$ for any
$r'\in\nats$. By Theorem \ref{thm: flatness and counting points}
we deduce that $\varphi$ is not almost $\{\mathbb{F}_{q}\}_{q}$-uniform
in $L^{\infty}$. 
\item Assume that $\varphi$ is almost $\{\ints_{q}\}_{q}$-uniform in $L^{\infty}$.
Item (1) implies that $\varphi_{\rats}$ is (FGI). By enlarging $S$
we may assume that $X$ is smooth over $\ints_{q}$ for every $q\in\mathcal{P}_{S}$,
so that $\mu_{X(\ints_{q})}$ and $\mu_{\underline{G}(\ints_{q})}$
exist by Lemma \ref{lem:canonical measure}. Since $\varphi$ is dominant,
$\varphi_{*}(\mu_{X(\ints_{q})})$ is absolutely continuous with respect
to $\mu_{\underline{G}(\ints_{q})}$ (see e.g.~\cite[Corollary 3.6]{AA16}).
Fix a prime $p\notin S$, and let $q=p^{r}\in\mathcal{P}_{S}$. For
every $g\in\underline{G}(\ints_{q})$ set $B(g,k)=r_{k}^{-1}(r_{k}(g))$,
then we have,
\begin{align*}
\tag{\ensuremath{\dagger}}\frac{\mu_{X(\ints_{q})}(\varphi^{-1}(B(g,k)))}{\mu_{\underline{G}(\ints_{q})}(B(g,k))} & =\frac{\left|\varphi^{-1}(r_{k}(g))\right|}{q^{k(\mathrm{dim}X_{\mathbb{Q}}-\mathrm{dim}\underline{G}_{\mathbb{Q}})}}=\tau_{\varphi,q,k}(r_{k}(g))\cdot\frac{\left|X(\ints_{q}/\frak{m}_{q}^{k})\right|}{q^{k(\mathrm{dim}X_{\mathbb{Q}}-\mathrm{dim}\underline{G}_{\mathbb{Q}})}}\\
 & \leq\sup\limits _{k\in\nats}\left\Vert \tau_{\varphi,q,k}\right\Vert _{\infty}\frac{\left|X(\ints_{q}/\frak{m}_{q}^{k})\right|}{q^{k\mathrm{dim}X_{\rats}}}\frac{q^{k\mathrm{dim}\underline{G}_{\mathbb{Q}}}}{\left|\underline{G}(\ints_{q}/\frak{m}_{q}^{k})\right|}.
\end{align*}
Since $\varphi$ is almost $\{\ints_{q}\}_{q}$-uniform in $L^{\infty}$,
by using the the Lang-Weil bounds, and by possibly enlarging $S$,
it follows that $(\dagger)$ is uniformly bounded (in $k$) by a constant
$C>0$. By the analytic criterion for the (FRS) property (Theorem
\ref{thm:analytic criterion of the (FRS) property} and \cite[Theorem 1.6]{Gla19})
combined with \cite[Lemma 3.15]{GH19}, the map $\varphi_{\rats}$
is (FRS). 

The ``if'' part follows from a family version of Theorem \ref{thm:number theoretic criterion for rational singularities},
which is proven in an upcoming work (see Remark \ref{rem:family version for rational singularities}).
For completeness, we prove here the weaker statement that $\varphi^{*2}$
is almost $\{\ints_{q}\}_{q}$-uniform in $L^{\infty}$. Set $\psi:X\rightarrow\underline{G}$
by $\psi(x):=\varphi(x)^{-1}$. Note that if $\varphi_{\rats}:X_{\rats}\rightarrow\underline{G}_{\rats}$
is (FRS) and (FGI) then by Proposition \ref{prop: properties preserved under convolution}
the map $\varphi_{\rats}*\psi_{\rats}:X_{\rats}\times X_{\rats}\rightarrow\underline{G}_{\rats}$
is (FRS) and (FGI) as well. In particular, the fiber $\left(X_{\rats}\times X_{\rats}\right)_{e,\varphi_{\rats}*\psi_{\rats}}$
is a reduced, geometrically irreducible, local complete intersection
scheme with rational singularities. By Theorem \ref{thm:number theoretic criterion for rational singularities}
(or \cite[Theorem 3.0.3]{AA18}) and by the Lang-Weil bounds, we have
$\left|\left|\underline{G}(\ints_{q}/\frak{m}_{q}^{k})\right|\tau_{\varphi,q,k}*\tau_{\psi,q,k}(e)-1\right|<Cq^{-1/2}$
for some $C>0$, so 
\begin{align*}
\left\Vert \tau_{\varphi,q,k}-\pi_{q,k}\right\Vert _{2}^{2} & =\left|\underline{G}(\ints_{q}/\frak{m}_{q}^{k})\right|\sum_{g\in\underline{G}(\ints_{q}/\frak{m}_{q}^{k})}\left(\tau_{\varphi,q,k}(g)-\pi_{q,k}(g)\right)^{2}\\
 & =\left(\left|\underline{G}(\ints_{q}/\frak{m}_{q}^{k})\right|\sum_{g\in\underline{G}(\ints_{q}/\frak{m}_{q}^{k})}\tau_{\varphi,q,k}(g)^{2}\right)-1\\
 & =\left|\underline{G}(\ints_{q}/\frak{m}_{q}^{k})\right|\tau_{\varphi,q,k}*\tau_{\psi,q,k}(e)-1<Cq^{-1/2}.
\end{align*}
By Young's inequality, 
\begin{equation}
\left\Vert \tau_{\varphi,q,k}^{*2}-\pi_{q,k}\right\Vert _{\infty}=\left\Vert (\tau_{\varphi,q,k}-\pi_{q,k})^{*2}\right\Vert _{\infty}\leq\left\Vert \tau_{\varphi,q,k}-\pi_{q,k}\right\Vert _{2}^{2}<Cq^{-1/2}.\label{eq:(9.1)}
\end{equation}

\item Follows from Theorem \ref{thm:epsilon jet flat and counting points}
by using the Lang-Weil bounds. 
\end{enumerate}
\end{proof}
The following theorem is a strengthening of \cite[Proposition 2.1]{LST19}.
\begin{thm}[Characterization of $L^{1}$-uniformity]
\label{thm:L1 bounds }Let $\varphi:X\rightarrow\underline{G}$ be
as in Theorem \ref{thm:dictionary for random walks Linfty}. Then
the following are equivalent: 
\begin{enumerate}
\item $\varphi$ is almost $\{\mathbb{F}_{q}\}_{q}$-uniform in $L^{1}$. 
\item $\varphi$ is almost $\{\ints_{q}\}_{q}$-uniform in $L^{1}$. 
\item $\varphi_{\rats}:X_{\rats}\rightarrow\underline{G}_{\rats}$ has a
geometrically irreducible generic fiber. 
\end{enumerate}
\end{thm}

\begin{proof}
We would like to prove $(1)\Rightarrow(3)\Rightarrow(2)\Rightarrow(1)$.
The implication $(2)\Rightarrow(1)$ is obvious. $(1)\Rightarrow(3)$
is essentially given in \cite{LST19}: if $\varphi$ is almost $\{\mathbb{F}_{q}\}_{q}$-uniform
in $L^{1}$, then for any prime $p\notin S$ we have $\lim\limits _{n\rightarrow\infty}\left\Vert \tau_{\varphi,p^{n}}-\pi_{p^{n}}\right\Vert _{1}=0$.
By \cite[Proposition 2.1]{LST19}, $\varphi_{\mathbb{F}_{p}}$ has
geometrically irreducible generic fiber for each $p\notin S$, which
implies that $\varphi_{\rats}$ has a geometrically irreducible generic
fiber.

It is left to show $(3)\Rightarrow(2)$. Let $U\subseteq\underline{G}$
be an open set such that $\varphi_{\rats}$ is smooth and (FGI) over
$U_{\rats}\subseteq\underline{G}_{\rats}$. Set 
\[
V_{q}:=\{g\in\underline{G}(\ints_{q}):r_{1}(g)\in U(\mathbb{F}_{q})\}.
\]
By a generalization of Hensel's lemma, we may find a finite set $S$
such that for every $q\in\mathcal{P}_{S}$, and every $g\in V_{q}$
we have 
\[
\left|\varphi^{-1}(r_{k}(g))\right|=q^{(k-1)(\mathrm{dim}X_{\rats}-\mathrm{dim}\underline{G}_{\rats})}\cdot\left|\varphi^{-1}(r_{1}(g))\right|.
\]
By the Lang-Weil bounds, we may find $C_{1}>0$ such that for any
$g\in V_{q}$ we have 
\[
\left|\tau_{\varphi,q,k}(r_{k}(g))-\frac{1}{\left|\underline{G}(\ints_{q}/\frak{m}_{q}^{k})\right|}\right|<\frac{C_{1}q^{-1/2}}{\left|\underline{G}(\ints_{q}/\frak{m}_{q}^{k})\right|}.
\]
Write $Z=U^{c}$ and notice that $\mathrm{dim}\varphi^{-1}(Z)_{\rats}<\mathrm{dim}X_{\rats}$
since $X_{\rats}$ is geometrically irreducible. We may therefore
find $C_{2}>0$, such that 
\begin{align*}
\left\Vert \tau_{\varphi,q,k}-\pi_{q,k}\right\Vert _{1} & =\sum_{g\in r_{k}(V_{q})}\left|\tau_{\varphi,q,k}(g)-\pi_{q,k}(g)\right|+\sum_{g\in r_{k}(V_{q}^{c})}\left|\tau_{\varphi,q,k}(g)-\pi_{q,k}(g)\right|\\
 & <C_{1}q^{-1/2}+\sum_{g\in r_{k}(V_{q}^{c})}\left|\tau_{\varphi,q,k}(g)\right|+\sum_{g\in r_{k}(V_{q}^{c})}\left|\pi_{q,k}(g)\right|\\
 & <C_{1}q^{-1/2}+\frac{\left|\varphi^{-1}(Z(\mathbb{F}_{q}))\right|}{\left|X(\mathbb{F}_{q})\right|}+\frac{\left|Z(\mathbb{F}_{q})\right|}{\left|\underline{G}(\mathbb{F}_{q})\right|}<C_{1}q^{-1/2}+C_{2}q^{-1},
\end{align*}
where the last inequality follows by applying the Lang-Weil bounds.
\end{proof}
We would now like to discuss the random walk on $\underline{G}(\hat{\ints})=\underset{\leftarrow}{\mathrm{lim}}\underline{G}(\ints/n\ints)$.
We say that a map $\varphi:X\rightarrow\underline{G}$ is \textsl{almost
$\{\ints/n\ints\}_{n\in\nats}$-uniform in $L^{a}$} if there exists
a set $S$ of primes such that $\left\Vert \tau_{\varphi,n}-\pi_{n}\right\Vert _{a}<\frac{1}{2}$
for every $n\in\nats$ which is a product of primes outside of $S$.
By the Chinese remainder theorem, it is tempting to say that if $\varphi:X\rightarrow\underline{G}$
is almost $\{\ints_{q}\}$-uniform in $L^{\infty}$, then it is also
\textsl{almost $\{\ints/n\ints\}_{n\in\nats}$-uniform in $L^{\infty}$.
}The following example shows this hope is too naive. 
\begin{rem}[Warning! Rational singularities+complete intersection does not imply
good point count over $\ints/n\ints$]
In fact, even smooth schemes might have bad point count over $\ints/n\ints$;
let $E$ be an elliptic curve over $\rats$. The Sato-Tate conjecture
(see \cite{CHT08,Tay08,HSBT10}) says that under some mild assumptions
on $E$, we have $\left|E(\mathbb{F}_{p})\right|=p+1-2\sqrt{p}\cos(\theta_{p})$,
where for any $0\leq\alpha\leq\beta\leq\pi$,
\begin{equation}
\underset{M\rightarrow\infty}{\lim}\left(\frac{\#\{\text{primes }p\leq M:\alpha\leq\theta_{p}\leq\beta\}}{\#\{\text{primes }p\leq M\}}\right)=\frac{2}{\pi}\int_{\alpha}^{\beta}\sin^{2}(\theta)d\theta.\label{eq:Sato-Tate}
\end{equation}
In particular, by (\ref{eq:Sato-Tate}) and the prime number theorem,
there exist $\varepsilon,\delta>0$ such that for arbitrary large
$M\in\nats$, we may find a subset $S$ of primes in the interval
$[\frac{1}{2}M,M]$, such that $\left|S\right|>\varepsilon\left|\frac{M}{\mathrm{log}M}\right|$
and such that $\frac{\left|E(\mathbb{F}_{p})\right|}{p}\geq1+\delta p^{-1/2}$,
for every $p\in S$. Now, taking $n=\prod\limits _{p\in S}p$, by
the Chinese remainder theorem we get the following:
\[
\frac{\left|E(\ints/n\ints)\right|}{n}=\prod_{p\in S}\frac{\left|E(\mathbb{F}_{p})\right|}{p}\geq\prod_{p\in S}\left(1+\delta p^{-1/2}\right)\geq\left(1+\frac{\delta}{\sqrt{M}}\right)^{\varepsilon\left|\frac{M}{\mathrm{log}M}\right|}\overset{M\rightarrow\infty}{\longrightarrow}\infty.
\]
\end{rem}

It turns out that this pathology can be fixed by applying a few more
convolutions. 
\begin{prop}
\label{prop:unifrm in Z/nZ}Let $\varphi:X\rightarrow\underline{G}$
be as in Theorem \ref{thm:dictionary for random walks Linfty}. Assume
that $\varphi_{\rats}$ is (FRS) and (FGI). Then $\varphi^{*3}$ is
almost $\{\ints/n\ints\}_{n\in\nats}$-uniform in $L^{\infty}$. 
\end{prop}

\begin{proof}
By (\ref{eq:(9.1)}), there exists $C>0$, such that for every $k\in\nats$
and almost every prime $p$ 
\[
\left\Vert \tau_{\varphi,p,k}-\pi_{p,k}\right\Vert _{\infty}<Cp^{-\frac{1}{2}}.
\]
By standard properties of mixing (see e.g.~\cite[Lemma 4.18]{LeP17}),
we have 
\[
\left\Vert \tau_{\varphi,p,k}^{*t}-\pi_{p,k}\right\Vert _{\infty}<C^{t}p^{-\frac{t}{2}}.
\]
Now let $n,N\in\nats$ be such that $n=\prod\limits _{i=1}^{N}p_{i}^{k_{i}}$
with $p_{i}$ distinct primes, larger than $M$ for some $M\in\nats$
sufficiently large. Then for each $i$ and each $g\in\underline{G}(\ints/p_{i}^{k_{i}}\ints)$,
\[
1-C^{t}p_{i}^{-\frac{t}{2}}<\left|\underline{G}(\ints/p_{i}^{k_{i}}\ints)\right|\tau_{\varphi,p_{i},k_{i}}^{*t}(g)<1+C^{t}p_{i}^{-\frac{t}{2}}.
\]
Therefore for every $g\in\underline{G}(\ints/n\ints)$, every $t\geq3$,
and $M$ sufficiently large we get, 
\[
\left|\underline{G}(\ints/n\ints)\right|\tau_{\varphi,n}^{*t}(g)<\mathrm{exp}\left(\sum_{M<p\in\text{primes}}\mathrm{log}\left(1+C^{t}p^{-\frac{t}{2}}\right)\right)<\mathrm{exp}\left(\sum_{M<p\in\text{primes}}C^{t}p^{-\frac{t}{2}}\right)<\mathrm{1+\frac{1}{2}},
\]
and similarly 
\[
1-\frac{1}{2}<\left|\underline{G}(\ints/n\ints)\right|\tau_{\varphi,n}^{*t}(g),
\]
which is precisely what we want. 
\end{proof}
\begin{rem}
The (FRS) property can also be used to obtain information on algebraic
random walks on real algebraic groups. Indeed, an Archimedean analogue
of \cite[Theorem 3.4]{AA16} was proven in \cite{Rei}: the pushforward
$\varphi_{*}(\mu)$ of any smooth compactly supported measure $\mu$
on $X(\reals)$ under an (FRS) morphism $\varphi:X\rightarrow Y$
between smooth $\reals$-varieties is a measure with continuous density. 
\end{rem}

\subsection{\label{subsec:Algebraic-random-walks word maps}Algebraic random
walks induced from word maps}

Let $\mathcal{G}$ denote the collection of all affine groups schemes
$\underline{G}$, with $\underline{G}_{\rats}$ semisimple and simply
connected. We study uniform properties of the family of random walks
induced by a word $w\in F_{r}$ on $\{\underline{G}(\ints_{q})\}_{q,\underline{G}\in\mathcal{G}}$
by considering the family of random walks induced on the quotients.
Let $\pi_{\underline{G},q,k}$ be the uniform probability measure
on $\underline{G}(\ints_{q}/\frak{m}_{q}^{k})$. The word map $\varphi_{w}$
induces a collection of probability measures $\{\tau_{w,\underline{G},q,k}\}_{\underline{G},q,k}$
on $\{\underline{G}(\ints_{q}/\frak{m}_{q}^{k})\}_{\underline{G},q,k}$
by $\tau_{w,\underline{G},q,k}:=(\varphi_{w})_{*}(\pi_{\underline{G},q,k}^{r})$.
As we saw before, the probability to reach an element $g\in\underline{G}(\ints_{q}/\frak{m}_{q}^{k})$
at the $t$-th step of the random walk is given by
\[
\tau_{w,\underline{G},q,k}^{*t}(g)=\frac{\left|(\varphi_{w}^{*t})^{-1}(g)\right|}{\left|\underline{G}(\ints_{q}/\frak{m}_{q}^{k})\right|^{rt}}.
\]

\begin{defn}
~\label{def: almost uniform mixing time}Let $w\in F_{r}$ be a word. 
\begin{enumerate}
\item $w$ is\textit{ almost $\{\ints_{q}\}_{q}$-uniform in $L^{a}$} if
$\varphi_{w}:\underline{G}^{r}\rightarrow\underline{G}$ is almost
$\{\ints_{q}\}_{q}$-uniform for every $\underline{G}\in\mathcal{G}$. 
\item $w$ has an \textit{almost} \textit{$\{\ints_{q}\}_{q}$-uniform $L^{a}$-mixing
time }$t_{a,w,\{\mathbb{Z}_{q}\}_{q}}$ if it is minimal such that
$w^{*t_{a,w,\{\mathbb{Z}_{q}\}_{q}}}$ is almost-\textit{\emph{$\{\ints_{q}\}_{q}$-uniform
in $L^{a}$.}}
\end{enumerate}
These notions are similarly defined for the collection $\{\mathbb{F}_{q}\}_{q}$. 
\end{defn}

The following corollary is an immediate consequence of Theorems \ref{thm:dictionary for random walks Linfty}
and \ref{thm:L1 bounds }. 
\begin{cor}
\label{cor:dictionary for word maps}Let $w\in F_{r}$. 
\begin{enumerate}
\item $w$ is almost $\{\mathbb{F}_{q}\}$-uniform in $L^{\infty}$ if and
only if $(\varphi_{w})_{\rats}:\underline{G}_{\rats}^{r}\rightarrow\underline{G}_{\rats}$
is (FGI) for every $\underline{G}\in\mathcal{G}$. 
\item $w$ is almost $\{\mathbb{Z}_{q}\}$-uniform in $L^{\infty}$ if and
only if $(\varphi_{w})_{\rats}:\underline{G}_{\rats}^{r}\rightarrow\underline{G}_{\rats}$
is (FRS) and (FGI) for every $\underline{G}\in\mathcal{G}$. 
\item The following are equivalent: 
\begin{enumerate}
\item $w$ is almost $\{\mathbb{F}_{q}\}$-uniform in $L^{1}$. 
\item $w$ is almost $\{\mathbb{Z}_{q}\}$-uniform in $L^{1}$. 
\item $(\varphi_{w})_{\rats}:\underline{G}_{\rats}^{r}\rightarrow\underline{G}_{\rats}$
has a geometrically irreducible generic fiber, for every $\underline{G}\in\mathcal{G}$. 
\end{enumerate}
\end{enumerate}
\end{cor}

\subsubsection{Uniform properties of random walks induced by word measures}

Combining the discussion on random walks in the last two subsections,
together with the results on word maps obtained in this paper, we
can summarize our applications to $p$-adic probabilistic Waring type
problem.

The following is a direct consequence of Theorem \ref{thm:jet flatness of word maps}
and Item $(3)$ of Theorem \ref{thm:dictionary for random walks Linfty}. 
\begin{cor}
\label{cor:bounds on fibers for SLn}Let $w\in F_{r}$ be a word.
Then for any $n\in\nats$ there exists a finite set $S(n)$ of primes
such that for every $q\in\mathcal{P}_{S(n)}$ and every $g\in\mathrm{SL}_{n}(\ints_{q}/\frak{m}_{q}^{k})$,
\[
\tau_{w,\mathrm{SL}_{n},q,k}(g)\leq\left|\mathrm{SL}_{n}(\ints_{q}/\frak{m}_{q}^{k})\right|^{-\frac{1}{2\cdot10^{17}\ell(w)^{9}}}.
\]
\end{cor}

Theorem \ref{thm: (FRS) at (e,...,e)} allows us to provide $L^{\infty}$-bounds
for random walks on congruence subgroups. For any $\underline{G}\in\mathcal{G}$
and any prime $p$, write $\underline{G}^{1}(\Zp)$ for the kernel
of the quotient map $\underline{G}(\Zp)\rightarrow\underline{G}(\mathbb{F}_{p})$. 
\begin{cor}
\label{cor:result for G1(ZP)}Let $w\in F_{r}$ be a word. Then there
exists $C>0$, such that for every $\underline{G}\in\mathcal{G}$,
every prime $p>p_{0}(\underline{G})$, and every $t\geq C\ell(w)^{4}$,
the pushforward measure $(\varphi_{w}^{*t})_{*}(\mu_{\underline{G}^{1}(\Zp)}^{rt})$
on $\underline{G}^{1}(\Zp)$ has continuous density with respect to
the normalized Haar measure. 
\end{cor}

\begin{proof}
By Theorem \ref{thm: (FRS) at (e,...,e)}, there exists $U\subseteq\underline{G}^{rt}$
such that $\varphi_{w}^{*t}:U_{\rats}\rightarrow\underline{G}_{\rats}$
is (FRS) for $t=C\ell(w)^{4}$. Note that for $p\gg1$, $U(\mathbb{F}_{p})$
contains $(e,...,e)\in\underline{G}(\mathbb{F}_{p})^{rt}$ and $U$
is smooth over $\Zp$. By a generalization of Hensel's lemma, we have
$U(\Zp)\supseteq\underline{G}^{1}(\Zp)^{rt}$. By the analytic criterion
for the (FRS) property (Theorem \ref{thm:analytic criterion of the (FRS) property}),
$(\varphi_{w}^{*t})_{*}(\mu_{\underline{G}^{1}(\Zp)}^{rt})$ has continuous
density.
\end{proof}
\begin{cor}[{$L^{1}$-mixing time, generalization of \cite[Theorems 1 and 2]{LST19}}]
\label{cor:-L1 mixing time of words}Let $w_{1}\in F_{r_{1}}$ and
$w_{2}\in F_{r_{2}}$ be words. Then $w=w_{1}*w_{2}$ is almost $\{\ints_{q}\}_{q}$-uniform
in $L^{1}$. 
\end{cor}

\begin{proof}
By Theorem \ref{thm:-convolution of two word maps is generically absolutely irreducible},
$(\varphi_{w})_{\rats}$ has geometrically irreducible generic fiber
for every \textit{$\underline{G}\in\mathcal{G}$.} By Corollary \ref{cor:dictionary for word maps}
we are done. 
\end{proof}
\begin{cor}[Mixing time for the commutator map]
\label{cor:mixing times for the commutator maps}Let $w_{0}=xyx^{-1}y^{-1}\in F_{r}$
be the commutator word. Then 
\begin{enumerate}
\item $w_{0}$ is almost $\{\ints_{q}\}_{q}$-uniform in $L^{1}$. 
\item $w_{0}^{*4}$ is almost $\{\ints_{q}\}_{q}$-uniform in $L^{\infty}$. 
\end{enumerate}
\end{cor}

\begin{proof}
By Theorems \ref{thm:Commutator is (FRS) after 4 convolutions}, \ref{Thm:commutator is flat}
and by Corollary \ref{cor: word maps are (FAI) }, the map $\varphi_{w_{0}}^{*4}:\underline{G}^{8}\rightarrow\underline{G}$
is (FRS) and (FGI) for every $\underline{G}\in\mathcal{G}$. By \cite{GS09},
$w_{0}$ is \textit{$\{\mathbb{F}_{q}\}$-uniform in $L^{1}$. }By
Corollary \ref{cor:dictionary for word maps} we are done. 
\end{proof}
\bibliographystyle{alpha}
\bibliography{singularity_properties_of_word_maps_arxiv_v2}

\end{document}